\documentclass[a4paper]{article}
\usepackage[margin=1in]{geometry}

\usepackage{amsmath,amsthm,amssymb,bbm,hyperref,enumitem,graphicx,caption,subcaption,centernot}
\usepackage[ruled,vlined]{algorithm2e}
\usepackage{natbib}
\usepackage{xcolor}
\hypersetup{colorlinks,linkcolor={red!70!black},citecolor={blue!65!black},urlcolor={blue!70!black}}

\usepackage{tikz}

\newcommand{\N}{\mathbb{N}}

\newcommand{\R}{\mathbb{R}}

\newcommand{\E}{\mathbb{E}}
\renewcommand{\Pr}{\mathbb{P}}
\newcommand{\Ind}{\mathbbm{1}}
\newcommand{\floor}[1]{\lfloor #1 \rfloor}

\newcommand{\norm}[1]{\|{#1}\|}
\newcommand{\ipr}[2]{\langle #1, #2 \rangle}

\newcommand{\iid}{\overset{\mathrm{iid}}{\sim}}

\newcommand{\cvp}{\overset{p}{\to}}
\newcommand{\cvd}{\overset{d}{\to}}
\newcommand{\eqd}{\overset{d}{=}}

\DeclareMathOperator\Exp{Exp}

\DeclareMathOperator\Var{Var}
\DeclareMathOperator\Cov{Cov}
\DeclareMathOperator\tr{tr}
\DeclareMathOperator{\supp}{supp}

\DeclareMathOperator{\Int}{Int}
\DeclareMathOperator{\sgn}{sgn}
\DeclareMathOperator*{\argmin}{argmin}
\DeclareMathOperator*{\argmax}{argmax}

\newtheorem{theorem}{Theorem}
\newtheorem{lemma}[theorem]{Lemma}
\newtheorem{proposition}[theorem]{Proposition}
\newtheorem{corollary}[theorem]{Corollary}

\theoremstyle{definition}
\newtheorem{example}[theorem]{Example}
\newtheorem{remark}[theorem]{Remark}
\newtheorem*{remark*}{Remark}

\title{Optimal convex $M$-estimation via score matching}
\date{\today}
\author{Oliver Y. Feng$^{*,\ddagger}$, Yu-Chun Kao$^\dagger$, Min Xu$^\dagger$ and Richard J. Samworth$^\ddagger$ \\ \\ $^*$Department of Mathematical Sciences, University of Bath \\ $^\dagger$Department of Statistics, Rutgers University \\ $^\ddagger$Statistical Laboratory, University of Cambridge}

\begin{document}

\maketitle

\begin{abstract}
In the context of linear regression, we construct a data-driven convex loss function with respect to which empirical risk minimisation yields optimal asymptotic variance in the downstream estimation of the regression coefficients.  At the population level, the negative derivative of the optimal convex loss is the best decreasing approximation of the derivative of the log-density of the noise distribution.  This motivates a fitting process via a nonparametric extension of score matching, corresponding to a log-concave projection of the noise distribution with respect to the Fisher divergence.  At the sample level, our semiparametric estimator is computationally efficient, and we prove that it attains the minimal asymptotic covariance among all convex $M$-estimators. As an example of a non-log-concave setting, the optimal convex loss function for Cauchy errors is Huber-like, and our procedure yields asymptotic efficiency greater than $0.87$ relative to the 
% parametric
maximum likelihood estimator of the regression coefficients that uses oracle knowledge of this error distribution.  In this sense, we provide robustness and facilitate computation without sacrificing much statistical 
efficiency. Numerical experiments using our accompanying \texttt{R} package \texttt{asm} confirm the practical merits of our proposal.
\end{abstract}

% natbib with citations listing at most 3 (rather than 2) authors?
\section{Introduction}

In linear models, the Gauss--Markov theorem is the primary justification for the use of ordinary least squares (OLS) in settings where the Gaussianity of our error distribution may be in doubt.  It states that, provided the errors have a finite second moment, OLS attains the minimal covariance among all linear unbiased estimators; recent papers on this topic include~\citet{hansen2022modern},~\citet{potscher2024comment} and~\citet{lei2022estimators}. On the other hand, it is now understood that biased, non-linear estimators can achieve lower mean squared error than OLS~\citep{stein1956inadmissibility,hoerl1970ridge}, especially when the noise distribution is appreciably non-Gaussian~\citep{zou2008composite,dumbgen2011approximation}. However, it remains unclear how best to fit linear models in a computationally efficient and adaptive fashion, i.e.~without knowledge of the error distribution. %indeed, our empirical results (\green{to be added}) demonstrate that considerable improvements in estimation accuracy are possible when the noise distribution is appreciably non-Gaussian.

Consider a linear model where $Y_i = X_i^\top\beta_0 + \varepsilon_i$ for $i = 1,\dotsc,n$. Recall that an \textit{$M$-estimator} of $\beta_0 \in \R^d$ based on a loss function $\ell \colon \R \to \R$ is defined as an empirical risk minimiser
\begin{equation}
\label{eq:linreg-M-est}
\hat{\beta} \in \argmin_{\beta \in \R^d} \frac{1}{n}\sum_{i=1}^n \ell(Y_i - X_i^\top \beta), 
\end{equation}
provided that this exists. If $\ell$ is differentiable on $\R$ with negative derivative $\psi = -\ell'$, then $\hat{\beta} \equiv \hat{\beta}_\psi$ solves the corresponding estimating equations 
\begin{equation}
\label{eq:linreg-Z-est}
\frac{1}{n}\sum_{i=1}^n X_i\psi(Y_i - X_i^\top\hat{\beta}_\psi) = 0    
\end{equation}
and is referred to as a \textit{$Z$-estimator}. We study a random design setting in which $(X_1,Y_1),\dotsc,(X_n,Y_n)$ are independent and identically distributed, with $X_1,\dotsc,X_n$ being $\R^d$-valued covariates that are independent of real-valued errors $\varepsilon_1,\dotsc,\varepsilon_n$ with density $p_0$. Suppose further that $\E\{X_1\psi(\varepsilon_1)\} = 0$. This means that~$\hat{\beta}_\psi$ is \emph{Fisher consistent} in the sense that the population analogue of~\eqref{eq:linreg-Z-est} is satisfied by the true parameter $\beta_0$, i.e.~$\E\{X_1\psi(Y_1 - X_1^\top\beta_0)\} = 0$. Then under suitable regularity conditions, including $\psi$ being differentiable and $\E(X_1 X_1^\top) \in \R^{d \times d}$ being invertible, we have
\begin{equation}
\label{eq:M-estimator-asymp}
\sqrt{n}(\hat{\beta}_\psi - \beta_0) \cvd N_d\bigl(0, V_{p_0}(\psi) \cdot \{\E(X_1 X_1^\top)\}^{-1}\bigr) \quad \text{as }n \to \infty, \quad\text{where }V_{p_0}(\psi) := \frac{\E\psi^2(\varepsilon_1)}{\{\E\psi'(\varepsilon_1)\}^2}
\end{equation}
\citep[e.g.][Theorems~5.21,~5.23 and~5.41]{vdV1998asymptotic}. Since the covariates and errors are assumed to be independent, they contribute separately to the limiting covariance above (a special case of the `sandwich' formula~\citep{huber1967behavior,young2023sandwich}): the matrix $\{\E(X_1 X_1^\top)\}^{-1}$ depends only on the covariate distribution, whereas the scalar $V_{p_0}(\psi)$ depends on the loss function $\ell$ (through $\psi = -\ell'$) and on the error distribution.

If the errors $\varepsilon_1,\varepsilon_2,\dotsc$ have a known absolutely continuous density $p_0$ on $\R$, then we can define the maximum likelihood estimator $\hat{\beta}^{\mathrm{MLE}}$ by taking $\ell = -\log p_0$ in~\eqref{eq:linreg-M-est}. In this case, $\psi = -\ell'$ is the \emph{score function (for location)}\footnote{The score is usually defined as a function of a parameter $\theta \in \R$ as the derivative of the log-likelihood; the link with our terminology comes from considering the location model $\{p_0(\cdot + \theta):\theta \in \mathbb{R}\}$, and evaluating the score at the origin.} $\psi_0 := (p_0'/p_0)\Ind_{\{p_0 > 0\}}$.  Under appropriate regularity conditions~\citep[e.g.][Theorem~5.39]{vdV1998asymptotic}, including that the \emph{Fisher information (for location)} $i(p_0) := \int_\R \psi_0^2\,p_0 = \int_{\{p_0 > 0\}}(p_0')^2/p_0$ is finite, we have
\begin{equation}
\label{eq:betahat-MLE-asymp}
\sqrt{n}\,(\hat{\beta}^{\mathrm{MLE}} - \beta_0) \cvd N_d\biggl(0, \frac{\{\E(X_1 X_1^\top)\}^{-1}}{i(p_0)}\biggr)
\end{equation}
as $n \to \infty$. The limiting covariance matrix $\{\E(X_1 X_1^\top)\}^{-1}/i(p_0)$ constitutes the usual efficiency lower bound~\citep[Chapter~8]{vdV1998asymptotic}. In fact, it can be seen directly that $1/i(p_0)$ is the smallest possible value of the \textit{asymptotic variance factor} $V_{p_0}(\psi)$ in the limiting covariance of $\sqrt{n}(\hat{\beta}_\psi - \beta_0)$ in~\eqref{eq:M-estimator-asymp}. Indeed, by the Cauchy--Schwarz inequality,
\begin{equation}
\label{eq:Vp0-Fisher}
V_{p_0}(\psi) = \frac{\int_\R\psi^2 p_0}{\bigl(\int_\R \psi' p_0\bigr)^2} = \frac{\int_\R\psi^2 p_0}{\bigl(\int_\R \psi p_0'\bigr)^2} \geq \frac{1}{\int_{\{p_0 > 0\}}(p_0')^2/p_0} = \frac{1}{i(p_0)} \in (0,\infty)
\end{equation}
whenever the integration by parts in the second step is justified, and equality holds if and only if there exists $\lambda \neq 0$ such that $\psi(\varepsilon_1) = \lambda\psi_0(\varepsilon_1)$ almost surely. This leads to an equivalent variational definition of the Fisher information; see~\citet[][Theorem~4.2]{huber2009robust}, which we restate as Proposition~\ref{prop:fisher-inf-variational} in Section~\ref{sec:auxiliary}. Thus, when~\eqref{eq:betahat-MLE-asymp} holds, $\hat{\beta}^{\mathrm{MLE}}$ has minimal asymptotic covariance among all $Z$-estimators~$\hat{\beta}_\psi$ for which~\eqref{eq:M-estimator-asymp} is valid, with the score function $\psi_0$ being the optimal choice of $\psi$.

Our goal in this work is to choose $\psi$ in a data-driven manner, such that the corresponding loss function~$\ell$ in~\eqref{eq:linreg-M-est} is convex, and such that the scale factor $V_{p_0}(\psi)$ in the asymptotic covariance~\eqref{eq:M-estimator-asymp} of the downstream estimator of $\beta_0$ is minimised.  Convexity is a particularly convenient property for a loss function, since for the purpose of $M$-estimation, it leads to more tractable theory and computation.  Indeed, the empirical risk in~\eqref{eq:linreg-M-est} becomes convex in $\beta$, so its local minimisers are global minimisers. In particular, when $\ell$ is also differentiable, $\hat{\beta}_\psi$ is a $Z$-estimator satisfying~\eqref{eq:linreg-Z-est} if and only if it is an $M$-estimator satisfying~\eqref{eq:linreg-M-est}. The existence, uniqueness and $\sqrt{n}$-consistency of $\hat{\beta}_\psi$ are then guaranteed under milder conditions on $\ell$ than for generic loss functions~\citep{yohai1979asymptotic,maronna1981asymptotic,portnoy1985asymptotic,mammen1989asymptotics,arcones1998asymptotic,he2000parameters}; see Proposition~\ref{prop:cvx-M-est-asymp}. Furthermore, an important practical advantage is that we can compute $\hat{\beta}_\psi$ efficiently using convex optimisation algorithms with guaranteed convergence~\citep[Chapter~9]{boyd2004convex}.

In view of the discussion above, our first main contribution in Section~\ref{sec:antitonic-proj} is to determine the optimal population-level convex loss function in the sense described in the previous paragraph.  For a uniformly continuous error density $p_0$, this amounts to finding
\begin{equation}
\label{eq:Vp0-minimiser}
\psi_0^* \in \argmin_{\psi \in \Psi_\downarrow(p_0)} V_{p_0}(\psi),
\end{equation}
where $\Psi_\downarrow(p_0)$ denotes the set of decreasing, right-continuous functions $\psi$ satisfying $\int_{\mathbb{R}} \psi^2 p_0 < \infty$. We will actually define the ratio $V_{p_0}(\psi)$ in a slightly more general way than in~\eqref{eq:Vp0-Fisher} to allow us to handle non-differentiable functions~$\psi$.  This turns out to be convenient because, for instance, the robust Huber loss $\ell_K$ given by
\begin{equation}
\label{eq:huber-loss}
\ell_K(z) :=
\begin{cases}
z^2/2 \quad &\text{if }|z| \leq K\\
K|z| - K^2/2 \quad &\text{if }|z| > K
\end{cases}
\end{equation}
for $K \in (0,\infty)$ has a non-differentiable negative derivative $\psi_K := -\ell_K'$ satisfying $\psi_K(z) = (-K) \vee (-z) \wedge K$ for $z \in \R$. 

In Section~\ref{subsec:antitonic-score-proj}, we show that minimising $V_{p_0}(\cdot)$ over $\Psi_\downarrow(p_0)$ is equivalent to minimising the \emph{score matching} objective
\begin{equation}
\label{eq:score-matching-objective}
D_{p_0}(\psi) := \mathbb{E}\bigl\{\psi^2(\varepsilon_1) + 2\psi'(\varepsilon_1)\bigr\},
\end{equation}
over $\psi \in \Psi_\downarrow(p_0)$, provided that we take appropriate care in defining this expression when $\psi$ is not absolutely continuous. This observation allows us to obtain an explicit characterisation of the solution to the optimisation problem as a `projected' score function $\psi_0^*$ in terms of $p_0$ and its distribution function $F_0$. Indeed, to obtain $\psi_0^*$ at $z \in \R$, we can first consider $p_0 \circ F_0^{-1}$ (whose domain is $[0,1]$), then compute the right derivative of its least concave majorant, before finally applying the resulting function to $F_0(z)$.  The negative antiderivative~$\ell_0^*$ of~$\psi_0^*$ is then the optimal convex loss function we seek.  An important property is that $\E\psi_0^*(Y_1 - X_1^\top \beta_0) = 0$, which ensures that $\ell_0^*$ correctly identifies the estimand $\beta_0$ on the population level; equivalently,~$\hat{\beta}_{\psi_0^*}$ is Fisher consistent.

\begin{figure}[htb!]
\centering
\includegraphics[width=0.32\textwidth,trim={4.65cm 9.55cm 4.45cm 10cm},clip]{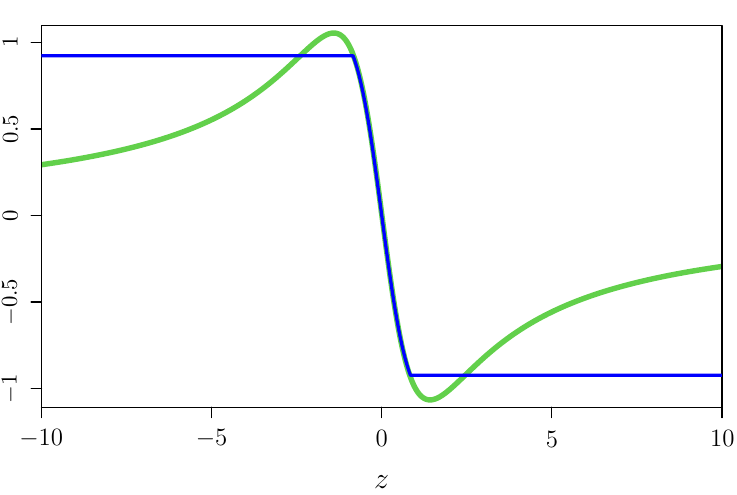}
\hfill
\includegraphics[width=0.32\textwidth,trim={4.65cm 9.55cm 4.45cm 10cm},clip]{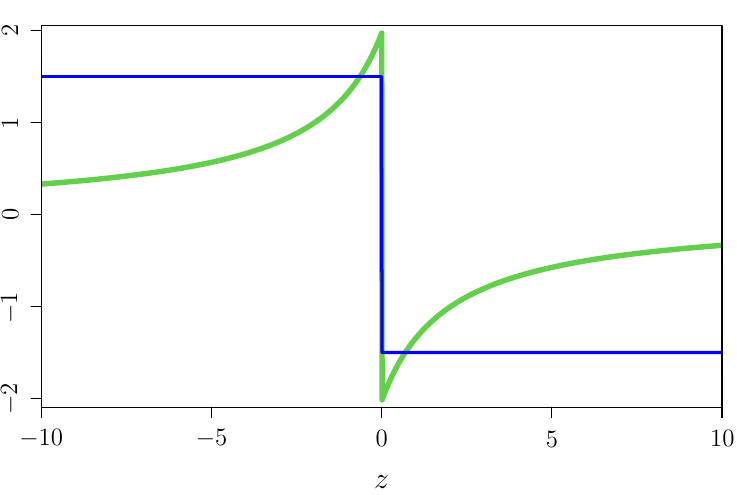}
\hfill
\includegraphics[width=0.32\textwidth,trim={4.65cm 9.55cm 4.45cm 10cm},clip]{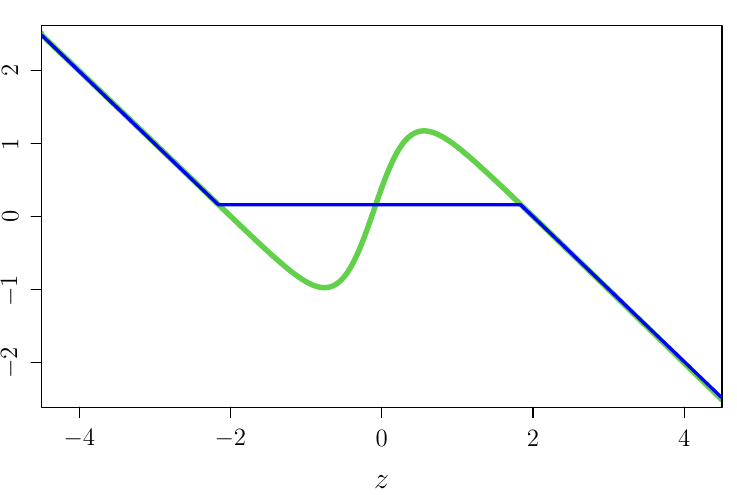}

\includegraphics[width=0.32\textwidth,trim={4.65cm 9.55cm 4.45cm 10cm},clip]{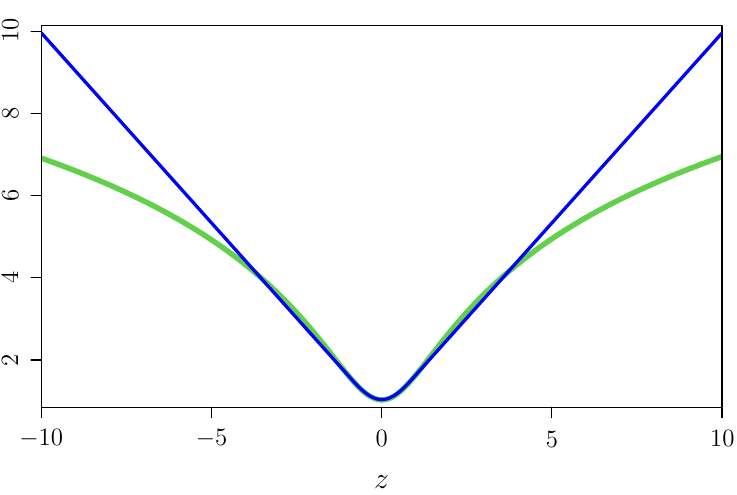}
\hfill
\includegraphics[width=0.32\textwidth,trim={4.65cm 9.55cm 4.45cm 10cm},clip]{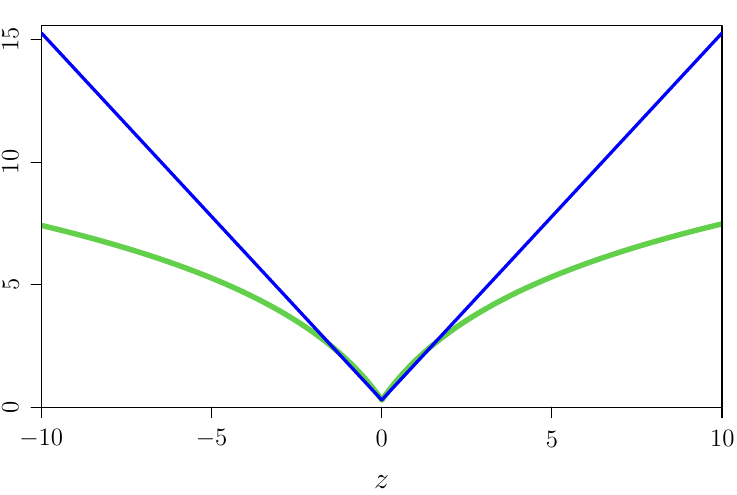}
\hfill
\includegraphics[width=0.32\textwidth,trim={4.65cm 9.55cm 4.45cm 10cm},clip]{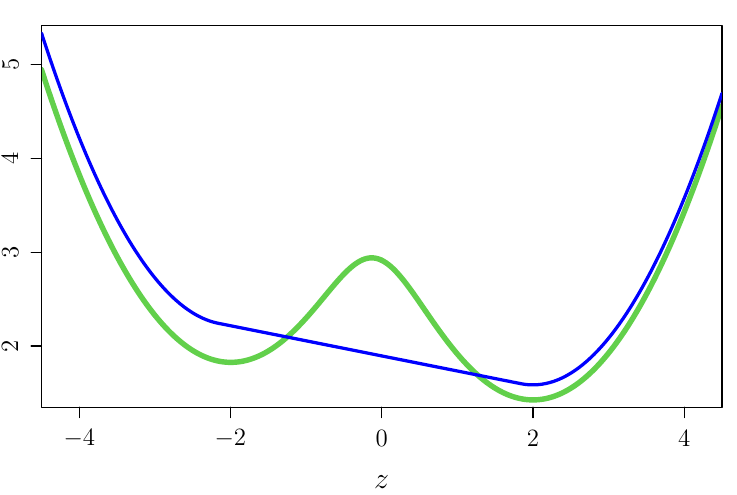}

\vspace{-0.3cm}
\caption{\textit{Top row}: Plots of the score function $\psi_0$ (\textit{green}) and projected score function $\psi_0^*$ (\textit{blue}); \textit{Bottom row}: their respective negative antiderivatives, namely the negative log-density $-\log p_0$ (\textit{green}) and optimal convex loss function $\ell_0^*$ (\textit{blue}), for each of the following non-log-concave distributions \textit{(from left to right)}: \textit{(a)} Student's $t_2$; \textit{(b)} symmetrised Pareto~\eqref{eq:pareto-sym} with $\sigma = 2$ and $\alpha = 3$; \textit{(c)} Gaussian mixture $0.4 N(-2,1) + 0.6 N(2,1)$.}
\label{fig:intro-plots}
\end{figure}

Note that $\hat{\beta}^{\mathrm{MLE}}$ is a convex $M$-estimator if and only if $\ell = -\log p_0$ is convex, i.e.\ $p_0$ is log-concave, in which case $\psi_0^* = \psi_0$ by~\eqref{eq:Vp0-Fisher}. We will be especially interested in error densities $p_0$ that are not log-concave, for which the efficiency lower bound in~\eqref{eq:Vp0-Fisher} cannot be achieved by a convex $M$-estimator corresponding to a decreasing function $\psi$.  We interpret the minimum ratio $V_{p_0}(\psi_0^*)$ as an analogue of the inverse Fisher information, serving as the crucial part of the efficiency lower bound for convex $M$-estimators.  To reinforce the link with score matching, we will see in Section~\ref{subsec:fisher-divergence-proj} that the density proportional to $e^{-\ell_0^*}$ is the best log-concave approximation to $p_0$ with respect to the Fisher divergence defined formally in~\eqref{eq:fisher-div} below.  This is typically different from the well-studied log-concave projection with respect to Kullback--Leibler divergence, and indeed the latter may yield considerably suboptimal covariance for the resulting convex $M$-estimator; see Proposition~\ref{prop:Vp0-MLE}. In concrete examples where $p_0$ has heavy tails (e.g.~a Cauchy density) or is multimodal (e.g.~a mixture density), we compute closed-form expressions for the projected score function and the optimal convex loss function in Section~\ref{subsec:examples}. In particular, $\ell_0^*$ turns out to be a robust Huber-like loss function in the Cauchy case. More generally, when the errors are heavy-tailed in the sense that their (two-sided) hazard function is bounded, Lemma~\ref{lem:hazard} shows that the projected score function $\psi_0^*$ is bounded, in which case the corresponding convex loss $\ell_0^*$ grows at most linearly in the tails and hence is robust to outliers; see the discussion immediately preceding Lemma~\ref{lem:hazard}. A major advantage of our framework over the use of the Huber loss is that it does not require the choice of a transition point $K$ (see~\eqref{eq:huber-loss} above) between quadratic and linear regimes (which in a regression context amounts to a choice of scale for the error distribution).  In fact, our antitonic\footnote{Antitonic means decreasing, in contrast to isotonic (increasing)~\citep[Section~2.1]{groeneboom14nonparametric}.} score projection, and hence the Fisher divergence projection, is affine equivariant (Remark~\ref{rem:affine-equivariance}), which reflects the fact that we optimise $V_{p_0}(\cdot)$ in~\eqref{eq:Vp0-minimiser} over a class $\Psi_\downarrow(p_0)$ that is closed under multiplication by non-negative scalars.

In Section~\ref{sec:regression}, we turn our attention to a linear regression setting where the error density $p_0$ is unknown. 
% To ensure that adaptive estimation of $\beta_0$ is possible in the sense of~\citet{bickel1982adaptive}, we assume either that $p_0$ is symmetric (Section~\ref{subsec:linreg-sym}), or that the model contains an explicit intercept term and the errors are appropriately centred (Section~\ref{subsec:linreg-intercept}).
We aim to construct a semiparametric $M$-estimator of~$\beta_0$ that achieves minimal covariance among all convex $M$-estimators, but since the optimal loss function $\ell_0^*$ is unknown, we seek to estimate $\beta_0$ and $\psi_0^*$ simultaneously. Our alternating optimisation procedure starts with a non-adaptive initialiser $\bar{\beta}_n$ and computes a kernel density estimate of the error distribution based on the residuals.  We can then apply the linear-time Pool Adjacent Violators Algorithm (PAVA) to obtain the projected score function of the density estimate, before minimising its negative antiderivative using Newton-type optimisation techniques to yield an updated estimator.  This process could then be iterated to convergence, but if we initialise with a $\sqrt{n}$-consistent pilot estimator $\bar{\beta}_n$, then one iteration of the alternating algorithm above suffices for our theoretical guarantees, and moreover it ensures that the procedure is computationally efficient. When~$p_0$ is symmetric, we prove that a three-fold cross-fitting version of our algorithm (with the different steps computed on different folds) yields an estimator $\hat{\beta}_n$ that is $\sqrt{n}$-consistent and asymptotically normal, with limiting covariance attaining our efficiency lower bound for convex $M$-estimators; see Theorem~\ref{thm:linreg-score-sym} in Section~\ref{subsec:linreg-sym}. We develop analogous methodology and theory for the setting where an explicit intercept is present in the linear model, and where the errors are appropriately centred though not necessarily symmetric; see Theorem~\ref{thm:linreg-score-intercept} in Section~\ref{subsec:linreg-intercept}. Consistent estimation of the limiting covariance matrices is straightforward using our nonparametric score matching procedure, so combining this with our asymptotic distributional results for $\hat{\beta}_n$, we can then perform inference for $\beta_0$ (Section~\ref{subsec:linreg-inference}).    

Section~\ref{sec:experiments} is devoted to a numerical study of the empirical performance and computational efficiency of our antitonic score matching estimator, which is implemented in the \texttt{R} package \texttt{asm}~\citep{kao24asm}, and whose output is designed to mimic that of the standard existing \texttt{lm} function in several aspects so as to appear familiar to practitioners.  These corroborate our theoretical findings: our proposed approach achieves (sometimes dramatically) smaller estimation error  compared with alternatives such as OLS, the least absolute deviation (LAD) estimator, a semiparametric one-step estimator, and a semiparametric $M$-estimator based on the log-concave MLE of the noise distribution. Moreover, the corresponding confidence sets for $\beta_0$ are smaller, while retaining nominal coverage. Finally, we perform a runtime analysis to show that the improved statistical performance comes without sacrificing computational scalability. 

The proofs of all results in Sections~\ref{sec:antitonic-proj} and~\ref{sec:regression} are given in the appendix in Sections~\ref{sec:antitonic-proj-proofs} and~\ref{subsec:regression-proofs} respectively.  The appendix also contains additional examples for Section~\ref{sec:antitonic-proj} (Section~\ref{subsec:appendix-examples}) and auxiliary results (Sections~\ref{sec:auxiliary} and~\ref{sec:isoproj}).

\subsection{Related work}

\textit{Score matching}~\citep{hyvarinen05score,lyu2012interpretation} is an estimation method designed for statistical models where the likelihood is only known up to a normalisation constant (e.g.~a partition function) that may be infeasible to compute; see the recent tutorial by~\citet{song2021train} on `energy-based' models. Instead of maximising an approximation to the likelihood, score matching circumvents this issue altogether by estimating the derivative of a log-density, i.e.~the score function. More precisely, given a differentiable density $p_0$ on $\R^d$ with score function $\psi_0 := (\nabla p_0 / p_0)\Ind_{\{p_0 > 0\}}$, the population version of the procedure aims to minimise
\begin{equation}
\label{eq:score-matching-1}
\E\bigl(\norm{\psi(\varepsilon) - \psi_0(\varepsilon)}^2\bigr)
\end{equation}
over a suitable class $\Psi$ of differentiable functions $\psi \equiv (\psi_1,\dotsc,\psi_d) \colon \R^d \to \R^d$, where $\varepsilon \sim p_0$.~\citet{hyvarinen05score} used integration by parts to show that it is equivalent to minimise 
\begin{equation}
\label{eq:score-matching-2}
D_{p_0}(\psi) := \E\bigl\{\norm{\psi(\varepsilon)}^2 + 2(\nabla \cdot \psi)(\varepsilon)\bigr\}
\end{equation}
over $\psi \in \Psi$, where $\nabla \cdot \psi := \sum_{j=1}^d \partial \psi_j/\partial x_j$. The score matching estimator based on data $\varepsilon_1,\dotsc,\varepsilon_n$ in~$\R^d$ is then defined as a minimiser of the empirical analogue $\hat{D}_n(\psi) := n^{-1}\sum_{i=1}^n\{\norm{\psi(\varepsilon_i)}^2 + 2(\nabla \cdot \psi)(\varepsilon_i)\}$ over $\psi \in \Psi$; see also~\citet{cox1985penalty}.  Such estimators are important in the context of Langevin Monte Carlo~\citep{parisi1981correlation,roberts1996exponential,betancourt2017geometric,cheng2018underdamped} and diffusion models~\citep{li2024towards}. The appearance of the score function in the (reverse-time) stochastic differential equations governing the Langevin and diffusion model dynamics can be related to Tweedie's formula, which underpins empirical Bayes denoising~\citep{efron2011tweedie,derenski23empirical}.

Likelihood maximisation corresponds to distributional approximation with respect to the Kullback--Leibler divergence.  On the other hand, score matching seeks to minimise the \textit{Fisher divergence}~\citep[Section~1.3]{johnson2004information} from a class of densities to the target $p_0$, in view of the equivalence between the optimisation objectives~\eqref{eq:score-matching-1} and~\eqref{eq:score-matching-2}; see~\eqref{eq:Dp0-L2P0} below.~\citet{sriperumbudur2017density} studied infinite-dimensional exponential families indexed by reproducing kernel Hilbert spaces, and proposed and analysed a density estimator that minimises a penalised empirical Fisher divergence. \citet{koehler2022statistical} used isoperimetric inequalities to investigate the statistical efficiency of score matching relative to maximum likelihood, thereby quantifying the effect of eliminating normalisation factors. \citet{lyu2012interpretation}
% ~\citet[Proposition~B.1]{sriperumbudur2017density}
observed that Fisher divergence and Kullback--Leibler divergence are related by an analogue of de Bruijn's identity~(\citealp[Appendix~C]{johnson2004information};~\citealp[Section~17.7]{cover2006elements}), which links Fisher information and Shannon entropy. From an information-theoretic perspective,~\citet{johnson2004fisher} proved central limit theorems that establish convergence in Fisher divergence to a limiting Gaussian distribution.~\citet{ley2013stein} extended Stein's method to derive information inequalities that bound a variety of integral probability distances in terms of the Fisher divergence.

Score matching has been generalised in different directions and applied to 
% including graphical modelling
a variety of statistical problems~\citep[e.g.][]{hyvarinen2007extensions,vincent2011connection,lyu2012interpretation,mardia2016score,song2020sliced,yu2020simultaneous,yu2022generalized,lederer2023extremes,benton2024denoising,ghosh2025stein}, where it exhibits excellent empirical performance while being computationally superior to full likelihood approaches. In particular, score-based algorithms for generative modelling, via Langevin dynamics~\citep{song19generative} and diffusion models~\citep{song21score}, have achieved remarkable success in machine learning tasks such as the reconstruction, inpainting and artificial generation of images; see e.g.~\citet{jolicoeur2020adversarial},~\citet{bortoli2022riemannian} and many other references therein. In these applications, score matching is applied to a class of functions parametrised by the weights of a deep neural network. On the other hand, different statistical considerations lead us to develop a nonparametric extension of score matching in Section~\ref{sec:antitonic-proj}, which we use to construct data-driven convex loss functions for efficient semiparametric estimation. We see that it is by minimising the Fisher divergence instead of the Kullback--Leibler divergence to the error distribution that one obtains a convex $M$-estimator with minimal asymptotic variance.

The framework in Section~\ref{subsec:linreg-sym} includes as a special case the classical location model in which we observe $Y_i = \theta_0 + \varepsilon_i$ for $i = 1,\dotsc,n$, where $\theta_0 \in \R$ is the parameter of interest and $\varepsilon_1,\dotsc,\varepsilon_n$ are independent errors with an unknown density $p_0$ that is symmetric about 0. Starting from the seminal paper of~\citet{stein1956efficient}, a series of works~\citep[e.g.][]{van1970efficiency,stone1975adaptive,beran1978efficient,bickel1982adaptive,schick1986asymptotically,faraway1992smoothing,dalalyan2006penalized,vdV2021stein,gupta2023finite} showed that adaptive, asymptotically efficient estimators of $\theta_0$ can be constructed; see also~\citet{doss2019univariate} and~\citet{laha2021adaptive} for approaches based on the further assumption that $p_0$ is log-concave. Many of these traditional semiparametric procedures have drawbacks that limit their practical utility. In particular, the estimated likelihood may have multiple local optima and it may be difficult to guarantee convergence of an optimisation algorithm to a global maximum~\citep[Example~5.50]{vdV1998asymptotic}. This is one of the reasons why prior works often study a one-step estimator resulting from a single iteration of Newton’s method~\citep{bickel1975one,jin1990empirical,mammen1997optimal,laha2021adaptive}, rather than full likelihood maximisation, though finite-sample performance may remain poor and sensitive to tuning (see Section~\ref{sec:experiments}).  By contrast, our focus is not on classical semiparametric adaptive efficiency per se; instead, we directly study the theoretical properties of a minimiser of the empirical risk with respect to an estimated loss function, whose convexity ensures that the estimator can be computed efficiently by iterating gradient descent or Newton's method to convergence.

Recently,~\citet{kao24choosing} constructed a location $M$-estimator that can adaptively attain rates of convergence faster than $n^{-1/2}$ when the symmetric error density is compactly supported and suitably irregular (e.g.\ discontinuous at the boundary of its support). They considered $\ell^q$-location estimators $\hat{\theta}_q := \argmin_{\theta \in \R} \sum_{i=1}^n |Y_i - \theta|^q$ based on univariate observations $Y_1,\dotsc,Y_n$, and used Lepski's method to select an exponent $\hat{q} \in [2,\infty)$ that minimises a proxy for the asymptotic variance of $\hat{\theta}_q$. The resulting estimator $\hat{\theta}_{\hat{q}}$ is shown to be minimax optimal up to poly-logarithmic factors, and the procedure is extended to linear regression models with unknown symmetric errors. By comparison with~\citet{kao24choosing}, we study `regular' regression models where the Fisher information is finite and minimax rates faster than $n^{-1/2}$ are impossible to achieve.  We aim to minimise the asymptotic variance as an end in itself, over the entire nonparametric class of convex loss functions rather than $\ell^q$-loss functions specifically.

Our work has connections with robust statistics, which deals with heavy-tailed noise distributions and data that may be contaminated by random or adversarial outliers. As mentioned previously, robust loss functions are designed to be tolerant to such data corruption; examples include the Huber loss~\eqref{eq:huber-loss}, a two-parameter family of loss functions considered by~\citet{barron2019general}, and the antiderivatives of Catoni's influence functions~\citep{catoni2012challenging}.  The Huber loss functions $\ell_K$ originally arose because they have the minimax asymptotic variance property that for every $\epsilon \in (0,1)$, there exists a unique $K \equiv K_\epsilon > 0$ such that
\[
\psi_K \equiv -\ell_K' = \argmin_{\psi \in \Psi} \sup_{P \in \mathcal{P}_\epsilon^{\mathrm{sym}}(\Phi)} \underbrace{\frac{\int_\R \psi^2\,dP}{\int_\R \psi'\,dP}}_{=:\,V_P(\psi)},
\]
where $\Psi$ consists of all `sufficiently regular' $\psi \colon \R \to \R$, and the symmetric \textit{$\epsilon$-contamination neighbourhood} $\mathcal{P}_\epsilon^{\mathrm{sym}}(\Phi)$ contains all univariate distributions of the form $P = (1 - \epsilon)N(0,1) + \epsilon Q$ for some symmetric distribution $Q$. The pioneering paper of~\citet{huber1964robust} also developed variational theory for minimising $\sup_{P \in \mathcal{P}}V_P(\cdot)$ more generally when $\mathcal{P}$ is a convex class of distributions, such as an $\epsilon$-contamination or a Kolmogorov neighbourhood of a symmetric log-concave density~\citep[Section~4.5]{huber2009robust}. See~\citet{donoho2015variance} for a high-dimensional extension of this line of work. An alternative to the Huber loss that seeks robustness without serious efficiency loss relative to OLS is the composite quantile regression (CQR) estimator \citep{zou2008composite,yang2024multiple}; in fact, our approach is always at least as efficient as CQR (see Lemma~\ref{lem:CQR-suboptimal}).  Other recent papers on robust convex $M$-estimation include~\citet{chinot2020robust} and~\citet{brunel2023geodesically}; see also the notes on robust statistical learning theory by~\citet{lerasle2019selected}.

More closely related to our optimisation problem~\eqref{eq:Vp0-minimiser} is the work of~\citet{hampel1974influence} on \textit{optimal $B$-robust estimators}, which have minimal asymptotic variance subject to an upper bound on the \textit{gross error sensitivity}~\citep[Section~2.4]{hampel2011robust}. 
% See also the very helpful table on page 176 in Section~2.7 that relates Huber's and Hampel's approaches
In our linear regression setting with $\varepsilon_1 \sim p_0$, this amounts to 
\begin{equation}
\label{eq:Vp0-hampel}
\text{minimising }V_{p_0}(\psi) \text{ over all `regular' $\psi$ such that }\int_\R \psi\,p_0 = 0 \text{ and }\sup_{z \in \R}|\psi(z)| \leq b
\end{equation}
for some suitable $b > 0$~(\citealp[p.~121 and Section~2.5d]{hampel2011robust}). 
% Not the same as \citet[Example~5.29]{vdV1998asymptotic}
In particular, when $p_0$ is a standard Gaussian density, $\psi_K$ is again optimal for some $K \equiv K_b > 0$ that depends non-linearly on $b$. By contrast with~\eqref{eq:Vp0-minimiser} however, the Fisher consistency condition $\E\psi(\varepsilon_1) = 0$ must be explicitly included as a constraint in~\eqref{eq:Vp0-hampel}, and moreover the $L^\infty$ bound on $\psi$ means that the set of feasible~$\psi$ is not closed under non-negative scalar multiplication. Consequently, the resulting optimal location $M$-estimators are generally not scale invariant~\citep[p.~105]{hampel2011robust}. In robust regression, adaptive selection of scale parameters is a non-trivial problem~(e.g.~\citealp[Section~5.4]{vdV1998asymptotic};~\citealp[Section~7.7]{huber2009robust};~\citealp{loh2021scale}); see also Figure~\ref{fig:cauchy-huber} below. 

Finally, we mention a different line of work on the performance of linear regression $M$-estimators in a proportional asymptotic regime where $n/d \to \kappa \in (1,\infty)$ and the covariates are Gaussian. \citet{bean2013optimal} derived the unpenalised convex $M$-estimator with minimal expected out-of-sample prediction error when the errors are log-concave.  Here, the optimisation objective is no longer $V_{p_0}(\psi)$ but instead the solution to a pair of non-linear equations involving the proximal operator of the convex loss function~\citep{elkaroui2013robust,elkaroui2018impact}.  For general error distributions with finite variance, \cite{donoho2016high} established exact asymptotics for convex $M$-estimators by means of an approximate message passing algorithm. Under prior structural information about the entries of~$\beta_0$, \citet{celentano2022fundamental} obtained precise characterisations of the asymptotic $\ell_2$-estimation error of convex-regularised least squared estimators, Bayes-optimal approximate message passing and the Bayes risk, quantifying the gaps between computational feasible and statistically optimal estimators.

\subsection{Notation}
\label{sec:notation}
Throughout this paper, we will adopt the convention $0/0 := 0$ and write $[n] := \{1,\dotsc,n\}$ for $n \in \N$. For a function $f \colon \R \to \R$, let $\norm{f}_\infty := \sup_{z \in \R}|f(z)|$. Recall that $f$ is \emph{symmetric} (i.e.~\textit{even}) if $f(z) = f(-z)$ for all $z \in \R$, and \emph{antisymmetric} (i.e.~\textit{odd}) if $f(z) = -f(-z)$ for all $z \in \R$. For an open set $U \subseteq \R$, we say that $f \colon U \to \R$ is \emph{locally absolutely continuous} on $U$ if it is absolutely continuous on every compact interval $I \subseteq U$, or equivalently if there exists a measurable function $g \colon U \to \R$ such that for every compact subinterval $I \subseteq U$, we have $\int_I |g| < \infty$ and $f(z_2) = f(z_1) + \int_{z_1}^{z_2} g$ for all $z_1,z_2 \in I$. In this case, $f$ is differentiable Lebesgue almost everywhere on $U$, with $f' = g$ almost everywhere.

Given a Borel probability measure $P$ on $\R$, we write $L^2(P)$ for the set of all Lebesgue measurable functions $f$ on $\R$ such that $\norm{f}_{L^2(P)} := \bigl(\int_\R f^2\,dP\bigr)^{1/2} < \infty$. Denote by $\ipr{f}{g}_{L^2(P)} := \int_\R fg\,dP$ the $L^2(P)$-inner product of $f,g \in L^2(P)$.
% For measures $\mu,\nu$ on a general measurable space $(\mathcal{X},\mathcal{A})$, we say that $\mu$ is \textit{absolutely continuous} with respect to~$\nu$, and write $\mu \ll \nu$, if $\mu(A) = 0$ whenever $\nu(A) = 0$ for $A \in \mathcal{A}$. The notation $\mu \nll \nu$ indicates that $\mu$ is not absolutely continuous with respect to $\nu$.

For a function $F \colon [0,1] \to \R$, we write $\hat{F}$ for its least concave majorant on $[0,1]$. Denote by $F^{(\mathrm{L})}(u)$ and $F^{(\mathrm{R})}(u)$ respectively the left and right derivatives of $F$ at $u \in [0,1]$, whenever these are well-defined.  Given an integrable function $f \colon (0,1) \to \R$ with antiderivative $F \colon [0,1] \to \R$ given by $F(u) := \int_0^u f$, define $\widehat{\mathcal{M}}_\mathrm{R}f \colon [0,1] \to [-\infty,\infty]$ by
\[
(\widehat{\mathcal{M}}_\mathrm{R}f)(u) := 
\begin{cases}
\hat{F}^{(\mathrm{R})}(u) \;&\text{for }u \in [0,1) \\
\hat{F}^{(\mathrm{L})}(1) \;&\text{for }u = 1,
\end{cases}
\]
so that $(\widehat{\mathcal{M}}_\mathrm{R}f)(1) = \lim_{u \nearrow 1} (\widehat{\mathcal{M}}_\mathrm{R}f)(u)$ by~\citet[Theorem~24.1]{rockafellar97convex}. Furthermore, define $\widehat{\mathcal{M}}_\mathrm{L}f \colon [0,1] \to [-\infty,\infty]$ by $(\widehat{\mathcal{M}}_\mathrm{L}f)(u) := (\widehat{\mathcal{M}}_\mathrm{R}g)(1 - u)$ for $u \in [0,1]$, where $g(u) := f(1 - u)$ for all such $u$.

\section{The antitonic score projection}
\label{sec:antitonic-proj}

\subsection{Construction and basic properties}
\label{subsec:antitonic-score-proj}

The aim of this section is to define formally and solve the optimisation problem~\eqref{eq:Vp0-minimiser} that yields the minimal asymptotic covariance of a convex $M$-estimator of $\beta_0$. Let $P_0$ be a probability measure on $\R$ with a uniformly continuous density $p_0$, which necessarily satisfies $p_0(\pm\infty) := \lim_{z \to \pm\infty}p_0(z) = 0$. 
% A density $p_0$ is uniformly continuous on $\R$ if and only if it is continuous on $\R$ and $p_0(\pm\infty) = 0$.
Letting $\supp p_0 := \{z \in \R : p_0(z) > 0\}$, define $\mathcal{S}_0 \equiv \mathcal{S}(p_0) := \bigl(\inf(\supp p_0), \sup(\supp p_0)\bigr)$, which is the smallest open interval that contains $\supp p_0$. We write $\Psi_\downarrow(p_0)$ for the set of all $\psi \in L^2(P_0)$ that are decreasing and right-continuous. Observe that $\Psi_\downarrow(p_0)$ is a convex cone, i.e.~$c_1\psi_1 + c_2\psi_2 \in \Psi_\downarrow(p_0)$ whenever $\psi_1,\psi_2 \in \Psi_\downarrow(p_0)$ and $c_1,c_2 \geq 0$. Moreover, every $\psi \in \Psi_\downarrow(p_0)$ is necessarily finite-valued on $\mathcal{S}_0$, so the corresponding Lebesgue--Stieltjes integral $\int_{\mathcal{S}_0}p_0\,d\psi \in [-\infty,0]$ is well-defined. 

For $\psi \in \Psi_\downarrow(p_0)$ with $\int_\R \psi^2\,dP_0 > 0$, let
\begin{equation}
\label{eq:Vp0}
V_{p_0}(\psi) := \frac{\int_\R\psi^2\,dP_0}{\bigl(\int_{\mathcal{S}_0}p_0\,d\psi\bigr)^2} \in [0,\infty],
\end{equation}
where we have modified the denominator in~\eqref{eq:Vp0-Fisher} to extend the original definition to non-differentiable functions in $\Psi_\downarrow(p_0)$ such as $z \mapsto -\sgn(z)$. That $V_{p_0}(\psi)$ is indeed the asymptotic variance factor for the corresponding convex $M$-estimator is justified formally by Proposition~\ref{prop:cvx-M-est-asymp}.  As a first step towards minimising $V_{p_0}(\psi)$ over $\psi \in \Psi_\downarrow(p_0)$, note that $V_{p_0}(c\psi) = V_{p_0}(\psi)$ for every $c > 0$, so any minimiser is at best unique up to a positive scalar.  Ignoring unimportant edge cases where the denominator in~\eqref{eq:Vp0} is zero or infinity, our optimisation problem can therefore be formulated as a constrained minimisation of the numerator in~\eqref{eq:Vp0} subject to the denominator being equal to 1. This motivates the definition of
\begin{equation}
\label{eq:score-matching-obj}
D_{p_0}(\psi) := \int_\R \psi^2\,dP_0 + 2\int_{\mathcal{S}_0} p_0\,d\psi \in [-\infty,\infty)
\end{equation}
for $\psi \in \Psi_\downarrow(p_0)$, which resembles a Lagrangian, though analogously to, e.g.~\citet[p.~798]{silverman82estimation} and \citet[][p.~705]{dumbgen2011approximation}, there is no need to introduce a Lagrange multiplier. If $\psi$ is locally absolutely continuous on $\mathcal{S}_0$ with derivative $\psi'$ Lebesgue almost everywhere, then
\begin{equation}
\label{eq:Dp0-deriv}
D_{p_0}(\psi) = \int_\R \psi^2\,dP_0 + 2\int_{\mathcal{S}_0} \psi'p_0 = \int_\R (\psi^2 + 2\psi')\,dP_0 = \E\bigl(\psi^2(\varepsilon_1) + 2\psi'(\varepsilon_1)\bigr)
\end{equation}
when $\varepsilon_1 \sim P_0$, which we recognise as the score matching objective~\eqref{eq:score-matching-objective} in the introduction.

The formal link between $V_{p_0}(\cdot)$ and $D_{p_0}(\cdot)$ is that for $\psi \in \Psi_\downarrow(p_0)$ with $\int_\R \psi^2\,dP_0 > 0$, we have $\int_{\mathcal{S}_0}p_0\,d\psi \leq 0$ and $c\psi \in \Psi_\downarrow(p_0)$ for all $c \geq 0$, so
\begin{equation}
\label{eq:V-D-equivalence}
\inf_{c \geq 0}D_{p_0}(c\psi) = \inf_{c \geq 0}\,\Bigl(c^2\int_\R\psi^2\,dP_0 + 2c\int_{\mathcal{S}_0}p_0\,d\psi\Bigr) = -\frac{\bigl(\int_{\mathcal{S}_0}p_0\,d\psi\bigr)^2}{\int_\R\psi^2\,dP_0} = -\frac{1}{V_{p_0}(\psi)}.
\end{equation}
Thus, minimising $V_{p_0}(\cdot)$ over $\Psi_\downarrow(p_0)$ is equivalent to minimising $D_{p_0}(\cdot)$ up to a scalar multiple, but $D_{p_0}(\cdot)$ is a convex function that is more tractable than $V_{p_0}(\cdot)$.

By exploiting this connection with score matching together with ideas from monotone function estimation, we prove in Theorem~\ref{thm:antitonic-score-proj}  below that the solution to our asymptotic variance minimisation problem is the function $\psi_0^*$ that we construct explicitly in the following lemma.

\begin{lemma}
\label{lem:psi0-star}
Let $P_0$ be a distribution with a uniformly continuous density $p_0$ on $\R$. Let $F_0 \colon [-\infty,\infty] \to [0,1]$ be the corresponding distribution function, and for $u \in [0,1]$, define
\[
F_0^{-1}(u) := \inf\{z \in [-\infty,\infty] : F_0(z) \geq u\} \quad\text{and}\quad J_0(u) := (p_0 \circ F_0^{-1})(u).
\] 
% $Q^0$ is right-continuous with $\lim_{v \nearrow u}Q^_0(v) = F_0^{-1}(u)$, so $J_0$ is right-continuous with left limits given by $J_0(u-) := \lim_{v \nearrow u}J_0(v) = (p_0 \circ F_0^{-1})(u)$ for $u \in (0,1]$.
Then both $J_0$ and its least concave majorant $\hat{J}_0$ on $[0,1]$ are continuous, with $p_0 = J_0 \circ F_0$ on $\R$, and
\[
\psi_0^* := \hat{J}_0^{(\mathrm{R})} \circ F_0
\]
is decreasing and right-continuous as a function from $\R$ to $[-\infty,\infty]$, provided that we set $\hat{J}_0^{(\mathrm{R})}(1) := \lim_{u \nearrow 1}\hat{J}_0^{(\mathrm{R})}(u)$.
% $= \hat{J}_0^{(\mathrm{L})}(1)$
Moreover, $\psi_0^*(z) \in \R$ if and only if $z \in \mathcal{S}_0$.
\end{lemma}
We refer to $J_0$ as the \emph{density quantile function} \citep{parzen1979nonparametric,jones1992estimating}. In the case where $p_0$ is a standard Cauchy density, Figure~\ref{fig:psi0-star-cauchy} presents a visualisation of $J_0$ and its least concave majorant $\hat{J}_0$, as well as the corresponding score functions $\psi_0 = p_0'/p_0$ and $\psi_0^*$.

\begin{figure}[htb]
\centering
\includegraphics[width=0.485\textwidth,trim={5.3cm 10cm 5.3cm 10cm},clip]{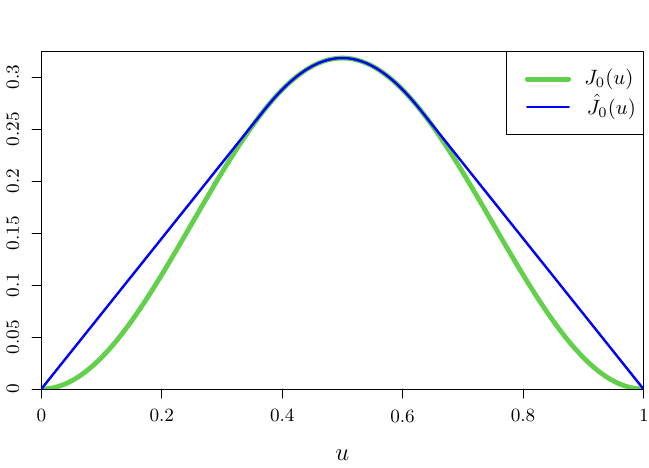}
\hfill
\includegraphics[width=0.485\textwidth,trim={5cm 9.82cm 5cm 10cm},clip]{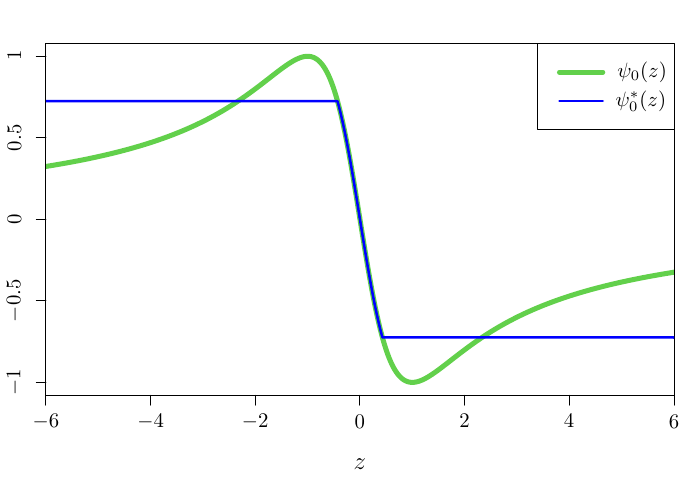}
\vspace{-0.3cm}
\caption{\emph{Left}: The density quantile function $J_0$ and its least concave majorant $\hat{J}_0$ for a standard Cauchy density. \emph{Right}: The corresponding score functions $\psi_0$ and $\psi_0^*$.}
\label{fig:psi0-star-cauchy}
\end{figure}

\begin{theorem}
\label{thm:antitonic-score-proj}
In the setting of Lemma~\ref{lem:psi0-star}, the following statements hold.
\begin{enumerate}[label=(\alph*)]
\item $\int_\R\psi_0^*\,dP_0 = 0$.
\item Let $i^*(p_0) := \int_\R(\psi_0^*)^2\,dP_0$. Then $\inf_{\psi \in \Psi_\downarrow(p_0)}D_{p_0}(\psi) = -i^*(p_0)$.
\item Suppose that $i^*(p_0) < \infty$. Then $\psi_0^*$ is the unique minimiser of $D_{p_0}(\cdot)$ over $\Psi_\downarrow(p_0)$. Moreover, for every $\psi \in \Psi_\downarrow(p_0)$ such that $\int_\R\psi^2\,dP_0 > 0$, we have 
\begin{equation}
\label{eq:antitonic-fisher-LB}
V_{p_0}(\psi) \geq V_{p_0}(\psi_0^*) = \frac{1}{i^*(p_0)} \in (0,\infty),
\end{equation}
with equality if and only if $\psi = \lambda\psi_0^*$ for some $\lambda > 0$.
\item Assume further that $p_0$ is absolutely continuous on $\R$ with derivative $p_0'$ Lebesgue almost everywhere, corresponding score function\footnote{Our convention $0/0 = 0$ means that $\psi_0 = (p_0'/p_0)\Ind_{\{p_0 > 0\}}$.} $\psi_0 := p_0'/p_0$ and Fisher information $i(p_0) := \int_\R \psi_0^2\,p_0$. Then
\begin{equation}
\label{eq:psi0-star}
\psi_0^* = \widehat{\mathcal{M}}_\mathrm{R}(\psi_0 \circ F_0^{-1}) \circ F_0
\end{equation}
and $0 < i^*(p_0) \leq i(p_0)$, with equality if and only if $p_0$ is log-concave.  In particular, if $i(p_0) < \infty$, then the conclusions of \textit{(c)} hold.
\end{enumerate}
\end{theorem}
Some remarks on the proof and implications of this result are in order.  The least concave majorant construction of $\psi_0^*$ in Lemma~\ref{lem:psi0-star} is reminiscent of isotonic regression, where the least squares estimator is the projection $\widehat{\theta} = \argmin_{\theta \in \mathcal{M}} \sum_{i=1}^n (Y_i - \theta_i)^2$ of the response vector $Y = (Y_1,\dotsc,Y_n) \in \R^n$ onto the monotone cone $\mathcal{M} := \bigl\{\theta = (\theta_1,\dotsc,\theta_n) : \theta_1 \leq \cdots \leq \theta_n\bigr\} \subseteq \R^n$. It is well-known that the entries of $\widehat{\theta}$ are left derivatives of the greatest convex minorant of a cumulative sum diagram, and other monotonicity-constrained estimators (such as the Grenander estimator of a decreasing density) have similar explicit representations (e.g.~\citealp[Chapter~1]{robertson1988order}; \citealp[Chapter~2]{groeneboom14nonparametric}; \citealp[Chapter~9]{samworth24modern}). An equivalent characterisation of the projection $\widehat{\theta}$ of $Y$ onto the convex cone $\mathcal{M}$ is that
\begin{equation}
\label{eq:cvx-proj}
(Y - \widehat{\theta})^\top\theta \leq 0 \;\;\text{for all }\theta \in \mathcal{M} \quad\text{and}\quad
(Y - \widehat{\theta})^\top\widehat{\theta} = 0.
\end{equation}
To minimise the score matching objective $D_{p_0}(\cdot)$ over our convex cone $\Psi_\downarrow(p_0)$ of antitonic functions, the key step of the proof of Theorem~\ref{thm:antitonic-score-proj} is to establish a similar first-order condition on the population level, namely
\begin{equation}
\label{eq:antiproj}
-\int_{\mathcal{S}_0} p_0\,d\psi \leq \int_\R\psi_0^*\psi\,dP_0 = \ipr{\psi_0^*}{\psi}_{L^2(P_0)}
\end{equation}
for $\psi \in \Psi_\downarrow(p_0)$, with equality when $\psi = \psi_0^*$. Since both sides of~\eqref{eq:antiproj} are linear in $\psi$, 
% so is preserved under non-negative linear combinations
it suffices to prove it for indicator functions of the form $\psi = \Ind_{(-\infty,t]}$ for $t \in \R$, which generate the cone $\Psi_\downarrow(p_0)$; see~\eqref{eq:KKT-ind}. This relies on key properties of the least concave majorant. Taking $\psi \equiv 1$ and $\psi \equiv -1$ in~\eqref{eq:antiproj} yields $\E\psi_0^*(\varepsilon_1) = \ipr{\psi_0^*}{1}_{L^2(P_0)} = 0$. This is part \textit{(a)} of the theorem, and
% This reflects the fact that $\psi + c \in \Psi_\downarrow(p_0)$ whenever $\psi \in \Psi_\downarrow(p_0)$ and $c \in \R$
ensures the Fisher consistency of the regression $Z$-estimator $\hat{\beta}_{\psi_0^*}$ defined in~\eqref{eq:linreg-Z-est}.  
% including the intuitively obvious Lemma~\ref{lem:lcm-affine} for the equality case
The optimality properties of $\psi_0^*$ in parts \textit{(b)} and \textit{(c)} follow readily from~\eqref{eq:antiproj}. In particular, combining this with the Cauchy--Schwarz inequality shows that~$\psi_0^*$ minimises the asymptotic variance factor over $\Psi_\downarrow(p_0)$, similarly to~\eqref{eq:Vp0-Fisher} in the introduction for the usual inverse Fisher information lower bound.

The parallels between~\eqref{eq:cvx-proj} and~\eqref{eq:antiproj} can be seen when $p_0$ is absolutely continuous with score function~$\psi_0$ satisfying $i(p_0) = \norm{\psi_0}_{L^2(P_0)}^2 < \infty$. In this case, integration by parts 
% Fubini's theorem strictly speaking
yields $-\int_{\mathcal{S}_0} p_0\,d\psi = \int_\R \psi\psi_0 p_0$,
and hence~\eqref{eq:antiproj} states that $\ipr{\psi_0 - \psi_0^*}{\psi}_{L^2(P_0)} \leq 0$ for $\psi \in \Psi_\downarrow(p_0)$ and $\ipr{\psi_0 - \psi_0^*}{\psi_0^*}_{L^2(P_0)} = 0$.
Consequently,
\[
D_{p_0}(\psi) = \int_\R (\psi - \psi_0)^2\,dP_0 - \int_\R \psi_0^2\,dP_0 = \norm{\psi - \psi_0}_{L^2(P_0)}^2 - i(p_0)
\]
for all $\psi \in \Psi_\downarrow(p_0)$, so if $i(p_0) < \infty$, then
\begin{equation}
\label{eq:Dp0-L2P0}
\psi_0^* \in \argmin_{\psi \in \Psi_\downarrow(p_0)}D_{p_0}(\psi) = \argmin_{\psi \in \Psi_\downarrow(p_0)}\norm{\psi - \psi_0}_{L^2(P_0)}^2.
\end{equation}
Thus, in the terminology of Section~\ref{sec:isoproj}, $\psi_0^*$ is a version of the \textit{$L^2(P_0)$-antitonic projection} of $\psi_0$ onto $\Psi_\downarrow(p_0)$. Indeed, the explicit representation~\eqref{eq:psi0-star} of $\psi_0^*$ as a `monotonisation' of $\psi_0$ (see the right panel of Figure~\ref{fig:psi0-star-cauchy}) is consistent with that given in Proposition~\ref{prop:isoproj} for a general $L^2(P)$-antitonic projection, where~$P$ is a univariate probability measure with a continuous distribution function. 

The inequality $i^*(p_0) = \int_\R (\psi_0^*)^2\,dP_0 \leq \int_\R \psi_0^2\,dP_0 = i(p_0)$ in Theorem~\ref{thm:antitonic-score-proj}\emph{(d)} follows from the fact that the $L^2(P_0)$-antitonic projection onto the convex cone $\Psi_\downarrow(p_0)$ is 1-Lipschitz with respect to $\norm{{\cdot}}_{L^2(P_0)}$; see~\eqref{eq:L2-proj-contraction} in Lemma~\ref{lem:L2-proj-ineq}. A statistical explanation of this information inequality arises from the fact that $1/i(p_0)$ is the infimum of the asymptotic variance factor $V_{p_0}(\psi)$ in~\eqref{eq:Vp0-Fisher} over all sufficiently regular $\psi \colon \R \to \R$; see Proposition~\ref{prop:fisher-inf-variational}. On the other hand, by~\eqref{eq:antitonic-fisher-LB}, $1/i^*(p_0)$ is the minimum value of $V_{p_0}(\cdot)$ over the restricted class $\Psi_\downarrow(p_0)$, so in view of our discussion in the introduction, it can be interpreted as an information lower bound for convex $M$-estimators. When $i^*(p_0) < \infty$, the \textit{antitonic relative efficiency} 
\[
\mathrm{ARE}^*(p_0) := \frac{i^*(p_0)}{i(p_0)}
\]
therefore quantifies the price we pay in statistical efficiency for insisting that our loss function be convex. By Theorem~\ref{thm:antitonic-score-proj}\emph{(d)}, $\mathrm{ARE}^*(p_0) \leq 1$ with equality if and only if $p_0$ is log-concave, so we can regard $1 - \mathrm{ARE}^*(p_0)$ as a measure of departure from log-concavity; see Section~\ref{subsec:fisher-divergence-proj} below. Example~\ref{ex:cauchy} shows that $\mathrm{ARE}^*(p_0) \approx 0.878$ when $p_0$ is the Cauchy density, whereas Example~\ref{rem:ARE-infinity} in Section~\ref{subsec:appendix-examples} yields a density $p_0$ for which $\mathrm{ARE}^*(p_0) = 0$. More generally, in Lemma~\ref{lem:ARE-lower-bound} below, we provide a simple lower bound on $\mathrm{ARE}^*(p_0)$ that is reasonably tight for heavy-tailed densities~$p_0$.

When $p_0$ is only uniformly continuous but not absolutely continuous, the score function and Fisher information cannot be defined as above, but we nevertheless refer to $\psi_0^*$ and $i^*(p_0)$ as the \emph{antitonic projected score function} (see Lemma~\ref{lem:p0-star} below) and \emph{antitonic information (for location)} respectively. 
% For general $\psi \in \Psi_\downarrow(p_0)$, it follows from~\eqref{eq:antiproj} that
% \[
% D_{p_0}(\psi) \geq \norm{\psi}_{L^2(P_0)}^2 - 2\ipr{\psi_0^*}{\psi}_{L^2(P_0)} = \norm{\psi - \psi_0^*}_{L^2(P_0)}^2 - \norm{\psi_0^*}_{L^2(P_0)}^2,
% \]
% which establishes Theorem~\ref{thm:antitonic-score-proj}\textit{(b)}. By applying the Cauchy--Schwarz inequality similarly to~\eqref{eq:Vp0-hampel} in the introduction, we can deduce that
% \[
% V_{p_0}(\psi) \geq \frac{\norm{\psi}_{L^2(P_0)}^2}{\ipr{\psi_0^*}{\psi}_{L^2(P_0)}^2} \geq \frac{1}{\norm{\psi_0^*}_{L^2(P_0)}^2} = \frac{1}{i^*(p_0)},
% \]
% to establish~\textit{(c)}.

\begin{remark}
\label{rem:fisher-J}
Since $F_0$ is continuous, we have $(F_0 \circ F_0^{-1})(u) = u$ for all $u \in (0,1)$. The concave function $\hat{J}_0$ is therefore absolutely continuous on $[0,1]$ with derivative $\hat{J}_0^{(\mathrm{R})} = \psi_0^* \circ F_0^{-1}$ Lebesgue almost everywhere~\citep[Corollary~24.2.1]{rockafellar97convex}, so
\[
i^*(p_0) = \int_\R (\psi_0^*)^2\,dP_0 = \int_0^1 (\psi_0^* \circ F_0^{-1})^2 = \int_0^1 \bigl(\hat{J}_0^{(\mathrm{R})}\bigr)^2.
\]
When $p_0$ is absolutely continuous on $\R$, a straightforward calculation (e.g.~\eqref{eq:psi0-J0} in the proof of Theorem~\ref{thm:antitonic-score-proj}) shows that the density quantile function $J_0 = p_0 \circ F_0^{-1}$ is absolutely continuous on $[0,1]$ with derivative $J_0' = \psi_0 \circ F_0^{-1}$ Lebesgue almost everywhere. Therefore, $\psi_0^* \circ F_0^{-1} = \hat{J}_0^{(\mathrm{R})} = \widehat{\mathcal{M}}_\mathrm{R}(\psi_0 \circ F_0^{-1})$ almost everywhere and
\[
i(p_0) = \int_\R \psi_0^2\,dP_0 = \int_0^1 (\psi_0 \circ F_0^{-1})^2 = \int_0^1 (J_0')^2.
\]
\end{remark}
It is well-known that the map $p_0 \mapsto i(p_0)$ is convex on the space of absolutely continuous densities on~$\R$~\citep[e.g.][Section~4.4, p.~78]{huber2009robust}. 
% For fixed $\theta \in (0,1)$ and any $\psi \colon \R \to \R$, we have
% \[
% 0 \leq (1 - \theta)(\psi - \psi_0)^2 p_0 + \theta(\psi - \psi_1)^2 p_1 = \bigl((1 - \theta)p_0 + \theta p_1\bigr)\psi^2 - 2\bigl((1 - \theta)p_0' + \theta p_1'\bigr)\psi + (1 - \theta)\psi_0^2 p_0 + \theta\psi_1^2 p_1.
% \]
% Taking $\psi := \frac{(1 - \theta)p_0' + \theta p_1'}{(1 - \theta)p_0 + \theta p_1}$ and integrating over $\R$, we obtain $i\bigl((1 - \theta)p_0 + \theta p_1\bigr) \leq (1 - \theta)i(p_0) + \theta i(p_1)$.
The following corollary of Theorem~\ref{thm:antitonic-score-proj}\textit{(b)} establishes the analogous property holds for antitonic information.
\begin{corollary}
\label{cor:istar-cvx}
The map $p_0 \mapsto i^*(p_0)$ is convex on the space of uniformly continuous densities on $\R$.
\end{corollary}

Finally in this subsection, we define the \emph{two-sided hazard function} $h_0 \colon \mathbb{R} \rightarrow [0,\infty)$ of $p_0$ by
\begin{equation}
\label{eq:two-sided-hazard}
h_0(z) := \frac{p_0(z)}{F_0(z) \wedge \bigl(1 - F_0(z)\bigr)} =
\begin{cases}
p_0(z)/F_0(z) & \text{if }F_0(z) \leq 1/2 \\
p_0(z)/\bigl(1 - F_0(z)\bigr) & \text{if }F_0(z) > 1/2,
\end{cases}
\end{equation}
where in accordance with our convention $0/0 = 0$, we have $h_0(z) = 0$ whenever $F_0(z) \in \{0,1\}$. The following simple lemma provides a necessary and sufficient condition on the two-sided hazard function for $\psi_0^*$ to be appropriately bounded, which means that any negative antiderivative $\ell_0^*$ grows at most linearly in the tails.  This is relevant because $\psi_0^*$ is bounded if and only if the corresponding $M$-estimator is robust in the sense of having positive \emph{finite-sample breakdown point} and uniformly bounded influence function, i.e.~finite gross error sensitivity~\citep[Sections~2.2 and~2.3]{hampel2011robust}.
\begin{lemma}
\label{lem:hazard}
In the setting of Lemma~\ref{lem:psi0-star}, define $z_{\min} := \inf(\supp p_0)$ and $z_{\max} := \sup(\supp p_0)$. Then
\begin{enumerate}[label=(\alph*)]
\item $\lim_{z \to -\infty}\psi_0^*(z) < \infty$ if and only if $\limsup_{z \searrow z_{\min}} h_0(z) < \infty$, in which case $z_{\min} = -\infty$;
\item $\lim_{z \to \infty}\psi_0^*(z) > -\infty$ if and only if $\limsup_{z \nearrow z_{\max}} h_0(z) < \infty$, in which case $z_{\max} = \infty$.
\end{enumerate}
\end{lemma}
Recall that a Laplace density has a constant two-sided hazard function, as well as a score function whose absolute value is constant. Roughly speaking, the conditions on $h_0$ in Lemma~\ref{lem:hazard} are satisfied by densities whose tails are heavier than those of the Laplace density~\citep{samworth2004convergence}, for which it is particularly attractive to have bounded projected score functions.

\subsection{The log-concave Fisher divergence projection}
\label{subsec:fisher-divergence-proj}

Let $P_0$ and $P_1$ be Borel probability measures on $\mathbb{R}$ such that $P_0$ is absolutely continuous with respect to $P_1$.  Write $\supp P_0$ for the support of $P_0$ (i.e.~the smallest closed set $S$ satisfying $P_0(S) = 1$), and $\Int(\supp P_0)$ for its interior. Suppose that there exists a Radon--Nikodym derivative $dP_0/dP_1$ that is continuous on $\Int(\supp P_0)$, and also strictly positive and differentiable on some subset $E \subseteq \Int(\supp P_0)$ such that $P_0(E^c) = 0$.\footnote{This condition precludes $P_0$ from having any isolated atoms, so in particular, $P_0$ cannot be a discrete measure.} The \textit{Fisher divergence} (also known as the \textit{Fisher information distance}\footnote{This is not to be confused with the \textit{Fisher information} (or \textit{Fisher--Rao}) \textit{metric}~\citep[Chapter~2]{amari2000methods}, a Riemannian metric on a manifold of probability distributions.}) from $P_1$ to $P_0$ is defined to be
\begin{equation}
\label{eq:fisher-div}
I(P_0,P_1) := \int_E \biggl(\Bigl(\log\frac{dP_0}{dP_1}\Bigr)'\biggr)^2\,dP_0.
\end{equation}
If $P_0,P_1$ do not satisfy the assumptions above, then we define $I(P_0,P_1) := \infty$. In the case where $P_0,P_1$ have Lebesgue densities $p_0,p_1$ respectively that are both locally absolutely continuous on $\R$, we have
\[
I(p_0,p_1) \equiv I(P_0,P_1) = 
\begin{cases}
\displaystyle \int_{\{p_0 > 0\}} \biggl(\Bigl(\log\frac{p_0}{p_1}\Bigr)'\biggr)^2\,p_0 = \int_\R (\psi_0 - \psi_1)^2\,dP_0
\;&\text{if }\supp p_0 \subseteq \supp p_1 \\
\infty \;&\text{otherwise},
\end{cases}
\]
where we denote by $\psi_j := (\log p_j)' \Ind_{\{p_j > 0\}} = p_j'/p_j$ the corresponding score functions for $j \in \{0,1\}$. For further background on the Fisher divergence, see~\citet[Definition~1.13]{johnson2004information},~\citet[Section~2]{yang2019variational} and references therein. 

The following lemma establishes the connection between the projected score function and the Fisher divergence.
\begin{lemma}
\label{lem:p0-star}
In the setting of Lemma~\ref{lem:psi0-star}, there is a unique continuous log-concave density $p_0^*$ on $\R$ such that $\supp p_0^* = \mathcal{S}_0$ and $\log p_0^*$ has right derivative $\psi_0^*$ on $\mathcal{S}_0$. In particular, $\psi_0^* = (\log p_0^*)'$ Lebesgue almost everywhere on $\mathcal{S}_0$.  Furthermore, if $p_0$ is absolutely continuous, then $p_0^*$ minimises $I(p_0,p)$ over the class $\mathcal{P}_{\mathrm{LC}}$ of all univariate log-concave densities $p$, and if $p_0 \in \mathcal{P}_{\mathrm{LC}}$, then $p_0^* = p_0$.
\end{lemma} 
Even when $p_0$ is only uniformly continuous, we refer to $p_0^*$ as the \textit{log-concave Fisher divergence projection} of $p_0$. In contrast, the log-concave maximum likelihood projection $p_0^{\mathrm{ML}}$ of the distribution $P_0$~\citep{dumbgen2011approximation,barber2021local} can be interpreted as a minimiser of Kullback--Leibler divergence rather than Fisher divergence over the class of upper semi-continuous log-concave densities. By~\citet[Theorem~2.2]{dumbgen2011approximation}, $p_0^{\mathrm{ML}}$ exists and is unique if and only if~$P_0$ is non-degenerate and has a finite mean (but not necessarily a Lebesgue density). On the other hand, moment conditions are not required for $p_0^*$ to exist and be unique, but $p_0^*$ is only defined in Lemma~\ref{lem:p0-star} when~$P_0$ has a uniformly continuous density on $\R$. As we will discuss in Section~\ref{subsec:score-estimation}, the non-existence of the Fisher divergence projection for discrete measures $P_0$ has consequences for our statistical methodology. 
% $p_0^{\mathrm{ML}}$ does not have an explicit characterisation, whereas $p_0^*$ does

When $p_0$ is not log-concave, $p_0^{\mathrm{ML}}$ usually does not coincide with $p_0^*$ even when both exist, and moreover the associated regression $M$-estimators 
\begin{equation}
\label{eq:betahat-ML-fisher}
\hat{\beta}_{\psi_0^{\mathrm{ML}}} \in \argmax_{\beta \in \R^d} \sum_{i=1}^n \log p_0^{\mathrm{ML}}(Y_i - X_i^\top\beta) \quad\text{and}\quad \hat{\beta}_{\psi_0^*} \in \argmax_{\beta \in \R^d} \sum_{i=1}^n \log p_0^*(Y_i - X_i^\top\beta)
\end{equation}
are generally different; see Examples~\ref{ex:t2} and~\ref{ex:laplace-mixture} in Section~\ref{subsec:appendix-examples}. In fact, the following result shows that there exist error distributions $P_0$ for which the asymptotic covariance of $\hat{\beta}_{\psi_0^{\mathrm{ML}}}$ is arbitrarily large compared with that of the optimal convex $M$-estimator $\hat{\beta}_{\psi_0^*}$, even when the latter is close to being asymptotically efficient in the sense of~\eqref{eq:betahat-MLE-asymp}.

\begin{figure}
\centering
\begin{tikzpicture}[xscale=1.6, yscale=2]
% Define constants
\def\a{3} % a = 3
\def\b{1} % b = 1
\def\epsilon{1} % epsilon = 1
\def\delta{0.15} % delta = 0.1
\def\constML{\epsilon * \a * exp(-\a * \delta) / 2} % p_0(1 + delta)

\def\constFish{\epsilon * \a / 2} % p_0(1 + delta)

% Define the piecewise function
\draw[domain=-2:-1, smooth, samples=100] 
plot (\x, {\epsilon * \a * exp(-\a * (abs(\x) - 1)) / 2});
\draw[domain=-1:1, smooth, samples=100] 
plot (\x, {\epsilon * \a * exp(-\b * (1 - abs(\x))) / 2});
\draw[domain=1:2, smooth, samples=100] 
plot (\x, {\epsilon * \a * exp(-\a * (abs(\x) - 1)) / 2});

% Modified density p_0^{\mathrm{ML}}(z)
\draw[domain=-2:-(1+\delta), smooth, thick, samples=100, red] 
plot (\x, {\epsilon * \a * exp(-\a * (abs(\x) - 1)) / 2});
\draw[domain=-(1+\delta):(1+\delta), smooth, thick, samples=100, red] 
plot (\x, {\constML});
\draw[domain=(1+\delta):2, smooth, thick, samples=100, red] 
plot (\x, {\epsilon * \a * exp(-\a * (abs(\x) - 1)) / 2});

\draw[domain=-2:-1, smooth, thick, samples=100, blue] 
plot (\x, {0.8 * \epsilon * \a * exp(-\a * (abs(\x) - 1)) / 2});
\draw[domain=-1:1, smooth, thick, samples=100, blue] 
plot (\x, {0.8 * \constFish});
\draw[domain=1:2, smooth, thick, samples=100, blue] 
plot (\x, {0.8 * \epsilon * \a * exp(-\a * (abs(\x) - 1)) / 2});

\draw (-1.5, 1.2) node [anchor=south, black] {$p_0(z)$};
\node[anchor=north, red] at (1.85, 0.95) {$p_0^{\mathrm{ML}}(z)$};
\node[anchor=north, blue] at (0.45, 1.65) {$p_0^*(z)$};

% Axes
\draw[->] (-2.2, 0) -- (2.5, 0) node[anchor=north] {$z$};
\draw[->] (0, 0) -- (0, 1.8);

% Labels for key points
\node[anchor=north] at (-1, 0) {$-1$};
\node[anchor=north] at (1, 0) {$1$};
\node[anchor=north] at (0, 0) {$0$};

% Optional: Dashed lines for key points
\draw[dashed] (-1, 0) -- (-1, 1.5);
\draw[dashed] (1, 0) -- (1, 1.5);
\end{tikzpicture}
\hspace{10pt}
\begin{tikzpicture}[xscale=1.3, yscale=0.5]
\def\a{3}
\def\b{1.2}
\def\delta{0.25}

% Axes
\draw[->] (-2.5, 0) -- (2.5, 0) node[anchor=north] {$z$};
\draw[->] (0, -4) -- (0, 4);

\draw[dashed] (-1, -4) -- (-1, 4);
\draw[dashed] (1, -4) -- (1, 4);

% Horizontal blue and black line at y = b
\draw[black] (0, \b) -- (1, \b);

% Horizontal blue and black line at y = -b
\draw[black] (-1, -\b) -- (0, -\b);

% Label functions
\node[anchor=west, blue] at (-2, \a + 0.65) {$\psi_0^*(z)$};
\node[anchor=west, red] at (-2.1, \a - 0.7) {$\psi_0^{\mathrm{ML}}(z)$};
\node[anchor=west, black] at (0.1, \b + 0.6) {$\psi_0(z)$};

% Horizontal line at y = a and -a
\draw[black] (-2.3, \a + 0.1) -- (-1, \a + 0.1);
\draw[black] (1, -\a + 0.1) -- (2.3, -\a + 0.1);

\draw[blue, thick] (-2.3, \a + 0.03) -- (-1, \a + 0.03);
\draw[blue, thick] (1, -\a + 0.03) -- (2.3, -\a + 0.03);

\draw[red, thick] (-2.3, \a - 0.05) -- (-1 - \delta, \a - 0.05);
\draw[red, thick] (1 + \delta, -\a - 0.05) -- (2.3, -\a - 0.05);

\draw[red, thick] (-1 - \delta, -0.05) -- (1 + \delta, -0.05);

\draw[blue, thick] (-1, 0.05) -- (1, 0.05);

% Origin and labels
\node[circle, fill, inner sep=1pt] at (0, 0) {};
\node[anchor=north west] at (0, 0) {0};

% Additional numeric labels
\node[anchor=north] at (-1 - 0.3, 0) {$-1$};
\node[anchor=north] at (1 + 0.2, 0) {1};
\end{tikzpicture}

\caption{Illustration of the construction in the proof of Proposition~\ref{prop:Vp0-MLE}. \textit{Left}: Plot of the density~$p_0$ (black) together with its log-concave maximum likelihood projection $p_0^{\mathrm{ML}}$ (red) and Fisher divergence projection $p_0^*$ (blue). \textit{Right}: Plot of the corresponding score functions.}
\label{fig:Vp0-MLE}
\end{figure}

\begin{proposition}
\label{prop:Vp0-MLE}
For every $\epsilon \in (0,1)$, there exists a distribution $P_0$ with a finite mean and an absolutely continuous density $p_0$ such that $i(p_0) < \infty$, and the log-concave maximum likelihood projection $q_0 \equiv p_0^{\mathrm{ML}}$ has corresponding score function $\psi_0^{\mathrm{ML}} := q_0^{(\mathrm{R})}/q_0 \in \Psi_\downarrow(p_0)$ satisfying
\begin{equation}
\label{eq:Vp0-MLE}
\frac{V_{p_0}(\psi_0^*)}{V_{p_0}(\psi_0^{\mathrm{ML}})} \leq \epsilon \quad\text{and}\quad \mathrm{ARE}^*(p_0) \geq 1 - \epsilon.
\end{equation}
\end{proposition}
The idea for the proof of Proposition~\ref{prop:Vp0-MLE} is to construct an absolutely continuous density $p_0$ whose score function is constant on each of $(-\infty,-1)$, $(-1,0)$, $(0,1)$ and $(1,\infty)$; see Figure~\ref{fig:Vp0-MLE}. The key point is that the log-concave maximum likelihood projection is constant on $[-1,1]$ and exactly matches the true density in the tails, so in order for $p_0^{\mathrm{ML}}$ to integrate to 1,  the densities and score functions can only agree on $\bigl(-\infty,-(1+\delta)\bigr)$ and $(1+\delta,\infty)$ for some $\delta > 0$. On the other hand, the log-concave Fisher divergence projection matches the score $\psi_0$ on the whole of $[-1,1]^c$ by merely being \emph{proportional} to the true density in the tails. The main contribution to the Fisher information of $p_0$ arises from the region $\bigl(-(1+\delta), -1\bigr) \cup (1, 1+\delta)$, on which the antitonic score function approximation is exact but $\psi_0^{\mathrm{ML}}$ is equal to 0. This means that the antitonic relative efficiency is close to~1, while the log-concave maximum likelihood projection incurs a considerable relative efficiency loss over this region.

Our next result provides a simple lower bound on the antitonic information.
\begin{lemma}
\label{lem:ARE-lower-bound}
Suppose that $p_0$ is an absolutely continuous density on $\R$ with $i(p_0) < \infty$. Then $p_0$ is bounded with $i^*(p_0) \geq 4\norm{p_0}_\infty^2$, so
\[
\mathrm{ARE}^*(p_0) \geq \frac{4\norm{p_0}_\infty^2}{i(p_0)},
\]
with equality if and only if $p_0^*$ is a Laplace density, i.e.~there exist $\mu \in \R$ and $\sigma > 0$ such that $p_0^*(z) = (2\sigma)^{-1}\exp(-|z - \mu|/\sigma)$ for all $z \in \R$.
\end{lemma}

\begin{remark}
\label{rem:affine-equivariance}
A reassuring property of the antitonic projection is its affine equivariance: if $p_0$ is a uniformly continuous density, then for $a > 0$ and $b \in \R$, the density $z \mapsto a p_0(az + b) =: p_{a,b}(z)$ has antitonic projected score function and log-concave Fisher divergence projection given by 
\[
\psi_{a,b}^*(z) := a \psi_0^*(az + b) \quad\text{and}\quad p_{a,b}^*(z) := a p_0^*(az + b)
\]
respectively for $z \in \R$. It follows that $1/V_{p_{a,b}}(\psi_{a,b}^*) = i^*(p_{a,b}) = \int_\R (\psi_{a,b}^*)^2\,p_{a,b} = a^2 i^*(p_0)$, so because $\norm{p_{a,b}}_\infty = a\norm{p_0}_\infty$, both the antitonic relative efficiency and the lower bound in Lemma~\ref{lem:ARE-lower-bound} are affine invariant in the sense that they remain unchanged if we replace $p_0$ with $p_{a,b}$.

Similarly, if $P_0$ has a finite mean, then by the affine equivariance of the log-concave maximum likelihood projection~\citep[Remark~2.4]{dumbgen2011approximation}, $p_{a,b}^{\mathrm{ML}}(z) = a p_0^{\mathrm{ML}}(az + b)$ and hence $\psi_{a,b}^{\mathrm{ML}} = a\psi_0^{\mathrm{ML}}(az + b)$ for $z \in \R$. Thus, $V_{p_{a,b}}(\psi_{a,b}^{\mathrm{ML}}) = V_{p_0}(\psi_0^{\mathrm{ML}})/a^2$, so the first ratio in~\eqref{eq:Vp0-MLE} is also affine invariant. Consequently, for any $C \in (0,\infty)$, $\mu \in \R$ and $\epsilon \in (0,1)$, there exists a density $p_0$ satisfying~\eqref{eq:Vp0-MLE} with $i(p_0) = C$ and $\int_\R z p_0(z)\,dz = \mu$.
\end{remark}

The final result in this subsection relates properties of densities and their log-concave Fisher divergence projections.  
\begin{proposition}
\label{prop:p0-star-fisher}
For a uniformly continuous density $p_0 \colon \R \to \R$, the log-concave Fisher divergence projection $p_0^*$ and its corresponding distribution function $F_0^* \colon [-\infty,\infty] \to \R$ have the following properties.
\begin{enumerate}[label=(\alph*)]
\item Denote by $\mathcal{T}$ the set of $z \in \mathcal{S}_0$ such that $\psi_0^*$ is non-constant on every open interval containing~$z$. 
For $z \in \mathcal{T}$, we have
\begin{equation}
\label{eq:p0-star-hazard}
\frac{p_0^*(z)}{F_0^*(z)} \leq \frac{p_0(z)}{F_0(z)} \quad\text{and}\quad \frac{p_0^*(z)}{1 - F_0^*(z)} \leq \frac{p_0(z)}{1 - F_0(z)},
\end{equation}
whence $p_0^*(z) \leq p_0(z)$.
\item $\norm{p_0^*}_\infty \leq \norm{p_0}_\infty$ and $i(p_0^*) = -\int_\R p_0^*\,d\psi_0^* \leq -\int_\R p_0\,d\psi_0^* = i^*(p_0)$.
% Question/conjecture: Similarly to Corollary~\ref{cor:istar-cvx}, does this hold for any information quantity defined by minimising $V_{p_0}(\psi)$ over a convex cone of functions $\psi$?
\end{enumerate}
\end{proposition}
We can define the two-sided hazard function $h_0^*$ of the log-concave Fisher divergence projection $p_0^*$ analogously to $h_0$ in~\eqref{eq:two-sided-hazard}. Since $p_0^*$ is log-concave, its density quantile function $J_0^* := p_0^* \circ (F_0^*)^{-1}$ has decreasing right derivative $(J_0^*)^{(\mathrm{R})} = \psi_0^* \circ (F_0^*)^{-1}$ by (the proof of) Lemma~\ref{lem:p0-star}, so $J_0^*$ is concave on $[0,1]$. Thus,
\[
0 = J_0^*(0) \leq J_0^*(u) - u(J_0^*)^{(\mathrm{R})}(u) \quad\text{and}\quad 0 = J_0^*(1) \leq J_0^*(u) + (1 - u)(J_0^*)^{(\mathrm{R})}(u)
\]
for all $u \in [0,1]$, so for $z \in \mathcal{T}$, Lemma~\ref{lem:psi0-star} and~\eqref{eq:p0-star-hazard} in Proposition~\ref{prop:p0-star-fisher}\emph{(a)} imply that
\begin{align*}
|\psi_0^*(z)| = \bigl|(J_0^*)^{(\mathrm{R})}\bigl(F_0^*(z)\bigr)\bigr| \leq \frac{J_0^*\bigl(F_0^*(z)\bigr)}{F_0^*(z) \wedge \bigl(1 - F_0^*(z)\bigr)} &= \frac{p_0^*(z)}{F_0^*(z) \wedge \bigl(1 - F_0^*(z)\bigr)} = h_0^*(z) \\
&\leq \frac{p_0(z)}{F_0(z) \wedge \bigl(1 - F_0(z)\bigr)} = h_0(z).
\end{align*}
Proposition~\ref{prop:p0-star-fisher}\emph{(b)} provides inequalities on the supremum norm and antitonic information of the log-concave Fisher divergence projection.  In particular, the Fisher information of the projected density is at most the antitonic information of the original density.

\subsection{Examples}
\label{subsec:examples}

\begin{example}
Let $p_0$ be the $\mathrm{Beta}(a,b)$ density given by $p_0(z) = z^{a-1}(1 - z)^{b-1}\mathbbm{1}_{\{z \in (0,1)\}}/\mathrm{B}(a,b)$ for $a,b > 1$, where $\mathrm{B}$ denotes the beta function. Then $p_0$ is uniformly continuous and log-concave on $\R$, so 
\[
\psi_0^*(z) = \psi_0(z) = (\log p_0)'(z) = \frac{a - 1}{z} - \frac{b - 1}{1 - z}
\]
for all $z \in (0,1)$, while $\psi_0^*(z) = \infty$ for $z \leq 0$ and $\psi_0^*(z) = -\infty$ for $z \geq 1$. We have $i^*(p_0) = i(p_0) = \int_0^1 \psi_0^2\,p_0 < \infty$ if and only if $a,b > 2$, in which case the conclusions of Theorem~\ref{thm:antitonic-score-proj}\textit{(c,\,d)} hold.
\end{example}

\begin{figure}[htb]
\centering
\includegraphics[width=0.49\textwidth,trim={5cm 9.82cm 5cm 10cm},clip]{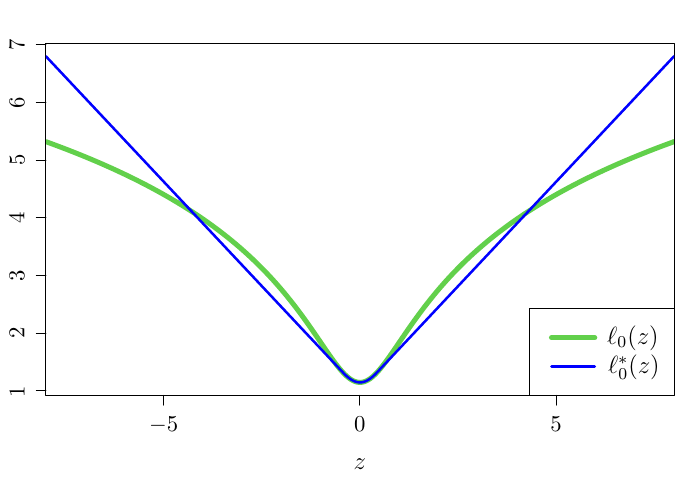}
\hfill
\includegraphics[width=0.49\textwidth,trim={5cm 9.82cm 5cm 10cm},clip]{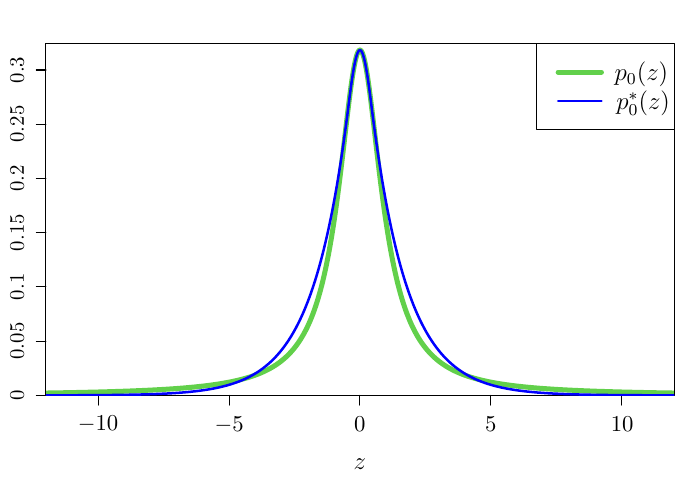}

\vspace{-0.3cm}
\caption{\emph{Left}: The negative log-density $\ell_0 = -\log p_0$ and the optimal convex loss function $\ell_0^* = -\log p_0^*$ when $p_0$ is the standard Cauchy density. \emph{Right}: The corresponding densities $p_0$ and $p_0^*$.}
\label{fig:cauchy}
\end{figure}

\begin{example}
\label{ex:cauchy}
Let $p_0$ be the standard Cauchy density given by $p_0(z) = 1/\bigl(\pi(1 + z^2)\bigr)$ for $z \in \R$. Then $p_0$ is absolutely continuous on $\R$ with $i(p_0) = \int_\R (p_0')^2/p_0 = 1/2$ and $\lim_{z \to \pm\infty} h_0(z) = 0$. We will derive an explicit expression for $\psi_0^*$, which is necessarily bounded by Lemma~\ref{lem:hazard}. In contrast to the previous example, $p_0$ is not log-concave, so $\psi_0^*$ does not coincide with $\psi_0 = p_0'/p_0 \colon z \mapsto -2z/(1 + z^2)$. Indeed, 
\begin{align*}
F_0(z) &= \frac{1}{2} + \frac{\arctan(z)}{\pi} \;\;\text{for }z \in \R, \qquad F_0^{-1}(u) = \tan\biggl(\pi\Bigl(u - \frac{1}{2}\Bigr)\biggr) = -\cot(\pi u) \;\;\text{for }u \in (0,1), \\
J_0(u) &= \frac{1}{\pi\bigl(1 + \cot^2(\pi u)\bigr)} = \frac{\sin^2(\pi u)}{\pi} = \frac{1 - \cos(2\pi u)}{2\pi} \;\;\text{for }u \in [0,1].
\end{align*}

Let $t_0 \approx 2.33$ be the unique $t \in (0,\pi)$ satisfying $t = \tan(t/2)$,
% $t\sin t + \cos t = 1$, 
and define $u_0 := t_0/(2\pi) \in (0,1/2)$. Then we can verify that $\hat{J}_0$ is linear on $[0,u_0]$ and on $[1 - u_0,1]$, with $\hat{J}_0(u) = J_0(u)$ for $u \in [u_0, 1 - u_0] \cup \{0,1\}$; see the left panel of Figure~\ref{fig:psi0-star-cauchy}. It follows that
\begin{align*}
\hat{J}_0^{(\mathrm{R})}(u) &= 
\begin{cases}
\,\sin t_0 &\text{for }u \in [0,u_0] \\
\,\sin(2\pi u) &\text{for }u \in [u_0,1 - u_0] \\
\,-\sin t_0 &\text{for }u \in [1 - u_0,1];
\end{cases}\\
\psi_0^*(z) = \hat{J}_0^{(\mathrm{R})}\bigl(F_0(z)\bigr) &= 
\begin{cases}
\,\sin t_0 = -2z_0/(1 + z_0^2) &\text{for }z \in (-\infty,-z_0] \\
\,-\sin(2\arctan z) = -2z/(1 + z^2) = \psi_0(z) &\text{for }z \in [-z_0,z_0] \\
\,-\sin t_0 &\text{for }z \in [z_0,\infty),
\end{cases}
\end{align*}
where $z_0 := \cot(\pi u_0) = \cot(t_0/2)\approx 0.43\in (0,1)$ satisfies $z_0\arctan(1/z_0) = 1/2$. Thus, $\psi_0^*(z) = \psi_0\bigl((z \wedge z_0) \vee (-z_0)\bigr)$ for $z \in \R$; see the right panel of Figure~\ref{fig:psi0-star-cauchy}. An antiderivative $\phi_0^*$ of $\psi_0^*$ is given by
\[
\phi_0^*(z) := \int_0^z \psi_0^* - \log\pi = 
\begin{cases}
\,-\log\bigl(\pi(1 + z^2)\bigr) &\text{for }z \in [-z_0,z_0] \\
\,-(|z| - z_0)\sin t_0 - \log\bigl(\pi(1 + z_0^2)\bigr) &\text{for }z \in \R \setminus [-z_0,z_0].
\end{cases}
\]
As illustrated in the left panel of Figure~\ref{fig:cauchy}, $\ell_0^* := -\phi_0^*$ is a symmetric convex function that is approximately quadratic on $[-z_0,z_0]$ and linear outside this interval, so in this respect, it resembles the Huber loss function~\eqref{eq:huber-loss}. This is significant as far as $M$-estimation is concerned, since Huber-like loss functions are designed precisely to be robust to outliers, such as those that arise in regression problems with heavy-tailed Cauchy errors. As discussed in the introduction, $\ell_0^*$ is optimal in the sense that the resulting regression $M$-estimator $\hat{\beta}_{\psi_0^*} \in \argmin_{\beta \in \R^d}\sum_{i=1}^n \ell_0^*(Y_i - X_i^\top\beta)$ has minimal asymptotic covariance among all convex $M$-estimators. By direct computation, 
\[
\frac{1}{V_{p_0}(\psi_0^*)} = i^*(p_0) = \frac{1}{2} - \frac{2t_0\cos(2t_0) - \sin(2t_0)}{4\pi} \approx 0.439, \;\;\text{so}\;\;\mathrm{ARE}^*(p_0) = \frac{i^*(p_0)}{i(p_0)} \approx 0.878
\]
in this case, meaning that the restriction to convex loss functions results in only a small loss of efficiency relative to the maximum likelihood estimator. This may well be outweighed by the increased computational convenience of optimising a convex empirical risk function as opposed to a Cauchy likelihood function, which typically has several local extrema; see~\citet[Example~5.50]{vdV1998asymptotic} for a discussion of the difficulties involved. Lemma~\ref{lem:ARE-lower-bound} yields the bound $\mathrm{ARE}^*(p_0) \geq 8\norm{p_0}_\infty^2 = 8/\pi^2 \approx 0.811$.

\begin{figure}[htb]
\centering
\includegraphics[width=0.6\textwidth,trim={4cm 9.8cm 5cm 10.5cm},clip]{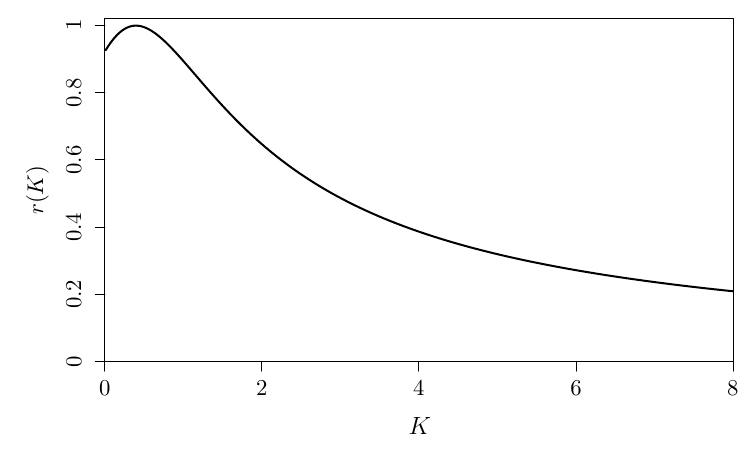}

\vspace{-0.3cm}
\caption{Plot of the asymptotic relative efficiency $r(K)$ of the Huber $M$-estimator $\hat{\beta}_{\psi_K}$ compared with the optimal convex $M$-estimator.}
\label{fig:cauchy-huber}
\end{figure}

For $K > 0$, the Huber regression $M$-estimator $\hat{\beta}_{\psi_K} \in \argmin_{\beta \in \R^d}\sum_{i=1}^n \ell_K(Y_i - X_i^\top\beta)$ defined with respect to~\eqref{eq:huber-loss} has asymptotic relative efficiency
\[
r(K) := \frac{V_{p_0}(\psi_0^*)}{V_{p_0}(\psi_K)} = \frac{4\arctan^2 K}{\pi\bigl(\pi K^2 + 2K - 2(1 + K^2)\arctan K\bigr)i^*(p_0)}
\]
compared with the optimal convex $M$-estimator $\hat{\beta}_{\psi_0^*}$; see Figure~\ref{fig:cauchy-huber}. The maximum value $\sup_{K > 0}r(K) \approx 0.9998$ is attained at $K^* \approx 0.394$. Moreover,
\[
\lim_{K \to 0} r(K) = \frac{4}{\pi^2 i^*(p_0)} \approx 0.922
\]
is the asymptotic relative efficiency $V_{p_0}(\psi_0^*)/V_{p_0}(\psi)$ of the least absolute deviation (LAD) estimator $\hat{\beta}_\psi \in \argmin_{\beta \in \R^d}\sum_{i=1}^n |Y_i - X_i^\top\beta|$, for which $\psi(\cdot) := -\sgn(\cdot)$ and $V_{p_0}(\psi) = 1/\bigl(4p_0(0)^2\bigr) = \pi^2/4$. On the other hand, the Huber loss $\ell_K$ in~\eqref{eq:huber-loss} converges pointwise to the squared error loss as $K \to \infty$, so $\lim_{K \to \infty} V_{p_0}(\psi_K) = \int_\R z^2\,p_0(z)\,dz = \infty$ and hence $\lim_{K \to \infty} r(K) = 0$. Recall from the discussion in the introduction the difficulties of choosing $K$, and its connection to the choice of scale. 

The Cauchy density $p_0$ and its log-concave Fisher divergence projection $p_0^* := e^{\phi_0^*}$ are plotted in the right panel of Figure~\ref{fig:cauchy}. Since $p_0 = p_0^*$ on $[-z_0,z_0]$ and $\psi_0^*$ is constant on $\R \setminus [-z_0,z_0]$, it turns out that $i^*(p_0) =  -\int_\R p_0\,d\psi_0^* = -\int_\R p_0^*\,d\psi_0^* = i(p_0^*)$, so both inequalities in Proposition~\ref{prop:p0-star-fisher}\textit{(b)} are in fact equalities in this example.
\end{example}

Section~\ref{subsec:appendix-examples} characterises the antitonic score projection for a variety of other densities $p_0$. In Example~\ref{ex:t2}, we take $P_0$ to be a scaled $t_2$ distribution, which has a finite first moment (unlike the Cauchy distribution in Example~\ref{ex:cauchy}), and verify that $\hat{\beta}_{\psi_0^{\mathrm{ML}}}$ and $\hat{\beta}_{\psi_0^*}$ in~\eqref{eq:betahat-ML-fisher} are different convex $M$-estimators. Example~\ref{ex:pareto} features a symmetrised Pareto density with polynomially decaying tails, where the optimal convex loss function $\ell_0^*$ is a scale transformation of the robust absolute error loss $z \mapsto |z|$. Moving on from heavy-tailed distributions, Example~\ref{ex:laplace-mixture} considers a density $p_0$ that fails to be log-concave because it is not unimodal, while Proposition~\ref{prop:gaussian-mixture} is a general result about the log-concave maximum likelihood and Fisher divergence projections of Gaussian mixtures.

\section{Semiparametric \texorpdfstring{$M$-estimation}{M-estimation} via antitonic score matching}
\label{sec:regression}

\subsection{Warm-up: Estimation of the projected score function from direct observations}
\label{subsec:score-estimation}

Section~\ref{sec:antitonic-proj} was concerned with the properties of the antitonic score projection on the population level.  Ultimately, our goal is to be able to incorporate these insights into a linear regression setting, but in this subsection we address an intermediate aim.  Specifically, we consider the nonparametric estimation of the antitonic projected score function $\psi_0^*$ based on a sample $\varepsilon_1,\dotsc,\varepsilon_n \iid p_0$.  This is an interesting problem in its own right, but an idealised one for the purposes of the regression setting that we have in mind, since there we do not observe the regression errors $\varepsilon_1,\dotsc,\varepsilon_n$ directly, and instead will need to rely on residuals from a pilot fit as proxies. 

We first explain why a naive approach to antitonic score matching fails before describing our alternative solution. For a locally absolutely continuous function $\psi \colon \R \to \R$ with derivative $\psi'$, the empirical analogue of $D_{p_0}(\psi) = \E\bigl(\psi^2(\varepsilon_1) + 2\psi'(\varepsilon_1)\bigr)$ in~\eqref{eq:Dp0-deriv} is 
\begin{equation}
\label{eq:score-matching-empirical}
\hat{D}_n(\psi) \equiv \hat{D}_n(\psi; \varepsilon_1,\dotsc,\varepsilon_n) := \frac{1}{n}\sum_{i=1}^n \{\psi^2(\varepsilon_i) + 2\psi'(\varepsilon_i)\}. 
\end{equation}
Recall from the introduction that score matching estimates the score function of a locally absolutely continuous density by an empirical risk minimiser $\hat{\psi}_n \in \argmin_{\psi \in \Psi}\hat{D}_n(\psi)$ over an appropriate class of functions $\Psi$. 

However, to obtain a monotone score estimate, we cannot minimise $\psi \mapsto \hat{D}_n(\psi)$ directly over the class $\Psi_\downarrow^{\mathrm{ac}}$ of all decreasing, locally absolutely continuous $\psi \colon \R \to \R$. Indeed, $\inf_{\psi \in \Psi_\downarrow^{\mathrm{ac}}}\hat{D}_n(\psi) = -\infty$, as can be seen by constructing differentiable approximations to a decreasing step function whose jumps are at the data points $\varepsilon_1,\dotsc,\varepsilon_n$. To circumvent this issue, we instead propose the following estimation strategy.

\medskip
\noindent
\textbf{Antitonic projected score estimation}:
Consider smoothing the empirical distribution of $\varepsilon_1,\dotsc,\varepsilon_n$, for example by convolving it with an  absolutely continuous kernel $K \colon \R \to \R$ to obtain a kernel density estimator $z \mapsto \tilde{p}_n(z) := n^{-1}\sum_{i=1}^n K_h(z - \varepsilon_i)$, where $h > 0$ is a suitable bandwidth and $K_h(\cdot) := h^{-1}K(\cdot/h)$. We can then define the smoothed empirical score matching objective
\[
\tilde{D}_n(\psi) := D_{\tilde{p}_n}(\psi) = \int_\R \psi^2 \,\tilde{p}_n + 2\int_{\mathcal{S}_0} \tilde{p}_n\,d\psi
\]
for $\psi \in \Psi_\downarrow(\tilde{p}_n)$, which approximates the population expectation in the definition of $D_{p_0}(\psi)$. Then by Theorem~\ref{thm:antitonic-score-proj},
\begin{equation}
\label{eq:kernel-proj-score}
\hat{J}_n^{(\mathrm{R})} \circ \tilde{F}_n \in \argmin_{\psi \in \Psi_\downarrow(\tilde{p}_n)}\tilde{D}_n(\psi),
\end{equation}
where $\tilde{F}_n$ denotes the distribution function corresponding to $\tilde{p}_n$, and $J_n := \tilde{p}_n \circ \tilde{F}_n^{-1}$. This is reminiscent of maximum smoothed likelihood estimation of a density or distribution function (e.g.~\citealp{eggermont2000maximum};~\citealp[Sections~8.2 and~8.5]{groeneboom14nonparametric}). By evaluating the antiderivative of $J_n$ on a suitably fine grid, we may obtain a piecewise affine approximation to $\hat{J}_n$ (and hence a piecewise constant approximation to its derivative) using PAVA, whose space and time complexities scale linearly with the size of the grid~\citep[Section~9.3.1]{samworth24modern}. 

\begin{figure}[htb]
\centering
\includegraphics[width=0.485\textwidth,trim={5cm 9.8cm 5cm 10cm},clip]{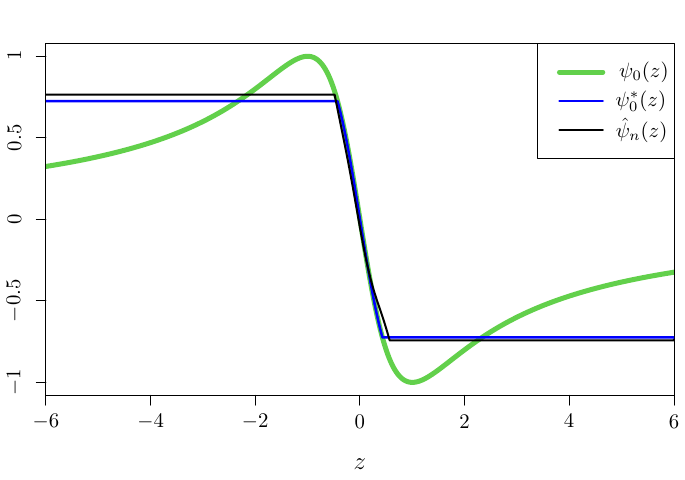}
\hfill
\includegraphics[width=0.485\textwidth,trim={5cm 9.8cm 5cm 10cm},clip]{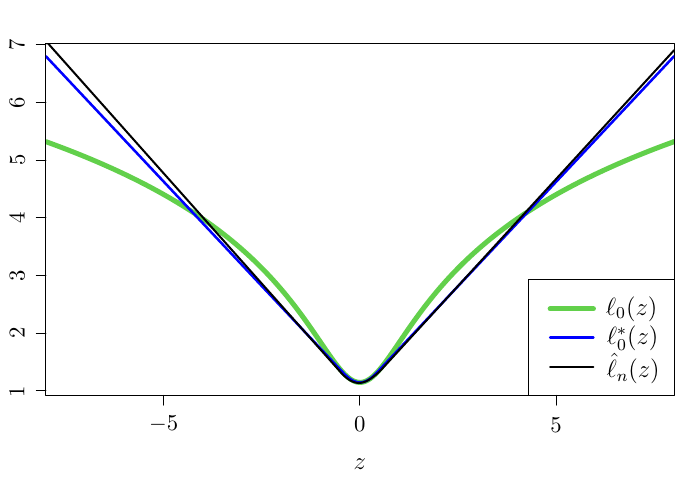}

\vspace{-0.3cm}
\caption{Kernel-based estimates of the projected score function and optimal convex loss function based on a sample of size $n = 2000$ from the Cauchy distribution.}
\label{fig:cauchy-kernel-score}
\end{figure}

More generally, we can use $\varepsilon_1,\dotsc,\varepsilon_n$ to construct a generic (not necessarily monotone) score estimator~$\tilde{\psi}_n$ and an estimate $\hat{F}_n$ of the distribution function $F_0$ corresponding to the density $p_0$. By analogy with the explicit representation~\eqref{eq:psi0-star} of $\psi_0^*$, we then define the decreasing score estimate
\begin{equation}
\label{eq:monotone-score}
\hat{\psi}_n := \widehat{\mathcal{M}}_\mathrm{R}(\tilde{\psi}_n \circ \hat{F}_n^{-1}) \circ \hat{F}_n.
\end{equation}
As explained above, (an approximation to) $\hat{\psi}_n$ can be computed efficiently using isotonic regression algorithms. In particular, if $\hat{F}_n$ is taken to be the empirical distribution function of $\varepsilon_1,\dotsc,\varepsilon_n$, then by Proposition~\ref{prop:isoproj}, $\hat{\psi}_n^{(\mathrm{L})} := \widehat{\mathcal{M}}_\mathrm{L}(\tilde{\psi}_n \circ \hat{F}_n^{-1}) \circ \hat{F}_n$ is an antitonic least squares estimator based on $\bigl\{\bigl(\varepsilon_i, \tilde{\psi}_n(\varepsilon_i)\bigr) : i \in [n]\bigr\}$. Our decreasing score estimate can be taken to be either $\hat{\psi}_n^{(\mathrm{L})}$ or the closely related $\hat{\psi}_n$: for every $z \in \R$, we have $\hat{\psi}_n(z) = \hat{\psi}_n^{(\mathrm{L})}\bigl(\varepsilon(z)\bigr)$, where $\varepsilon(z)$ is the smallest element of $\{\varepsilon_1,\dotsc,\varepsilon_n\}$ that is strictly greater than $z$ (where such an element exists) or equal to $\max_{i \in [n]} \varepsilon_i$ otherwise.

\begin{figure}[htb]
\centering
\includegraphics[width=0.485\textwidth,trim={5cm 9.8cm 5cm 10cm},clip]{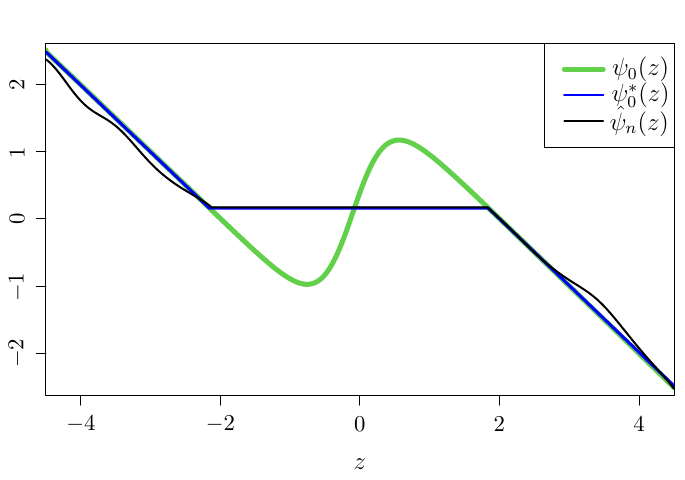}
\hfill
\includegraphics[width=0.485\textwidth,trim={5cm 9.8cm 5cm 10cm},clip]{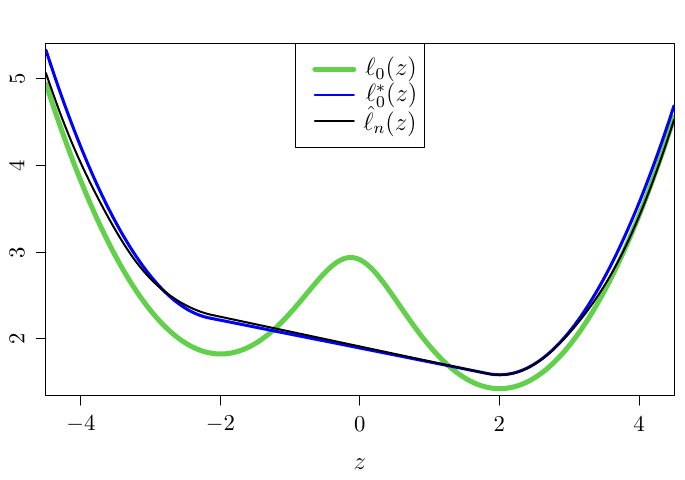}

\vspace{-0.3cm}
\caption{Kernel-based estimates of the projected score function and optimal convex loss function based on a sample of size $n = 10^4$ from the Gaussian mixture distribution $0.4 N(-2,1) + 0.6 N(2,1)$.}
\label{fig:gaussian-kernel-score}
\end{figure}

The transformation~\eqref{eq:monotone-score} may be applied to any appropriate initial score estimator $\tilde{\psi}_n$. Since a misspecified parametric method may introduce significant error at the outset, we seek a nonparametric estimator. For instance, we may take $\tilde{\psi}_n$ to be a ratio of kernel density estimates of $p_0'$ and $p_0$, which may be truncated for theoretical and practical convenience to avoid instability in low-density regions; see~\eqref{eq:psi-kernel} in Section~\ref{subsec:linreg-sym}. Observe that~\eqref{eq:kernel-proj-score} is a special case of~\eqref{eq:monotone-score} with $\tilde{\psi}_n = \tilde{p}_n'/\tilde{p}_n$ and $\hat{F}_n = \tilde{F}_n$.

\subsection{Linear regression: alternating algorithm outline}
\label{subsec:lin_reg_id_alt}

Suppose that we observe independent and identically distributed pairs $(X_1,Y_1),\dotsc,(X_n,Y_n)$ satisfying
\begin{equation}
\label{eq:linear-model}
Y_i = X_i^\top\beta_0 + \varepsilon_i
\end{equation}
for $i \in [n]$, where $X_1,\dotsc,X_n$ are $\R^d$-valued covariates that are independent of errors $\varepsilon_1,\dotsc,\varepsilon_n$ with an unknown absolutely continuous (Lebesgue) density $p_0$ on $\R$. A necessary condition for $\beta_0$ to be identifiable is that $\E(X_1 X_1^\top)$ is positive definite, since otherwise there exists $v \in \R^d \setminus \{0\}$ such that $X_i^\top v = 0$ for all $i$ almost surely. In this case, the joint distribution of our observed data is unchanged if we replace $\beta_0$ with $\beta_0 + v$. To accommodate heavy-tailed (e.g.~Cauchy) errors, we do not necessarily insist that $\E(\varepsilon_1) = 0$, or even suppose that~$\varepsilon_1$ is integrable, though see also the discussion at the start of Section~\ref{subsec:linreg-intercept}.

As the regression coefficients $\beta_0$, the error density $p_0$ and hence the antitonic score projected score $\psi_0^*$ are unknown, a natural estimation strategy on the population level is to alternate between the following two steps:
\begin{enumerate}[label=\Roman*.]
\item For a fixed $\beta$, minimise the (convex) score matching objective $D_{q_\beta}(\psi)$ based on the density $q_\beta$ of $Y_1 - X_1^\top\beta$.
\item For a fixed decreasing and right-continuous $\psi$, minimise the convex function $\beta \mapsto \E\ell(Y_1 - X_1^\top \beta)$, where $\ell$ is a negative antiderivative of $\psi$.  
\end{enumerate}
This approach can be motivated by a joint optimisation problem over a set of pairs $(\beta,\psi)$; see Section~\ref{subsec:joint-opt}. In an empirical version of this alternating algorithm based on $(X_1,Y_1),\dotsc,(X_n,Y_n)$, we can estimate $\psi_0^*$ in Step~I by minimising a sample analogue of $D_{q_\beta}(\psi)$ over $\psi$. As discussed in Section~\ref{subsec:score-estimation}, the unknown density of $Y_1 - X_1^\top\beta$ can be approximated by smoothing the empirical distribution of the residuals $(Y_i - X_i^\top\beta)_{i=1}^n$ from the current estimate of $\beta_0$. Step II then involves finding an $M$-estimator~\eqref{eq:linreg-M-est} of $\beta_0$ based on the convex loss function induced by the current estimate of $\psi_0^*$. In practice, we can apply any suitable convex optimisation algorithm such as gradient descent, with Newton's method being a faster alternative when the score estimate is differentiable. Steps I and II can then be iterated to convergence.

The following two subsections focus in turn on linear regression with symmetric errors and with an explicit intercept term. We will analyse a specific version of the above procedure that is initialised with a pilot estimator $\bar{\beta}_n$ of $\beta_0$. Provided that $(\bar{\beta}_n)$ is $\sqrt{n}$-consistent, we show that a single iteration of Steps~I and~II yields a semiparametric convex $M$-estimator of $\beta_0$ that achieves `antitonic efficiency' as $n \to \infty$.

\subsection{Linear regression with symmetric errors}
\label{subsec:linreg-sym}

Under the assumption that $p_0$ is symmetric, we first approximate $\psi_0^*$ via antitonic projection of kernel-based score estimators. We do not observe the errors $\varepsilon_1,\ldots,\varepsilon_n$ directly, so in view of the discussion above, in the algorithm below we use the residuals from the pilot regression estimator to construct our initial score estimators. Assume throughout that $n \geq 3$.

\begin{enumerate}[leftmargin=0.45cm]
\item \textbf{Sample splitting}: Partition the observations into three folds indexed by disjoint $I_1,I_2,I_3 \subseteq [n]$ such that $|I_1| = |I_2| = \floor{n/3}$ and $|I_3| = n - |I_1| - |I_2|$ respectively. For notational convenience, let $I_{j+3} := I_j$ for $j \in \{1,2\}$.
\item \textbf{Pilot estimators}: Fix a convex function $L \colon \R \to \R$, and for $j \in \{1,2,3\}$, let $\bar{\beta}_n^{(j)}$ be an $M$-estimator
\begin{equation}
\label{eq:pilot-est}
\bar{\beta}_n^{(j)} \in \argmin_{\beta \in \R^d} \sum_{i \in I_j} L(Y_i - X_i^\top\beta).
\end{equation}
\item \textbf{Antitonic projected score estimation}: For $j \in \{1,2,3\}$ and $i \in I_{j+1}$, define out-of-sample residuals $\hat{\varepsilon}_i := Y_i - X_i^\top\bar{\beta}_n^{(j)}$. Letting $K \colon \R \to [0,\infty)$ be a differentiable kernel and $h \equiv h_n > 0$ be a (deterministic) bandwidth, define a kernel density estimator $\tilde{p}_{n,j}$ of $p_0$ by
\[
\tilde{p}_{n,j}(z) := \frac{1}{|I_{j+1}|}\sum_{i \in I_{j+1}} K_h(z - \hat{\varepsilon}_i)
\]
for $z \in \R$, where $K_h(\cdot) = h^{-1}K(\cdot/h)$. In addition, let $\tilde{S}_{n,j} := \bigl\{z \in \R : |\tilde{p}_{n,j}'(z)| \leq \alpha_n,\, \tilde{p}_{n,j}(z) \geq \gamma_n\bigr\}$, where $\alpha_n \in (0,\infty]$ and $\gamma_n \in (0,\infty)$ are truncation parameters, and define $\tilde{\psi}_{n,j} \colon \R \to \R$ by
\begin{equation}
\label{eq:psi-kernel}
\tilde{\psi}_{n,j}(z) := \frac{\tilde{p}_{n,j}'(z)}{\tilde{p}_{n,j}(z)}\Ind_{\{z \in \tilde{S}_{n,j}\}}.
\end{equation}
Writing $\tilde{F}_{n,j}$ for the distribution function corresponding to $\tilde{p}_{n,j}$, let $\hat{\psi}_{n,j} := \widehat{\mathcal{M}}_\mathrm{R}(\tilde{\psi}_{n,j} \circ \tilde{F}_{n,j}^{-1}) \circ \tilde{F}_{n,j}$ be an antitonic projected score estimate, in accordance with~\eqref{eq:monotone-score}. Finally, define an estimator $\hat{\psi}_{n,j}^{\mathrm{anti}} \in \Psi_\downarrow^{\mathrm{anti}}(p_0)$ of $\psi_0^*$ by
\[
\hat{\psi}_{n,j}^{\mathrm{anti}}(z) := \frac{\hat{\psi}_{n,j}(z) - \hat{\psi}_{n,j}(-z)}{2}
\]
for $z \in \R$.
\item \textbf{Plug-in cross-fitted convex $M$-estimator}: For $j \in \{1,2,3\}$, let $\hat{\ell}_{n,j}^{\mathrm{sym}} \colon \R \to \R$ be the induced convex loss function given by $\hat{\ell}_{n,j}^{\mathrm{sym}}(z) := -\int_0^z \hat{\psi}_{n,j}^{\mathrm{anti}}$, and define
\begin{equation}
\label{eq:betahat-sym}
\hat{\beta}_n^{(j)} \in \argmin_{\beta \in \R^d} \sum_{i \in I_{j+2}} \hat{\ell}_{n,j}^{\mathrm{sym}}(Y_i - X_i^\top\beta)
\end{equation}
to be a corresponding $M$-estimator of $\beta_0$. Finally, let
\begin{equation}
\label{eq:DML1}
\hat{\beta}_n^\dagger := \frac{\hat{\beta}_n^{(1)} + \hat{\beta}_n^{(2)} + \hat{\beta}_n^{(3)}}{3}.
\end{equation}
\end{enumerate}

In~\eqref{eq:betahat-sym}, $\hat{\beta}_n^{(j)}$ always exists because either $\hat{\psi}_{n,j}^{\mathrm{anti}} \equiv 0$ and any $\beta \in \R^d$ is a minimiser, or otherwise $\inf_{z \in \R}\hat{\psi}_{n,j}^{\mathrm{anti}}(z) < 0 < \sup_{z \in \R}\hat{\psi}_{n,j}^{\mathrm{anti}}(z)$ and hence the convex function $\hat{\ell}_{n,j}^{\mathrm{sym}}$ is \emph{coercive} in the sense that~$\hat{\ell}_{n,j}^{\mathrm{sym}}(z) \to \infty$ as $|z| \to \infty$. Moreover, $\hat{\beta}_n^{(j)}$ is unique if $\hat{\psi}_{n,j}^{\mathrm{anti}}$ is strictly decreasing and the design matrix $X$ has full column rank, which happens with probability tending to 1 as $n \to \infty$ if $\E(X_1 X_1^\top)$ is invertible; see Proposition~\ref{prop:cvx-M-est-asymp}\textit{(a)} for elementary justifications of these claims. In practice, our antitonic score estimates may have constant pieces, but the conclusion of Theorem~\ref{thm:linreg-score-sym} below applies to all sequences of minimisers $(\hat{\beta}_n^{(j)})$, so is unaffected by non-uniqueness issues.
% For each $j \in \{1,2,3\}$, we claim that $\hat{\beta}_n^{(j)}$ always exists in Step~4. Indeed, if $\hat{\psi}_{n,j}^{\mathrm{anti}} \equiv 0$, then $\hat{\ell}_{n,j}^{\mathrm{sym}} \equiv 0$ and any $\beta \in \R^d$ minimises the objective function in~\eqref{eq:betahat-sym}. Otherwise, $\inf_{z \in \R}\hat{\psi}_{n,j}^{\mathrm{anti}}(z) < 0 < \sup_{z \in \R}\hat{\psi}_{n,j}^{\mathrm{anti}}(z)$ and hence $\hat{\ell}_{n,j}^{\mathrm{sym}}$ is a finite convex function on $\R$ that is \emph{coercive} in the sense that~$\hat{\ell}_{n,j}^{\mathrm{sym}}(z) \to \infty$ as $|z| \to \infty$. Thus, $\theta \equiv (\theta_1,\dotsc,\theta_n) \mapsto \sum_{i \in I_{j+2}} \hat{\ell}_{n,j}^{\mathrm{sym}}(Y_i - \theta_i) =: \mathcal{L}_{n,j}(\theta)$ is convex and coercive on $\{X\beta : \beta \in \R^d\}$, and hence $\beta \mapsto \mathcal{L}_{n,j}(X\beta)$ attains its minimum on $\R^d$. A necessary condition for $\hat{\beta}_n^{(j)}$ to be unique is that $X$ has full column rank, which happens with probability tending to 1 as $n \to \infty$ if $\E(X_1 X_1^\top)$ is invertible. In this case, uniqueness holds when $\hat{\psi}_{n,j}^{\mathrm{anti}}$ is strictly decreasing and hence $\mathcal{L}_{n,j}$ is strictly convex on $\R^n$.

The above procedure uses the observations indexed by $I_1,I_2,I_3$ to construct the pilot estimator~$\bar{\beta}_n^{(1)}$, antitonic score estimate $\hat{\psi}_{n,1}^{\mathrm{anti}}$ and semiparametric $M$-estimator $\hat{\beta}_n^{(1)}$ respectively. Since each fold (specifically the last one) contains only about one-third of all the data, sample splitting reduces the efficiency of $\hat{\beta}_n^{(1)}$. We remedy this by \textit{cross-fitting}~(\citealp[p.~393]{vdV1998asymptotic}; \citealp[Section~3]{chernozhukov2018double}), which involves cyclically permuting the folds to obtain $\hat{\beta}_n^{(2)},\hat{\beta}_n^{(3)}$ analogously to $\hat{\beta}_n^{(1)}$, and then averaging these three estimators. This reduces the limiting covariance of $\hat{\beta}_n^{(1)}$ by a factor of three in the theory below, where we show that $\hat{\beta}_n^{(1)},\hat{\beta}_n^{(2)},\hat{\beta}_n^{(3)}$ are `asymptotically independent' in a precise sense. 

A different version of cross-fitting~\citep[Definition~3.2]{chernozhukov2018double} instead averages the empirical risk functions across all three folds, and outputs a single estimator
\begin{equation}
\label{eq:DML2}
\hat{\beta}_n^\ddagger \in \argmin_{\beta \in \R^d} \sum_{j=1}^3 \sum_{i \in I_{j+2}} \hat{\ell}_{n,j}^{\mathrm{sym}}(Y_i - X_i^\top\beta),
\end{equation}
whose existence is similarly guaranteed by the convexity of $\hat{\ell}_{n,j}^{\mathrm{sym}}$ for $j \in \{1,2,3\}$.

To introduce our theoretical guarantees for this procedure, for a sequence of regression models~\eqref{eq:linear-model} indexed by $n \in \N$, we make the following assumptions on the model and the parameters in our procedure. 
\begin{enumerate}[label=(A\arabic*)]
\item \label{ass:fisher-finite} $p_0$ is an absolutely continuous density on $\R$ such that $i(p_0) < \infty$ and $\int_\R |z|^\delta\,p_0(z)\,dz < \infty$ for some~$\delta > 0$.
\item \label{ass:psi0-star} There exists $t_0 > 0$ such that for $t \in \{-t_0,t_0\}$, the antitonic projected score function $\psi_0^*$ satisfies $\int_\R \psi_0^*(z + t)^2\,p_0(z)\,dz < \infty$.
\item \label{ass:kernel} The kernel $K$ is non-negative, twice continuously differentiable and supported on $[-1,1]$.
\item \label{ass:alpha-gamma-h} $\alpha_n \to \infty$,\; $\gamma_n, h_n \to 0$,\; $nh_n^3\gamma_n^2 \to \infty$ and $(h_n \vee n^{-2\rho/3})(\alpha_n/\gamma_n)^2 \to 0$ for some $\rho \in \bigl(0,\delta/(\delta + 1)\bigr)$.
\item \label{ass:covariates} $\E(X_1 X_1^\top) \in \R^{d \times d}$ is positive definite and $\max_{i \in [n]}\norm{X_i}\,\alpha_n/\gamma_n = o_p(n^{1/2})$ as $n \to \infty$.
\end{enumerate}

The conditions~\ref{ass:fisher-finite} and~\ref{ass:psi0-star} 
% only exclude densities with extremely heavy or extremely light tails respectively, and 
are satisfied by a wide variety of commonly-encountered densities $p_0$ ranging from all $t_\nu$ densities with $\nu > 0$ degrees of freedom (including the Cauchy density as a special case $\nu = 1$) to lighter-tailed Weibull, Laplace, Gaussian and Gumbel densities. In particular, $P_0$ need not have a finite mean, and our procedure does not require knowledge of the exponent $\delta > 0$ in~\ref{ass:fisher-finite}. By Lemma~\ref{lem:psi0-star}, $\{z \in \R : \psi_0^*(z) \in \R\} = \mathcal{S}_0 = \bigl(\inf(\supp p_0), \sup(\supp p_0)\bigr)$, so if~\ref{ass:psi0-star} holds, then $\mathcal{S}_0 = \R$ and $\psi_0^*$ must be finite-valued on $\R$. The truncation parameters $\alpha_n,\gamma_n$ and bandwidth $h_n$ can be chosen quite flexibly; for instance,~\ref{ass:alpha-gamma-h} holds if $\alpha_n = \gamma_n^{-1} = \log n$ and $h_n = n^{-b}$ for some $b \in (0,1/3)$. As for~\ref{ass:covariates}, the fact that $\norm{X_1}^2,\dotsc,\norm{X_n}^2$ are identically distributed and integrable means that $\max_{i \in [n]}\norm{X_i} = o_p(n^{1/2})$; see~\eqref{eq:expected-max}. Our condition is slightly stronger than this to account for the score estimators being uniformly bounded in absolute value by $\alpha_n/\gamma_n$. In particular, if $\E(\norm{X_1}^3) < \infty$, then $\max_{i \in [n]}\norm{X_i} = o_p(n^{1/3})$, so under~\ref{ass:alpha-gamma-h}, we have $\alpha_n/\gamma_n = o_p(h_n^{-1/2}) = o_p(n^{1/6})$ and hence~\ref{ass:covariates} holds automatically.

We also require the pilot estimators in Step~2 to exist and be $\sqrt{n}$-consistent for $\beta_0$. By a classical result (see Proposition~\ref{prop:cvx-M-est-asymp}), this is guaranteed if the loss function $L$ in this step has negative right derivative~$\varphi$ satisfying both $\E\varphi(\varepsilon_1) = 0$ and the following condition.

\begin{enumerate}[label=(B)]
\item \label{ass:zeta} $\varphi$ is decreasing and right-continuous with $\inf_{z \in \R}\varphi(z) < 0 < \sup_{z \in \R}\varphi(z)$. In addition,
\[
V_{p_0}(\varphi) = \frac{\int_\R \varphi^2\,p_0}{(\int_\R p_0\,d\varphi)^2} \in (0,\infty)
\]
and there exists $t_0 > 0$ such that $\int_\R \varphi(z + t)^2\,p_0(z)\,dz < \infty$ for $t \in \{-t_0,t_0\}$.
\end{enumerate}
In particular, if $L$ is a symmetric and twice differentiable convex function whose derivative is strictly increasing and bounded, then~\ref{ass:zeta} holds and $\int_\R \varphi\,p_0 = 0$ for all symmetric densities $p_0$.  Under the assumptions above, we first prove the $L^2(P_0)$-consistency of the initial estimates $\tilde{\psi}_{n,j}$ of the score function $\psi_0$, from which it follows that the antitonic functions $\hat{\psi}_{n,j}$ consistently estimate the population-level projected score $\psi_0^*$ in $L^2(P_0)$; see Lemmas~\ref{lem:kernel-score-consistency} and~\ref{lem:proj-score-consistency} in Section~\ref{subsubsec:score-estimation}. This enables us to establish the following result.

\begin{theorem}
\label{thm:linreg-score-sym}
Assume that~\emph{\ref{ass:fisher-finite}--\ref{ass:covariates}} hold for the linear model~\eqref{eq:linear-model} with symmetric error density~$p_0$.  If $\bar{\beta}_n^{(j)} - \beta_0 = O_p(n^{-1/2})$ for $j \in \{1,2,3\}$, then for any sequence of estimators $(\hat{\beta}_n^{\mathrm{sym}})$ with $\hat{\beta}_n^{\mathrm{sym}} \in \{\hat{\beta}_n^\dagger,\hat{\beta}_n^\ddagger\}$ for each $n$, we have
\[
\sqrt{n}(\hat{\beta}_n^{\mathrm{sym}} - \beta_0) \cvd N_d\biggl(0, \frac{\{\E(X_1 X_1^\top)\}^{-1}}{i^*(p_0)}\biggr)
\]
as $n \to \infty$.
\end{theorem}

Therefore, our semiparametric convex $M$-estimators $\hat{\beta}_n^\dagger,\hat{\beta}_n^\ddagger$ are $\sqrt{n}$-consistent and have the same limiting Gaussian distribution as the `oracle' convex $M$-estimator $\hat{\beta}_{\psi_0^*} := \argmin_{\beta \in \R^d} \sum_{i=1}^n \ell_0^*(Y_i - X_i^\top\beta)$, where $\ell_0^*$ denotes an optimal convex loss function with right derivative $\psi_0^*$.

To understand the form of the asymptotic covariance matrix in Theorem~\ref{thm:linreg-score-sym}, observe that since $X_1$ and $\varepsilon_1$ are independent, $(X_1,Y_1)$ has joint density $(x,y) \mapsto p_0(y - x^\top\beta_0)$ with respect to the product measure $P_X \otimes \mathrm{Leb}$ on $\R^d \times \R$, where we write $P_X$ for the distribution of $X_1$, and $\mathrm{Leb}$ for Lebesgue measure on $\R$. Therefore, in a parametric model where $p_0$ is known, the score function $\dot{\ell}_{\beta_0} \colon \R^d \times \R \to \R^d$ for $\beta_0$ is given by
\[
\dot{\ell}_{\beta_0}(x,y) := -x\psi_0(y - x^\top\beta_0),
\]
where $\psi_0 = p_0'/p_0$. If $i(p_0) = \E\bigl(\psi_0(\varepsilon_1)^2\bigr)$ is finite, then because $\E\psi_0(\varepsilon_1) = 0$, the Fisher information matrix for $\beta_0$ is
\begin{equation}
\label{eq:Ibeta0}
I_{\beta_0} := \Cov\dot{\ell}_{\beta_0}(X_1,Y_1) = \Cov\{X_1\psi_0(\varepsilon_1)\} = \E\bigl\{X_1 X_1^\top\psi_0(\varepsilon_1)^2\bigr\} = \E(X_1 X_1^\top)\,i(p_0) \in \R^{d \times d}.
\end{equation}
Since $i(p_0) \geq i^*(p_0) > 0$ by Theorem~\ref{thm:antitonic-score-proj}\textit{(d)}, $I_{\beta_0}$ is positive definite if and only if $\E(X_1 X_1^\top)$ is positive definite. In this case, the convolution and local asymptotic minimax theorems~\citep[Chapter~8]{vdV1998asymptotic} indicate that $\sqrt{n}(\hat{\beta}^{\mathrm{MLE}} - \beta_0)$ in~\eqref{eq:betahat-MLE-asymp} has the `optimal' limiting distribution
\[
N_d(0, I_{\beta_0}^{-1}) = N_d\biggl(0, \frac{\{\E(X_1 X_1^\top)\}^{-1}}{i(p_0)}\biggr)
\]
among all (regular) sequences of estimators of $\beta_0$. By analogy with the previous display, the limiting covariance in Theorem~\ref{thm:linreg-score-sym} can be written as the inverse $(I^*_{\beta_0})^{-1}$ of the \textit{antitonic information matrix} $I^*_{\beta_0} := \E(X_1 X_1^\top)\,i^*(p_0)$. By Theorem~\ref{thm:antitonic-score-proj}, $1/i^*(p_0) = V_{p_0}(\psi_0^*) = \min_{\psi \in \Psi_\downarrow(p_0)}V_{p_0}(\psi)$, so by Proposition~\ref{prop:cvx-M-est-asymp}, $(I^*_{\beta_0})^{-1}$ is the smallest possible limiting covariance among all convex $M$-estimators $\hat{\beta}_\psi$ based on a fixed $\psi \in \Psi_\downarrow(p_0)$. We can therefore interpret $(I^*_{\beta_0})^{-1}$ as an \textit{antitonic efficiency lower bound}.

\subsection{Linear regression with an intercept term}
\label{subsec:linreg-intercept}
%In Proposition~\ref{prop:joint-min}\emph{(b)}, we do not assume that $p_0$ is symmetric, but instead consider a linear model where $d \geq 2$ and
For $d \geq 2$, now consider the linear model
\begin{equation}
\label{eq:linear-model-intercept}
Y_i = \mu_0 + \tilde{X}_i^\top\theta_0 + \varepsilon_i \quad\text{for }i \in [n],
\end{equation}
where $\mu_0$ is an explicit intercept term, so that $\beta_0 = (\theta_0,\mu_0)$ and $X_i = (\tilde{X}_i,1)$ in~\eqref{eq:linear-model} for $i \in [n]$.  In the absence of further restrictions on the distribution of $\varepsilon_1$, the intercept term in this model is non-identifiable since we may add a scalar to $\mu_0$ and make a corresponding location shift to the distribution of $\varepsilon_1$ without changing the distribution of $(X_1,Y_1)$. One could restore identifiability by including an assumption that $\E(\varepsilon_1) = 0$. However, to incorporate the potential for heavy-tailed error distributions without a finite first moment (such as the Cauchy distribution), we instead impose a more general centring condition of the form $\E\zeta(\varepsilon_1) = 0$ for some pre-specified decreasing function $\zeta \colon \R \to \R$ that satisfies condition~\ref{ass:zeta} with $\varphi = \zeta$. In this case, define
\[
\upsilon_{p_0} := V_{p_0}(\zeta) \in (0,\infty).
\]
Naturally, taking $\zeta$ to be the function $z \mapsto -z$ yields the usual mean-zero assumption on the errors; on the other hand, for $\tau \in (0,1)$, letting $\zeta(z) = \Ind_{\{z < 0\}} - \tau$ for $z \in \R$ constrains the errors to have $\tau$-quantile equal to 0. In these two examples, $\upsilon_{p_0} = \E(\varepsilon_1^2)$ and $\upsilon_{p_0} = \tau(1 - \tau)/p_0(0)^2$ 
% J_0(\tau) = p_0(0)
respectively, and when $\upsilon_{p_0} \in (0,\infty)$, we automatically have $\int_\R \zeta(z + t)^2\,p_0(z)\,dz < \infty$ for all $t \in \R$.
In general, the fact that $\int_\R p_0\,d\zeta \in (-\infty,0)$ ensures that $\E\zeta(\varepsilon_1 - c) = 0$ if and only if $c = 0$, and hence that $\mu_0$ is
% the unique solution to
identified by the equation $\E\zeta(Y_1 - \mu_0 - \tilde{X}_1^\top\theta_0) = 0$; see~\eqref{eq:zeta-Lambda-deriv}. 

Similarly to Section~\ref{subsec:linreg-sym}, we employ three-fold cross-fitting with the convention $I_{j+3} = I_j$ for $j \in \{1,2\}$, and obtain pilot estimators $\bar{\beta}_n^{(j)} = (\bar{\theta}_n^{(j)},\bar{\mu}_n^{(j)})$ of $\beta_0 = (\theta_0,\mu_0)$ given by~\eqref{eq:pilot-est} for $j \in \{1,2,3\}$ based on a fixed loss function $L$. We require the estimators $\bar{\theta}_n^{(j)}$ to be $\sqrt{n}$-consistent for $\theta_0$ as $n \to \infty$. This is guaranteed by Proposition~\ref{prop:cvx-M-est-asymp} if condition~\ref{ass:zeta} is satisfied by $\varphi = -L^{(\mathrm{R})}$, so here we can either take $L$ to be a twice differentiable convex loss function with a strictly decreasing and bounded derivative, or let $L$ be a negative antiderivative of $\zeta$.
% In the latter case, $\bar{\beta}_n^{(j)}$ is $\sqrt{n}$-consistent for $\beta_0$

We make some modifications to subsequent steps of the previous antitonic score matching procedure.

\begin{enumerate}[leftmargin=0.55cm,label=\arabic*$'$.]
\setcounter{enumi}{2}
\item \textbf{Antitonic projected score estimation}: For $j \in \{1,2,3\}$ and $i \in I_{j+1}$, use the out-of-sample residuals $\hat{\varepsilon}_i := Y_i - \tilde{X}_i^\top\bar{\theta}_n^{(j)}$ to construct the initial kernel-based score estimator $\tilde{\psi}_{n,j}$ and its antitonic projection $\hat{\psi}_{n,j} := \widehat{\mathcal{M}}_\mathrm{R}(\tilde{\psi}_{n,j} \circ \tilde{F}_{n,j}^{-1}) \circ \tilde{F}_{n,j}$ as before. Since $p_0$ is not symmetric in general, we use $\hat{\psi}_{n,j}$ instead of $\hat{\psi}_{n,j}^{\mathrm{anti}}$.
\item \textbf{Plug-in cross-fitted convex $M$-estimator}: For $j \in \{1,2,3\}$, let $\hat{\ell}_{n,j} \colon \R \to \R$ be the induced convex loss function given by $\hat{\ell}_{n,j}(z) := -\int_0^z \hat{\psi}_{n,j}$, and define
\begin{align}
\label{eq:betahat-intercept}
\hat{\theta}_n^{(j)} \in \argmin_{\theta \in \R^{d-1}} \sum_{i \in I_{j+2}} \hat{\ell}_{n,j}\bigl(Y_i - \bar{X}_{n,j}^\top\bar{\theta}_n^{(j)} - (\tilde{X}_i - \bar{X}_{n,j})^\top\theta\bigr),
\end{align}
where $\bar{X}_{n,j} := |I_{j+2}|^{-1}\sum_{i \in I_{j+2}}\tilde{X}_i$. Finally, let $\hat{\theta}_n^\dagger := \bigl(\hat{\theta}_n^{(1)} + \hat{\theta}_n^{(2)} + \hat{\theta}_n^{(3)}\bigr)/3$. Alternatively, define
\[
\hat{\theta}_n^\ddagger \in \argmin_{\theta \in \R^{d-1}} \sum_{j=1}^3 \sum_{i \in I_{j+2}} \hat{\ell}_{n,j}\bigl(Y_i - \bar{X}_{n,j}^\top\bar{\theta}_n^{(j)} - (\tilde{X}_i - \bar{X}_{n,j})^\top\theta\bigr).
\]
\item \textbf{Intercept estimation}: Taking either $\hat{\theta}_n = \hat{\theta}_n^\dagger$ or $\hat{\theta}_n = \hat{\theta}_n^\ddagger$, let
\begin{equation}
\label{eq:intercept-est}
\hat{\mu}_n^\zeta \in \argmin_{\mu \in \R} \sum_{i=1}^n L_\zeta(Y_i - \tilde{X}_i^\top\hat{\theta}_n - \mu),
\end{equation}
where $L_\zeta$ is a negative antiderivative of $\zeta$. Finally, output $\hat{\beta}_n^\zeta := (\hat{\theta}_n,\hat{\mu}_n^\zeta)$. 
\end{enumerate}

In summary, $\hat{\theta}_n$ minimises an empirical risk based on centred covariates and an estimated convex loss (via antitonic score matching), while $\hat{\mu}_n^\zeta$ is defined as a location $M$-estimator with respect to the residuals from $\hat{\theta}_n$ and the fixed loss function used to centre the regression errors. There exists a minimiser $\hat{\theta}_n^{(j)}$ in~\eqref{eq:betahat-intercept} if $\inf_{z \in \R}\hat{\psi}_{n,j}(z) < 0 < \sup_{z \in \R}\hat{\psi}_{n,j}(z)$, and $\hat{\theta}_n^{(j)}$ is unique if $\hat{\psi}_{n,j}$ is strictly decreasing and the design matrix has full column rank; see Proposition~\ref{prop:cvx-M-est-asymp}\textit{(a)}. 

% We cannot hope to estimate the full vector $\beta_0$ since the intercept $\mu_0$ is unidentifiable in the absence of further restrictions on $p_0$. We therefore view~\eqref{eq:linear-model-intercept} as a semiparametric model and aim to estimate $\theta_0 \in \R^{d-1}$ in the presence of nuisance parameters $\mu_0$ and $p_0$. 
% As mentioned in Section~\ref{subsec:lin_reg_id_alt}, a necessary condition for $\theta_0$ to be identifiable is that
% \[
% \E(X_1 X_1^\top) = 
% \begin{pmatrix} 
% \E(\tilde{X}_1 \tilde{X}_1^\top) & \!\!\!\E(\tilde{X}_1) 
% \\ \E(\tilde{X}_1)^\top & \!\!\!1 
% \end{pmatrix}
% \]  
% is positive definite, which is equivalent to the Schur complement of $1$ in $\E(X_1 X_1^\top)$, i.e.~$\Sigma := \Cov(\tilde{X}_1)$, being positive definite. 

Step $4'$ can be motivated by 
% standard 
semiparametric calculations that are summarised briefly below and presented formally in Section~\ref{subsec:linreg-semiparametric}. We verify in Proposition~\ref{prop:linreg-semiparametric}\textit{(a)} that if $\mu_0$ is viewed as a nuisance parameter, then
% In a parametric model where $p_0$ is known and $\beta_0$ is unknown, the efficiency lower bound for $\theta_0$ is the top-left $(d-1) \times (d-1)$ submatrix of $I_{\beta_0}^{-1}$, which by the blockwise inversion formula~\citep[e.g.][Proposition~10.10.2]{samworth24modern} is $\Sigma^{-1}/i(p_0) \in \R^{(d-1) \times (d-1)}$. If instead $p_0$ is unknown and we regard both $\mu_0$ and $p_0$ as nuisance parameters, then 
the \textit{efficient score function} $\tilde{\ell}_{\theta_0} \colon \R^d \times \R \to \R^{d-1}$~\citep[Section~25.4]{vdV1998asymptotic} for~$\theta_0$ is given by
\begin{equation}
\label{eq:efficient-score}
\tilde{\ell}_{\theta_0}(x,y) := -(\tilde{x} - \tilde{m})\,\psi_0(y - x^\top\beta_0),
\end{equation}
where $\tilde{m} := \E(\tilde{X}_1) \in \R^{d-1}$, $x = (x_1,\ldots,x_d)$ and $\tilde{x} = (x_1,\ldots,x_{d-1})$. 
% This is an $L^2$ orthogonal projection of the parametric score function for $\theta_0$ (namely the first $d - 1$ components of $\dot{\ell}_{\beta_0}$ above) onto the orthogonal complement of the \textit{nuisance tangent set}, and 
This is the same regardless of whether $p_0$ is known or unknown.
% so a form of \textit{Neyman orthogonality} holds~\citep{neyman1959optimal}
The \textit{efficient information matrix} for $\theta_0$ is
\begin{equation}
\label{eq:efficient-information-theta}
\tilde{I}_{\theta_0} := \E\bigl(\tilde{\ell}_{\theta_0}(X_1,Y_1)\,\tilde{\ell}_{\theta_0}(X_1,Y_1)^\top\bigr) = \E\bigl(\psi_0(\varepsilon_1)^2\bigr)\Cov(\tilde{X}_1) = i(p_0)\Sigma,
\end{equation}
where $\Sigma := \Cov(\tilde{X}_1)$, and the semiparametric efficiency lower bound \citep[e.g.][p.~367]{vdV1998asymptotic} is $\tilde{I}_{\theta_0}^{-1} = \Sigma^{-1}/i(p_0)$. This is the top-left $(d-1) \times (d-1)$ submatrix of $I_{\beta_0}^{-1}$ in~\eqref{eq:Ibeta0}, 
% so intuitively the difficulty of estimating $\theta_0$ when $p_0$ is unknown is the same as in a parametric model where $p_0$ is known
and is asymptotically attained by an adaptive estimator $\tilde{\theta}_n$ of $\theta'$ that solves a version of the \textit{efficient score equations} $\sum_{i=1}^n \tilde{\ell}_{\tilde{\theta}_n}(X_i,Y_i) = 0$ with $\psi_0$ replaced with an $L^2(P_0)$-consistent score estimate; see~\citet[Example~3]{bickel1982adaptive} and~\citet[Chapter~25.8]{vdV1998asymptotic}.
In~\eqref{eq:betahat-intercept}, we instead target a surrogate of the efficient score $\tilde{\ell}_{\theta_0}$ above with~$\psi_0$ replaced by $\psi_0^*$.  If $\hat{\psi}_{n,j}$ is continuous, then any minimiser $\hat{\theta}_n^{(j)}$ satisfies
\[
\sum_{i \in I_{j+2}} (\tilde{X}_i - \bar{X}_{n,j}) \cdot \hat{\psi}_{n,j}\bigl(Y_i - \bar{X}_{n,j}^\top\bar{\theta}_n^{(j)} - (\tilde{X}_i - \bar{X}_{n,j})^\top\hat{\theta}_n^{(j)}\bigr) = 0,
\]
which is a variant of the efficient score equations for $\theta_0$. Since we do not assume knowledge of the population mean $\tilde{m} = \E(\tilde{X}_1)$, we replace it with its sample analogue $\bar{X}_{n,j}$ in each of the three folds. The rationale for the final Step~$5'$ is that $\hat{\beta}_n^\zeta = (\hat{\theta}_n,\hat{\mu}_n^\zeta)$ solves a version of the efficient score equations for $\beta_0$ when $p_0$ is regarded as a nuisance parameter; see~\eqref{eq:efficient-score-eqn-beta} in Section~\ref{subsec:linreg-semiparametric}.

Under the regularity conditions in Section~\ref{subsec:linreg-sym}, it turns out that, up to a translation by $\mu_0$, the projected score functions $\hat{\psi}_{n,j}$ are also $L^2(P_0)$-consistent estimators of $\psi_0^*$ in this setting (Lemma~\ref{lem:proj-score-consistency}). It follows that with probability tending to~1 as $n \to \infty$, we have $\lim_{z \to -\infty}\hat{\psi}_{n,j}(z) > 0 > \lim_{z \to \infty}\hat{\psi}_{n,j}(z)$ and hence $\hat{\theta}_n^\dagger$ and $\hat{\theta}_n^\ddagger$ exist. We can then adapt the proof strategy for Theorem~\ref{thm:linreg-score-sym} above to establish the following main result.

\begin{theorem}
\label{thm:linreg-score-intercept}
Assume that~\emph{\ref{ass:fisher-finite}--\ref{ass:covariates}} hold, and moreover that~\emph{\ref{ass:zeta}} is satisfied by $\varphi = \zeta$ in the linear model~\eqref{eq:linear-model-intercept} with $\E\zeta(\varepsilon_1) = 0$. Suppose further that $\bar{\theta}_n^{(j)} - \theta_0 = O_p(n^{-1/2})$ for $j \in \{1,2,3\}$. Then for any sequence of estimators $\hat{\beta}_n^\zeta = (\hat{\theta}_n,\hat{\mu}_n^\zeta)$ such that $\hat{\theta}_n \in \{\hat{\theta}_n^\dagger,\hat{\theta}_n^\ddagger\}$ for each $n$, we have
\[
\sqrt{n}(\hat{\beta}_n^\zeta - \beta_0) \cvd N_d\bigl(0, (\tilde{I}_{\beta_0}^*)^{-1}\bigr);
\]
here,
\begin{align}
\label{eq:antitonic-efficient-information-beta}
(\tilde{I}_{\beta_0}^*)^{-1} =
\begin{pmatrix}
\dfrac{\Sigma^{-1}}{i^*(p_0)} & -\dfrac{\Sigma^{-1}\tilde{m}}{i^*(p_0)} \\[12pt]
-\dfrac{\tilde{m}^\top\Sigma^{-1}}{i^*(p_0)} & \upsilon_{p_0} + \dfrac{\tilde{m}^\top\Sigma^{-1}\tilde{m}}{i^*(p_0)}
\end{pmatrix}
\in \R^{d \times d}
\end{align}
is the inverse of
\[
\tilde{I}_{\beta_0}^* := 
\begin{pmatrix}
i^*(p_0)\Sigma + \tilde{m}\tilde{m}^\top/\upsilon_{p_0} & \tilde{m}/\upsilon_{p_0} \\[3pt]
\tilde{m}^\top/\upsilon_{p_0} & 1/\upsilon_{p_0}
\end{pmatrix} = i^*(p_0)\,\E(X_1 X_1^\top) - \Bigl(i^*(p_0) - \frac{1}{\upsilon_{p_0}}\Bigr)\,\E(X_1)\E(X_1)^\top.
\]
\end{theorem}

As an immediate consequence of Theorem~\ref{thm:linreg-score-intercept},
\begin{equation}
\label{eq:thetahat}
\sqrt{n}(\hat{\theta}_n - \theta_0) \cvd N_{d-1}\biggl(0,\,\frac{\Sigma^{-1}}{i^*(p_0)}\biggr)
\end{equation}
as $n \to \infty$, and the proof reveals that this holds even without the centring constraint $\E\zeta(\varepsilon_1) = 0$; see also Proposition~\ref{prop:linreg-semiparametric}\textit{(a)}. Therefore, $\hat{\theta}_n$ is an data-driven convex $M$-estimator of $\theta_0$ whose limiting covariance $\Sigma^{-1}/i^*(p_0)$ is identical to $\tilde{I}_{\theta_0}^{-1}$ in~\eqref{eq:efficient-information-theta}, except with $i(p_0)$ in place of $i^*(p_0)$. Similarly, $\tilde{I}_{\beta_0}^*$ is the antitonic analogue of the efficient information matrix $\tilde{I}_{\beta_0}$ for $\beta_0$; see~\eqref{eq:efficient-information-beta} in Proposition~\ref{prop:linreg-semiparametric}\textit{(b)}. Consequently, we can view Theorem~\ref{thm:linreg-score-intercept} as an `antitonic efficiency' result. 

% The main interest lies in the last row and column of $\tilde{I}_{\beta_0}^{-1}$ since the top-left $(d-1) \times (d-1)$ block is precisely the limiting covariance of $\sqrt{n}(\hat{\theta}_n - \theta_0)$ in Theorem~\ref{thm:linreg-score-intercept}, namely the previous semiparametric efficiency lower bound $\tilde{I}_{\theta_0}^{-1}$ for $\theta_0$.
To interpret the bottom-right entry in~\eqref{eq:antitonic-efficient-information-beta}, first suppose that $\theta_0$ is known and $p_0$ is unknown. Then the problem of estimating $\mu_0$ reduces to one-dimension location estimation based on 
$\tilde{Y}_i := Y_i - \tilde{X}_i^\top\theta_0 = \mu_0 + \varepsilon_i$ for $i \in [n]$. Since $\E\zeta(\tilde{Y}_1 - \mu_0) = 0$, the corresponding semiparametric efficiency lower bound is~$\upsilon_{p_0}$, which is the one-dimensional version of~\eqref{eq:M-estimator-asymp}; see~\citet[Example~25.24]{vdV1998asymptotic} for the special case of mean estimation. On the other hand, when $\theta_0$ is also unknown, the efficiency lower bound for $\mu_0$ is instead the bottom-right entry $(\tilde{I}_{\beta_0}^{-1})_{d,d} = \upsilon_{p_0} + \tilde{m}^\top\Sigma^{-1}\tilde{m}/i(p_0)$ in~\eqref{eq:efficient-information-beta}. Therefore, it is strictly harder to estimate $\mu_0$ in our semiparametric setting unless $\tilde{m} = 0$, and the asymptotic variance of our $\hat{\mu}_n^\zeta$ is the antitonic counterpart of $(\tilde{I}_{\beta_0}^{-1})_{d,d}$.  Observe also that $i^*(p_0) \geq 1/\upsilon_{p_0}$ by~\eqref{eq:antitonic-fisher-LB}, so the limiting covariance matrices in Theorems~\ref{thm:linreg-score-sym} and~\ref{thm:linreg-score-intercept} satisfy $(I_{\beta_0}^*)^{-1} \preceq (\tilde{I}_{\beta_0}^*)^{-1}$ in the positive semidefinite (Loewner) ordering~$\preceq$, with equality if and only if $\zeta$ is proportional to $-\psi_0^*$.

Lemma~\ref{lem:CQR-suboptimal} in Section~\ref{subsec:CQR} establishes that $\hat{\theta}_n$ always has asymptotic efficiency at least that of the composite quantile estimator~\citep{zou2008composite}, and also that there exist log-concave densities $p_0$ for which the latter has arbitrarily low efficiency relative to the former.

\subsection{Inference}
\label{subsec:linreg-inference}

To perform asymptotically valid inference for $\beta_0$ based on Theorems~\ref{thm:linreg-score-sym} and~\ref{thm:linreg-score-intercept} when $p_0$ is unknown, we require a consistent estimator of the antitonic information $i^*(p_0)$. This can be constructed using residuals $\check{\varepsilon}_i := Y_i - X_i^\top\bar{\beta}_n^{(j)}$ for $j \in \{1,2,3\}$ and $i \in I_{j+2}$ in place of the unobserved errors, where $\bar{\beta}_n^{(j)}$ are $\sqrt{n}$-consistent estimators of $\beta_0$ given by~\eqref{eq:pilot-est}.

\begin{lemma}
\label{lem:observed-information}
In the setting of Theorem~\ref{thm:linreg-score-sym}, let $\check{\psi}_{n,j} \in \{\hat{\psi}_{n,j},\hat{\psi}_{n,j}^{\mathrm{anti}}\}$ for $n \in \N$ and $j \in \{1,2,3\}$. Then
\[
\hat{\imath}_n := \frac{1}{n}\sum_{j=1}^3 \sum_{i \in I_{j+2}}\check{\psi}_{n,j}(\check{\varepsilon}_i)^2 \cvp i^*(p_0).
\]
The same conclusion holds in the setting of Theorem~\ref{thm:linreg-score-intercept} if instead $\check{\psi}_{n,j}(\cdot) = \hat{\psi}_{n,j}(\,\cdot + \bar{\mu}_n^{(j)})$ for all $n,j$.
\end{lemma}

Therefore, $I_{\beta_0}^* = \E(X_1 X_1^\top)\,i^*(p_0)$ can be estimated consistently by 
% the \textit{observed antitonic information matrix}
\[
\hat{I}_n := \frac{\hat{\imath}_n}{n}\sum_{i=1}^n X_i X_i^\top = \frac{\hat{\imath}_n}{n}X^\top X,
\]
where $X = (X_1 \; \cdots \; X_n)^\top \in \mathbb{R}^{n \times d}$ has full column rank with probability tending to 1 under~\ref{ass:covariates}. When $p_0$ is symmetric, it follows from Theorem~\ref{thm:linreg-score-sym} and Lemma~\ref{lem:observed-information} that
\[
\sqrt{n}\hat{I}_n^{1/2}(\hat{\beta}_n^{\mathrm{sym}} - \beta_0) \cvd N_d(0, I_d).
\]
Thus, writing $\hat{\beta}_{n,j}^{\mathrm{sym}}$ for the $j$th component of $\hat{\beta}_n^{\mathrm{sym}}$ and $\hat{s}_j := (\hat{I}_n^{-1})_{jj}$ for the $j$th diagonal entry of $\hat{I}_n^{-1}$, we deduce that
\[
\biggl[\hat{\beta}_{n,j}^{\mathrm{sym}} - z_{\alpha/2}\sqrt{\frac{\hat{s}_j}{n}} \, , \, \hat{\beta}_{n,j}^{\mathrm{sym}} + z_{\alpha/2}\sqrt{\frac{\hat{s}_j}{n}}\biggl]
\]
is an asymptotic $(1-\alpha)$-level confidence interval for the $j$th component of $\beta_0$, where $z_{\alpha/2}$ denotes the $(1 - \alpha/2)$-quantile of the standard normal distribution. Moreover,
\[
\bigl\{b \in \R^d : n(\hat{\beta}_n^{\mathrm{sym}} - b)^\top\hat{I}_n(\hat{\beta}_n^{\mathrm{sym}} - b) \leq \chi_d^2(\alpha)\bigr\}
\]
is an asymptotic $(1 - \alpha)$-confidence ellipsoid for $\beta_0$, where $\chi_d^2(\alpha)$ denotes the $(1-\alpha)$-quantile of the $\chi_d^2$ distribution. 

On the other hand, in the setting of Section~\ref{subsec:linreg-intercept}, $\tilde{I}_n := n^{-1}\sum_{i=1}^n \hat{\imath}_n(\tilde{X}_i - \bar{\tilde{X}}_n)(\tilde{X}_i - \bar{\tilde{X}}_n)^\top$ is a consistent estimator of $i^*(p_0)\Sigma$, where $\bar{\tilde{X}}_n := n^{-1}\sum_{i=1}^n \tilde{X}_i$. In view of~\eqref{eq:thetahat}, we can therefore construct asymptotically valid confidence sets for $\theta_0$ and confidence intervals for its components, as above. To obtain a sample approximation to the asymptotic variance of the intercept estimate $\hat{\mu}_n^\zeta$, we also require a consistent estimate $\hat{\upsilon}_n$ of $\upsilon_{p_0} = V_{p_0}(\zeta)$.  If the errors are mean-centred via $\zeta \colon z \mapsto -z$, then a natural estimator of $\upsilon_{p_0} = \E(\varepsilon_1^2)$ is $n^{-1}\sum_{i=1}^n \check{\varepsilon}_i^2$. On the other hand, if the $\tau$-quantile of the errors are centred via $\zeta \colon z \mapsto \Ind_{\{z < 0\}} - \tau$ for some $\tau \in (0,1)$, then we can approximate $\upsilon_{p_0} = \tau(1 - \tau)/p_0(0)^2$ using a kernel density estimator of $p_0(0)$. The following lemma verifies that using the residuals $\check{\varepsilon}_1,\dotsc,\check{\varepsilon}_n$ is justified more generally in this context.
\begin{lemma}
\label{lem:upsilon-estimate}
Suppose that $\zeta = \zeta_{\mathrm{ac}} - \sum_{m=1}^M \zeta_m \Ind_{[z_m,\infty)}$ for some $M \in \N_0$, where $\zeta_{\mathrm{ac}}$ is absolutely continuous on $\R$, while $\zeta_m > 0$ and $z_m \in \R$ for all $m \in [M]$. Define $\tilde{p}_n \colon \R \to \R$ by $\tilde{p}_n(z) := n^{-1}\sum_{i=1}^n K_h(z - \check{\varepsilon}_i)$ for some square-integrable kernel $K$ and bandwidth $h \equiv  h_n$ satisfying $h_n \to 0$ and $nh_n \to \infty$. If $\zeta_{\mathrm{ac}}'$ is continuous Lebesgue almost everywhere on $\R$, then under the hypotheses of Theorem~\ref{thm:linreg-score-intercept},
\[
\hat{\upsilon}_n := \frac{n^{-1}\sum_{i=1}^n \zeta(\check{\varepsilon}_i)^2}{\bigl\{n^{-1}\sum_{i=1}^n \zeta_{\mathrm{ac}}'(\check{\varepsilon}_i) - \sum_{m=1}^M \zeta_m\tilde{p}_n(z_m)\bigr\}^2} \cvp \upsilon_{p_0}.
\]
\end{lemma}
Now define
\[
\hat{I}_n^\zeta := \hat{I}_n - \Bigl(\hat{\imath}_n - \frac{1}{\hat{\upsilon}_n}\Bigr)\bar{X}_n\bar{X}_n^\top,
\]
where $\bar{X}_n := n^{-1}\sum_{i=1}^n X_i$, and denote by $\tilde{s}_j$ the $j$th diagonal entry of $(\hat{I}_n^\zeta)^{-1}$ for $j \in [d]$. We deduce from Theorem~\ref{thm:linreg-score-intercept} together with Lemmas~\ref{lem:observed-information} and~\ref{lem:upsilon-estimate} that
\[
\biggl[\hat{\beta}_{n,j}^\zeta - z_{\alpha/2}\sqrt{\frac{\tilde{s}_j}{n}} \, , \, \hat{\beta}_{n,j}^\zeta + z_{\alpha/2}\sqrt{\frac{\tilde{s}_j}{n}}\biggl]
\]
is an asymptotic $(1-\alpha)$-level confidence interval for the $j$th component of $\beta_0$, and
\begin{align}
\label{eq:confidence ellipsoid}
\bigl\{b \in \R^d: n(\hat{\beta}_n^\zeta - b)^\top\hat{I}_n^\zeta (\hat{\beta}_n^\zeta - b) \leq \chi_d^2(\alpha)\bigr\}
\end{align}
is an asymptotic $(1 - \alpha)$-confidence ellipsoid for $\beta_0$. 
Lemma~\ref{lem:observed-information} also ensures that standard linear model diagnostics, either based on heuristics such as Cook's distances~\citep{cook1977detection}, or formal goodness-of-fit tests~\citep{jankova2020goodness}, can be applied. 

\section{Numerical experiments}
\label{sec:experiments}

In our numerical experiments, we generate covariates $\tilde{X}_1,\dotsc, \tilde{X}_n \iid N_{d-1}(\mathbf{1}_{d-1}, I_{d-1})$, where $\mathbf{1}_{d-1}$ denotes the $(d-1)$-dimensional all-ones vector, and responses $Y_1, \dotsc, Y_n$ according to the linear model~\eqref{eq:linear-model-intercept}. Independently of the covariates, we draw errors $\varepsilon_1, \ldots, \varepsilon_n \iid P_0$ for the following choices of $P_0$:

\begin{enumerate}[label=(\roman*)]
\item $P_0$ is standard Gaussian.
\item $P_0$ is standard Cauchy.
\item Gaussian scale mixture: $P_0 = \frac{1}{2} N(0, 1) + \frac{1}{2} N(0, 16)$.
\item Gaussian location mixture: $P_0 = \frac{1}{2} N\bigl(- \frac{3}{2}, \frac{1}{100} \bigr) + \frac{1}{2} N\bigl( \frac{3}{2}, \frac{1}{100} \bigr)$.  This distribution is similar to the one constructed in the proof of Proposition~\ref{prop:Vp0-MLE} to show that log-concave maximum likelihood estimation of the error distribution may result in arbitrarily large efficiency loss.
\item Smoothed uniform: $P_0$ is the distribution of $U + \frac{1}{10} Z$, where $U \sim \mathrm{Unif}[-1,1]$ and $Z \sim N(0,1)$ are independent.
\item Smoothed exponential: $P_0$ is the distribution of $W - 1 + \frac{\sqrt{3}}{10} Z$, where $W \sim \Exp(1)$ and $Z \sim N(0,1)$ are independent. We choose the standard deviation $\frac{\sqrt{3}}{10}$ for the Gaussian component so that the ratio of the variances of the non-Gaussian and the Gaussian components is the same as that in the smoothed uniform setting.
\end{enumerate}
In cases (i),~(v) and~(vi), $P_0$ is log-concave, while for the other settings it is not.  Since the OLS and LAD estimators are targeting different population intercepts in cases where the mean and median of $P_0$ are not equal (so at most one of these estimators can be Fisher consistent), we focus on estimation and inference for the $(d-1)$-dimensional subvector $\theta_0$ of~$\beta_0$ in order to present a fair comparison.  

% \begin{table}[ht]
% \centering
% \begin{tabular}{|l|r||r|r|r|r|r|r|}
% \hline
% & Oracle & ASM & Alt & LCMLE & 1S & LAD & OLS \\ 
% \hline
% Standard Gaussian & 8.34 & 8.88 & 8.86 & 10.73 & 9.72 & 13.11 & \textbf{8.34} \\ 
% Standard Cauchy & 20.07 & $\bf{20.61}$ & 20.76 & 23.85 & 22.13 & 21.36 & $2.03 \times 10^6$ \\ 
% Gaussian scale mixture & 31.36 & $\bf{32.01}$ & 32.33 & 34.89 & 36.18 & 34.97 & 72.34 \\ 
% Gaussian location mixture & 0.16 & 0.17 & $\bf{0.16}$ & 1.51 & 18.07 & 319.65 & 18.51 \\ 
% Smoothed uniform & 1.02 & 1.29 & $\bf{1.15}$  & 1.52 & 2.07 & 7.92 & 2.78 \\ 
% Smoothed exponential & 1.78 & 2.13 & $\bf{2.02}$ & 2.54 & 3.36 & 8.27 & 8.65 \\ 
% \hline
% \end{tabular}
% \caption{Squared estimation error ($\times 10^3$) for different estimators, with $n = 600$ and $d=6$.}
% \label{tab:MSE-comparison}
% \end{table}

\begin{table}[ht]
\centering
\begin{tabular}{|l|r||r|r|r|r|r|r|}
\hline
& Oracle & ASM & Alt & LCMLE & 1S & LAD & OLS \\ 
\hline
Standard Gaussian & 8.51 & 8.96 & 9.01 & 9.43 & 9.77 & 12.70 & $\bf{8.51}$ \\ 
Standard Cauchy & 19.98 & $\bf{20.44}$ & 20.68 & 48.31 & 21.74 & 21.15 & $3.9 \times 10^6$ \\ 
Gaussian scale mixture & 30.71 & $\bf{31.58}$ & 31.78 & 34.49 & 36.48 & 34.67 & 73.45 \\ 
Gaussian location mixture & 0.17 & 0.18 & $\bf{0.17}$ & 0.74 & 18.28 & 332.08 & 18.72 \\ 
Smoothed uniform & 1.03 & 1.36 & $\bf{1.18}$ & 1.31 & 2.10 & 8.38 & 2.99 \\ 
Smoothed exponential & 1.91 & 2.24 & $\bf{2.10}$ & 2.26 & 3.33 & 8.60 & 8.86 \\ 
\hline
\end{tabular}
\caption{Squared estimation error ($\times 10^3$) for different estimators of $\theta_0$, with $n = 600$ and $d=6$.}
\label{tab:MSE-comparison}
\end{table}

\begin{figure}[ht]
\centering
\includegraphics[width=0.47\textwidth]{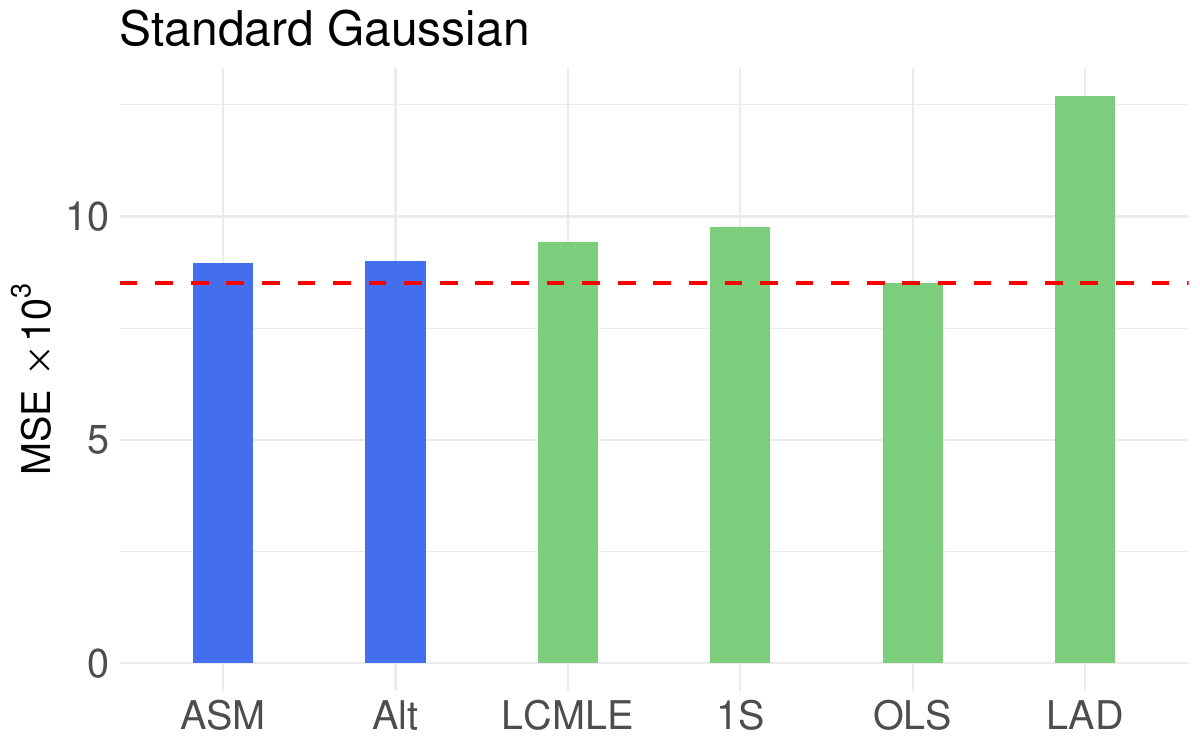}
\hfill
\includegraphics[width=0.47\textwidth]{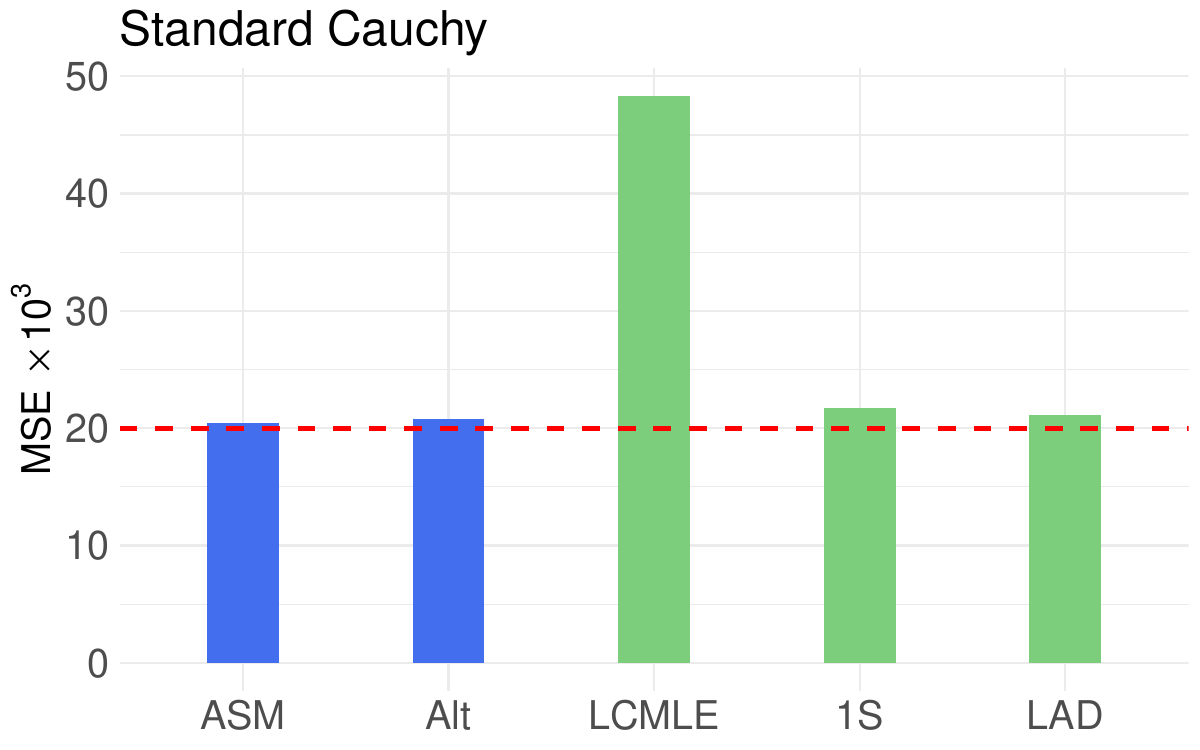}
\includegraphics[width=0.47\textwidth]{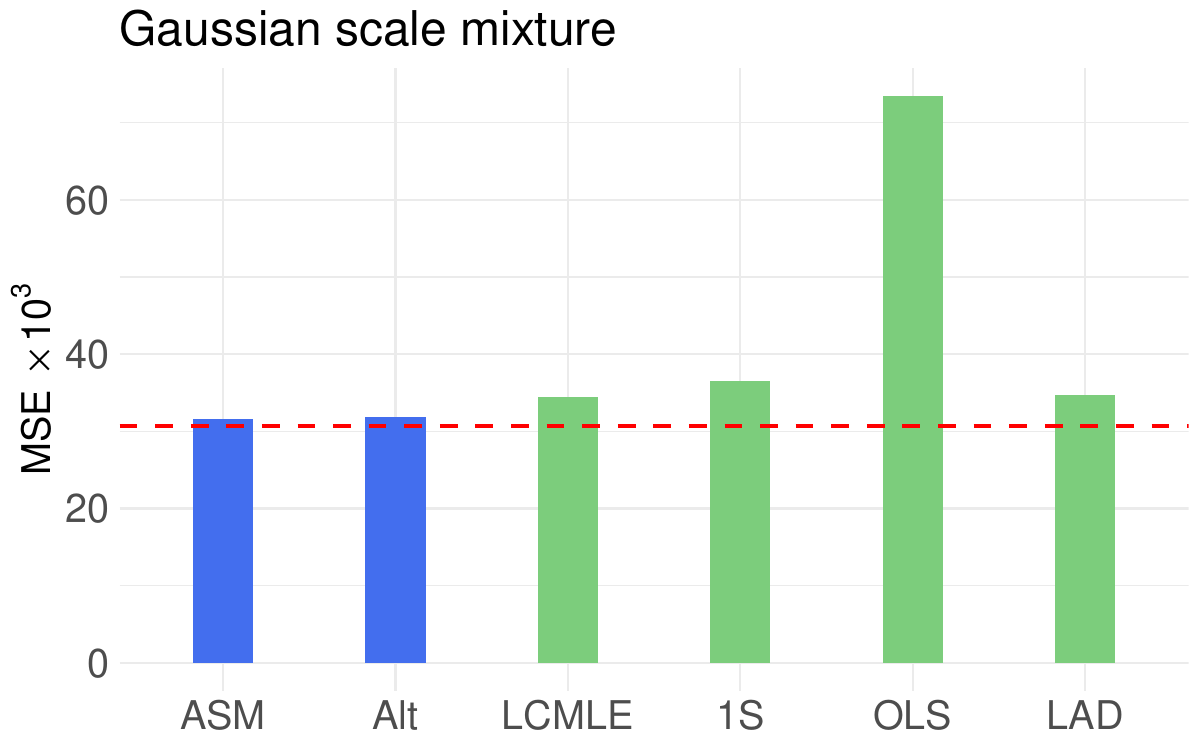}
\hfill
\includegraphics[width=0.47\textwidth]{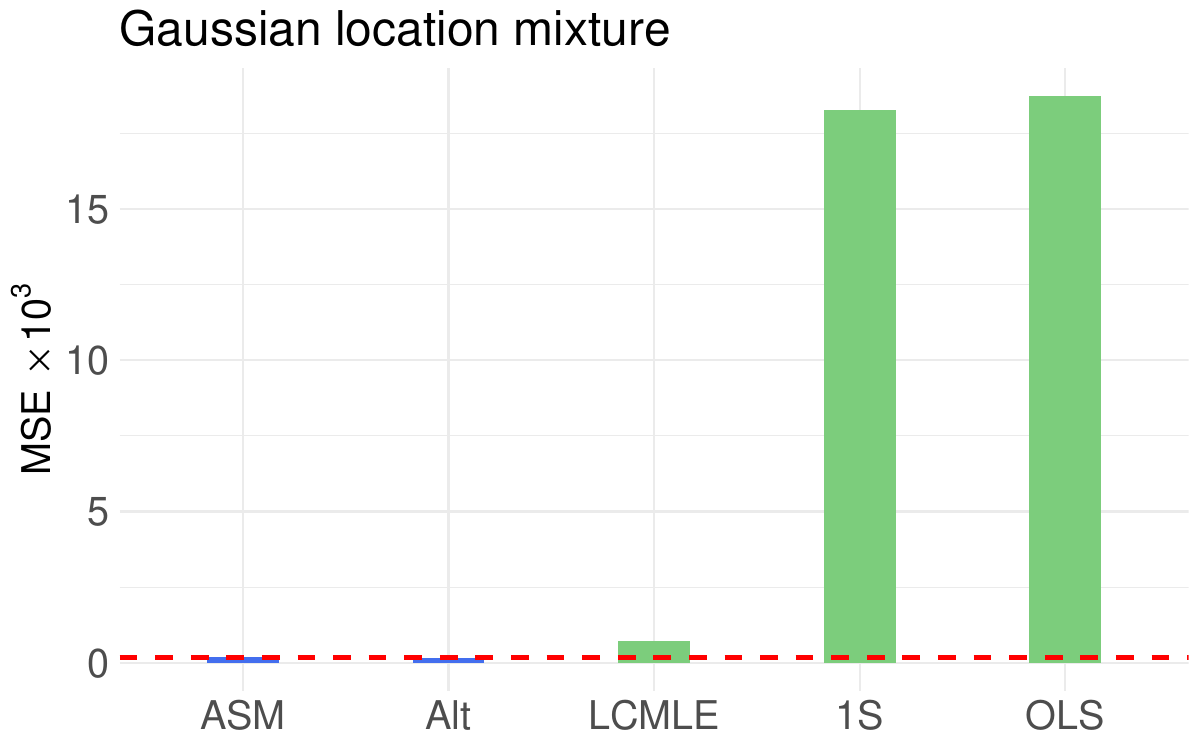}
\includegraphics[width=0.47\textwidth]{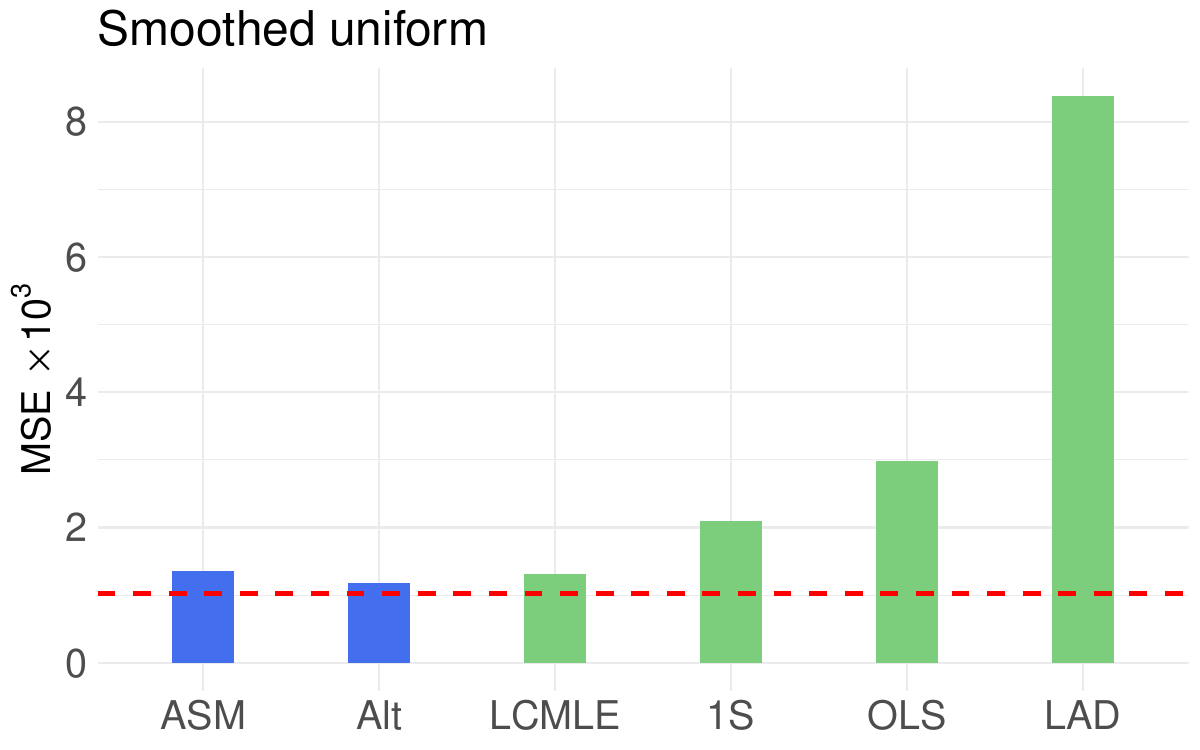}
\hfill
\includegraphics[width=0.47\textwidth]{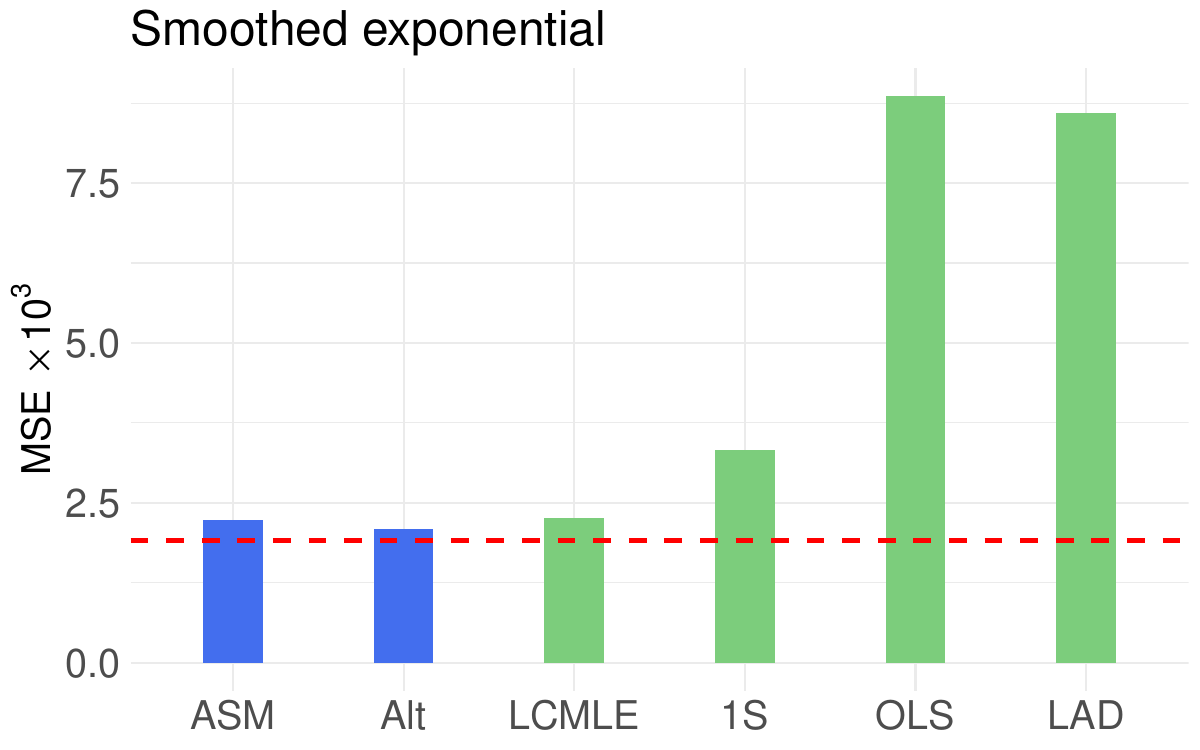}
% \vspace{-0.3cm}   
\caption{Plots of the average squared error loss ($\times 10^3$) of different estimators of $\theta_0$ for noise distributions (i)--(vi), with $n = 600$ and $d = 6$. In each plot, the red dashed line indicates the corresponding value for the oracle convex $M$-estimator, and we omit the estimators that have very large estimation error (see Table~\ref{tab:MSE-comparison} for full details).}
\label{fig:MSE-comparison}
\end{figure}

We compare the performance of two versions of our procedure with an oracle approach and four existing methods. The first variant of our procedure, which we refer to as \textbf{ASM} (antitonic score matching) in all of the plots, is as described in Section~\ref{subsec:linreg-intercept}, except that we do not perform sample splitting, cross-fitting or truncation of the initial score estimates. These devices are convenient for theoretical analysis but not essential in practice. More precisely, \textbf{ASM} first constructs a pilot estimator $\bar{\beta}_n = (\bar{\theta}_n,\bar{\mu}_n)$, and then uses the vector of residuals $\bigl( \hat{\varepsilon}_1,\dotsc,\hat{\varepsilon}_n \bigr)$, where $\hat{\varepsilon}_i := Y_i - \tilde{X}_i^\top\bar{\theta}_n$, to obtain an initial kernel-based score estimator $\tilde{\psi}_n$, formed using a Gaussian kernel and the default Silverman's choice of bandwidth~\citep[p.~48]{silverman1986density}. Following Step~3$'$ in Section~\ref{subsec:linreg-intercept}, we estimate the antitonic projected score and the corresponding convex loss function by $\hat{\psi}_n := \widehat{\mathcal{M}}_{\mathrm{R}}(\tilde{\psi}_n \circ \tilde{F}_n^{-1}) \circ \tilde{F}_n$ and $\hat{\ell}_n$ respectively, where $\tilde{F}_n$ denotes the distribution function associated with the kernel density estimate. Finally, we use Newton's algorithm with Hessian modification\footnote{This modification involves adding the identity matrix to the Hessian prior to its version.}~\citep[Section~3.4]{nocedal2006numerical} to compute our semiparametric estimator
\[
\hat{\theta}_n^{\mathrm{ASM}} \in \argmin_{\theta \in \mathbb{R}^{d-1}} \sum_{i=1}^n \hat{\ell}_n\bigl(Y_i - \bar{X}_n^\top \bar{\theta}_n - (\tilde{X}_i - \bar{X}_n)^\top \theta \bigr),
\]
where $\bar{X}_n := n^{-1} \sum_{i=1}^n \tilde{X}_i$.

In the second version of our procedure, which we refer to as \textbf{Alt} in our plots, we implement an empirical analogue of the alternating optimisation procedure described in Section~\ref{subsec:lin_reg_id_alt}. We start with an uninformative initialiser $(\hat{\theta}_n^{(0)}, \hat{\mu}_n^{(0)}) = (0, 0) \in \R^{d-1} \times \R$, and then alternate between the following steps for $t \in \mathbb{N}$:
\begin{enumerate}[label=\Roman*.]
\item Compute residuals $\hat{\varepsilon}_i^{(t-1)} := Y_i - \hat{\mu}_n^{(t-1)} - \tilde{X}_i^\top \hat{\theta}_n^{(t-1)}$ for $i \in [n]$ and hence estimate the antitonic projected score $\hat{\psi}_n^{(t-1)}$ and the corresponding convex loss $\hat{\ell}_n^{(t-1)}$ as for \textbf{ASM}.
\item Update $(\hat{\theta}_n^{(t)}, \hat{\mu}_n^{(t)}) \in \argmin_{(\theta, \mu) \in \mathbb{R}^{d-1} \times \mathbb{R}} \sum_{i=1}^n \hat{\ell}_n^{(t-1)}(Y_i - \mu - \tilde{X}_i^\top \theta)$.
\end{enumerate}
We iterate these steps until convergence of the empirical score matching objective
\[
\hat{D}_n(\hat{\psi}_n^{(t)}; \hat{\varepsilon}_1^{(t)},\dotsc,\hat{\varepsilon}_n^{(t)}) = \frac{1}{n}\sum_{i=1}^n \bigl\{ \hat{\psi}_n^{(t)}(\hat{\varepsilon}_i^{(t)}) + 2(\hat{\psi}_n^{(t)})'(\hat{\varepsilon}_i^{(t)}) \bigr\}
\]
defined in~\eqref{eq:score-matching-empirical}; in all of our experiments both $\hat{\theta}_n^{(t)}$ and the score matching objective values did indeed converge.  
% The iterates $\hat{\mu}_n^{(t)}$ and $\hat{\psi}_n^{(t)}$ are not guaranteed to converge due to non-identifiability of the intercept term $\mu_0$ in~\eqref{eq:linear-model-intercept}, but $\hat{\theta}_n^{(t)}$ and the score matching objective values did indeed converge in all of our experiments. 

The alternative approaches that we consider are as follows:
\begin{itemize}
\item \textbf{Oracle}: The $M$-estimator $\hat{\theta}_n^{\mathrm{oracle}}$, where
\[
\begin{pmatrix}
\hat{\theta}_n^{\mathrm{oracle}} \\
\hat{\mu}_n^{\mathrm{oracle}} 
\end{pmatrix}
\in \argmin_{(\theta, \mu) \in \R^{d-1} \times \R}\;\sum_{i=1}^n \ell_0^*(Y_i - \mu - \tilde{X}_i^\top\theta)
\]
is defined with respect to the optimal convex loss function~$\ell_0^*$. Although $\hat{\theta}_n^{\mathrm{oracle}}$ is antitonically efficient in the sense of~\eqref{eq:thetahat}, it is not a valid estimator in our semiparametric framework since it requires knowledge of $p_0$. 
\item \textbf{LAD}: The least absolute deviation estimator $\hat{\theta}_n^{\mathrm{LAD}}$, where
\[
\hat{\beta}_n^{\mathrm{LAD}} =
\begin{pmatrix}
\hat{\theta}_n^{\mathrm{LAD}} \\
\hat{\mu}_n^{\mathrm{LAD}}
\end{pmatrix} 
\in \argmin_{(\theta, \mu) \in \R^{d-1} \times \mathbb{R}}\;\sum_{i=1}^n |Y_i - \mu - \tilde{X}_i^\top\theta|.
\]
\item \textbf{OLS}: The ordinary least squares estimator $\hat{\theta}_n^{\mathrm{OLS}}$, where
\[
\hat{\beta}_n^{\mathrm{OLS}} =
\begin{pmatrix}
\hat{\theta}_n^{\mathrm{OLS}} \\
\hat{\mu}_n^{\mathrm{OLS}}
\end{pmatrix} 
\in \argmin_{(\theta, \mu) \in \R^{d-1} \times \mathbb{R}}\;\sum_{i=1}^n (Y_i - \mu - \tilde{X}_i^\top\theta)^2.
\]
\item \textbf{1S}: The semiparametric one-step method, where we start with a pilot estimator $(\bar{\theta}_n,\bar{\mu}_n)$, compute a (not necessarily decreasing) nonparametric score estimate, and then update $\bar{\theta}_n$ with a single Newton step instead of solving the estimating equations exactly. Our implementation follows~\citet[Chapter~25.8]{vdV1998asymptotic}: we split the data into two folds of equal size indexed by $I_1$ and~$I_2$, and then use the residuals $\hat{\varepsilon}_i := Y_i - \tilde{X}_i^\top \bar{\theta}_n - \bar{\mu}_n$ for $i \in I_1$ and $i \in I_2$ separately to obtain kernel density estimates $\hat{p}_{n,1},\hat{p}_{n,2}$ of $p_0$ (constructed as for \textbf{ASM}).  Defining the score estimates $\hat{\psi}_{n,j} := \hat{p}_{n,j}'/\hat{p}_{n,j}$ for $j \in \{1,2\}$, we output the cross-fitted estimator $\hat{\theta}_n^{\mathrm{1S}}$, where
\begin{align*}
% where we need to recentre the covariates?
\hat{\theta}_n^{\mathrm{1S}} &:= \bar{\theta}_n - \biggl(\sum_{i \in I_1}\hat{\psi}_{n,2}(\hat{\varepsilon}_i)^2 \tilde{X}_i \tilde{X}_i^\top + \sum_{i \in I_2}\hat{\psi}_{n,1}(\hat{\varepsilon}_i)^2 \tilde{X}_i \tilde{X}_i^\top\biggr)^{-1}\biggl(\sum_{i \in I_1}\hat{\psi}_{n,2}(\hat{\varepsilon}_i)\tilde{X}_i + \sum_{i \in I_2}\hat{\psi}_{n,1}(\hat{\varepsilon}_i)\tilde{X}_i\biggr).
\end{align*}

\item \textbf{LCMLE}: We estimate the error density $p_0$ using the log-concave maximum likelihood estimator~\citep{cule2010maximum,dumbgen2011approximation,dumbgen2013stochastic}. More precisely, again writing $\mathcal{P}_{\mathrm{LC}}$ for the set of univariate log-concave densities and defining $Q(\theta,\mu; p) := \sum_{i=1}^n \log p(Y_i - \mu - \tilde{X}_i^\top\theta)$ for $\theta \in \R^{d-1}$, $\mu \in \R$ and $p \in \mathcal{P}_{\mathrm{LC}}$, we start with a pilot estimator $(\hat{\theta}_n^{(0)},\hat{\mu}_n^{(0)})$ and alternate the following two steps for $t \in \N$ until convergence of $Q\bigl(\hat{\theta}_n^{(t)},\hat{\mu}_n^{(t)};\hat{p}_n^{(t)}\bigr)$:
% \red{extrapolation outside the convex support of the residuals?}:
\begin{align*}
\hat{p}_n^{(t)} \in \argmax_{p \in \mathcal{P}_{\mathrm{LC}}}\,Q\bigl(\hat{\theta}_n^{(t-1)},\hat{\mu}_n^{(t-1)}; p\bigr), \qquad
\begin{pmatrix} \hat{\theta}_n^{(t)} \\ \hat{\mu}_n^{(t)} \end{pmatrix} \in \argmax_{(\theta, \mu) \in \R^{d-1} \times \R}\,Q(\theta,\mu;\hat{p}_n^{(t)}).
\end{align*}
\end{itemize}

For all of our error densities $p_0$, condition~\ref{ass:zeta} is satisfied by $\varphi = -\sgn$ and hence by Proposition~\ref{prop:cvx-M-est-asymp}, $\hat{\beta}_n^{\mathrm{LAD}}$ is $\sqrt{n}$-consistent with asymptotic variance factor $V_{p_0}(\varphi) = 1/\bigl(4p_0(0)^2\bigr)$. Therefore, we took $\hat{\beta}_n^{\mathrm{LAD}}$ to be our pilot estimator for all methods except in the Gaussian location mixture setting~(iv), where instead we used $\hat{\beta}_n^{\mathrm{OLS}}$ since $p_0(0)$ is close to 0 and hence $V_{p_0}(\varphi)$ for $\hat{\beta}_n^{\mathrm{LAD}}$ is very large.

\begin{table}[ht]
\begin{subtable}{0.33\linewidth}
\centering
\begin{tabular}{|l|r|r|r|}
\hline
$n$ & 600 & 1200 & 2400 \\
\hline
ASM & 0.26 & 0.61 & 1.93 \\ 
Alt & 4.35 & 9.99 & 27.85 \\ 
LCMLE & 0.64 & 2.21 & 8.34 \\ 
1S & 0.10 & 0.15 & 0.29 \\ 
\hline
\end{tabular}
\caption*{Standard Gaussian}
\end{subtable}
\begin{subtable}{0.33\linewidth}
\centering
\begin{tabular}{|l|r|r|r|}
\hline
$n$ & 600 & 1200 & 2400 \\ 
\hline
ASM & 0.27 & 0.61 & 1.81 \\ 
Alt & 5.03 & 10.52 & 28.64 \\ 
LCMLE & 0.71 & 1.36 & 2.72 \\ 
1S & 0.11 & 0.15 & 0.29 \\ 
\hline
\end{tabular}
\caption*{Gaussian scale mixture}
\end{subtable}
\begin{subtable}{0.33\linewidth}
\centering
\begin{tabular}{|l|r|r|r|}
\hline
$n$ & 600 & 1200 & 2400 \\
\hline
ASM & 0.21 & 0.58 & 1.96 \\ 
Alt & 1.84 & 3.72 & 11.36 \\ 
LCMLE & 3.15 & 7.09 & 9.29 \\ 
1S & 0.12 & 0.15 & 0.29 \\ 
\hline
\end{tabular}
\caption*{Gaussian location mixture}
\end{subtable}
\caption{Mean execution time in seconds for different semiparametric estimators when $d = 100$.}
\label{tab:execution-time}
\end{table}

\begin{figure}[ht]
\centering
\includegraphics[width=0.32\textwidth]{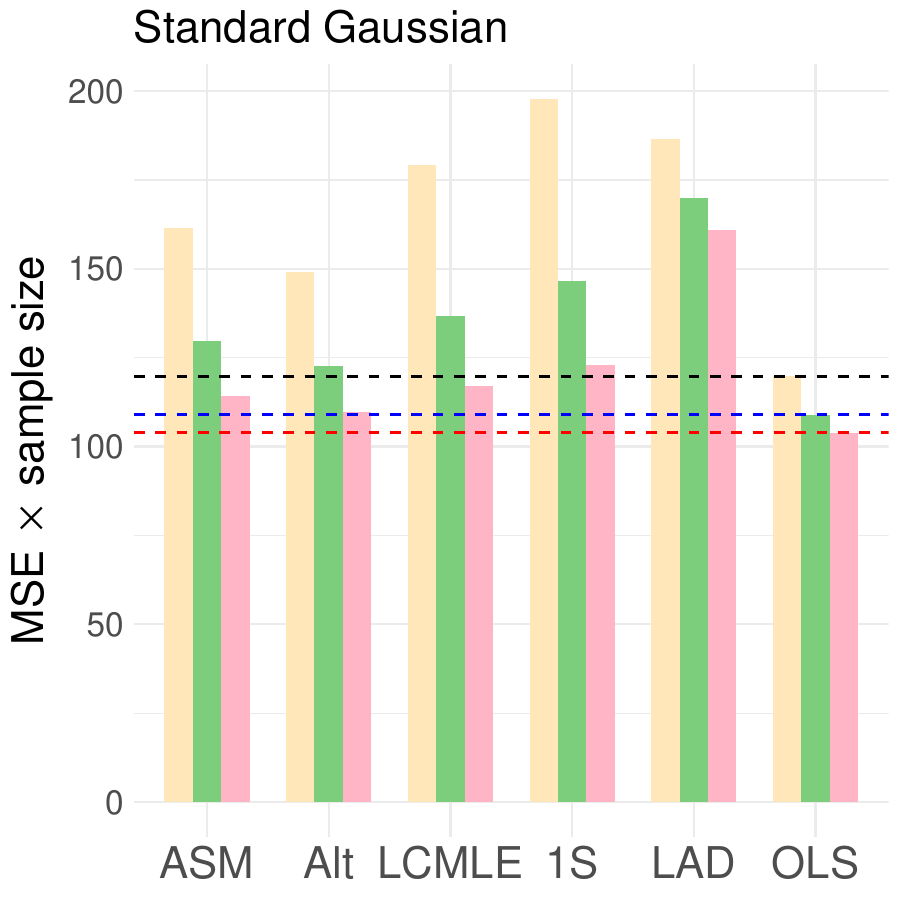}
\includegraphics[width=0.32\textwidth]{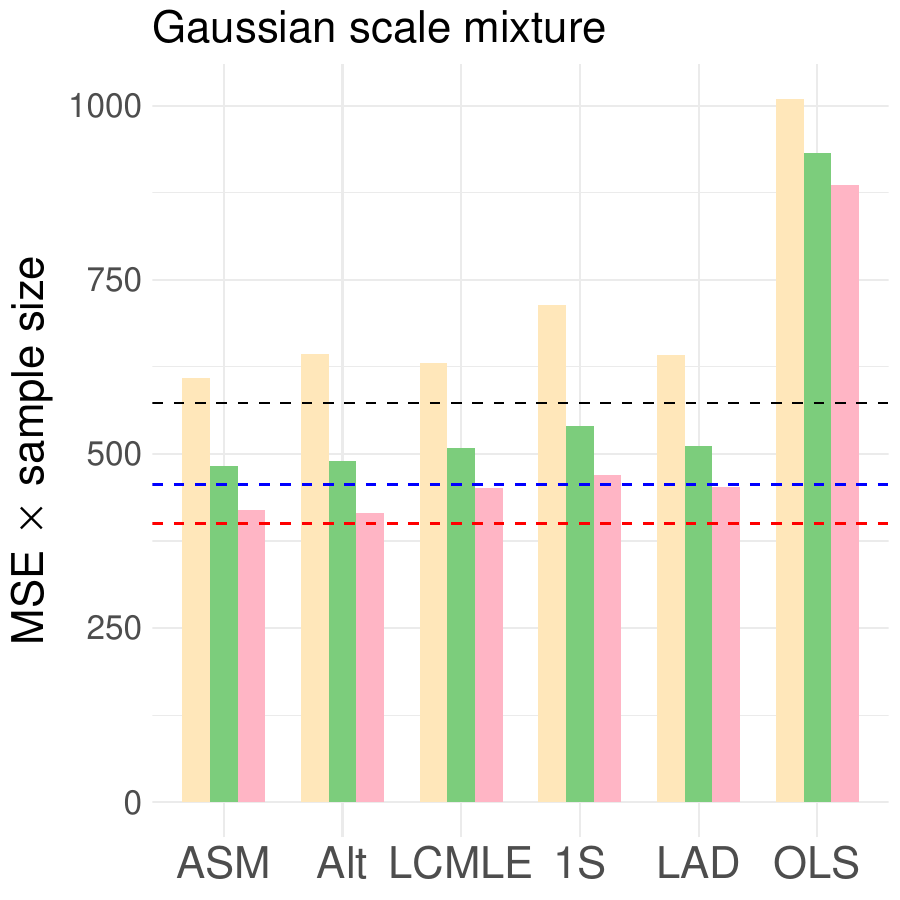}
\includegraphics[width=0.32\textwidth]{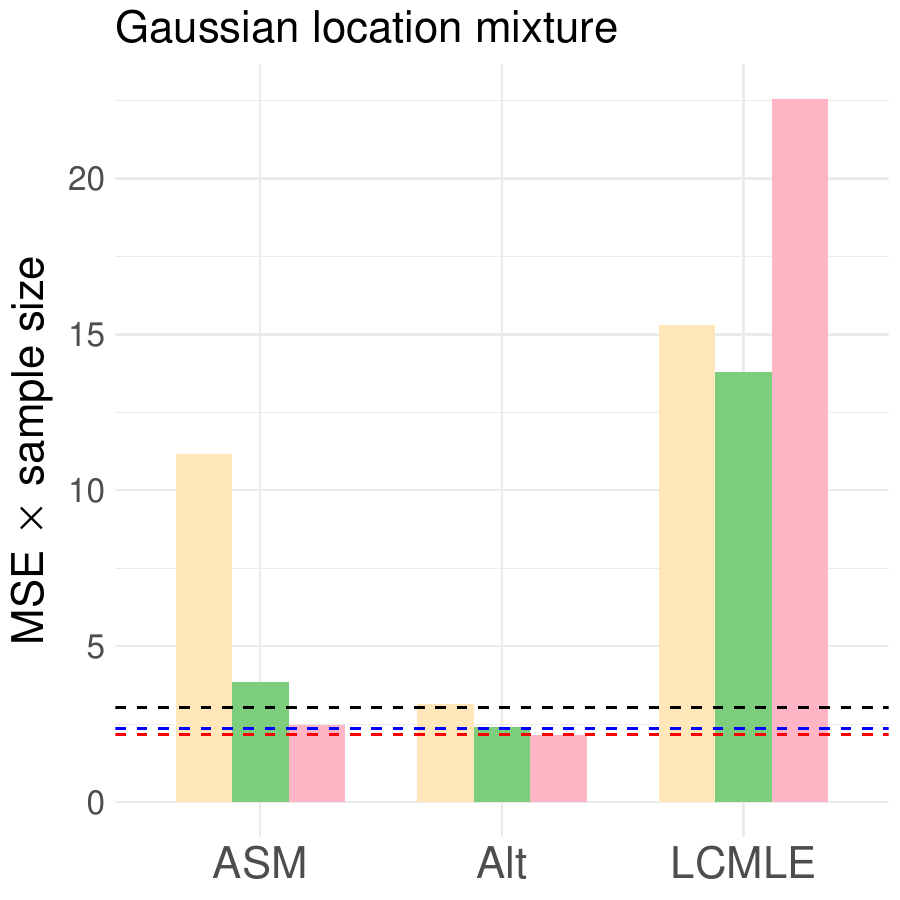}

% \vspace{-0.3cm}   
\caption{Plot of $n \times (\text{average squared loss})$ for different estimators of $\theta_0$, with $d = 100$ and sample sizes $n \in \{600, 1200, 2400\}$ corresponding to the yellow, green, and pink bars respectively. The black, blue and red dashed lines indicate the values of $n \times (\text{average squared loss})$ for the oracle convex $M$-estimator when $n = 600, 1200, 2400$ respectively.}
\label{fig:high_dim_sim}
\end{figure}

\begin{figure}[ht]
\centering
\includegraphics[width=0.32\textwidth]{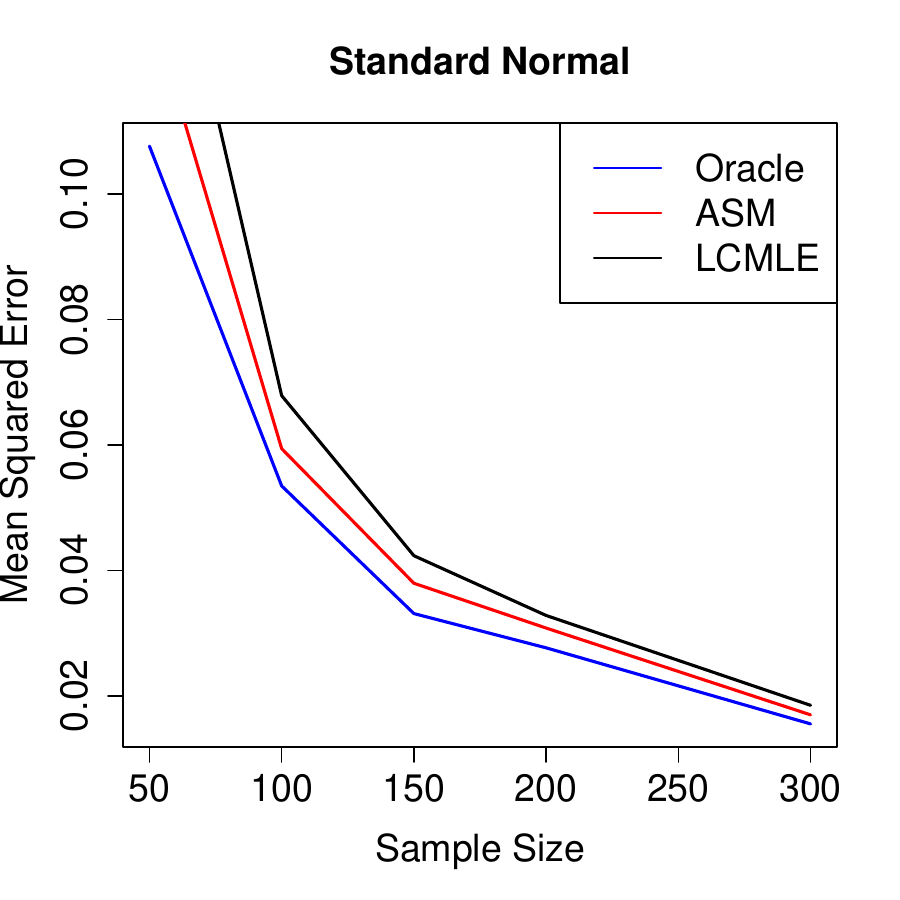}
\includegraphics[width=0.32\textwidth]{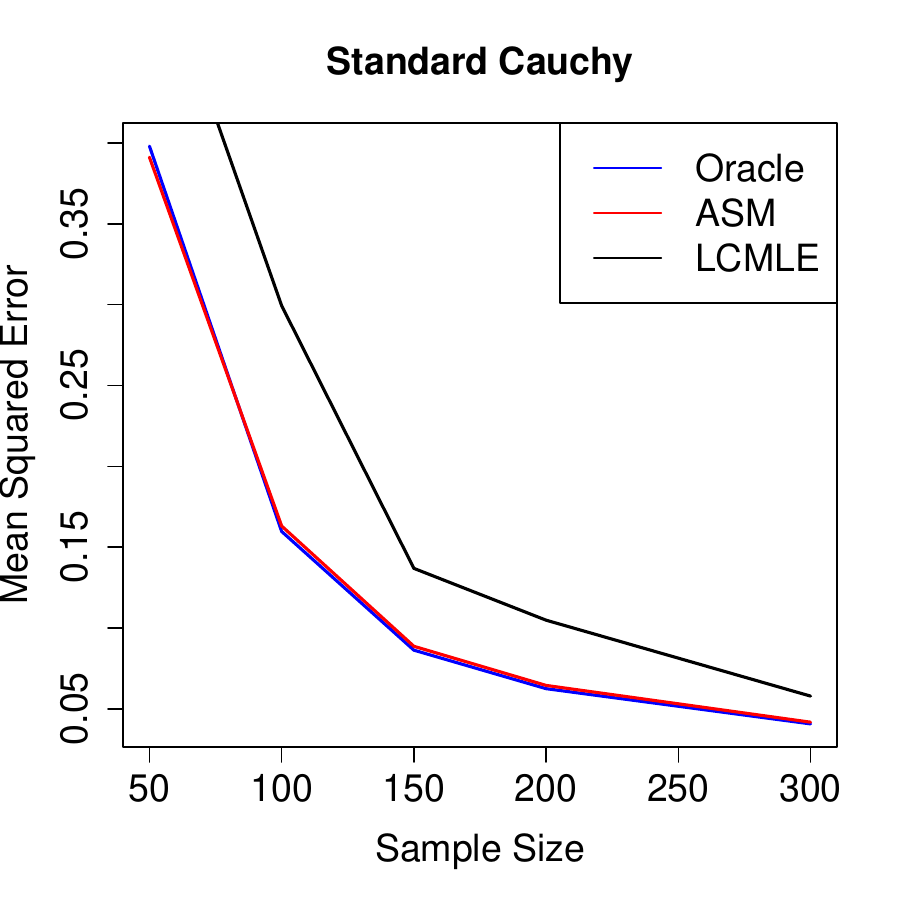}
\includegraphics[width=0.32\textwidth]{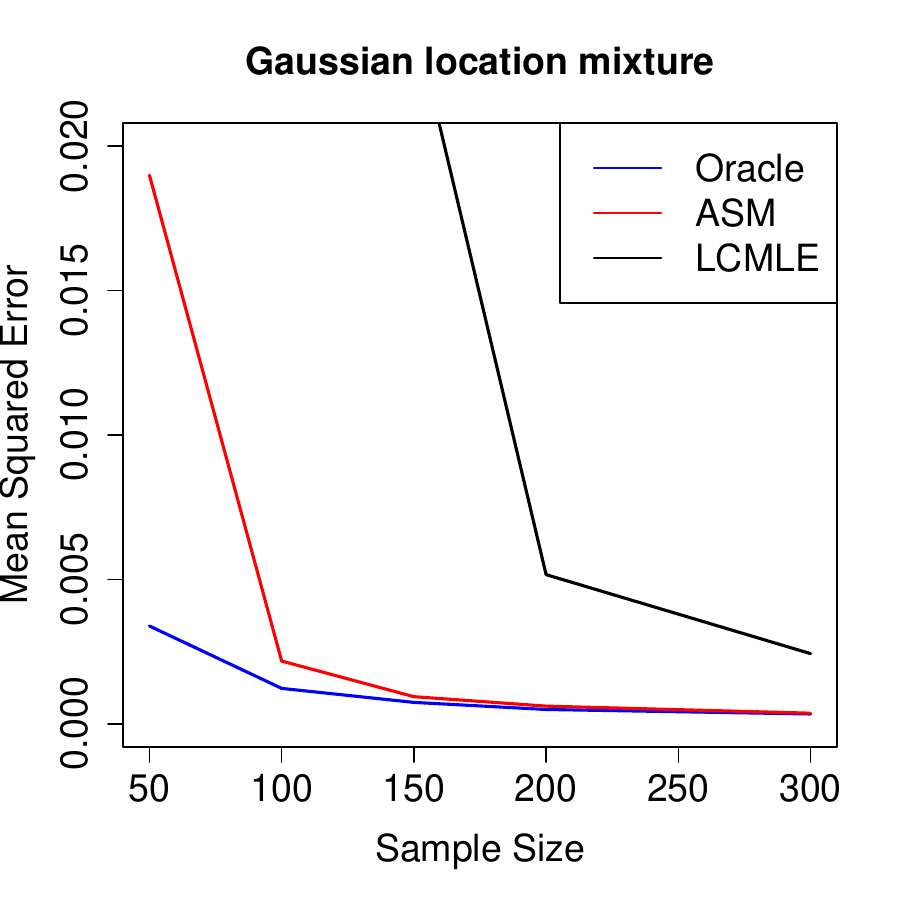}

\vspace{-0.3cm}   
\caption{Comparison of the mean squared errors of different convex $M$-estimators for sample sizes $n \in \{50, 100, 150, 200, 250, 300\}$ and $d = 6$.}
\label{fig:MSE vs sample size}
\end{figure}

\begin{figure}[ht]
\centering
\includegraphics[width=0.49\textwidth]{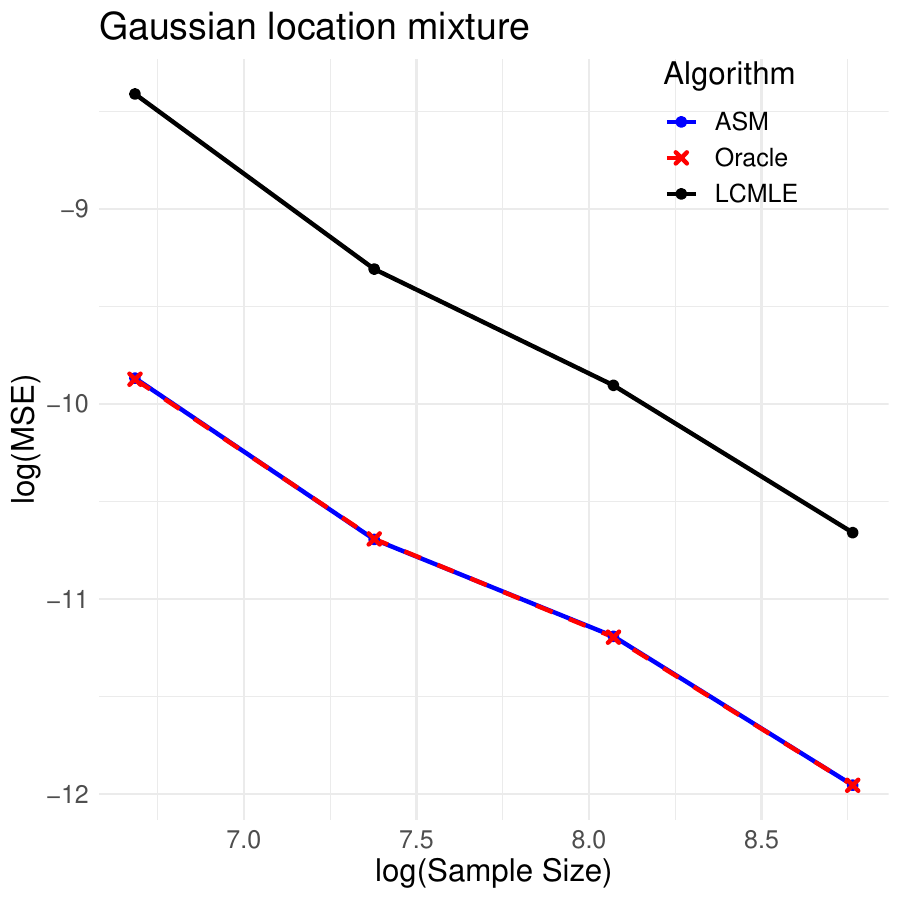}
\includegraphics[width=0.49\textwidth]{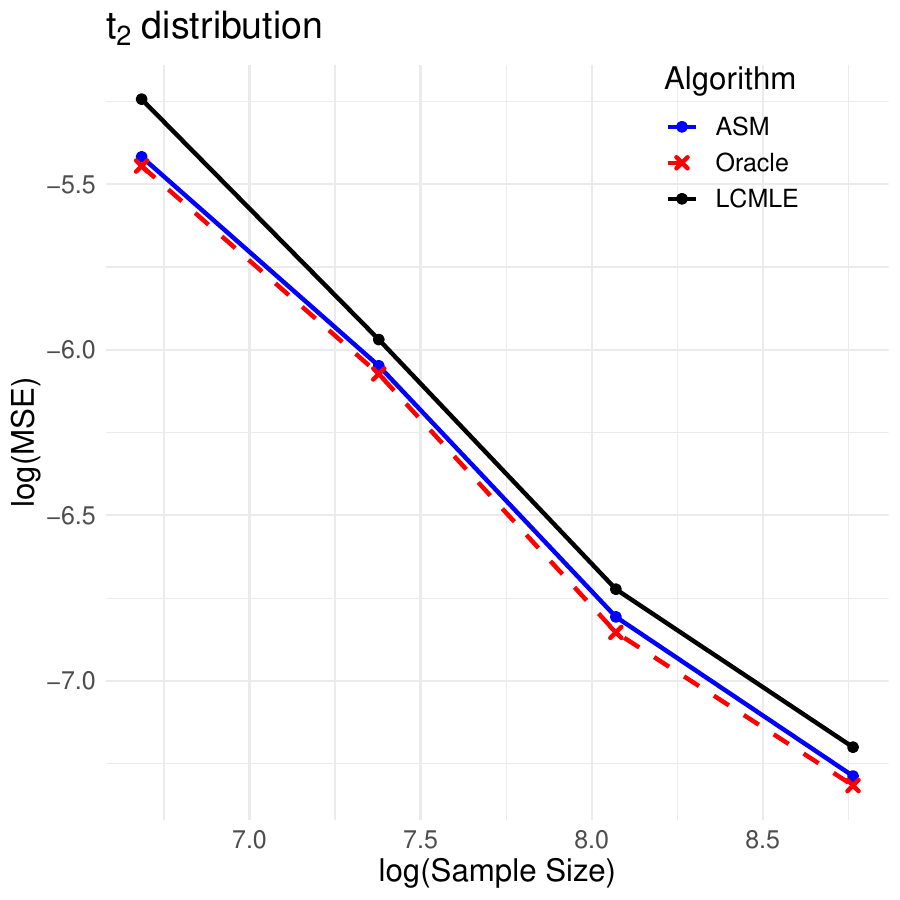}
\vspace{-0.3cm}   
\caption{Log-log plots of the average squared error loss against sample size for three convex $M$-estimators. In each experiment, we set $n \in \{800, 1600, 3200, 6400\}$ and $d = 2$.}
\label{fig:MSE larger sample size}
\end{figure}

\begin{figure}[htp]
\centering
\includegraphics[width=0.32\textwidth]{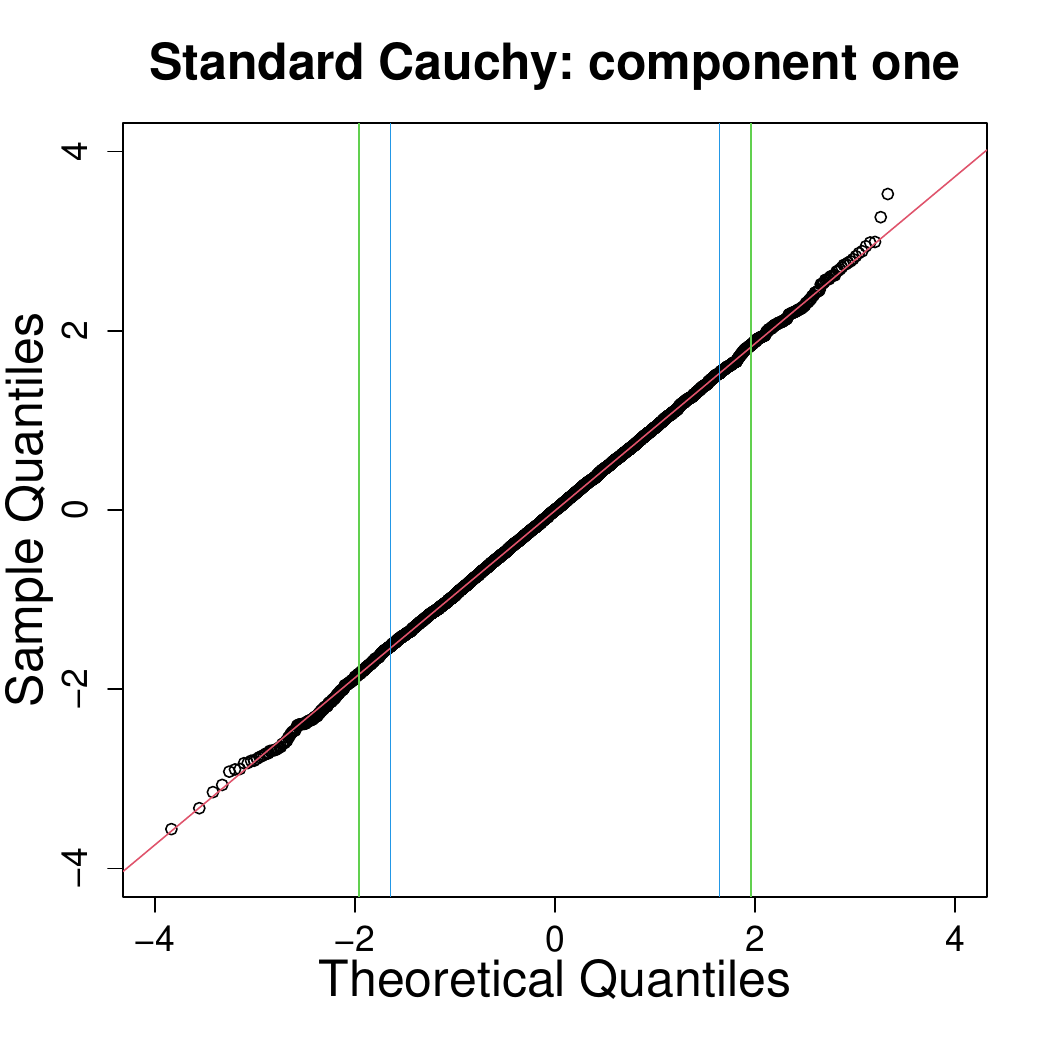}
\includegraphics[width=0.32\textwidth]{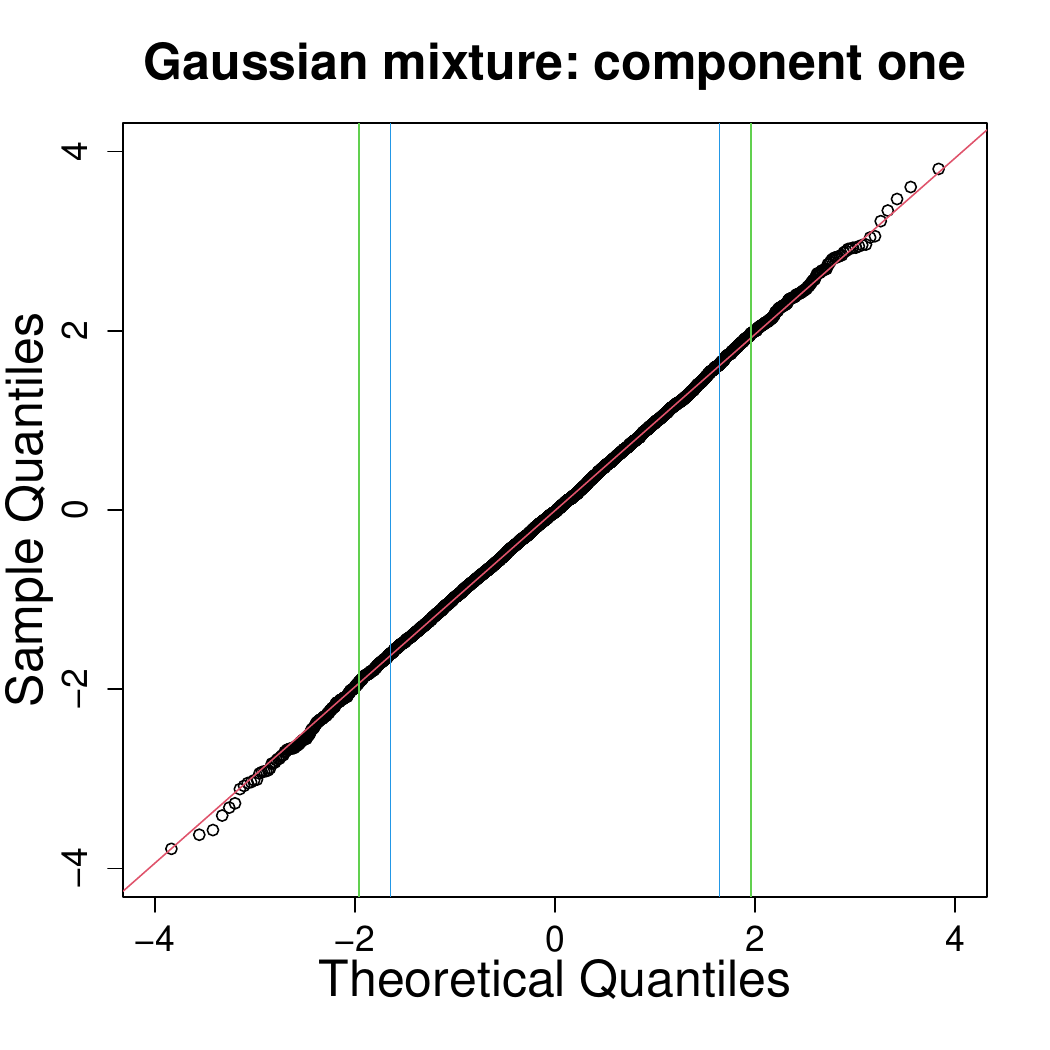}
\includegraphics[width=0.32\textwidth]{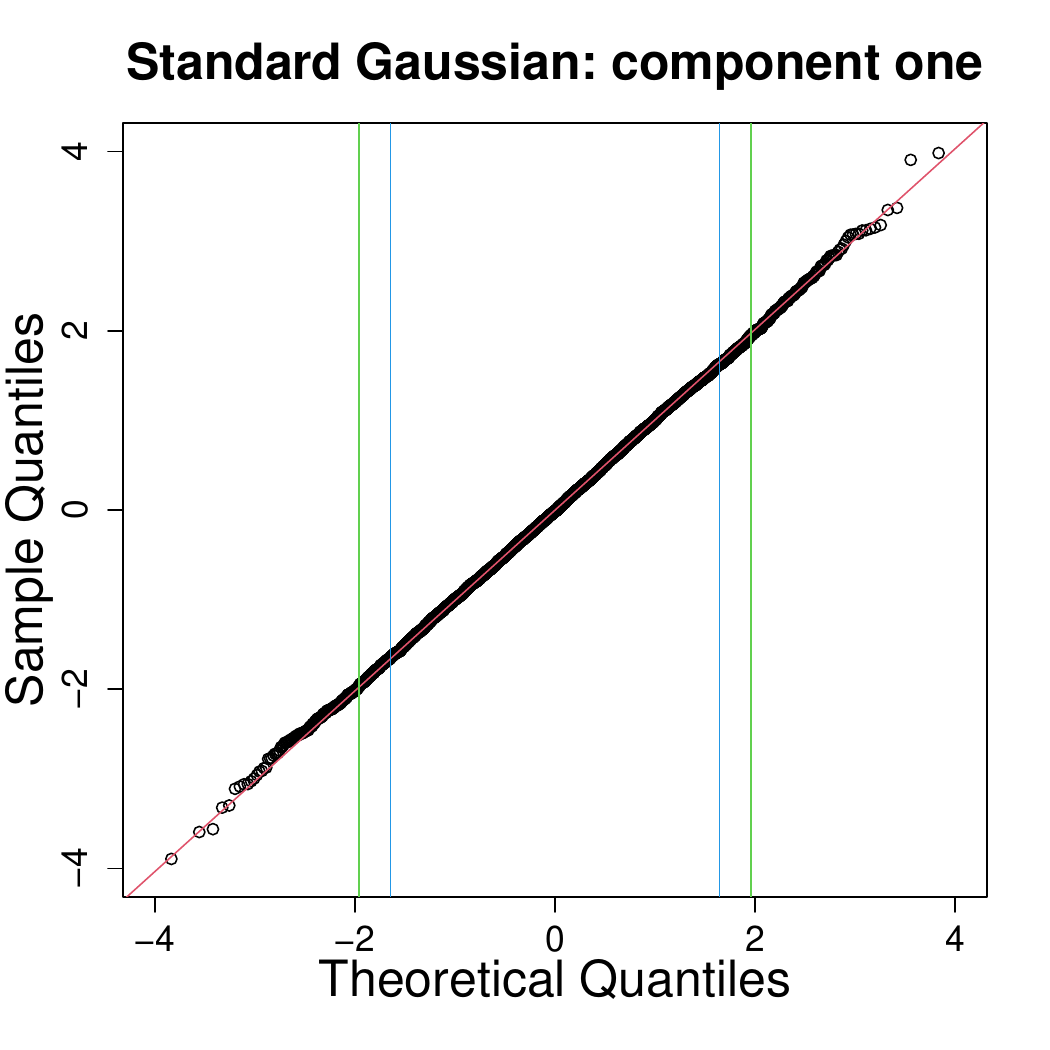}
\includegraphics[width=0.32\textwidth]{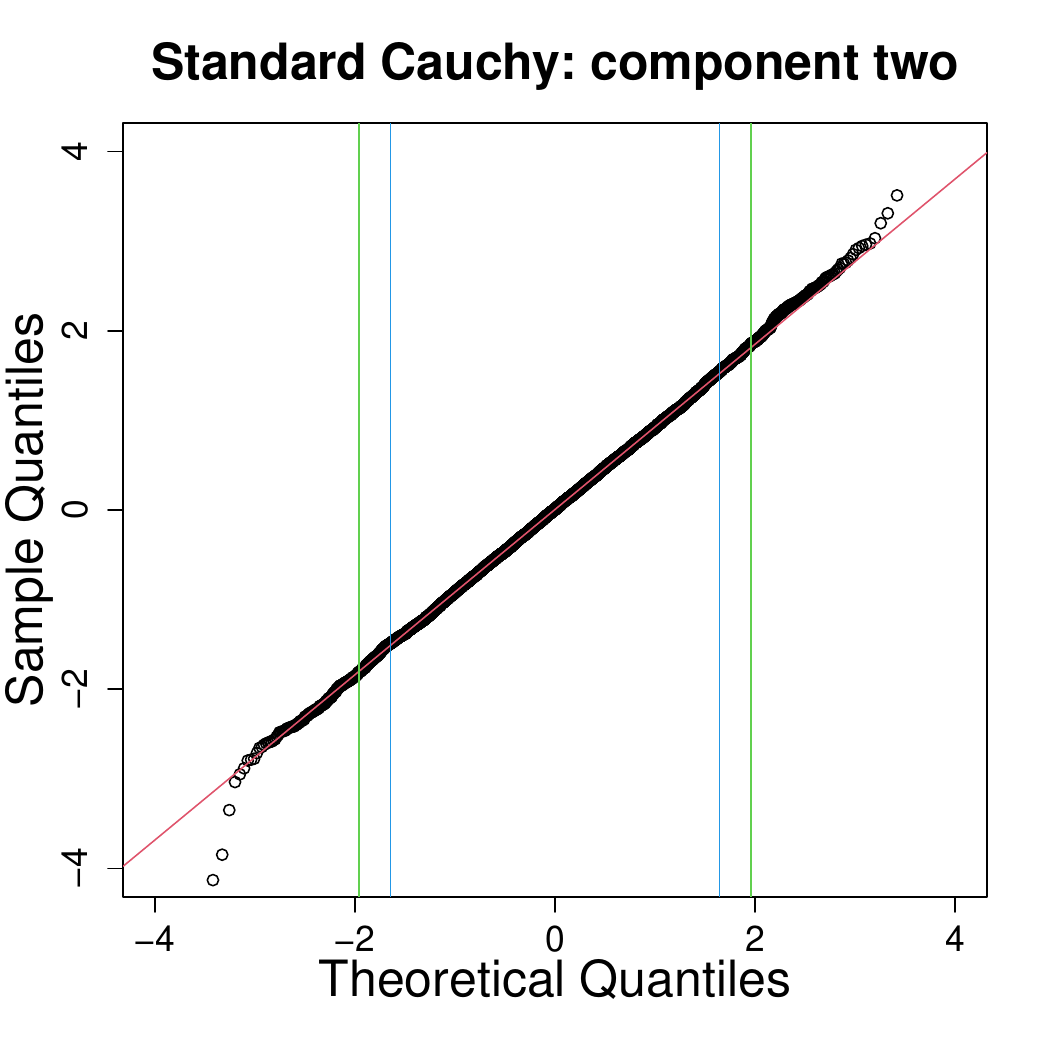}
\includegraphics[width=0.32\textwidth]{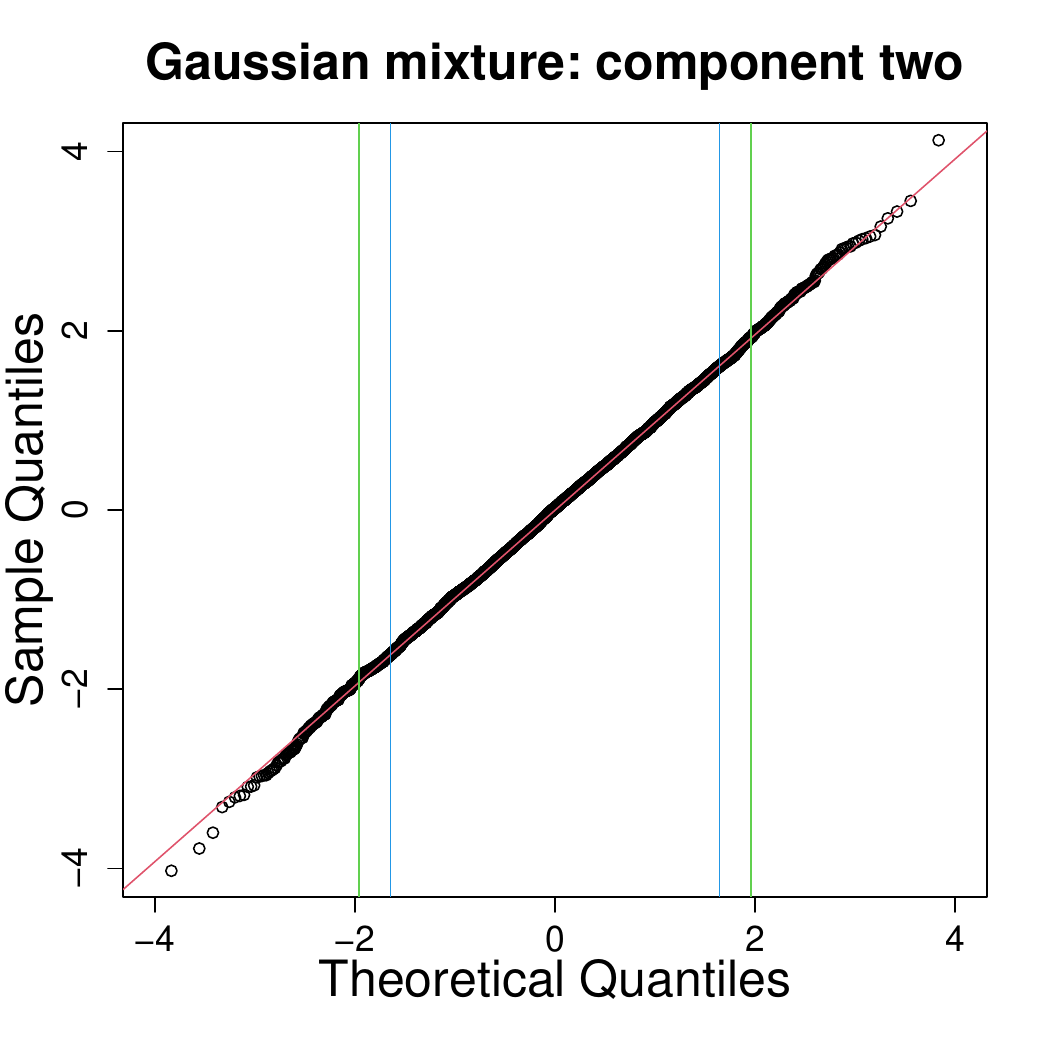}
\includegraphics[width=0.32\textwidth]{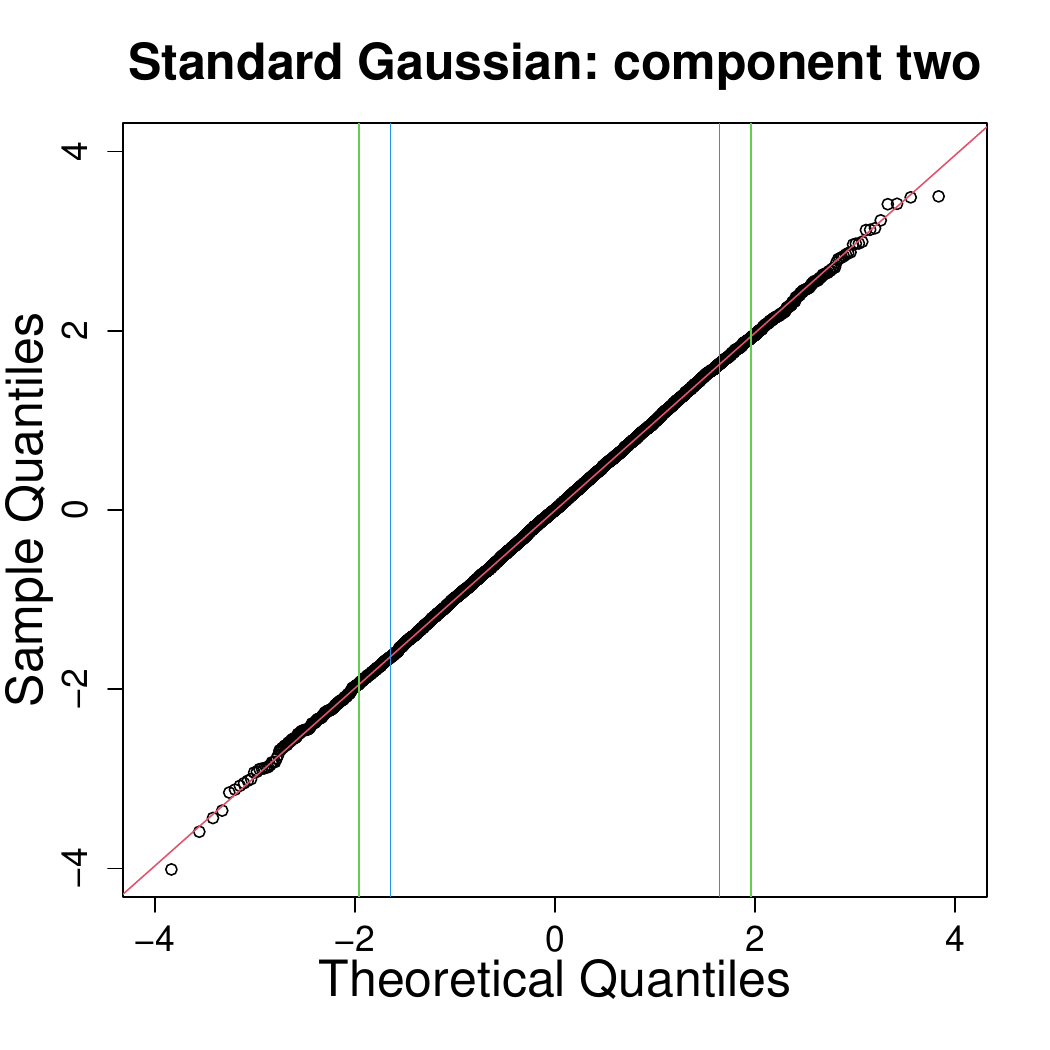}
\vspace{-0.3cm}   
\caption{Q-Q plots of the standardised errors of estimates of two components of $\theta_0$. The blue and green vertical lines mark the 90\% and 95\% theoretical quantiles respectively. In each experiment, we set $n = 600$ and $d = 4$, and perform 8000 repetitions.}
\label{fig:QQ-plot-1}
\end{figure}

\begin{figure}[htp]
\centering
\begin{subfigure}{0.49\textwidth}
\includegraphics[width=\linewidth, trim=0 30 0 0, clip]{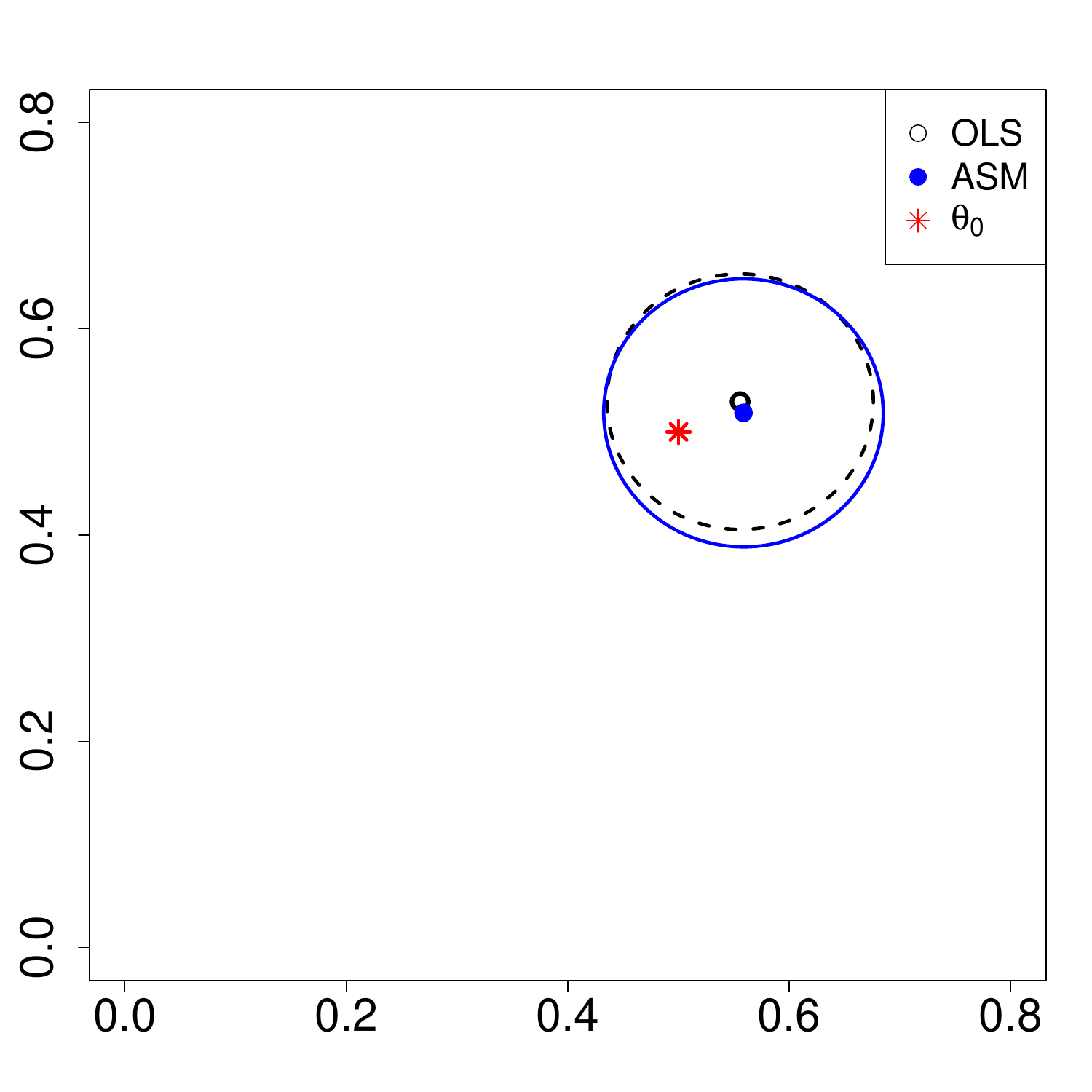}
\caption{Standard normal}
\end{subfigure}
\hfill
\begin{subfigure}{0.49\textwidth}
\includegraphics[width=\linewidth, trim=0 30 0 0, clip]{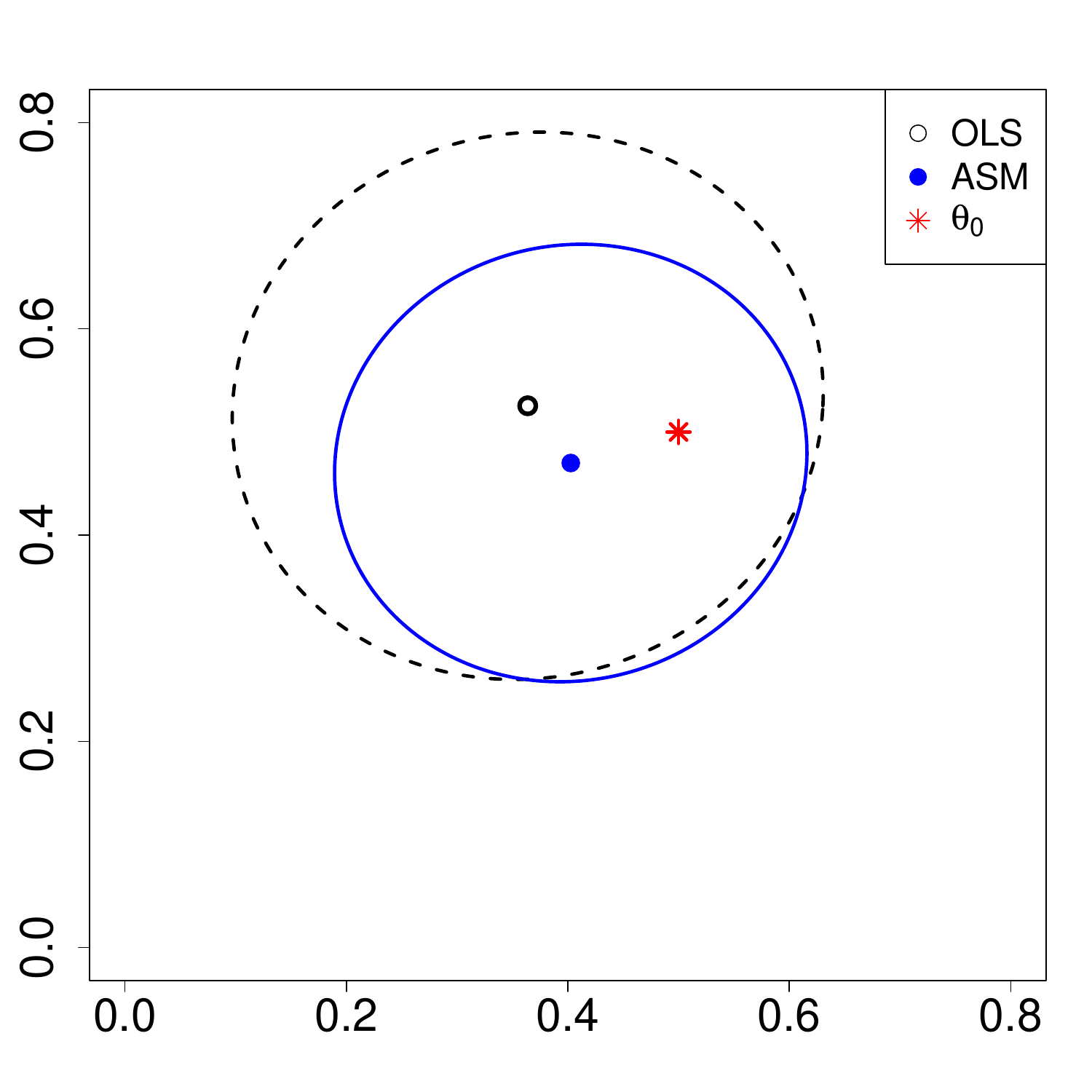}
\caption{Gaussian mixture}
\end{subfigure}
\vspace{-0.3cm}   
\caption{The projection of the $95\%$ confidence ellipsoid of $\hat{\theta}^{\mathrm{ASM}}$ onto the first two dimensions (\textit{blue}) and the projection of the $95\%$ confidence ellipsoid of $\hat{\theta}_n^{\mathrm{OLS}}$ onto the first two dimensions (\textit{black dashed}). The blue dot, black circle and red star correspond to $\hat{\theta}^{\mathrm{ASM}}$, $\hat{\theta}_n^{\mathrm{OLS}}$, and $\theta_0$ respectively.}
\label{fig:confidence-ellipsoids}
\end{figure}

\begin{table}[ht]
\begin{subtable}{0.33\linewidth}
\centering
\begin{tabular}{|r||r|r|r|}
\hline
& $\theta_{0,1}$ & $\theta_{0,2}$ & $\theta_{0,3}$ \\ 
\hline
95\% CI & 0.965 & 0.965 & 0.964 \\ 
90\% CI & 0.924 & 0.925 & 0.924 \\ 
\hline
\end{tabular}
\caption*{Standard Cauchy}
\end{subtable}
\begin{subtable}{0.33\linewidth}
\centering
\begin{tabular}{|r||r|r|r|}
\hline
& $\theta_{0,1}$ & $\theta_{0,2}$ & $\theta_{0,3}$ \\
\hline
95\% CI & 0.952 & 0.955 & 0.955 \\ 
90\% CI & 0.905 & 0.908 & 0.908 \\ 
\hline
\end{tabular}
\caption*{Gaussian mixture}
\end{subtable}
\begin{subtable}{0.33\linewidth}
\centering
\begin{tabular}{|r||r|r|r|}
\hline
& $\theta_{0,1}$ & $\theta_{0,2}$ & $\theta_{0,3}$ \\
\hline
95\% CI & 0.950 & 0.954 & 0.952 \\ 
90\% CI & 0.903 & 0.902 & 0.901 \\ 
\hline
\end{tabular}
\caption*{Standard Gaussian}
\end{subtable}
\caption{Empirical coverage of the confidence intervals for each coordinate of $\theta_0$. In each experiment, we set $n = 600$ and $d = 4$, and perform 8000 repetitions.}
\label{tab:coverage_ratio}
\end{table}

\begin{table}[htbp]
\begin{subtable}{0.5\linewidth}
\centering
\begin{tabular}{|r|r|r|r|}
\hline
& Gaussian & Cauchy & Mixture \\ 
\hline
$i^*(p_0)$ & 1.0 & 0.44 & 0.46   \\
RMSE($\hat{\jmath}_n$) & 0.1 & 0.01 & 0.05  \\ 
% CF-ASM & $7.63$ & $8.68$ & $13.5$ \\ 
\hline
\end{tabular}
\caption{Root mean squared errors of $\hat{\jmath}_n$.}
\label{tab:mse_ihat}
\end{subtable}
\begin{subtable}{0.5\linewidth}
\centering
\begin{tabular}{|r|r|r|r|}
\hline
& Gaussian & Cauchy & Mixture  \\ 
\hline
% ASM & 1.67 & 1.54 & 1.04 \\ 
% OLS & 160.39 & 1.96 & 1.00 \\ 
% Mean ratio $\hat{\gamma}_n$ & 0.07 & 0.79 & 1.04 \\
\rule{0pt}{3ex} $\dfrac{\mathrm{Vol}(\hat{C}_n^{\mathrm{ASM}})}{\mathrm{Vol}(\hat{C}_n^{\mathrm{OLS}})}$ & 1.14 & 0.001 & 0.49 \\[1ex]
\hline
\end{tabular}
\caption{Mean ratio of volumes of $\hat{C}_n^{\mathrm{ASM}}$ and $\hat{C}_n^{\mathrm{OLS}}$.}
\label{tab:volume_ratio}
\end{subtable}
\caption{Accuracy of estimates of the antitonic information, and comparison of the volumes of the two confidence ellipsoids. In each experiment, we set $n = 600$ and $d = 4$, and perform 8000 repetitions.}
\end{table}

\subsection{Estimation accuracy}
\label{subsec:sim_estimation}

In all of the experiments in this subsection, we drew $d = 6$, $\mu_0 = 2$ and drew $\theta_0$ uniformly at random from the centred Euclidean sphere in $\mathbb{R}^{d-1}$ of radius 3.  More precisely, for each of the estimators above, we computed the average squared Euclidean norm errors $\norm{\hat{\theta}_n - \theta_0}^2$ over 200 repetitions. We considered each error distribution in turn and compare the average squared estimation error when $n = 600$. The results are presented in Table~\ref{tab:MSE-comparison} and Figure~\ref{fig:MSE-comparison}. We observe that our proposed procedures ASM and Alt have the lowest estimation error except in the case of standard Gaussian error, where OLS coincides with the oracle convex loss estimator and has a slightly lower estimation error.  It is interesting that the one-step estimator (1S) can perform very poorly in finite samples, and in particular, may barely improve on its initialiser.  In all settings considered, ASM and Alt have comparable error. 
% We therefore recommend ASM (or a few iterations of ASM if there is a particular concern that the pilot estimator may be very poor) in practice due to its lower computational cost. 

% \begin{figure}[ht]
% \centering
% \includegraphics[width=0.32\textwidth]{Figures/MSE_d=100/scaled_MSE_gaussian_d=100.pdf}
% %
% \includegraphics[width=0.32\textwidth]{Figures/MSE_d=100/scaled_MSE_gaussian_scale_mix_d=100.pdf}
% %
% \includegraphics[width=0.32\textwidth]{Figures/MSE_d=100/scaled_MSE_gaussian_location_mix_d=100.pdf}

% % \vspace{-0.3cm}   
% \caption{Plot of the mean squared errors of different estimators. In each experiment, we set $d = 100$ and $n = 600$, and perform 200 repetitions.}
% \end{figure}

% \begin{table}[ht]
% \centering
% \begin{tabular}{|l|r|r|r|r|r|}
%   \hline
%  & ASM & Alt & LCMLE & 1S \\ 
%   \hline
%   Standard Gaussian & 0.06 & 0.59 & 3.00 & 0.05 \\ 
%   % Standard Cauchy & 0.45 & 5.30 & 4.50 & 1.13 & 0.62 \\ 
%   Gaussian scale mixture & 0.06 & 0.58 & 1.56 & 0.05 \\ 
%   Gaussian location mixture & 0.06 & 0.38 &4.11 & 0.05 \\ 
%   % Smoothed uniform & 0.06 & 0.49 & 0.39 & 5.30 & 0.05 \\ 
%   % Smoothed exponential & 0.06 & 0.52 & 0.41 & 5.43 & 0.05 \\ 
%    \hline
% \end{tabular}
% \caption{Mean execution time in seconds for different estimation procedures. In each experiment, we set $d = 6$ and $n = 600$, and perform 400 repetitions. \green{How do these values change as $n$ increases? A plot might be helpful.}}
% \label{tab:execution-time-comparison}
% \end{table}

Next, we investigate the running time and estimation accuracy of the semiparametric estimators ASM, Alt, LCMLE and 1S for moderately large sample sizes $n \in \{600, 1200, 2400\}$ and dimension $d = 100$. We consider in turn the error distributions (i), (iii) and (iv); see Table~\ref{tab:execution-time} and Figure~\ref{fig:high_dim_sim}.  We see that ASM in particular is competitive in terms of its running time, and that the performance of our procedures does not deteriorate relative to its competitors for this larger choice of $d$.

Our third numerical experiment demonstrates that the estimation error of ASM is comparable to that of the oracle convex $M$-estimator even for small sample sizes. For $n \in \{50, 100, 150, 200, 300\}$ and $d = 6$, and error distributions~(i),~(ii) and~(iv), we record in Figure~\ref{fig:MSE vs sample size} the average estimation error of Oracle, ASM, and LCMLE.  With the exception of the smallest sample size $n = 50$ in the Gaussian location mixture setting, the estimation error of ASM tracks that of the oracle very closely; on the other hand, LCMLE is somewhat suboptimal, particularly in the Gaussian location mixture setting.

Finally in this subsection, we verify empirically the suboptimality of LCMLE relative to ASM for larger sample sizes. We let $d = 2$ and $n \in \{800, 1600, 3200, 6400\}$, and consider error distribution (iv) above as well as the $t_2$ distribution (Example~\ref{ex:t2}).  We include the $t_2$ distribution in place of the standard Cauchy because when $P_0$ does not have a finite first moment, its population level log-concave maximum likelihood projection does not exist (and indeed we found empirically that LCMLE was highly unstable in the Cauchy setting). Figure~\ref{fig:MSE larger sample size} displays the average squared error loss of the oracle convex $M$-estimator, ASM, and LCMLE. We see that the relative efficiency of ASM with respect to Oracle approaches 1 as $n$ increases, while LCMLE has an efficiency gap that does not vanish even for large $n$, in agreement with the asymptotic calculation in Example~\ref{ex:t2}. 

\subsection{Inference}

In this subsection, we assess the finite-sample performance of the inferential  procedures for $\theta_0$ described in Section~\ref{subsec:linreg-inference}, but do not perform sample splitting or cross-fitting. Here, we estimate the antitonic information $i^*(p_0)$ by
\[
\hat{\jmath}_n := \frac{1}{n}\sum_{i=1}^n \hat{\psi}_n(\hat{\varepsilon}_i)^2,
\]
rather than $\hat{\imath}_n$, where $\hat{\varepsilon}_i := Y_i - \bar{\mu}_n - \tilde{X}_i^\top \bar{\theta}_n$ for $i \in [n]$ is the $i$th residual of the pilot estimator $(\bar{\theta}_n,\bar{\mu}_n)$. We then compute $\tilde{\mathcal{J}}_n := \frac{\tilde{\jmath}_n}{n} \sum_{i=1}^n (\tilde{X}_i - \bar{X}_n) (\tilde{X}_i - \bar{X}_n)^\top$. 

Figure~\ref{fig:QQ-plot-1} displays Q-Q plots, computed over 8000 repetitions and with $n = 600$ and $d = 4$, of $\sqrt{n}\tilde{\mathcal{J}}_n^{1/2} (\hat{\theta}_n^{\mathrm{ASM}} - \theta_0)$ against $N_{d-1}(0, I_{d-1})$ for three different noise distributions: standard Gaussian, standard Cauchy and the Gaussian mixture $P_0 = \frac{2}{3} N(0, 1) + \frac{1}{3} N\bigl(\frac{1}{2}, 9\bigr)$.  These reveal that the distribution of $\sqrt{n}\tilde{\mathcal{J}}_n^{1/2} (\hat{\theta}_n^{\mathrm{ASM}} - \theta_0)$ is well-approximated by that of $N_{d-1}(0, I_{d-1})$ even quite far into the tails of the distribution.  Moreover, Table~\ref{tab:mse_ihat} demonstrates that the root mean squared estimation error of $\hat{\jmath}_n$ as an estimator of $i^*(p_0)$ is small. 

In Table~\ref{tab:coverage_ratio}, we present coverage probabilities of the confidence ellipsoid 
\[
\hat{C}_n^{\mathrm{ASM}} := \bigl\{v \in \mathbb{R}^{d-1} : n (\hat{\theta}_n^{\mathrm{ASM}} - v)^\top \tilde{\mathcal{J}}_n (\hat{\theta}_n^{\mathrm{ASM}} - v) \leq \chi^2_{d-1}(\alpha) \bigr\}
\]
in the same settings as in the previous paragraph.  For the standard Gaussian and Gaussian mixture noise distributions above, $\hat{C}_n^{\mathrm{ASM}}$ maintains nominal coverage, while for standard Cauchy errors, it is slightly conservative.  We also compare, in Table~\ref{tab:volume_ratio}, the volume of $\hat{C}_n^{\mathrm{ASM}}$ with that of the OLS confidence ellipsoid
\[
\hat{C}_n^{\mathrm{OLS}} := \bigl\{v \in \mathbb{R}^{d-1} : n (\hat{\theta}_n^{\mathrm{OLS}} - v)^\top \hat{I}_n^{\mathrm{OLS}} (\hat{\theta}_n^{\mathrm{OLS}} - v) \leq \chi^2_{d-1}(\alpha)\bigr\},
\]
where $\hat{I}_n^{\mathrm{OLS}} := n^{-1}\sum_{i=1}^n \hat{\sigma}^{-2}_n(\tilde{X}_i - \bar{X}_n)(\tilde{X}_i - \bar{X}_n)^\top$ and $\hat{\sigma}_n^2 := n^{-1} \sum_{i=1}^n (Y_i  - \hat{\mu}^{\mathrm{OLS}}_n - \tilde{X}_i^\top \hat{\theta}_n^{\mathrm{OLS}})^2$.  The $\hat{C}_n^{\mathrm{ASM}}$ ellipsoid is dramatically smaller in the Cauchy example, and also appreciably smaller for the Gaussian mixture, while on average it is a little larger in the Gaussian case.  Examples of $\hat{C}_n^{\mathrm{ASM}}$ and $\hat{C}_n^{\mathrm{OLS}}$ are plotted in Figure~\ref{fig:confidence-ellipsoids}.

\section{Discussion}

One of the messages of this paper is that, despite the Gauss--Markov theorem, the success of ordinary least squares is relatively closely tied to Gaussian or near-Gaussian error distributions.  Our antitonic score matching approach aims to free the practitioner from the Gaussian straitjacket while retaining the convenience and stability of working with convex loss functions.  The Fisher divergence projection framework brings together previously disparate ideas on shape-constrained estimation, score matching, information theory and classical robust statistics.  Given the prevalence of procedures in statistics and machine learning that are constructed as optimisers of pre-specified loss functions, we look forward to seeing how related insights may lead to more flexible, data-driven and computationally feasible approaches that combine robustness and efficiency.

\medskip
\noindent
\textbf{Acknowledgements}: The authors thank David Firth, Elliot Young and Cun-Hui Zhang for helpful discussions, as well as an Associate Editor and two anonymous reviewers whose constructive comments led to several improvements in the paper. The research of YCK and MX was supported by National Science Foundation grants DMS-2311299 and DMS-2113671; OYF and RJS were supported by Engineering and Physical Sciences Research Council Programme Grant EP/N031938/1, while RJS was also supported by European Research Council Advanced Grant 101019498.

% https://openreview.net/forum?id=PxTIG12RRHS

\bibliographystyle{apalike}
\bibliography{bib}

\begin{thebibliography}{}

\bibitem[Amari and Nagaoka, 2000]{amari2000methods}
Amari, S.-i. and Nagaoka, H. (2000).
\newblock {\em Methods of Information Geometry}, volume 191.
\newblock American Mathematical Society.

\bibitem[Arcones, 1998]{arcones1998asymptotic}
Arcones, M.~A. (1998).
\newblock Asymptotic theory for {$M$}-estimators over a convex kernel.
\newblock {\em Econometric Theory}, 14(4):387--422.

\bibitem[Barber and Samworth, 2021]{barber2021local}
Barber, R.~F. and Samworth, R.~J. (2021).
\newblock Local continuity of log-concave projection, with applications to
  estimation under model misspecification.
\newblock {\em Bernoulli}, 27(4):2437--2472.

\bibitem[Barron, 2019]{barron2019general}
Barron, J.~T. (2019).
\newblock A general and adaptive robust loss function.
\newblock In {\em Proceedings of the IEEE/CVF Conference on Computer Vision and
  Pattern Recognition}, pages 4331--4339.

\bibitem[Bean et~al., 2013]{bean2013optimal}
Bean, D., Bickel, P.~J., El~Karoui, N., and Yu, B. (2013).
\newblock Optimal {$M$}-estimation in high-dimensional regression.
\newblock {\em Proceedings of the National Academy of Sciences},
  110(36):14563--14568.

\bibitem[Benton et~al., 2024]{benton2024denoising}
Benton, J., Shi, Y., De~Bortoli, V., Deligiannidis, G., and Doucet, A. (2024).
\newblock From denoising diffusions to denoising {M}arkov models.
\newblock {\em Journal of the Royal Statistical Society, Series B},
  86(2):286--301.

\bibitem[Beran, 1978]{beran1978efficient}
Beran, R. (1978).
\newblock An efficient and robust adaptive estimator of location.
\newblock {\em The Annals of Statistics}, 6(2):292--313.

\bibitem[Betancourt et~al., 2017]{betancourt2017geometric}
Betancourt, M., Byrne, S., Livingstone, S., and Girolami, M. (2017).
\newblock The geometric foundations of {H}amiltonian {M}onte {C}arlo.
\newblock {\em Bernoulli}, 23(4A):2257--2298.

\bibitem[Bickel, 1975]{bickel1975one}
Bickel, P.~J. (1975).
\newblock One-step {H}uber estimates in the linear model.
\newblock {\em Journal of the American Statistical Association},
  70(350):428--434.

\bibitem[Bickel, 1982]{bickel1982adaptive}
Bickel, P.~J. (1982).
\newblock On adaptive estimation.
\newblock {\em The Annals of Statistics}, 10(3):647--671.

\bibitem[Bobkov, 1996]{bobkov1996extremal}
Bobkov, S.~G. (1996).
\newblock Extremal properties of half-spaces for log-concave distributions.
\newblock {\em The Annals of Probability}, 24(1):35--48.

\bibitem[Boyd and Vandenberghe, 2004]{boyd2004convex}
Boyd, S.~P. and Vandenberghe, L. (2004).
\newblock {\em Convex Optimization}.
\newblock Cambridge University Press.

\bibitem[Brunel, 2023]{brunel2023geodesically}
Brunel, V.-E. (2023).
\newblock Geodesically convex {$M$}-estimation in metric spaces.
\newblock In {\em The Thirty Sixth Annual Conference on Learning Theory}, pages
  2188--2210. PMLR.

\bibitem[Catoni, 2012]{catoni2012challenging}
Catoni, O. (2012).
\newblock Challenging the empirical mean and empirical variance: a deviation
  study.
\newblock {\em Annales de l'Institut Henri Poincaré -- Probabilités et
  Statistiques}, 48(4):1148--1185.

\bibitem[Celentano and Montanari, 2022]{celentano2022fundamental}
Celentano, M. and Montanari, A. (2022).
\newblock Fundamental barriers to high-dimensional regression with convex
  penalties.
\newblock {\em The Annals of Statistics}, 50(1):170--196.

\bibitem[Chen and Samworth, 2013]{chen2013smoothed}
Chen, Y. and Samworth, R.~J. (2013).
\newblock Smoothed log-concave maximum likelihood estimation with applications.
\newblock {\em Statistica Sinica}, 23(3):1373--1398.

\bibitem[Cheng et~al., 2018]{cheng2018underdamped}
Cheng, X., Chatterji, N.~S., Bartlett, P.~L., and Jordan, M.~I. (2018).
\newblock Underdamped {L}angevin {MCMC}: a non-asymptotic analysis.
\newblock In {\em Conference on Learning Theory}, pages 300--323. PMLR.

\bibitem[Chernozhukov et~al., 2018]{chernozhukov2018double}
Chernozhukov, V., Chetverikov, D., Demirer, M., Duflo, E., Hansen, C., Newey,
  W., and Robins, J. (2018).
\newblock Double/debiased machine learning for treatment and structural
  parameters.
\newblock {\em The Econometrics Journal}, 21(1):C1--C68.

\bibitem[Chinot et~al., 2020]{chinot2020robust}
Chinot, G., Lecu\'e, G., and Lerasle, M. (2020).
\newblock Robust statistical learning with {L}ipschitz and convex loss
  functions.
\newblock {\em Probability Theory and Related Fields}, 176(3):897--940.

\bibitem[Cook, 1977]{cook1977detection}
Cook, R.~D. (1977).
\newblock Detection of influential observation in linear regression.
\newblock {\em Technometrics}, 19(1):15--18.

\bibitem[Cover and Thomas, 2006]{cover2006elements}
Cover, T.~M. and Thomas, J.~A. (2006).
\newblock {\em Elements of {I}nformation {T}heory}.
\newblock John Wiley \& Sons, 2nd edition.

\bibitem[Cox, 1985]{cox1985penalty}
Cox, D.~D. (1985).
\newblock A penalty method for nonparametric estimation of the logarithmic
  derivative of a density function.
\newblock {\em Annals of the Institute of Statistical Mathematics},
  37(2):271--288.

\bibitem[Cule et~al., 2010]{cule2010maximum}
Cule, M., Samworth, R., and Stewart, M. (2010).
\newblock Maximum likelihood estimation of a multi-dimensional log-concave
  density.
\newblock {\em Journal of the Royal Statistical Society Series B: Statistical
  Methodology}, 72(5):545--607.

\bibitem[Dalalyan et~al., 2006]{dalalyan2006penalized}
Dalalyan, A.~S., Golubev, G.~K., and Tsybakov, A.~B. (2006).
\newblock Penalized maximum likelihood and semiparametric second-order
  efficiency.
\newblock {\em The Annals of Statistics}, 34(1):169--201.

\bibitem[De~Bortoli et~al., 2022]{bortoli2022riemannian}
De~Bortoli, V., Mathieu, E., Hutchinson, M., Thornton, J., Teh, Y.~W., and
  Doucet, A. (2022).
\newblock Riemannian score-based generative modelling.
\newblock {\em Advances in Neural Information Processing Systems},
  35:2406--2422.

\bibitem[Derenski et~al., 2023]{derenski23empirical}
Derenski, J., Fan, Y., James, G., and Xu, M. (2023).
\newblock An empirical {B}ayes shrinkage method for functional data.
\newblock {\em Submitted}.

\bibitem[Donoho and Montanari, 2016]{donoho2016high}
Donoho, D. and Montanari, A. (2016).
\newblock High dimensional robust {M}-estimation: asymptotic variance via
  approximate message passing.
\newblock {\em Probability Theory and Related Fields}, 166:935--969.

\bibitem[Donoho and Montanari, 2015]{donoho2015variance}
Donoho, D.~L. and Montanari, A. (2015).
\newblock Variance breakdown of {H}uber {$M$}-estimators: $n/p \in (1,\infty)$.
\newblock {\em arXiv preprint arXiv:1503.02106}.

\bibitem[Doss and Wellner, 2019]{doss2019univariate}
Doss, C.~R. and Wellner, J.~A. (2019).
\newblock Univariate log-concave density estimation with symmetry or modal
  constraints.
\newblock {\em Electronic Journal of Statistics}, 13(2):2391--2461.

\bibitem[D\"umbgen et~al., 2011]{dumbgen2011approximation}
D\"umbgen, L., Samworth, R., and Schuhmacher, D. (2011).
\newblock Approximation by log-concave distributions, with applications to
  regression.
\newblock {\em The Annals of Statistics}, 39(2):702--730.

\bibitem[D\"umbgen et~al., 2013]{dumbgen2013stochastic}
D\"umbgen, L., Samworth, R.~J., and Schuhmacher, D. (2013).
\newblock Stochastic search for semiparametric linear regression models.
\newblock In {\em From Probability to Statistics and Back: High-Dimensional
  Models and Processes--A Festschrift in Honor of Jon A. Wellner}, volume~9,
  pages 78--91. Institute of Mathematical Statistics.

\bibitem[Efron, 2011]{efron2011tweedie}
Efron, B. (2011).
\newblock Tweedie’s formula and selection bias.
\newblock {\em Journal of the American Statistical Association},
  106(496):1602--1614.

\bibitem[Eggermont and LaRiccia, 2000]{eggermont2000maximum}
Eggermont, P. P.~B. and LaRiccia, V.~N. (2000).
\newblock Maximum likelihood estimation of smooth monotone and unimodal
  densities.
\newblock {\em The Annals of Statistics}, 28(3):922--947.

\bibitem[El~Karoui, 2018]{elkaroui2018impact}
El~Karoui, N. (2018).
\newblock On the impact of predictor geometry on the performance on
  high-dimensional ridge-regularized generalized robust regression estimators.
\newblock {\em Probability Theory and Related Fields}, 170:95--175.

\bibitem[El~Karoui et~al., 2013]{elkaroui2013robust}
El~Karoui, N., Bean, D., Bickel, P.~J., Lim, C., and Yu, B. (2013).
\newblock On robust regression with high-dimensional predictors.
\newblock {\em Proceedings of the National Academy of Sciences},
  110(36):14557--14562.

\bibitem[Faraway, 1992]{faraway1992smoothing}
Faraway, J.~J. (1992).
\newblock Smoothing in adaptive estimation.
\newblock {\em The Annals of Statistics}, 20(1):414--427.

\bibitem[Folland, 1999]{folland1999real}
Folland, G.~B. (1999).
\newblock {\em Real Analysis: Modern Techniques and their Applications},
  volume~40.
\newblock John Wiley \& Sons.

\bibitem[Ghosh et~al., 2025]{ghosh2025stein}
Ghosh, S., Ignatiadis, N., Koehler, F., and Lee, A. (2025).
\newblock Stein's unbiased risk estimate and {H}yv\"arinen's score matching.
\newblock {\em arXiv preprint arXiv:2502.20123}.

\bibitem[Groeneboom and Jongbloed, 2014]{groeneboom14nonparametric}
Groeneboom, P. and Jongbloed, G. (2014).
\newblock {\em Nonparametric Estimation under Shape Constraints}.
\newblock Cambridge University Press.

\bibitem[Gupta et~al., 2023]{gupta2023finite}
Gupta, S., Lee, J. C.~H., and Price, E. (2023).
\newblock Finite-sample symmetric mean estimation with {F}isher information
  rate.
\newblock In {\em The Thirty Sixth Annual Conference on Learning Theory}, pages
  4777--4830. PMLR.

\bibitem[Hampel, 1974]{hampel1974influence}
Hampel, F.~R. (1974).
\newblock The influence curve and its role in robust estimation.
\newblock {\em Journal of the American Statistical Association},
  69(346):383--393.

\bibitem[Hampel et~al., 2011]{hampel2011robust}
Hampel, F.~R., Ronchetti, E.~M., Rousseeuw, P.~J., and Stahel, W.~A. (2011).
\newblock {\em Robust Statistics: The approach based on influence functions}.
\newblock John Wiley \& Sons.

\bibitem[Hansen, 2022]{hansen2022modern}
Hansen, B.~E. (2022).
\newblock A modern {G}auss--{M}arkov theorem.
\newblock {\em Econometrica}, 90(3):1283--1294.

\bibitem[He and Shao, 2000]{he2000parameters}
He, X. and Shao, Q.-M. (2000).
\newblock On parameters of increasing dimensions.
\newblock {\em Journal of Multivariate Analysis}, 73(1):120--135.

\bibitem[Hoerl and Kennard, 1970]{hoerl1970ridge}
Hoerl, A.~E. and Kennard, R.~W. (1970).
\newblock Ridge regression: biased estimation for nonorthogonal problems.
\newblock {\em Technometrics}, 12(1):55--67.

\bibitem[Huber, 1964]{huber1964robust}
Huber, P.~J. (1964).
\newblock Robust estimation of a location parameter.
\newblock {\em The Annals of Mathematical Statistics}, 35(1):73--101.

\bibitem[Huber, 1967]{huber1967behavior}
Huber, P.~J. (1967).
\newblock The behavior of maximum likelihood estimates under nonstandard
  conditions.
\newblock In {\em Proceedings of the fifth Berkeley Symposium on Mathematical
  Statistics and Probability}, volume~1, pages 221--233. Berkeley, CA:
  University of California Press.

\bibitem[Huber and Ronchetti, 2009]{huber2009robust}
Huber, P.~J. and Ronchetti, E.~M. (2009).
\newblock {\em Robust Statistics}.
\newblock Wiley, 2nd edition.

\bibitem[Hyv\"arinen, 2005]{hyvarinen05score}
Hyv\"arinen, A. (2005).
\newblock Estimation of non-normalized statistical models by score matching.
\newblock {\em Journal of Machine Learning Research}, 6:695--709.

\bibitem[Hyv{\"a}rinen, 2007]{hyvarinen2007extensions}
Hyv{\"a}rinen, A. (2007).
\newblock Some extensions of score matching.
\newblock {\em Computational Statistics \& Data Analysis}, 51(5):2499--2512.

\bibitem[Jankov{\'a} et~al., 2020]{jankova2020goodness}
Jankov{\'a}, J., Shah, R.~D., B{\"u}hlmann, P., and Samworth, R.~J. (2020).
\newblock Goodness-of-fit testing in high dimensional generalized linear
  models.
\newblock {\em Journal of the Royal Statistical Society Series B: Statistical
  Methodology}, 82(3):773--795.

\bibitem[Jin, 1990]{jin1990empirical}
Jin, K. (1990).
\newblock {\em Empirical Smoothing Parameter Selection in Adaptive Estimation}.
\newblock University of California, Berkeley.

\bibitem[Johnson, 2004]{johnson2004information}
Johnson, O. (2004).
\newblock {\em Information Theory and the Central Limit Theorem}.
\newblock World Scientific.

\bibitem[Johnson and Barron, 2004]{johnson2004fisher}
Johnson, O. and Barron, A. (2004).
\newblock Fisher information inequalities and the central limit theorem.
\newblock {\em Probability Theory and Related Fields}, 129:391--409.

\bibitem[Jolicoeur-Martineau et~al., 2020]{jolicoeur2020adversarial}
Jolicoeur-Martineau, A., Pich{\'e}-Taillefer, R., des Combes, R.~T., and
  Mitliagkas, I. (2020).
\newblock Adversarial score matching and improved sampling for image
  generation.
\newblock {\em arXiv preprint arXiv:2009.05475}.

\bibitem[Jones, 1992]{jones1992estimating}
Jones, M.~C. (1992).
\newblock Estimating densities, quantiles, quantile densities and density
  quantiles.
\newblock {\em Annals of the Institute of Statistical Mathematics},
  44:721--727.

\bibitem[Kao et~al., 2024a]{kao24asm}
Kao, Y.-C., Xu, M., Feng, O.~Y., and Samworth, R.~J. (2024a).
\newblock {\em asm: Optimal convex {$M$}-estimation for linear regression via
  antitonic score matching}.
\newblock R package version 0.2.4,
  \url{https://CRAN.R-project.org/package=asm}.

\bibitem[Kao et~al., 2024b]{kao24choosing}
Kao, Y.-C., Xu, M., and Zhang, C.-H. (2024b).
\newblock Choosing the $p$ in {$L_p$} loss: adaptive rates for symmetric mean
  estimation.
\newblock In {\em The Thirty Seventh Annual Conference on Learning Theory},
  pages 2795--2839. PMLR.

\bibitem[Koehler et~al., 2022]{koehler2022statistical}
Koehler, F., Heckett, A., and Risteski, A. (2022).
\newblock Statistical efficiency of score matching: The view from isoperimetry.
\newblock {\em arXiv preprint arXiv:2210.00726}.

\bibitem[Kosorok, 2008]{kosorok2008introduction}
Kosorok, M.~R. (2008).
\newblock {\em Introduction to Empirical Processes and Semiparametric
  Inference}.
\newblock Springer.

\bibitem[Laha, 2021]{laha2021adaptive}
Laha, N. (2021).
\newblock Adaptive estimation in symmetric location model under log-concavity
  constraint.
\newblock {\em Electronic Journal of Statistics}, 15(1):2939--3014.

\bibitem[Lederer and Oesting, 2023]{lederer2023extremes}
Lederer, J. and Oesting, M. (2023).
\newblock Extremes in high dimensions: methods and scalable algorithms.
\newblock {\em arXiv preprint arXiv:2303.04258}.

\bibitem[Lei and Wooldridge, 2022]{lei2022estimators}
Lei, L. and Wooldridge, J. (2022).
\newblock What estimators are unbiased for linear models?
\newblock {\em arXiv preprint arXiv:2212.14185}.

\bibitem[Lerasle, 2019]{lerasle2019selected}
Lerasle, M. (2019).
\newblock Selected topics on robust statistical learning theory.
\newblock {\em arXiv preprint arxiv:1908.10761}.

\bibitem[Ley and Swan, 2013]{ley2013stein}
Ley, C. and Swan, Y. (2013).
\newblock Stein's density approach and information inequalities.
\newblock {\em Electronic Communications in Probability}, 18:1--14.

\bibitem[Li et~al., 2024]{li2024towards}
Li, G., Wei, Y., Chen, Y., and Chi, Y. (2024).
\newblock Towards non-asymptotic convergence for diffusion-based generative
  models.
\newblock In {\em The Twelfth International Conference on Learning
  Representations}.

\bibitem[Loh, 2021]{loh2021scale}
Loh, P.-L. (2021).
\newblock Scale calibration for high-dimensional robust regression.
\newblock {\em Electronic Journal of Statistics}, 15(2):5933--5994.

\bibitem[Lyu, 2012]{lyu2012interpretation}
Lyu, S. (2012).
\newblock Interpretation and generalization of score matching.
\newblock {\em arXiv preprint arXiv:1205.2629}.

\bibitem[Mammen, 1989]{mammen1989asymptotics}
Mammen, E. (1989).
\newblock Asymptotics with increasing dimension for robust regression with
  applications to the bootstrap.
\newblock {\em The Annals of Statistics}, 17(1):382--400.

\bibitem[Mammen and Park, 1997]{mammen1997optimal}
Mammen, E. and Park, B.~U. (1997).
\newblock Optimal smoothing in adaptive location estimation.
\newblock {\em Journal of Statistical Planning and Inference}, 58(2):333--348.

\bibitem[Mardia et~al., 2016]{mardia2016score}
Mardia, K.~V., Kent, J.~T., and Laha, A.~K. (2016).
\newblock Score matching estimators for directional distributions.
\newblock {\em arXiv preprint arXiv:1604.08470}.

\bibitem[Maronna and Yohai, 1981]{maronna1981asymptotic}
Maronna, R.~A. and Yohai, V.~J. (1981).
\newblock Asymptotic behavior of general {$M$}-estimates for regression and
  scale with random carriers.
\newblock {\em Zeitschrift f\"ur Wahrscheinlichkeitstheorie und verwandte
  Gebiete}, 58:7--20.

\bibitem[Nocedal and Wright, 2006]{nocedal2006numerical}
Nocedal, J. and Wright, S.~J. (2006).
\newblock {\em Numerical Optimization}.
\newblock Springer, 2nd edition.

\bibitem[Parisi, 1981]{parisi1981correlation}
Parisi, G. (1981).
\newblock Correlation functions and computer simulations.
\newblock {\em Nuclear Physics B}, 180(3):378--384.

\bibitem[Parzen, 1979]{parzen1979nonparametric}
Parzen, E. (1979).
\newblock Nonparametric statistical data modeling.
\newblock {\em Journal of the American Statistical Association},
  74(365):105--121.

\bibitem[Portnoy, 1985]{portnoy1985asymptotic}
Portnoy, S. (1985).
\newblock Asymptotic behavior of {$M$} estimators of $p$ regression parameters
  when $p^2/n $ is large; {II}. normal approximation.
\newblock {\em The Annals of Statistics}, 13(4):1403--1417.

\bibitem[P\"otscher and Preinerstorfer, 2024]{potscher2024comment}
P\"otscher, B.~M. and Preinerstorfer, D. (2024).
\newblock A comment on ``{A} modern {G}auss--{M}arkov theorem''.
\newblock {\em Econometrica}, 92(3):913--924.

\bibitem[Roberts and Tweedie, 1996]{roberts1996exponential}
Roberts, G.~O. and Tweedie, R.~L. (1996).
\newblock Exponential convergence of {L}angevin distributions and their
  discrete approximations.
\newblock {\em Bernoulli}, 2:341--363.

\bibitem[Robertson et~al., 1988]{robertson1988order}
Robertson, T., Wright, F.~T., and Dykstra, R. (1988).
\newblock {\em Order Restricted Statistical Inference}.
\newblock Wiley.

\bibitem[Rockafellar, 1997]{rockafellar97convex}
Rockafellar, R.~T. (1997).
\newblock {\em Convex Analysis}.
\newblock Princeton University Press.

\bibitem[Samworth and Johnson, 2004]{samworth2004convergence}
Samworth, R. and Johnson, O. (2004).
\newblock Convergence of the empirical process in {M}allows distance, with an
  application to bootstrap performance.
\newblock {\em arXiv preprint math/0406603}.

\bibitem[Samworth and Shah, 2025]{samworth24modern}
Samworth, R.~J. and Shah, R.~D. (2025+).
\newblock {\em Modern Statistical Methods and Theory}.
\newblock Cambridge University Press.

\bibitem[Schick, 1986]{schick1986asymptotically}
Schick, A. (1986).
\newblock On asymptotically efficient estimation in semiparametric models.
\newblock {\em The Annals of Statistics}, 14(3):1139--1151.

\bibitem[Serrin and Varberg, 1969]{serrin1969general}
Serrin, J. and Varberg, D.~E. (1969).
\newblock A general chain rule for derivatives and the change of variables
  formula for the {L}ebesgue integral.
\newblock {\em The American Mathematical Monthly}, 76(5):514--520.

\bibitem[Silverman, 1982]{silverman82estimation}
Silverman, B.~W. (1982).
\newblock On the estimation of a probability density function by the maximum
  penalized likelihood method.
\newblock {\em The Annals of Statistics}, 10(3):795--810.

\bibitem[Silverman, 1986]{silverman1986density}
Silverman, B.~W. (1986).
\newblock {\em Density Estimation}.
\newblock Chapman \& Hall.

\bibitem[Song and Ermon, 2019]{song19generative}
Song, Y. and Ermon, S. (2019).
\newblock Generative modeling by estimating gradients of the data distribution.
\newblock {\em Advances in Neural Information Processing Systems},
  32:11895--11907.

\bibitem[Song et~al., 2020]{song2020sliced}
Song, Y., Garg, S., Shi, J., and Ermon, S. (2020).
\newblock Sliced score matching: A scalable approach to density and score
  estimation.
\newblock In {\em Uncertainty in Artificial Intelligence}, pages 574--584.

\bibitem[Song and Kingma, 2021]{song2021train}
Song, Y. and Kingma, D.~P. (2021).
\newblock How to train your energy-based models.
\newblock {\em arXiv preprint arXiv:2101.03288}.

\bibitem[Song et~al., 2021]{song21score}
Song, Y., Sohl-Dickstein, J., Kingma, D.~P., Kumar, A., Ermon, S., and Poole,
  B. (2021).
\newblock Score-based generative modeling through stochastic differential
  equations.
\newblock In {\em The Ninth International Conference on Learning
  Representations}.

\bibitem[Sriperumbudur et~al., 2017]{sriperumbudur2017density}
Sriperumbudur, B., Fukumizu, K., Gretton, A., Hyv\"arinen, A., and Kumar, R.
  (2017).
\newblock Density estimation in infinite dimensional exponential families.
\newblock {\em Journal of Machine Learning Research}, 18:1--59.

\bibitem[Stein, 1956a]{stein1956efficient}
Stein, C. (1956a).
\newblock Efficient nonparametric testing and estimation.
\newblock In {\em Proceedings of the Third Berkeley Symposium on Mathematical
  Statistics and Probability}, volume~1, pages 187--195.

\bibitem[Stein, 1956b]{stein1956inadmissibility}
Stein, C. (1956b).
\newblock Inadmissibility of the usual estimator for the mean of a multivariate
  normal distribution.
\newblock In {\em Proceedings of the Third Berkeley Symposium on Mathematical
  Statistics and Probability, Volume 1: Contributions to the Theory of
  Statistics}, volume~3, pages 197--207. University of California Press.

\bibitem[Stone, 1975]{stone1975adaptive}
Stone, C.~J. (1975).
\newblock Adaptive maximum likelihood estimators of a location parameter.
\newblock {\em The Annals of Statistics}, 3(2):267--284.

\bibitem[van~der Vaart, 1998]{vdV1998asymptotic}
van~der Vaart, A.~W. (1998).
\newblock {\em Asymptotic {S}tatistics}.
\newblock Cambridge University Press.

\bibitem[van~der Vaart and Wellner, 2021]{vdV2021stein}
van~der Vaart, A.~W. and Wellner, J.~A. (2021).
\newblock Stein 1956: {E}fficient nonparametric testing and estimation.
\newblock {\em The Annals of Statistics}, 49(4):1836--1849.

\bibitem[van Eeden, 1970]{van1970efficiency}
van Eeden, C. (1970).
\newblock Efficiency-robust estimation of location.
\newblock {\em The Annals of Mathematical Statistics}, 41(1):172--181.

\bibitem[Vincent, 2011]{vincent2011connection}
Vincent, P. (2011).
\newblock A connection between score matching and denoising autoencoders.
\newblock {\em Neural Computation}, 23(7):1661--1674.

\bibitem[Wainwright, 2019]{wainwright2019high}
Wainwright, M.~J. (2019).
\newblock {\em High-Dimensional Statistics: A Non-Asymptotic Viewpoint},
  volume~48.
\newblock Cambridge University Press.

\bibitem[Yang and Wang, 2024]{yang2024multiple}
Yang, X. and Wang, T. (2024).
\newblock Multiple-output composite quantile regression through an optimal
  transport lens.
\newblock In {\em The Thirty Seventh Annual Conference on Learning Theory},
  volume 247, pages 5076--5122. PMLR.

\bibitem[Yang et~al., 2019]{yang2019variational}
Yang, Y., Martin, R., and Bondell, H. (2019).
\newblock Variational approximations using {F}isher divergence.
\newblock {\em arXiv preprint arXiv:1905.05284}.

\bibitem[Yohai and Maronna, 1979]{yohai1979asymptotic}
Yohai, V.~J. and Maronna, R.~A. (1979).
\newblock Asymptotic behavior of {$M$}-estimators for the linear model.
\newblock {\em The Annals of Statistics}, 7:258--268.

\bibitem[Young and Shah, 2024]{young2023sandwich}
Young, E.~H. and Shah, R.~D. (2024).
\newblock Sandwich boosting for accurate estimation in partially linear models
  for grouped data.
\newblock {\em Journal of the Royal Statistical Society, Series B},
  86:1286--1311.

\bibitem[Yu et~al., 2020]{yu2020simultaneous}
Yu, M., Gupta, V., and Kolar, M. (2020).
\newblock Simultaneous inference for pairwise graphical models with generalized
  score matching.
\newblock {\em Journal of Machine Learning Research}, 21(91):1--51.

\bibitem[Yu et~al., 2022]{yu2022generalized}
Yu, S., Drton, M., and Shojaie, A. (2022).
\newblock Generalized score matching for general domains.
\newblock {\em Information and Inference: A Journal of the IMA},
  11(2):739--780.

\bibitem[Zou and Yuan, 2008]{zou2008composite}
Zou, H. and Yuan, M. (2008).
\newblock {Composite quantile regression and the oracle model selection
  theory}.
\newblock {\em The Annals of Statistics}, 36(3):1108--1126.

\end{thebibliography}

\section{Appendix}
\label{sec:appendix}

\subsection{Proofs for Section~\ref{sec:antitonic-proj}}
\label{sec:antitonic-proj-proofs}

Throughout this subsection, we work in the setting of Lemmas~\ref{lem:psi0-star} and~\ref{lem:p0-star}. Moreover, define $Q_0(u) := \sup\{z \in [-\infty,\infty] : F_0(z) \leq u\}$ for $u \in [0,1]$. 
% where $\sup\emptyset = -\infty$ by convention, and also set $F_0^{-1}(0) := -\infty$ for convenience. 
Then $z_{\min} := \inf(\supp p_0) = Q_0(0)$ and $z_{\max} := \sup(\supp p_0) = F_0^{-1}(1)$, and $\mathcal{S}_0 = \{z \in \R : F_0(z) \in (0,1)\}$.

\begin{proof}[Proof of Lemma~\ref{lem:psi0-star}]
If $F_0^{-1}(v) < Q_0(v)$ for some $v \in [0,1]$, then because $p_0(\pm\infty) = 0$ and $p_0,F_0$ are both continuous, we have $p_0(z) = 0$ for all $z \in [F_0^{-1}(v),Q_0(v)] = \{z \in [-\infty,\infty] : F_0(z) = v\}$. Therefore, $p_0(z) = J_0\bigl(F_0(z)\bigr)$ for all $z \in \R$ and $J_0(v) = p_0\bigl(F_0^{-1}(v)\bigr) = p_0\bigl(Q_0(v)\bigr)$ for all $v \in [0,1]$, with $J_0(0) = J_0(1) = 0$. Since $\lim_{u \nearrow v} F_0^{-1}(u) = F_0^{-1}(v)$ and $\lim_{u \searrow v} F_0^{-1}(u) = Q_0(v)$ for all $v \in [0,1]$, it follows from the continuity of $p_0$ on $[-\infty,\infty]$ that $J_0$ is continuous on $[0,1]$. Moreover, the least concave majorant $\hat{J}_0$ satisfies $\sup_{u \in [0,1]}\hat{J}_0(u) = \sup_{u \in [0,1]}J_0(u) < \infty$ and
\begin{equation}
\label{eq:J01}
\hat{J}_0(0) = J_0(0) = 0 = J_0(1) = \hat{J}_0(1).
\end{equation}
In particular, $\hat{J}_0$ is concave and bounded, so it is also continuous on $[0,1]$. By~\citet[Theorem~24.1]{rockafellar97convex}, the right derivative $\hat{J}_0^{(\mathrm{R})}$ is decreasing and right-continuous on $[0,1)$, with $\hat{J}_0^{(\mathrm{R})}(u) \in \R$ for all $u \in (0,1)$. Thus, $\psi_0^*$ is finite-valued on $\mathcal{S}_0$. If $z_n \searrow z$ for some $z \in \R$, then either $u_n := F_0(z_n) \searrow F_0(z) =:u < 1$ or $u_n = u = 1$ for all $n$. In both cases, $\psi_0^*(z_n) = \hat{J}_0^{(\mathrm{R})}(u_n) \nearrow \hat{J}_0^{(\mathrm{R})}(u) = \psi_0^*(z)$, so $\psi_0^*$ is right-continuous as a function from $\R$ to $[-\infty,\infty]$. Since $F_0$ is increasing, $\psi_0^*$ is decreasing.

We have $z_{\min} = \inf\{z \in \R : F_0(z) > 0\}$ and $\log F_0(z) \searrow -\infty$ as $z \searrow z_{\min}$. Thus, if $z_{\min} > -\infty$, then
\begin{equation}
\label{eq:left-hazard}
\limsup_{u \searrow 0}\frac{J_0(u)}{u} = \limsup_{z \searrow z_{\min}}\frac{J_0\bigl(F_0(z)\bigr)}{F_0(z)} = \limsup_{z \searrow z_{\min}}\frac{p_0(z)}{F_0(z)} = \limsup_{z \searrow z_{\min}}\,(\log F_0)'(z) = \infty,
\end{equation}
where the final equality follows from the mean value theorem. Together with Lemma~\ref{lem:lcm-deriv-0}, this implies that $\psi_0^*(z) = \hat{J}_0^{(\mathrm{R})}(0) = \infty$ for all $z \leq z_{\min}$. Similarly, if $z_{\max} < \infty$, then $\psi_0^*(z) = \hat{J}_0^{(\mathrm{R})}(1) = -\infty$ for all $z \geq z_{\max}$. Therefore, in all cases, $\psi_0^*(z) \in \R$ if and only if $z \in \mathcal{S}_0$.
\end{proof}

The next lemma characterises precisely the class of density quantile functions of uniformly continuous densities on $\R$, and shows explicitly how to recover a density from a continuous density quantile function that is strictly positive on $(0,1)$.
\begin{lemma}
\label{lem:density-quantile-reverse}
A function $J \colon [0,1] \to [0,\infty)$ is the density quantile function of a uniformly continuous density on $\R$ if and only if
\begin{equation}
\label{eq:quantile-J}
\text{ $J$ is continuous on $[0,1]$ with $J(0) = J(1) = 0$ and }Q_J(u) := \int_{1/2}^u \frac{1}{J} < \infty
\end{equation}
for all $u \in (0,1)$. In this case, $Q_J$ is a strictly increasing function from $(0,1)$ to $\bigl(Q_J(0),Q_J(1)\bigr)$, whose inverse is continuously differentiable. Moreover, the function $p_J \colon \R \to \R$ given by
\[
p_J(z) := 
\begin{cases}
(Q_J^{-1})'(z) \;&\text{for }z \in \bigl(Q_J(0),Q_J(1)\bigr) \\
0 \;&\text{otherwise}
\end{cases}
\]
is a uniformly continuous density with corresponding density quantile function $J$. If in addition $J > 0$ on $(0,1)$, then a density $p_0$ has density quantile function $J$ if and only if $p_0(\cdot) = p_J(\cdot - \mu)$ for some $\mu \in \R$.
\end{lemma}

\begin{proof}
If $p_0$ is a uniformly continuous density with density quantile function $J_0 = p_0 \circ F_0^{-1}$, then Lemma~\ref{lem:psi0-star} shows that $J_0$ is continuous with $J_0(0) = J_0(1) = 0$. Introducing $U \sim U(0,1)$ and $Z := F_0^{-1}(U) \sim P_0$, we have $p_0(Z) = J_0(U)$ and $\Pr\bigl(J_0(U) = 0\bigr) = \Pr\bigl(p_0(Z) = 0\bigr) = 0$. Moreover, since $F_0$ is continuous, $(F_0 \circ F_0^{-1})(u) = u$ for all $u \in (0,1)$, so $U = F_0(Z)$. Therefore,
\begin{equation}
\label{eq:density-quantile-reverse}
\int_{1/2}^u \frac{1}{J_0} = \E\biggl(\frac{1}{J_0(U)} \Ind_{\{U \in [1/2,u)\}}\biggr) = \E\biggl(\frac{1}{p_0(Z)} \Ind_{\{Z \in [F_0^{-1}(1/2),F_0^{-1}(u))\}}\biggr) \leq F_0^{-1}(u) - F_0^{-1}\Bigl(\frac{1}{2}\Bigr) < \infty
\end{equation}
for $u \in [1/2,1)$, with equality if and only if $p_0(z) > 0$ Lebesgue almost everywhere on $\bigl[F_0^{-1}(1/2),F_0^{-1}(u)\bigr)$. Similarly, for $u \in (0,1/2]$,~\eqref{eq:density-quantile-reverse} remains true, with an analogous equality condition.

Conversely, if $J \colon [0,1] \to \R$ satisfies~\eqref{eq:quantile-J}, then $Q_J$ is strictly increasing bijection from $(0,1)$ to $\bigl(Q_J(0),Q_J(1)\bigr)$ with $Q_J'(u) = 1/J(u)$ for all $u \in (0,1)$. Thus, by the inverse function theorem, $F_J := Q_J^{-1} \colon \bigl(Q_J(0),Q_J(1)\bigr) \to (0,1)$ is a well-defined, continuously differentiable bijection with derivative $p_J := F_J' = 1/(Q_J' \circ F_J) = J \circ F_J$ satisfying $J = p_J \circ F_J^{-1}$ on $(0,1)$. We have $\lim_{z \searrow Q_J(0)}F_J(z) = 0$ and $\lim_{z \nearrow Q_J(1)}F_J(z) = 1$, while $\lim_{z \searrow Q_J(0)}p_J(z) = \lim_{z \nearrow Q_J(1)}p_J(z) = 0$ since $J$ is continuous at $0$ and $1$. Therefore, setting $p_J(z) = 0$ for $z \in \R \setminus \bigl(Q_J(0),Q_J(1)\bigr)$, we conclude that $p_J$ is a uniformly continuous density on $\R$ with corresponding distribution function $F_J$, quantile function $Q_J$ and density quantile function $J$. Moreover, for every $\mu \in \R$, the density $p_J(\cdot - \mu)$ also has density quantile function $J$.

Finally, if $p_0$ is a uniformly continuous density with density quantile function $J = p_0 \circ F_0^{-1} > 0$ on $(0,1)$, then for every $u \in (0,1)$, equality holds in~\eqref{eq:quantile-J} and hence
\[
Q_J(u) = \int_{1/2}^u \frac{1}{J} = F_0^{-1}(u) - F_0^{-1}\Bigl(\frac{1}{2}\Bigr).
\]
Therefore, letting $\mu := F_0^{-1}(1/2)$, we deduce that $p_0(\cdot) = p_J(\cdot - \mu)$, which completes the proof.
\end{proof}

\begin{lemma}
\label{lem:J0-J0hat}
Let $\mathcal{T}$ be as in Proposition~\ref{prop:p0-star-fisher}\textit{(a)}. If $t \in \mathcal{T}$, then $p_0(t) = \hat{J}_0\bigl(F_0(t)\bigr) > 0$. 
\end{lemma}

\begin{proof}
Since $\hat{J}_0$ is a non-negative concave function, we must have $\hat{J}_0 > 0$ on $(0,1)$, as otherwise $J_0 = \hat{J}_0 = 0$ and hence $p_0 = J_0 \circ F_0 = 0$ on $\R$ by Lemma~\ref{lem:psi0-star}, which is a contradiction. Moreover, $\psi_0^*(t) \in \R$, so again by Lemma~\ref{lem:psi0-star}, $t \in \mathcal{S}_0$. Thus, $v := F_0(t) \in (0,1)$ and $p_0(t) = (J_0 \circ F_0)(t) = J_0(v)$. Suppose for a contradiction that $J_0(v) < \hat{J}_0(v)$. Then by Lemma~\ref{lem:lcm-affine} and the fact that $J_0$ and $\hat{J}_0$ are continuous at~$v$, there exists $\delta > 0$ such that $\hat{J}_0$ is affine on $[v - \delta, v + \delta]$. Thus, $\hat{J}_0^{(\mathrm{R})}$ is constant on $[v - \delta, v + \delta)$, so $\psi_0^* = \hat{J}_0^{(\mathrm{R})} \circ F_0$ is constant on the open interval $\bigl(Q_0(v - \delta),F_0^{-1}(v + \delta)\bigr)$, which contains~$t$. This contradicts the definition of $\mathcal{T}$, so $p_0(t) = J_0(v) = \hat{J}_0(v) > 0$, as required.
\end{proof}

\begin{lemma}
\label{lem:psi-unique}
If $\psi \in \Psi_\downarrow(p_0)$ satisfies $\psi = \psi_0^*$ $P_0$-almost everywhere, then in fact $\psi = \psi_0^*$ on $\R$.
\end{lemma}

\begin{proof}
First consider $z \in \mathcal{S}_0$, let $u := F_0(z) \in (0,1)$ and let $u_1 := \inf\bigl\{v \in [0,1] : \hat{J}_0^{(\mathrm{R})}(v) = \hat{J}_0^{(\mathrm{R})}(u)\bigr\}$. Then $\hat{J}_0^{(\mathrm{R})}(u_1) = \hat{J}_0^{(\mathrm{R})}(u)$ by the right-continuity of $\hat{J}_0^{(\mathrm{R})}$. Moreover, define $z_1 := F_0^{-1}(u_1)$ and $z_2 := Q_0(u)$, so that $z_1 \leq z \leq z_2$ and
\[
F_0(z_1 - \delta) < u_1 = F_0(z_1) \leq F_0(z_2) = u < F_0(z_2 + \delta)
\]
for all $\delta > 0$. Then $\psi_0^*(z_1 - \delta) = \hat{J}_0^{(\mathrm{R})}\bigl(F_0(z_1 - \delta)\bigr) > \hat{J}_0^{(\mathrm{R})}(u_1) = \psi_0^*(z_1)$ for all such $\delta$, so $z_1 \in \mathcal{T}$ and hence $p_0(z_1) > 0$ by Lemma~\ref{lem:J0-J0hat}. It follows that $P_0\bigl([z_j, z_j + \delta)\bigr) > 0$ for $j \in \{1,2\}$, so because $\psi$ and $\psi_0^*$ are decreasing, right-continuous functions that agree $P_0$-almost everywhere,
\[
\psi(z) \leq \psi(z_1) = \psi_0^*(z_1) = \hat{J}_0^{(\mathrm{R})}(u_1) = \hat{J}_0^{(\mathrm{R})}(u) = \psi_0^*(z_2) = \psi(z_2) \leq \psi(z).
\]
Thus, $\psi(z) = \psi_0^*(z)$ for all $z \in \mathcal{S}_0$. Together with Lemma~\ref{lem:psi0-star}, this implies that $\lim_{z \searrow z_{\min}} \psi(z) = \lim_{z \searrow z_{\min}}  \psi_0^*(z) = \infty$ and $\lim_{z \nearrow z_{\max}} \psi(z) = \lim_{z \nearrow z_{\max}}  \psi_0^*(z) = -\infty$, so $\psi = \psi_0^* = \infty$ on $(-\infty,z_{\min}]$ and $\psi = \psi_0^* = -\infty$ on $[z_{\max},\infty)$. Thus, $\psi = \psi_0^*$ on $\R$, as required.
\end{proof}

\begin{proof}[Proof of Theorem~\ref{thm:antitonic-score-proj}]
Let $\chi_t := \Ind_{(-\infty,t)} \in \Psi_\downarrow(p_0)$ for $t \in (\infty,\infty]$. Then, letting $U \sim U(0,1)$, we have $(F_0 \circ F_0^{-1})(U) = U$, so
\begin{align*}
\int_\R\psi_0^*\chi_t\,dP_0 = \int_{(-\infty,t)}(\hat{J}_0^{(\mathrm{R})} \circ F_0)\,dP_0 &= \E\bigl\{\bigl(\hat{J}_0^{(\mathrm{R})} \circ F_0 \circ F_0^{-1}\bigr)(U)\,\Ind_{\{F_0^{-1}(U) \leq t\}}\bigr\} \\
&= \E\bigl\{\hat{J}_0^{(\mathrm{R})}(U)\,\Ind_{\{U \leq F_0(t)\}}\bigr\} = \int_0^{F_0(t)}\hat{J}_0^{(\mathrm{R})} \\
&= \hat{J}_0\bigl(F_0(t)\bigr) - \hat{J}_0(0) \geq J_0\bigl(F_0(t)\bigr) = p_0(t),
\end{align*}
where the last two equalities are due to~\citet[Corollary~24.2.1]{rockafellar97convex} and Lemma~\ref{lem:psi0-star} respectively. By~\eqref{eq:J01} and Lemma~\ref{lem:J0-J0hat}, equality holds throughout in the final line if $F_0(t) \in \{0,1\}$ or $t \in \mathcal{T}$, so
\begin{equation}
\label{eq:KKT-ind}
\int_\R\psi_0^*\chi_t\,dP_0\,
\begin{cases}
\geq p_0(t) &\;\text{for all }t \in (\infty,\infty] \\
= p_0(t) &\;\text{for all }t \in (-\infty,z_{\min}] \cup [z_{\max},\infty] \cup \mathcal{T}.
\end{cases}
\end{equation}
\textit{(a)} In particular, taking $t = \infty$, we have $\int_\R \psi_0^*\,dP_0 = 0$.

\medskip
\noindent
\textit{(b,\,c)} \textbf{Case 1}: Suppose that $i^*(p_0) = \int_\R(\psi_0^*)^2\,dP_0 < \infty$. Then $\psi_0^* \in \Psi_\downarrow(p_0)$ since it is decreasing and right-continuous. For measurable functions $\psi$ on $\R$, we write $\norm{\psi} \equiv \norm{\psi}_{L^2(P_0)} := (\int_\R\psi^2\,dP_0)^{1/2}$, and for $\psi_1,\psi_2 \in L^2(P_0)$, denote by $\ipr{\psi_1}{\psi_2} \equiv \ipr{\psi_1}{\psi_2}_{L^2(P_0)} := \int_\R \psi_1\psi_2\,dP_0$ their $L^2(P_0)$ inner product. Recalling that $D_{p_0}(\cdot)$ is a convex (quadratic) function on the convex cone $\Psi_\downarrow(p_0)$, we will deduce from~\eqref{eq:KKT-ind} that $\psi_0^*$ satisfies the first-order stationarity condition
\begin{align}
\label{eq:KKT}
-\int_{\mathcal{S}_0} p_0\,d\psi \leq \int_\R\psi_0^*\psi\,dP_0 = \ipr{\psi_0^*}{\psi} \in \R
\end{align}
for all $\psi \in \Psi_\downarrow(p_0)$, and that equality holds when $\psi = \psi_0^*$. To see this, fix $t_0 \in \mathcal{S}_0$ and define $g(z,t) := \Ind_{\{z < t \leq t_0\}} - \Ind_{\{t_0 < t \leq z\}}$ for $z,t \in \R$. Then $g(\cdot,t) = \chi_t(\cdot) - \Ind_{\{t > t_0\}}$ for all $t \in \R$ and $\psi(t_0) - \psi(z) = \int_{\mathcal{S}_0} g(z,t)\,d\psi(t)$ for all $z \in \mathcal{S}_0$. Since $\int_\R|\psi_0^*||\psi(t_0) - \psi|\,dP_0 \leq \norm{\psi_0^*}\bigl(|\psi(t_0)| + \norm{\psi}\bigr) < \infty$ and $\int_\R\psi_0^*\,dP_0 = 0$ by~\eqref{eq:KKT-ind}, we can apply Fubini's theorem to obtain
\begin{align}
\label{eq:KKT-Fubini}
\int_\R\psi_0^*\psi\,dP_0 &= \psi(t_0)\int_\R\psi_0^*\,dP_0-\int_\R\psi_0^*(z)\int_{\mathcal{S}_0} g(z,t)\,d\psi(t)\,dP_0(z) = -\int_{\mathcal{S}_0}\int_\R \psi_0^*(z)\,g(z,t)\,dP_0(z)\,d\psi(t).
\end{align}
On the right-hand side of~\eqref{eq:KKT-Fubini}, the outer Lebesgue--Stieltjes integral is with respect to a negative measure, so it follows from~\eqref{eq:KKT-ind} that
\begin{align}
\int_{\mathcal{S}_0}\int_\R \psi_0^*(z)\,g(z,t)\,dP_0(z)\,d\psi(t) &= \int_{(z_{\min},t_0]}\biggl(\int_\R\psi_0^*\chi_t\,dP_0\biggr)\,d\psi(t) + \int_{(t_0,z_{\max})}\biggl(\int_\R\psi_0^*(\chi_t - 1)\,dP_0\biggr)\,d\psi(t)\notag \\
\label{eq:KKT-integrals}
&\leq \int_{(z_{\min},t_0]}p_0(t)\,d\psi(t) + \int_{(t_0,z_{\max})}p_0(t)\,d\psi(t) = \int_{\mathcal{S}_0}p_0\,d\psi.
\end{align}
% where the final equality holds since $p_0 = 0$ on $[-\infty,\infty] \setminus (z_{\min},z_{\max})$ and $\psi$ is right-continuous. 
Together with~\eqref{eq:KKT-Fubini}, this proves the inequality~\eqref{eq:KKT} for all $\psi \in \Psi_\downarrow(p_0)$. It remains to show that~\eqref{eq:KKT} holds with equality when $\psi = \psi_0^*$. Let $\nu^*$ for the Lebesgue--Stieltjes measure induced by $\psi_0^*$, so that the set $\mathcal{T}$ in Lemma~\ref{lem:J0-J0hat} is the support of $\nu^*$. Then $\mathcal{T}^c$ is an open set with $\nu^*(\mathcal{T}^c) = 0$, so $\psi_0^*(t_0) - \psi_0^*(z) = \int_\mathcal{T} g(z,t)\,d\psi_0^*(t)$ for all $z \in \R$. Thus,~\eqref{eq:KKT-Fubini} holds for $\psi = \psi_0^*$ when the Lebesgue--Stieltjes integrals (with respect to $\nu^*$) are restricted to $\mathcal{S}_0\cap\mathcal{T}$. Moreover,~\eqref{eq:KKT-integrals} is still valid if we intersect the domains of all outer integrals with $\mathcal{T}$, again because $\nu^*(\mathcal{T}^c) = 0$. We established that~\eqref{eq:KKT-ind} is an equality for all $t \in \mathcal{T}$, so~\eqref{eq:KKT-integrals} and hence~\eqref{eq:KKT} hold with equality when $\psi = \psi_0^*$.

\medskip
\noindent
For $\psi \in \Psi_\downarrow(p_0)$, it follows from~\eqref{eq:KKT} that $\int_{\mathcal{S}_0} p_0\,d\psi > -\infty$ and hence $D_{p_0}(\psi) \in \R$, so
\begin{align}
D_{p_0}(\psi) &= \int_\R(\psi - \psi_0^*)^2\,dP_0 + \int_\R(\psi_0^*)^2\,dP_0 + 2\,\biggl\{\int_\R\psi_0^*(\psi - \psi_0^*)\,dP_0 + \int_{\mathcal{S}_0} p_0\,d\psi\biggr\} \notag \\
\label{eq:inf-Dp0}
&\geq \int_\R(\psi - \psi_0^*)^2\,dP_0 -\int_\R (\psi_0^*)^2\,dP_0 \\
&= \int_\R(\psi - \psi_0^*)^2\,dP_0 + \int_\R (\psi_0^*)^2\,dP_0 + 2\int_{\mathcal{S}_0} p_0\,d\psi_0^* \notag \\
&= \int_\R(\psi - \psi_0^*)^2\,dP_0 + D_{p_0}(\psi_0^*) \geq D_{p_0}(\psi_0^*) \notag.
\end{align}
Thus, $\psi \in \argmin_{\psi \in \Psi_\downarrow(p_0)}D_{p_0}(\psi)$ if and only if $\psi = \psi_0^*$ $P_0$-almost everywhere. By Lemma~\ref{lem:psi-unique}, this holds if and only if $\psi = \psi_0^*$ on $\R$, which proves the uniqueness assertion in \textit{(b)}.

Moreover, since $\hat{J}_0^{(\mathrm{R})}(0) = \lim_{u \searrow 0}\hat{J}_0^{(\mathrm{R})}(u) > 0$ and $F_0(z) > 0$ for all $z > Q_0(0) = z_{\min}$, we have $\int_{(z_{\min},z)}(\psi_0^*)^2\,dP_0 = \int_{(z_{\min},z)}(\hat{J}_0^{(\mathrm{R})} \circ F_0)^2\,dP_0 > 0$ for all $z > z_{\min}$, so $i^*(p_0) = \norm{\psi_0^*}^2 > 0$. Now for $\psi \in \Psi_\downarrow(p_0)$ such that $\int_\R\psi^2\,dP_0 > 0$, we have $-\int_{\mathcal{S}_0} p_0\,d\psi \geq 0$ since $\psi$ is decreasing. It follows from~\eqref{eq:KKT} and the Cauchy--Schwarz inequality that
\begin{align*}
V_{p_0}(\psi) = \frac{\int_\R\psi^2\,dP_0}{\bigl(\int_{\mathcal{S}_0} p_0\,d\psi\bigr)^2} \geq \frac{\norm{\psi}^2}{\ipr{\psi_0^*}{\psi}^2} \geq \frac{1}{\norm{\psi_0^*}^2} = \frac{1}{i^*(p_0)} \in (0,\infty),
\end{align*}
with equality if and only if $\psi = \lambda\psi_0^*$ $P_0$-almost everywhere for some $\lambda > 0$. By Lemma~\ref{lem:psi-unique}, this holds if and only if $\psi = \lambda\psi_0^*$ on $\R$, so the proof of \emph{(c)} is complete.

\medskip
\noindent
\textbf{Case 2}: Suppose instead that $i^*(p_0) = \int_\R (\psi_0^*)^2\,dP_0 = \infty$. For each $n \in \N$, define $\psi_n^* := (\psi_0^* \wedge n) \vee (-n)$. This is bounded, decreasing and right-continuous, so $\psi_n^* \in \Psi_\downarrow(p_0)$, and we claim that
\begin{align}
\label{eq:psi_n}
-\int_{\mathcal{S}_0} p_0\,d\psi_n^* = \int_\R\psi_0^*\psi_n^*\,dP_0  \in \R.
\end{align}
Indeed, by the quantile transform used to establish~\eqref{eq:KKT-ind}, $|\int_\R\psi_0^*\psi_n^*\,dP_0| \leq n\int_\R|\psi_0^*|\,dP_0 = n\int_0^1|\hat{J}_0^{(\mathrm{R})}| \leq 2n\sup_{u \in [0,1]}\hat{J}_0(u) < \infty$ since $\hat{J}_0$ is concave and bounded. Therefore, we may apply Fubini's theorem as above to see that~\eqref{eq:KKT-Fubini} above holds when $\psi = \psi_n^*$. Furthermore, writing $\nu_n^*$ for the Lebesgue--Stieltjes measure associated with $\psi_n^*$, we have $\nu_n^*(\mathcal{T}^c) = \nu^*(\mathcal{T}^c\cap I_n) = 0$, where $I_n := \{z \in \R : \psi_0^*(z) \in (-n,n]\}$.
% $I_n$ is a closed interval because $\psi_0^*$ is right-continuous
Consequently, by intersecting the domains of all outer integrals in~\eqref{eq:KKT-integrals} with $\mathcal{T}$ and arguing as above, we deduce that~\eqref{eq:KKT-integrals} holds with equality for $\psi = \psi_n^*$. This yields~\eqref{eq:psi_n} for every $n \in \N$. It then follows by the monotone convergence theorem that
\begin{align*}
D_{p_0}(\psi_n^*) = \int_\R(\psi_n^*)^2\,dP_0 - 2\int_\R\psi_0^*\psi_n^*\,dP_0 \leq -\int_\R(\psi_n^*)^2\,dP_0 &= -\int_\R\bigl((\psi_0^*)^2\wedge n^2\bigr)\,dP_0 \\
&\searrow -\int_\R(\psi_0^*)^2\,dP_0 = -\infty
\end{align*}
as $n \to \infty$. Thus, $\inf_{\psi \in \Psi_\downarrow(p_0)}D_{p_0}(\psi) = -\infty$. 

Together with~\eqref{eq:inf-Dp0}, this shows that $\inf_{\psi \in \Psi_\downarrow(p_0)}D_{p_0}(\psi) = -i^*(p_0)$ in both cases, which completes the proof of \emph{(b)}.

\medskip
\noindent
\textit{(d)} Suppose that $p_0$ is absolutely continuous on $\R$ with score function $\psi_0 = p_0'/p_0$. Recall that
$(F_0 \circ F_0^{-1})(u) = u$ for all $u \in (0,1)$, so that letting $U \sim U(0,1)$ and $Z = F_0^{-1}(U) \sim P_0$, we have
\begin{align}
\label{eq:psi0-J0}
\int_0^u \psi_0 \circ F_0^{-1} &= \E\bigl((\psi_0 \circ F_0^{-1})(U) \Ind_{\{U \leq (F_0 \circ F_0^{-1})(u)\}}\bigr) \\
&= \E\bigl(\psi_0(Z) \Ind_{\{Z \leq F_0^{-1}(u)\}}\bigr) = \int_{(-\infty,F_0^{-1}(u)]} \psi_0\,dP_0 = \int_{-\infty}^{F_0^{-1}(u)} p_0' = (p_0 \circ F_0^{-1})(u) = J_0(u) \notag
\end{align}
for all $u \in (0,1)$. This yields the representation~\eqref{eq:psi0-star} of $\psi_0^*$.
% and also shows that $J_0 = p_0 \circ F_0^{-1}$ is absolutely continuous on $[0,1]$ with derivative $\psi_0 \circ F_0^{-1}$
% NB: The concave function $\hat{J}_0$ is always absolutely continuous on $[0,1]$ with derivative $\hat{J}_0^{(\mathrm{R})} = \psi_0^* \circ F_0^{-1}$

We assume henceforth that $i(p_0) = \int_\R\psi_0^2\,dP_0 = \norm{\psi_0}^2 < \infty$, since otherwise the inequality $i^*(p_0) \leq i(p_0)$ holds trivially. Then $\int_\R|\psi\psi_0|\,dP_0 = \ipr{|\psi|}{|\psi_0|} \leq \norm{\psi}\,\norm{\psi_0} < \infty$ for all $\psi \in \Psi_\downarrow(p_0)$. In particular, $\int_\R|p_0'| = \int_\R|\psi_0|\,dP_0 < \infty$, so by the dominated convergence theorem, 
\[
\int_\R\psi_0\,dP_0 = \int_\R p_0' = \lim_{t \to \infty}\int_{-t}^t p_0' = \lim_{t \to \infty}\bigl(p_0(t) - p_0(-t)\bigr) = p_0(\infty) - p_0(-\infty) = 0.
\]
Recalling that $g(z,t) = \Ind_{\{z < t \leq t_0\}} - \Ind_{\{t_0 < t \leq z\}}$ for $z,t \in \R$ and $\psi(t_0) - \psi(z) = \int_{\mathcal{S}_0} g(z,t)\,d\psi(t)$ for all $z \in \mathcal{S}_0$, we similarly have $\int_\R p_0'(z)g(z,t)\,dz = -\bigl(p_0(\infty) - p_0(t)\bigr) = p_0(t)$ for all $t > t_0$ and $\int_\R p_0'(z)g(z,t)\,dz = p_0(t) - p_0(-\infty) = p_0(t)$ for all $t \leq t_0$. For $\psi \in \Psi_\downarrow(p_0)$, it then follows from Fubini's theorem that
\begin{align*}
\int_\R\psi\psi_0\,dP_0 = \psi(t_0)\int_\R\psi_0\,dP_0 - \int_\R p_0'(z)\int_{\mathcal{S}_0}g(z,t)\,d\psi(t)\,dz = -\int_{\mathcal{S}_0} p_0(t)\,d\psi(t)
\end{align*}
and hence that
\begin{align}
\label{eq:Dp0-psi0}
D_{p_0}(\psi) = \int_\R\psi^2\,dP_0 + 2\int_{\mathcal{S}_0} p_0\,d\psi = \norm{\psi}^2 - 2\ipr{\psi}{\psi_0} &= \norm{\psi - \psi_0}^2 - \norm{\psi_0}^2 \\
&\geq -\norm{\psi_0}^2 = -i(p_0) > -\infty. \notag
\end{align}
Thus, by \emph{(b)}, $-i^*(p_0) = \inf_{\psi \in \Psi_\downarrow(p_0)}D_{p_0}(\psi) \geq -i(p_0) > -\infty$, so $i^*(p_0) \leq i(p_0) < \infty$. Equality holds if and only if $\inf_{\psi \in \Psi_\downarrow(p_0)}\norm{\psi - \psi_0} = 0$, i.e.\ $\psi_0 \in \Psi_\downarrow(p_0)$, which is equivalent to $p_0$ being log-concave.
% Alternatively, by the previous display and the equality case of~\eqref{eq:KKT}, $\ipr{\psi_0^*}{\psi_0 - \psi_0^*} = 0$, so $i(p_0) = \norm{\psi_0}^2 = \norm{\psi_0^*}^2 + \norm{\psi_0 - \psi_0^*}^2 \geq \norm{\psi_0^*}^2 = i^*(p_0)$.
\end{proof}

\begin{proof}[Proof of Corollary~\ref{cor:istar-cvx}]
Fix $t \in (0,1)$. If $p_0,p_1$ are uniformly continuous densities on $\R$, then the same is true of $p_t := (1 - t)p_0 + tp_1$. For any $\psi \in \Psi_\downarrow(p_t)$, we have $\max\bigl\{(1 - t)\int_\R \psi^2 p_0,\,t\int_\R \psi^2 p_1\bigr\} \leq \int_\R \psi^2 p_t < \infty$, so $\psi \in \Psi_\downarrow(p_0) \cap \Psi_\downarrow(p_1)$ and $D_{p_t}(\psi) = (1 - t)D_{p_0}(\psi) + tD_{p_1}(\psi)$. Therefore, by Theorem~\ref{thm:antitonic-score-proj}\emph{(b)},
\[
-i^*(p_t) = \inf_{\psi \in \Psi_\downarrow(p_t)} D_{p_t}(\psi) \geq (1 - t)\inf_{\psi \in \Psi_\downarrow(p_0)} D_{p_0}(\psi) + t\inf_{\psi \in \Psi_\downarrow(p_1)} D_{p_1}(\psi) = -(1 - t)i^*(p_0) - t\,i^*(p_1),
\]
as required.
\end{proof}

\begin{proof}[Proof of Lemma~\ref{lem:hazard}]
\textit{(a)} By Lemma~\ref{lem:lcm-deriv-0} and~\eqref{eq:J01},
\[
\lim_{z \to -\infty}\psi_0^*(z) = \lim_{z \to -\infty} \hat{J}_0^{(\mathrm{R})}\bigl(F_0(z)\bigr) = \hat{J}_0^{(\mathrm{R})}(0) = \sup_{u \in (0,1)}\frac{J_0(u)}{u} > 0.
\]
On the other hand, recalling from Lemma~\ref{lem:psi0-star} that $p_0 = J_0 \circ F_0$, we have
\[
\limsup_{z \searrow z_{\min}}\,h_0(z) = \limsup_{z \searrow z_{\min}}\frac{p_0(z)}{F_0(z)} = \limsup_{u \searrow 0}\frac{J_0(u)}{u},
\]
which by~\eqref{eq:left-hazard} can only be finite if $z_{\min} = -\infty$. Since $J_0$ is continuous, $\sup_{u \in (0,1)} J_0(u)/u < \infty$ if and only if $\limsup_{u \searrow 0} J_0(u)/u < \infty$, which implies the desired conclusion.

\medskip
\noindent
\textit{(b)} Likewise, 
\[
\lim_{z \to \infty}\psi_0^*(z) = \hat{J}_0^{(\mathrm{L})}(1) = -\sup_{u \in (0,1)}\frac{J_0(u)}{1 - u} \in [-\infty,0),
\]
and is finite if and only if
\[
\limsup_{z \nearrow z_{\max}}\,h_0(z) = \limsup_{z \nearrow z_{\max}}\frac{p_0(z)}{1 - F_0(z)} = \limsup_{u \nearrow 1}\frac{J_0(u)}{1 - u}
\]
is finite, in which case $z_{\max} = \infty$ by arguing similarly to~\eqref{eq:left-hazard}.
\end{proof}

Lemma~\ref{lem:p0-psi0-ineq} is used in the proofs of Lemma~\ref{lem:p0-star} and Proposition~\ref{prop:p0-star-fisher} below.

\begin{lemma}
\label{lem:p0-psi0-ineq}
Let $s_0 := \inf\{z \in \R : \psi_0^*(z) \leq 0\}$ and $t_0 := \sup\{z \in \R : \psi_0^*(z) < 0\}$. For $s,t \in \mathcal{S}_0$ such that $s \leq t$, we have
\begin{equation}
\label{eq:psi0-int-integral}
\int_s^t \psi_0^* \,
\begin{cases}
\geq \log\hat{J}_0\bigl(F_0(t)\bigr) - \log\hat{J}_0\bigl(F_0(s)\bigr) \;\;&\text{if }t \leq t_0 \\
\leq \log\hat{J}_0\bigl(F_0(t)\bigr) - \log\hat{J}_0\bigl(F_0(s)\bigr)  \;\;&\text{if }s \geq s_0,
\end{cases}
\end{equation}
and hence
\begin{equation}
\label{eq:p0-psi0-ineq}
\begin{cases}
p_0(t) \leq p_0(s)\exp\bigl(\int_s^t \psi_0^*\bigr) \;&\text{if }t \leq t_0 \text{ and }s \in \mathcal{T} \\
p_0(t) \geq p_0(s)\exp\bigl(\int_s^t \psi_0^*\bigr) \;&\text{if }s \geq s_0 \text{ and }t \in \mathcal{T}.
\end{cases}
\end{equation}
\end{lemma}

\begin{proof}
For $s,t \in \mathcal{S}_0$ with $s \leq t$, we have $0 < F_0(s) \leq F_0(t) < 1$. By Lemma~\ref{lem:psi0-star} and the fact that $\psi_0^*$ is decreasing,
\[
-\infty < \psi_0^*(t)(t - s) \leq \int_s^t \psi_0^* \leq \psi_0^*(s)(t - s) < \infty.
\]
Introducing $U \sim U(0,1)$ and $Z := F_0^{-1}(U)\sim P_0$, we have $F_0(Z) = U$ by the continuity of $F_0$. Therefore, $\psi_0^*(Z) = \hat{J}_0^{(\mathrm{R})}(U)$ and $\Pr\bigl(J_0(U) = 0\bigr) = \Pr\bigl(p_0(Z) = 0\bigr) = 0$. If $t \leq t_0$, then $\hat{J}_0^{(\mathrm{R})} \circ F_0 = \psi_0^* \geq 0$ on $(s,t]$ and hence $\hat{J}_0^{(\mathrm{R})} \geq 0$ on $(F_0(s),F_0(t)]$. Therefore,
\begin{align*}
\int_s^t \psi_0^* &\geq \E\biggl(\frac{\psi_0^*(Z)}{p_0(Z)}\Ind_{\{Z \in (s,t]\}}\biggr) = \E\biggl(\frac{\hat{J}_0^{(\mathrm{R})}(U)}{J_0(U)}\Ind_{\{U \in (F_0(s),F_0(t)]\}}\biggr) \\
&= \int_{F_0(s)}^{F_0(t)} \frac{\hat{J}_0^{(\mathrm{R})}}{J_0} \geq \int_{F_0(s)}^{F_0(t)} \frac{\hat{J}_0^{(\mathrm{R})}}{\hat{J}_0}
= \log\hat{J}_0\bigl(F_0(t)\bigr) - \log\hat{J}_0\bigl(F_0(s)\bigr),
\end{align*}
where the final equality holds because $\log\hat{J}_0$ is finite-valued, Lipschitz and hence absolutely continuous on $[F_0(s),F_0(t)]$ with derivative equal to $\hat{J}_0^{(\mathrm{R})}/\hat{J}_0$, Lebesgue almost everywhere. This proves the first line of~\eqref{eq:psi0-int-integral}.  If $s \in \mathcal{T}$, then by Lemma~\ref{lem:J0-J0hat},
\[
\log\hat{J}_0\bigl(F_0(t)\bigr) - \log\hat{J}_0\bigl(F_0(s)\bigr) = \log\frac{\hat{J}_0\bigl(F_0(t)\bigr)}{p_0(s)} \geq \log\frac{J_0\bigl(F_0(t)\bigr)}{p_0(s)} = \log\frac{p_0(t)}{p_0(s)},
\]
which yields the first line of~\eqref{eq:p0-psi0-ineq}. The proofs in the case $s_0 \leq s \leq t$ are analogous, except that now $\hat{J}_0^{(\mathrm{R})} \circ F_0 = \psi_0^* \leq 0$ on $(s,t]$ and hence $\hat{J}_0^{(\mathrm{R})} \leq 0$ on $(F_0(s),F_0(t)]$, so the directions of all inequalities above are reversed.
\end{proof}

\begin{proof}[Proof of Lemma~\ref{lem:p0-star}]
Since $\sup_{u \in (0,1)}\hat{J}_0(u) = \sup_{u \in (0,1)}J_0(u) = \sup_{z \in \mathcal{S}_0}p_0(z) > 0$ and $\hat{J}_0$ is concave, we have
\[
\hat{J}_0^{(\mathrm{R})}(0) = \lim_{u \searrow 0}\hat{J}_0^{(\mathrm{R})}(u) > 0 > \lim_{u \nearrow 1}\hat{J}_0^{(\mathrm{R})}(u) = \hat{J}_0^{(\mathrm{R})}(1).
\]
We have $(F_0 \circ F_0^{-1})(u) = u$ for all $u \in (0,1)$, so if $\delta > 0$ is sufficiently small, then $z_1 := F_0^{-1}(\delta)$ and $z_2 := F_0^{-1}(1 - \delta)$ are elements of $\mathcal{S}_0$ satisfying $\psi_0^*(z_1) = \hat{J}_0^{(\mathrm{R})}(\delta) > 0 > \hat{J}_0^{(\mathrm{R})}(1 - \delta) = \psi_0^*(z_2)$. By Lemma~\ref{lem:psi0-star}, $\psi_0^*$ is decreasing and $\{z \in \R : \psi_0^*(z) \in \R\} = \mathcal{S}_0$, so
\[
\phi_0^*(z) := \int_{z_1}^z\psi_0^*
\]
is well-defined in $[-\infty,\infty)$ for all $z \in\R$, with $\phi_0^*(z) \in \R$ for all $z \in \mathcal{S}_0$. By Lemma~\ref{lem:p0-psi0-ineq},
\[
\phi_0^*(z) =
\begin{cases}
-\int_z^{z_1}\psi_0^* \leq \log\hat{J}_0\bigl(F_0(z)\bigr) - \log\hat{J}_0(\delta) \;\;&\text{for }z \leq z_1 \\
\phi_0^*(z_2) + \int_{z_2}^z \psi_0^* \leq \phi_0^*(z_2) + \log\hat{J}_0\bigl(F_0(z)\bigr) - \log\hat{J}_0(1 - \delta) \;\;&\text{for }z \geq z_2.
\end{cases}
\]
Together with~\eqref{eq:J01}, this shows that $\phi_0^*(z) = \phi_0^*(z_{\min}) = \lim_{z' \searrow z_{\min}}\phi_0^*(z') = -\infty$ for all $z \leq z_{\min}$, and $\phi_0^*(z) = \phi_0^*(z_{\max}) = \lim_{z' \nearrow z_{\max}}\phi_0^*(z') = -\infty$ for all $z \geq z_{\max}$. We deduce that $\phi_0^*$ is concave and continuous as a function from $\R$ to $[-\infty,\infty)$, with $\{z \in \R : \phi_0^*(z) > -\infty\} = \mathcal{S}_0$.
By~\citet[Theorem~24.2]{rockafellar97convex}, $\phi_0^*$ has right derivative $\psi_0^*$ on $\mathcal{S}_0$, with $(\phi_0^*)' = \psi_0^*$ Lebesgue almost everywhere on $\mathcal{S}_0$. Since $\psi_0^*$ is decreasing,
\[
\phi_0^*(z) \leq
\begin{cases}
-\psi_0^*(z_1)(z_1 - z)&\;\;\text{for }z \leq z_1 \\ 
\phi_0^*(z_2) + \int_{z_2}^z \psi_0^* \leq \phi_0^*(z_2) - |\psi_0^*(z_2)|(z - z_2)&\;\;\text{for }z \geq z_2,
\end{cases}
\]
so 
\[
\int_\R e^{\phi_0^*} \leq \frac{e^{\phi_0^*(z_1)}}{\psi_0^*(z_1)} + \int_{z_1}^{z_2}e^{\phi_0^*} + \frac{e^{\phi_0^*(z_2)}}{|\psi_0^*(z_2)|} < \infty.
\]
Therefore, $p_0^* := e^{\phi_0^*}/\int_\R e^{\phi_0^*}$ is a continuous log-concave density such that $\supp p_0^* = \mathcal{S}_0$ and $\log p_0^* = \phi_0^* - \log\bigl(\int_{\mathbb{R}} e^{\phi_0^*}\bigr)$ has right derivative $\psi_0^*$ on $\mathcal{S}_0$.

If $\tilde{p}$ is another continuous log-concave density such that $\log\tilde{p}$ has right derivative $\psi_0^*$ on $\mathcal{S}_0$, then by~\citet[Corollary~24.2.1]{rockafellar97convex}, $\log\tilde{p}(z) = \log\tilde{p}(z_1) + \int_{z_1}^z\psi_0^*$ for all $z \in\R$. This implies that $\tilde{p} = p_0^*$, so $p_0^*$ is the unique density with the required properties. 

\medskip
\noindent
Finally, suppose that $p_0$ is a continuous log-concave density on $\R$. Then $\phi_0 := \log p_0$ is concave, and on $\mathcal{S}_0 = \supp p_0 = \{z \in\R : \phi_0(z) > -\infty\}$, its right derivative $\psi_0 := \phi_0^{(\mathrm{R})}$ is decreasing and right-continuous. 
% Note that $\phi_0$ and hence $p_0 = e^{\phi_0}$ are locally Lipschitz on $\mathcal{S}_0$ and hence differentiable Lebesgue almost everywhere, with 
Thus, $p_0$ has right derivative $p_0^{(\mathrm{R})} = \psi_0 p_0$ on $\mathcal{S}_0$. Moreover, $F_0$ is strictly increasing and differentiable on $\mathcal{S}_0 = \{z \in\R : F_0(z) \in (0,1)\}$, so $F_0^{-1}$ is differentiable on $(0,1)$ with derivative $(F_0^{-1})'(u) = 1/(p_0 \circ F_0^{-1})(u)$ for $u \in (0,1)$, and $(F_0^{-1} \circ F_0)(z) = z$ for all $z \in \mathcal{S}_0$. Consequently, $J_0 = p_0 \circ F_0^{-1}$ 
% is locally Lipschitz on $(0,1)$
has right derivative $J_0^{(\mathrm{R})}(u) = p_0^{(\mathrm{R})}\bigl(F_0^{-1}(u)\bigr)/p_0\bigl(F_0^{-1}(u)\bigr) = (\psi_0 \circ F_0^{-1})(u)$ for $u \in (0,1)$. Since $\psi_0$ is decreasing, $J_0$ is therefore concave on $(0,1)$ by~\citet[Theorem~24.2]{rockafellar97convex}; see also~\citet[Proposition~A.1(c)]{bobkov1996extremal}. Therefore, $\hat{J}_0 = J_0$ and $\hat{J}_0^{(\mathrm{R})} = J_0^{(\mathrm{R})} = \psi_0 \circ F_0^{-1}$ on $(0,1)$, so $\psi_0^* = \psi_0 \circ F_0^{-1} \circ F_0 = \psi_0$ on $\mathcal{S}_0$. Together with the arguments in the previous paragraph, this implies that $\phi_0^* = \phi_0$ and hence $p_0^* = p_0$, as claimed.
\end{proof}

\begin{proof}[Proof of Proposition~\ref{prop:Vp0-MLE}]
Given $\epsilon \in (0,1)$, let $a \geq (1 - \epsilon^2)/\epsilon^2 > (1 - \epsilon)/\epsilon$ and define $p_0 \colon \R \to [0,\infty)$ by
\[
p_0(z) :=
\begin{cases}
\epsilon a e^{-a(|z| - 1)}/2 \;\;&\text{for }|z| \geq 1 \\
\epsilon a e^{-b(1 - |z|)}/2 \;\;&\text{for }|z| < 1,
\end{cases}
\]
where $b > 0$ uniquely satisfies $1 - \epsilon = \int_{-1}^1 \epsilon a e^{b(|z| - 1)}/2\,dz = \epsilon a(1 - e^{-b})/b$. Then $\int_\R p_0 = 1$, so $p_0$ is a symmetric, absolutely continuous and piecewise log-affine density, and $b = (1 - e^{-b})a\epsilon/(1 - \epsilon) \leq a\epsilon/(1 - \epsilon)$. 
% Letting $c := a\epsilon/(1 - \epsilon)$, we have $b = c(1 - e^{-b}) \leq c$ and hence $b \leq c(1 - e^{-c})$. 
Moreover,
\[
\psi_0(z) := \frac{p_0'(z)}{p_0(z)} =
\begin{cases}
-a \;\;&\text{for }z > 1 \\
b \;\;&\text{for }z \in (0,1) \\
-\psi_0(-z) \;\;&\text{for }z < 0,
\end{cases}
\]
so $i(p_0) = \int_\R \psi_0^2\,p_0 = b^2(1 - \epsilon) + a^2\epsilon$ and $\psi_0$ is increasing on $(-1,0) \cup (0,1)$. Thus, $p_0$ is log-convex on $[-1,1]$, and log-affine on both $(-\infty,1]$ and $[1,\infty)$. Since $\int_\R zp_0(z)\,dz = 0$, the proof of~\citet[Remark~2.11(ii)]{dumbgen2011approximation} ensures that the log-concave maximum likelihood projection $q_0 \equiv p_0^{\mathrm{ML}}$ of $p_0$ satisfies
\[
p_0^{\mathrm{ML}}(z) =
\begin{cases}
p_0(z) \;\;&\text{for }|z| \geq 1 + \delta \\
p_0(1 + \delta) = \dfrac{\epsilon a e^{-a\delta}}{2} \;\;&\text{for }|z| \leq 1 + \delta,
\end{cases} 
\]
where $\delta > 0$ uniquely solves $1 = \int_\R p_0^{\mathrm{ML}} = \epsilon\bigl(a(1 + \delta) + 1\bigr)e^{-a\delta}$. Then the associated score function $\psi_0^{\mathrm{ML}} := q_0^{(\mathrm{R})}/q_0$ is given by
\[
\psi_0^{\mathrm{ML}}(z) = 
\begin{cases}
a \;\;&\text{for }z < -(1 + \delta) \\
0 \;\;&\text{for }z \in [-(1 + \delta), 1 + \delta) \\
-a \;\;&\text{for }z \geq 1 + \delta.
\end{cases}
\]
Therefore, $\psi_0^{\mathrm{ML}} \in \Psi_\downarrow(p_0)$ with
\begin{align*}
\int_\R (\psi_0^{\mathrm{ML}})^2\,dP_0 = 2a^2\int_{1+\delta}^\infty p_0 = a^2\epsilon e^{-a\delta} = a\bigl(p_0(1 + \delta) + p_0(-1 - \delta)\bigr) = -\int_\R p_0 \,d\psi_0^{\mathrm{ML}},
\end{align*}
so 
\[
\frac{1}{V_{p_0}(\psi_0^{\mathrm{ML}})} = \frac{\bigl(\int_\R p_0 \,d\psi_0^{\mathrm{ML}}\bigr)^2}{\int_\R (\psi_0^{\mathrm{ML}})^2\,dP_0} = -\int_\R p_0\,d\psi_0^{\mathrm{ML}} = a^2\epsilon e^{-a\delta}.
\]
Furthermore, $J_0'(u) = (\psi_0 \circ F_0^{-1})(u)$ whenever $u = F_0(z)$ for $z \neq \{-1,0,1\}$, so $J_0$ is symmetric about $1/2$, convex on $[F_0(-1),F_0(1)]$, and linear on both $[0,F_0(-1)]$ and $[F_0(1),1]$. Hence, $\hat{J}_0 = J_0$ on $[0,F_0(-1)] \cup [F_0(1),1]$ and $\hat{J}_0 = p_0(1)$ on $[F_0(-1),F_0(1)]$, so
\[
\psi_0^*(z) = \hat{J}_0^{(\mathrm{R})}\bigl(F_0(z)\bigr) =
\begin{cases}
a \;\;&\text{for }z < -1 \\
0 \;\;&\text{for }z \in [-1,1) \\
-a \;\;&\text{for }z \geq 1.
\end{cases}
\]
Thus, by Theorem~\ref{thm:antitonic-score-proj}\textit{(c)}, $1/V_{p_0}(\psi_0^*) = i^*(p_0) = \int_\R (\psi_0^*)^2\,p_0 = a^2\epsilon$,
so
\[
\frac{V_{p_0}(\psi_0^*)}{V_{p_0}(\psi_0^{\mathrm{ML}})} = e^{-a\delta} = \frac{1}{\epsilon\bigl(a(1 + \delta) + 1\bigr)} < \frac{1}{\epsilon(a + 1)} \leq \epsilon
\]
and
\[
\mathrm{ARE}^*(p_0) = \frac{i^*(p_0)}{i(p_0)} = \frac{a^2\epsilon}{b^2(1 - \epsilon) + a^2\epsilon} \geq 
% \frac{a^2\epsilon}{c^2(1 - e^{-c})^2(1 - \epsilon) + a^2\epsilon} = 
\frac{a^2\epsilon}{(a\epsilon)^2/(1 - \epsilon) + a^2\epsilon} = 1 - \epsilon,
\]
as required.
\end{proof}

\begin{proof}[Proof of Lemma~\ref{lem:ARE-lower-bound}]
By Lemma~\ref{lem:psi0-star}, $\hat{J}_0$ is continuous on $[0,1]$ with $\hat{J}_0(0) = \hat{J}_0(1) = 0$, and we can find $u^* \in (0,1)$ such that $\hat{J}_0(u^*) = J_0(u^*) = \norm{J_0}_\infty = \norm{p_0}_\infty < \infty$. Then by the Cauchy--Schwarz inequality and~\citet[Corollary~24.2.1]{rockafellar97convex},
\begin{align*}
\int_0^{u^*} \bigl(\hat{J}_0^{(\mathrm{R})}\bigr)^2 &\geq \frac{\bigl(\int_0^{u^*} \hat{J}_0^{(\mathrm{R})}\bigr)^2}{u^*} = \frac{\hat{J}_0(u^*)^2}{u^*} = \frac{\norm{p_0}_\infty^2}{u^*}, \\
\int_{u^*}^1 \bigl(\hat{J}_0^{(\mathrm{R})}\bigr)^2 &\geq \frac{\bigl(\int_{u^*}^1 \hat{J}_0^{(\mathrm{R})}\bigr)^2}{1 - u^*} = \frac{\hat{J}_0(u^*)^2}{1 - u^*} = \frac{\norm{p_0}_\infty^2}{1 - u^*},
\end{align*}
with equality if and only if $\hat{J}_0$ is linear on both $[0,u^*]$ and $[u^*,1]$. Thus, by Remark~\ref{rem:fisher-J},
\[
i^*(p_0) = \int_0^1 \bigl(\hat{J}_0^{(\mathrm{R})}\bigr)^2 \geq \norm{p_0}_\infty^2 \Bigl(\frac{1}{u^*} + \frac{1}{1 - u^*}\Bigr) \geq 4\norm{p_0}_\infty^2.
\]
Equality holds if and only if $u^* = 1/2$ and $\hat{J}_0(u) = \min(u, 1 - u)/\sigma$ for some $\sigma > 0$. In this case, letting $\mu := F_0^{-1}(1/2)$, we deduce that $\psi_0^*(z) = (\hat{J}_0^{(\mathrm{R})} \circ F_0)(z) = -\sgn(z - \mu)/\sigma$ for all $z \in \R$, so $p_0^*$ is a Laplace density of the stated form.
\end{proof}

To prove Proposition~\ref{prop:p0-star-fisher}, we require some further definitions and lemmas. Let $\mathcal{T}$ be as in the proposition. Then $\mathcal{T}^c$ is an open subset of $\mathbb{R}$ and hence has a unique representation as a countable disjoint union $\bigcup_{k=1}^K (s_k,t_k)$ of open intervals, where $K \in \N_0 \cup \{\infty\}$ and $-\infty \leq s_k < t_k \leq \infty$ for every $k$. Recall from Lemma~\ref{lem:p0-psi0-ineq} that $s_0 = \inf\{z \in \R : \psi_0^*(z) \leq 0\}$ and $t_0 = \sup\{z \in \R : \psi_0^*(z) < 0\}$. For $t \in \R$, let $\psi_0(t-) := \lim_{z \nearrow t}\psi_0^*(z)$. For $k \in \N$ such that $k \leq K$, inductively define $p_k \colon [-\infty,\infty] \to \R$ by
\[
p_k(z) := 
\begin{cases}
p_{k-1}(z) \;\;&\text{for }z \leq s_k \\
p_{k-1}(s_k)e^{\psi_0^*(s_k)(z - s_k)} \;\;&\text{for } z \in [s_k,t_k] \\
p_{k-1}(z) \cdot \dfrac{p_{k-1}(s_k)e^{\psi_0^*(s_k)(t_k - s_k)}}{p_{k-1}(t_k)} \;\;&\text{for }z > t_k
\end{cases}
\]
if $s_k > s_0$, and otherwise let 
\[
p_k(z) := 
\begin{cases}
p_{k-1}(z) \cdot \dfrac{p_{k-1}(t_k)e^{\psi_0^*(t_k-)(s_k - t_k)}}{p_{k-1}(s_k)}  \;\;&\text{for }z < s_k \\[6pt]
p_{k-1}(t_k)e^{\psi_0^*(t_k-)(z - t_k)} \;\;&\text{for } z \in [s_k,t_k] \\
p_{k-1}(z) \;\;&\text{for }z \geq t_k.
\end{cases}
\]
if $s_k \leq s_0$. When $K < \infty$, define $p_k := p_K$ for all $k \in \N$ with $k > K$. By Lemmas~\ref{lem:psi0-star} and~\ref{lem:J0-J0hat}, $\psi_0^*(z) \in \R$ and $p_0(z) > 0$ for all $z \in \mathcal{T}$. Since $s_k,t_k \in \mathcal{T} \cup \{-\infty,\infty\}$ and $p_k(s_0) = p_{k-1}(s_0)$ for $k \in \N$ such that $k \leq K$, it follows by induction that $p_k > 0$ on $\mathcal{T}$ for all $k \in \N$, with $p_k(s_0) = p_0(s_0)$. In particular, $p_{k-1}(s_k) > 0$ whenever $s_k > -\infty$ and $p_{k-1}(t_k) > 0$ whenever $t_k < \infty$.

\begin{lemma}
\label{lem:pk-induction}
For every $k \in \N_0$, the following statements hold.
\begin{enumerate}[label=(\alph*)]
\item $p_k \leq p_0$ on $\mathcal{T}$, and $\norm{p_k}_\infty = \norm{p_0}_\infty$;
\item $\displaystyle p_0(z) \int_s^t p_k \geq p_k(z) \int_s^t p_0$ if $z \in \{s,t\}$ for some $s,t \in \mathcal{T} \cup \{-\infty,\infty\}$ such that $s \leq t$;
% If either $t_k \leq t \wedge t_0$ or $s_k \geq s \vee s_0$
% NB: if $s_0 \in [s,t]$, then since $p_0(s_k) = p_k(s_0)$, we obtain the tighter inequality $\int_s^t p_k = \int_s^{s_0} p_k + \int_{s_0}^t p_k \geq \int_s^{s_0} p_{k-1} + \int_{s_0}^t p_{k-1} = \int_s^t p_{k-1}$.
\item $\dfrac{p_k(z)}{p_k(s_j)} = 
\begin{cases}
p_0(z)/p_0(s_j) &\text{if }z \in [s_j,t_j]\text{ for some }j \in \N\text{ with }j > k \\
p_j(z)/p_j(s_j) &\text{if }z \in [s_j,t_j]\text{ for some }j \in \N\text{ with }j < k;
\end{cases}$
\item $\displaystyle \psi_0^*(t-) \int_s^t p_k \leq p_k(t) - p_k(s) \leq \psi_0^*(s)\int_s^t p_k$ for all $s,t \in \mathcal{T} \cup \{-\infty,\infty\}$ such that $s \leq t$.
\end{enumerate}
\end{lemma}

\begin{proof}
We will proceed by induction on $k$. When $k = 0$, \textit{(a)},~\textit{(b)} and~\textit{(c)} hold trivially, so it remains to prove \textit{(d)}. This holds with equality when $s = t$, so now let $s,t \in \mathcal{T} \cup \{-\infty,\infty\}$ be such that $s < t$. By the concavity of $\hat{J}_0$ together with Lemmas~\ref{lem:psi0-star} and~\ref{lem:J0-J0hat}, $v := F_0(s)$ and $w := F_0(t)$ satisfy $0 \leq v < w \leq 1$ and
\begin{equation}
\label{eq:p0-F0-ineq-1}
p_0(t) - p_0(s) = \hat{J}_0(w)  - \hat{J}_0(v) 
\begin{cases}
\leq \hat{J}_0^{(\mathrm{R})}(v)(w - v) = \psi_0^*(s)\int_s^t p_0 \\[3pt]
\geq \hat{J}_0^{(\mathrm{L})}(w)(w - v) = \displaystyle\lim_{u \nearrow w}\hat{J}_0^{(\mathrm{R})}(u)(w - v) = \textstyle\psi_0^*(t-)\int_s^t p_0,
\end{cases}
\end{equation}
where $\hat{J}_0^{(\mathrm{L})}(w) = \lim_{u \nearrow w}\hat{J}_0^{(\mathrm{R})}(u)$ by~\citet[Theorem~24.1]{rockafellar97convex}. 

Next, consider a general $k \in \N$. If $k > K$, then $p_k = p_{k-1}$ and hence \textit{(a)}--\textit{(d)} hold by induction. Supposing now that $k \leq K$, let $a_k := \psi_0^*(s_k)$ if $s_k > -\infty$ and otherwise let $a_k := \psi_0^*(t_k-)$, so that $\psi_0^* = a_k$ on $(s_k,t_k)$.

\medskip
\noindent
\textit{(a)} By part \textit{(c)} of the inductive hypothesis and~\eqref{eq:p0-psi0-ineq},
\begin{equation}
\label{eq:rk}
r_k := \frac{p_{k-1}(t_k)e^{a_k(s_k - t_k)}}{p_{k-1}(s_k)} = \frac{p_0(t_k)e^{a_k(s_k - t_k)}}{p_0(s_k)}
\begin{cases}
\leq 1 \;\;&\text{if }s_k \leq s_0 \\
\geq 1 \;\;&\text{if }s_k \geq s_0.
\end{cases}
\end{equation}
Therefore,
\begin{equation}
\label{eq:pk-pk-1}
\begin{cases}
p_k = r_k p_{k-1} \leq p_{k-1} \text{ on }(-\infty,s_k] \quad\text{and}\quad p_k = p_{k-1} \text{ on }[t_k,\infty) \;\;&\text{if }s_k \leq s_0 \\
p_k = p_{k-1} \text{ on }(-\infty,s_k] \quad\text{and}\quad p_k = r_k^{-1}p_{k-1} \leq p_{k-1} \text{ on }[t_k,\infty) \;\;&\text{if }s_k \geq s_0,
\end{cases}
\end{equation}
while for $z \in [s_k,t_k]$, we have
\[
p_k(z) = 
\begin{cases}
p_{k-1}(t_k)e^{a_k(z - t_k)} \leq p_{k-1}(t_k) \;\;&\text{if }s_k \leq s_0 \\
p_{k-1}(s_k)e^{a_k(z - s_k)} \leq p_{k-1}(s_k) \;\;&\text{if }s_k \geq s_0.
\end{cases}
\]
This shows that $p_k \leq p_{k-1}$ on $\R \setminus (s_k,t_k) \supseteq \mathcal{T}$ and $\norm{p_k}_\infty = \norm{p_{k-1}}_\infty$, so \textit{(a)} holds by induction.

\medskip
\noindent
% Let $a_k := \psi_0^*(s_k) = \psi_0^*(t_k-)$ and $r_k := p_{k-1}(t_k)e^{a_k(s_k - t_k)}/p_{k-1}(s_k)$.
\textit{(b)} Taking $s = s_k$ and $t = t_k$ in~\eqref{eq:p0-F0-ineq-1}, we deduce from part~\textit{(c)} of the inductive hypothesis that
\begin{equation}
\label{eq:ak-eq}
\frac{p_{k-1}(t_k) - p_{k-1}(s_k)}{\int_{s_k}^{t_k} p_{k-1}} = \frac{p_0(t_k) - p_0(s_k)}{\int_{s_k}^{t_k} p_0} = a_k
\begin{cases}
\geq 0 \;\;&\text{if }s_k \leq s_0 \\
\leq 0 \;\;&\text{if }s_k \geq s_0.
\end{cases}
\end{equation}
% If either $t_k \leq s \wedge t_0$ or $s_k \geq t \vee s_0$, then $p_k = p_{k-1}$ on $(s,t)$ and the result holds trivially.
There is nothing to prove when $s = t$, so let $s,t \in \mathcal{T} \cup \{-\infty,\infty\}$ be such that $s < t$. If either $(s,t) \subseteq (-\infty,s_k)$ or $(s,t) \subseteq (t_k,\infty)$, then~\eqref{eq:pk-pk-1} implies that $p_k = rp_{k-1}$ on $[s,t]$ for some $r \in \{1,r_k,r_k^{-1}\}$, so by part~\textit{(c)} of the inductive hypothesis,
\[
p_0(z) \int_s^t p_k = rp_0(z) \int_s^t p_{k-1} \geq rp_{k-1}(z) \int_s^t p_0 = p_k(z) \int_s^t p_0
\]
% $s < t_k \leq t \wedge t_0$, in which case $s \leq s_k \leq s_0$ and $a_k \in [0,\infty)$
for $z \in \{s,t\}$. It remains to consider the case $(s_k,t_k) \subseteq (s,t)$. Assume first that $s_k \leq s_0$, so that $a_k \in [0,\infty)$. If $a_k = 0$, then by part~\textit{(c)} of the inductive hypothesis and~\eqref{eq:p0-psi0-ineq},
\[
p_{k-1}(z) = \frac{p_0(z)}{p_0(s_k)} \cdot p_{k-1}(s_k) \leq p_{k-1}(s_k) \leq p_{k-1}(s_k) = p_{k-1}(t_k)
\]
for all $z \in [s_k,t_k]$, so
\[
\int_s^t p_k = \frac{p_{k-1}(t_k)}{p_{k-1}(s_k)}\int_{-\infty}^{s_k} p_{k-1} + \int_{s_k}^{t_k}p_{k-1}(t_k)\,dz + \int_{t_k}^t p_{k-1} \geq \int_s^{s_k} p_{k-1} + \int_{s_k}^{t_k}p_{k-1} + \int_{t_k}^t p_{k-1} = \int_s^t p_{k-1}.
\]
Suppose instead that $a_k > 0$. Taking $s = -\infty$ and $t = s_k$ in part~\textit{(d)} of the inductive hypothesis, we obtain $p_{k-1}(s_k) \geq \psi_0^*(s_k-)\int_{-\infty}^{s_k}p_{k-1} \geq a_k\int_{-\infty}^{s_k}p_{k-1}$, so
\[
\frac{\int_{-\infty}^{s_k}p_{k-1}}{p_{k-1}(s_k)} \leq \frac{1}{a_k}.
\]
Together with~\eqref{eq:ak-eq}, this implies that
\begin{align*}
\int_{-\infty}^{t_k}p_k = \int_{-\infty}^{s_k} r_k p_{k-1} + p_{k-1}(t_k) \cdot \frac{1 - e^{a_k(s_k - t_k)}}{a_k} &= p_{k-1}(t_k)e^{a_k(s_k - t_k)}\biggl(\frac{\int_{-\infty}^{s_k} p_{k-1}}{p_{k-1}(s_k)} - \frac{1}{a_k}\biggr) + \frac{p_{k-1}(t_k)}{a_k} \\
&\geq p_{k-1}(s_k)\biggl(\frac{\int_{-\infty}^{s_k} p_{k-1}}{p_{k-1}(s_k)} - \frac{1}{a_k}\biggr) + \frac{p_{k-1}(t_k)}{a_k} \\
&= \int_{-\infty}^{s_k} p_{k-1} + \frac{p_{k-1}(t_k) - p_{k-1}(s_k)}{a_k} = \int_{-\infty}^{t_k} p_{k-1}.
\end{align*}
Combining this with~\eqref{eq:rk} and~\eqref{eq:pk-pk-1} yields
\[
\int_s^t p_k = \int_{-\infty}^{t_k}p_k - \int_{-\infty}^s r_k p_{k-1} + \int_{t_k}^t p_{k-1} \geq \int_{-\infty}^{t_k}p_{k-1} - \int_{-\infty}^s p_{k-1} + \int_{t_k}^t p_{k-1} = \int_s^t p_{k-1}.
\]
In view of~\eqref{eq:pk-pk-1} and part~\textit{(b)} of the inductive hypothesis, we conclude that if $(s_k,t_k) \subseteq (s,t)$ and $s_k \leq s_0$, then
\[
p_0(z) \int_s^t p_k  = p_0(z) \int_s^t p_{k-1} \geq p_{k-1}(z) \int_s^t p_0 \geq p_k(z) \int_s^t p_0
\]
for $z \in \{s,t\}$, which proves \textit{(b)} in this case. On the other hand, if $(s_k,t_k) \subseteq (s,t)$ and $s_k > s_0$, then instead $a_k \in (-\infty,0]$, but the arguments are similar and hence omitted.

\medskip
\noindent
\textit{(c)} For $j \in [K]$ such that $j > k$, we have $(s_j,t_j) \cap (s_k,t_k) = \emptyset$, so by~\eqref{eq:pk-pk-1} and part~\textit{(c)} of the inductive hypothesis, 
\[
\frac{p_k(z)}{p_k(s_j)} = \frac{p_{k-1}(z)}{p_{k-1}(s_j)} =
\begin{cases}
p_0(z)/p_0(s_j) &\text{if }z \in [s_j,t_j]\text{ for }j \in \N\text{ with }j > k \\
p_j(z)/p_j(s_j) &\text{if }z \in [s_j,t_j]\text{ for }j \in \N\text{ with }j < k.
\end{cases}
\]
\textit{(d)} There is nothing to prove when $s = t$, so let $s,t \in \mathcal{T} \cup \{-\infty,\infty\}$ be such that $s < t$. If either $(s,t) \subseteq (-\infty,s_k)$ or $(s,t) \subseteq (t_k,\infty)$, then by~\eqref{eq:pk-pk-1}, $p_k = rp_{k-1}$ on $[s,t]$ for some $r \in \{1,r_k,r_k^{-1}\}$. Thus, by part \textit{(d)} of the inductive hypothesis,
\[
p_k(t) - p_k(s) = r\bigl(p_{k-1}(t) - p_{k-1}(s)\bigl) \:
\begin{cases}
\leq r\psi_0^*(s)\int_s^t p_{k-1} = \psi_0^*(s)\int_s^t p_k \\[3pt]
\geq r\psi_0^*(t-)\int_s^t p_{k-1} = \psi_0^*(t-)\int_s^t p_k,
\end{cases}
\]
as required. In the remaining case, $(s_k,t_k) \subseteq (s,t)$, and similarly by~\eqref{eq:pk-pk-1} and part \textit{(d)} of the inductive hypothesis,
\begin{align*}
\psi_0^*(s_k-)\int_s^{s_k} p_k &\leq p_k(s_k) - p_k(s) \leq \psi_0^*(s)\int_s^{s_k} p_k, \\
\psi_0^*(t-)\int_{t_k}^t p_k &\leq p_k(t) - p_k(t_k) \leq \psi_0^*(t_k)\int_{t_k}^t p_k.
\end{align*}
Moreover, $a_k\int_{s_k}^{t_k}p_k = a_k p_k(t_k)\int_{s_k}^{t_k}e^{a(z - t_k)}\,dz = p_k(t_k) - p_k(s_k)$ and $a_k = \psi_0^*(s_k) = \psi_0^*(t_k-)$, so
\[
\psi_0^*(t_k-)\int_{s_k}^{t_k} p_k = p_k(t_k) - p_k(s_k) = \psi_0^*(s_k)\int_{s_k}^{t_k} p_k.
\]
Since $\psi_0^*$ is decreasing, 
% $\psi_0^*(t-) \leq \psi_0^*(t_1-) \leq \psi_0^*(s_1-)$ and $\psi_0^*(t_1) \leq \psi_0^*(s_1) \leq \psi_0^*(s)$
combining this with the two inequalities above yields \textit{(d)}.
\end{proof}

\begin{lemma}
\label{lem:pk-limit}
For $z \in \R$, we have
\begin{equation}
\label{eq:pk-limit}
\lim_{k \to \infty}p_k(z) = \frac{p_0(s_0)}{p_0^*(s_0)} \cdot p_0^*(z).
\end{equation}
Moreover, let $t_{\min} := \inf\mathcal{T}$ and $t_{\max} := \sup\mathcal{T}$. If $s,t \in \mathcal{T} \cup \{-\infty,\infty\}$ are such that $-\infty < s \vee t_{\min} \leq t \wedge t_{\max} < \infty$, then
\begin{equation}
\label{eq:pk-int-limit}
\lim_{k \to \infty}\int_s^t p_k = \frac{p_0^*(s_0)}{p_0(s_0)} \int_s^t p_0^*.
\end{equation}
% Is this true in general if $|s| \vee |t| = \infty$?
\end{lemma}

\begin{proof}
For $k \in \N$, let $\phi_k := \log p_k$ and define $r_k$ as in~\eqref{eq:rk}. Then as noted above, $\phi_k > -\infty$ on $\mathcal{T}$, and Lemma~\ref{lem:pk-induction}\textit{(c)} ensures that
\begin{equation}
\label{eq:log-rk}
\log r_k = \psi_0^*(s_k)(t_k - s_k) - \bigl(\phi_{k-1}(t_k) - \phi_{k-1}(s_k)\bigr) = \int_{s_k}^{t_k}\psi_0^*(z)\,dz - \bigl(\phi_0(t_k) - \phi_0(s_k)\bigr).
\end{equation}
Fix $t \in \mathcal{T}$ such that $t \geq s_0$, and for convenience, define $[K] = \emptyset$ when $K = 0$ and $[K] = \N$ when $K = \infty$. Let $\mathcal{K}$ be the set of $\ell \in [K]$ such that $(s_\ell,t_\ell) \subseteq (s_0,t)$, so that $(s_\ell,t_\ell) \cap (s_0,t) = \emptyset$ for $\ell \in [K] \setminus \mathcal{K}$ and hence $\mathcal{T}^c \cap (s_0,t) = \bigcup_{\ell \in \mathcal{K}}(s_\ell,t_\ell)$. Similarly to the proof of Lemma~\ref{lem:p0-psi0-ineq}, let $U \sim U(0,1)$, so that $Z := F_0^{-1}(U)\sim P_0$, $F_0(Z) = U$ and $\psi_0^*(Z) = \hat{J}_0^{(\mathrm{R})}(U)$. We have $\psi_0^* \leq 0$ and $p_0 > 0$ on $\mathcal{T} \cap (s_0,t)$, and moreover $\hat{J}_0$ is Lipschitz on $(F_0(s_0),F_0(t)]$ with $\hat{J}_0^{(\mathrm{R})} \leq 0$, so
\begin{align}
\int_{s_0}^t \psi_0^* \Ind_{\mathcal{T}} &= \E\biggl(\frac{\psi_0^*(Z)}{p_0(Z)}\Ind_{\{Z \in \mathcal{T} \cap (s_0,t)\}}\biggr) = \E\biggl(\frac{\hat{J}_0^{(\mathrm{R})}(U)}{\hat{J}_0(U)}\Ind_{\{F_0^{-1}(U) \in \mathcal{T} \cap (s_0,t)\}}\biggr) \notag \\
&= \E\biggl(\frac{\hat{J}_0^{(\mathrm{R})}(U)}{\hat{J}_0(U)}\Ind_{\{F_0^{-1}(U) \in (s_0,t]\}}\biggr) - \sum_{\ell \in \mathcal{K}} \E\biggl(\frac{\hat{J}_0^{(\mathrm{R})}(U)}{\hat{J}_0(U)}\Ind_{\{F_0^{-1}(U) \in (s_\ell,t_\ell]\}}\biggr) \notag \\
% &= \E\biggl(\frac{\hat{J}_0^{(\mathrm{R})}(U)}{\hat{J}_0(U)}\Ind_{\{U \in (F_0(s_0),F_0(t)]\}}\biggr) - \sum_{k \in \mathcal{K}} \E\biggl(\frac{\hat{J}_0^{(\mathrm{R})}(U)}{\hat{J}_0(U)}\Ind_{\{U \in (F_0(s_0),F_0(t)]\}}\biggr) \notag \\
&= \int_{F_0(s_0)}^{F_0(t)} \frac{\hat{J}_0^{(\mathrm{R})}}{\hat{J}_0} - \sum_{\ell \in \mathcal{K}} \int_{F_0(s_\ell)}^{F_0(t_\ell)} \frac{\hat{J}_0^{(\mathrm{R})}}{\hat{J}_0} \notag \\
&= \log\hat{J}_0\bigl(F_0(t)\bigr) - \log\hat{J}_0\bigl(F_0(s_0)\bigr) - \sum_{\ell \in \mathcal{K}} \bigl\{\log\hat{J}_0\bigl(F_0(t_\ell)\bigr) - \log\hat{J}_0\bigl(F_0(s_\ell)\bigr)\bigr\} \notag \\
\label{eq:psi0-T}
&= \phi_0(t) - \phi_0(s_0) - \sum_{\ell \in \mathcal{K}}\bigl(\phi_0(t_\ell) - \phi_0(s_\ell)\bigr),
\end{align}
where the second and final equalities above follow from Lemma~\ref{lem:J0-J0hat} and the fact that $s_0,t,s_\ell,t_\ell \in \mathcal{T} \cup \{-\infty,\infty\}$ for $\ell \in [K]$. In addition, for all such $\ell$, we have $\phi_\ell(t) = \phi_{\ell-1}(t) + (\log r_\ell)\Ind_{\{\ell \in \mathcal{K}\}}$ by the definition of $p_\ell$, and $\phi_0^* := \log p_0^*$ satisfies $\phi_0^*(t) - \phi_0^*(s_0) = \int_{s_0}^t \psi_0^*$. Thus, by induction together with~\eqref{eq:log-rk} and~\eqref{eq:psi0-T},
\begin{align*}
\phi_k(t) - \phi_0(s_0) &= \phi_0(t) - \phi_0(s_0) + \sum_{\ell \in \mathcal{K} \cap [k]}\log r_\ell = \phi_0(t) - \phi_0(s_0) - \sum_{\ell \in \mathcal{K} \cap [k]} \Bigl(\phi_0(t_\ell) - \phi_0(s_\ell) - \int_{s_\ell}^{t_\ell}\psi_0^*\Bigr) \\
&= \biggl(\int_{s_0}^t \psi_0^* \Ind_{\mathcal{T}} + \sum_{\ell \in \mathcal{K}} \int_{s_\ell}^{t_\ell} \psi_0^*\biggr) + \sum_{\ell \in \mathcal{K} : \ell > k} \Bigl(\phi_0(t_\ell) - \phi_0(s_\ell) - \int_{s_\ell}^{t_\ell}\psi_0^*\Bigr) \\
&= \int_{s_0}^t \psi_0^* + \sum_{\ell \in \mathcal{K} : \ell > k} \Bigl(\phi_0(t_\ell) - \phi_0(s_\ell) - \int_{s_\ell}^{t_\ell}\psi_0^*\Bigr) \to \phi_0^*(t) - \phi_0^*(s_0)
% \begin{cases}
% = \phi_0^*(t) - \phi_0^*(s_0) &\text{if }k = K < \infty \\[3pt]
% \to \phi_0^*(t) - \phi_0^*(s_0) &\text{as }k \to \infty \text{ if }K = \infty,
% \end{cases}
\end{align*}
as $k \to \infty$, which yields~\eqref{eq:pk-limit} for $z = t$. On the other hand, if $t \in \mathcal{T}$ and $t < s_0$, then we can instead let $\mathcal{K}$ be the set of $\ell \in [K]$ such that $(s_\ell,t_\ell) \subseteq (t,s_0)$, and deduce by similar reasoning that
\begin{align*}
\phi_k(t) - \phi_0(s_0) &= \phi_0(t) - \phi_0(s_0) + \sum_{\ell \in \mathcal{K} \cap [k]} \Bigl(\phi_0(t_\ell) - \phi_0(s_\ell) - \int_{s_\ell}^{t_\ell}\psi_0^*\Bigr) \\
&= -\int_t^{s_0}\psi_0^* - \sum_{\ell \in \mathcal{K} : \ell > k} \Bigl(\phi_0(t_\ell) - \phi_0(s_\ell) - \int_{s_\ell}^{t_\ell}\psi_0^*\Bigr) \to \int_{s_0}^t \psi_0^* = \phi_0^*(t) - \phi_0^*(s_0)
% \begin{cases}
% = \phi_0^*(t) - \phi_0^*(s_0) &\text{if }k = K < \infty \\[3pt]
% \to \phi_0^*(t) - \phi_0^*(s_0) &\text{as }k \to \infty \text{ if }K = \infty,
% \end{cases}
\end{align*}
as $k \to \infty$. Having proved that~\eqref{eq:pk-limit} holds for all $z \in \mathcal{T}$, we now consider $z \in \mathcal{T}^c$, for which there exists a unique $\ell \in [K]$ such that $z \in (s_\ell,t_\ell)$. If $s_\ell > s_0$, then $p_\ell(z) = p_\ell(s_\ell)e^{\psi_0^*(s_\ell)(z - s_\ell)}$, while if $s_\ell \leq s_0$, then $p_\ell(z) = p_\ell(t_\ell)e^{\psi_0^*(s_\ell)(z - t_\ell)}$.
% \[
% p_\ell(z) = 
% \begin{cases}
% p_\ell(s_\ell)e^{\psi_0^*(s_\ell)(z - s_\ell)} &\text{if }s_\ell > s_0 \\
% p_\ell(t_\ell)e^{\psi_0^*(s_\ell)(z - t_\ell)} &\text{if }s_\ell \leq s_0.
% \end{cases}
% \]
For $k > \ell$, it follows from Lemma~\ref{lem:pk-induction}\textit{(c)} that
\[
p_k(z) = 
\begin{cases}
p_k(t_\ell)e^{\psi_0^*(s_\ell)(z - t_\ell)} &\text{if }s_\ell \leq s_0 \\
p_k(s_\ell)e^{\psi_0^*(s_\ell)(z - s_\ell)} &\text{if }s_\ell > s_0.
\end{cases}
\]
Moreover, $s_\ell,t_\ell \in \mathcal{T} \cup \{-\infty,\infty\}$, so if $s_\ell \leq s_0$, then
\[
p_k(z) = p_k(t_\ell)e^{\psi_0^*(s_\ell)(z - t_\ell)} \to \frac{p_0(s_0)}{p_0^*(s_0)} \cdot p_0^*(t_\ell)e^{\psi_0^*(s_\ell)(z - t_\ell)} = \frac{p_0(s_0)}{p_0^*(s_0)} \cdot p_0^*(z)
\]
as $K \to \infty$, and~\eqref{eq:pk-limit} holds similarly when $s_\ell > s_0$.
% \[
% \frac{p_0(s_0)}{p_0^*(s_0)} \cdot p_0^*(z) = \frac{p_0(s_0)}{p_0^*(s_0)} \cdot p_0^*(s_\ell)e^{\psi_0^*(s_\ell)(z - t_\ell)} =
% \begin{cases}
% p_K(s_\ell)e^{\psi_0^*(s_\ell)(z - t_\ell)} = p_K(z) &\text{if }K < \infty \\
% \lim_{k \to K}p_k(s_\ell)e^{\psi_0^*(s_\ell)(z - t_\ell)} = \lim_{k \to K}p_k(z) &\text{if }K = \infty,
% \end{cases}
% \]
% as required.

Finally, let $s,t \in \mathcal{T} \cup \{-\infty,\infty\}$ be such that $-\infty < s \vee t_{\min} \leq t \wedge t_{\max} < \infty$. If $t > t_{\max}$, then there exists a unique $j \in [K]$ such that $t_{\max} = s_j < t_j = t = \infty$, and by Lemma~\ref{lem:p0-star} and its proof, $\psi_0^*(t_{\max}) = \lim_{z \to \infty}\psi_0^*(z) \in (-\infty,0)$. Moreover, by Lemma~\ref{lem:pk-induction}\textit{(c)},
\[
p_k(z) = p_k(t_{\max})e^{\psi_0^*(t_{\max})(z - t_{\max})}
\]
for all $z > t_{\max}$ and $k \geq j$. Similarly, if $s < t_{\min}$, then there exists a unique $\ell \in [K]$ such that $-\infty = s = s_\ell < t_\ell = t_{\min}$, and $\psi_0^*(t_{\min}-) = \lim_{z \to -\infty}\psi_0^*(z) \in (0,\infty)$. In addition, by Lemma~\ref{lem:pk-induction}\textit{(c)},
\[
p_k(z) = p_k(t_{\min})e^{\psi_0^*(t_{\min}-)(z - t_{\min})}
\]
for all $z < t_{\min}$ and $k \geq \ell$. Thus, for $k \geq j \vee \ell$, it follows from Lemma~\ref{lem:pk-induction}\textit{(a)} that in all cases,
% $p_k(t) \leq p_0(t)$ for $t \in \{t_{\min},t_{\max}\} \subseteq \mathcal{T}$, so
\[
p_k(z) \leq
\begin{cases}
p_0(t_{\min})e^{\psi_0^*(t_{\min}-)(z - t_{\min})} \;\;&\quad\text{for }z \in (s,s \vee t_{\min}) \\
\norm{p_0}_\infty \;\;&\quad\text{for }z \in [s \vee t_{\min},t \wedge t_{\max}] \\
p_0(t_{\max})e^{\psi_0^*(t_{\max})(z - t_{\max})} \;\;&\quad\text{for }z \in (t \wedge t_{\max},t),
\end{cases}
\]
so the pointwise supremum $\sup_{k \geq j \vee \ell} p_k$ is integrable on $(s,t)$, and hence~\eqref{eq:pk-int-limit} follows from~\eqref{eq:pk-limit} and the dominated convergence theorem.
\end{proof}

\begin{proof}[Proof of Proposition~\ref{prop:p0-star-fisher}]
\textit{(a)} We will prove the following stronger statement: for $s,t \in \mathcal{T} \cup \{-\infty,\infty\}$ such that $s \leq t$, and $z \in \{s,t\}$, we have
\begin{equation}
\label{eq:p0-star-s-t}
\frac{p_0^*(z)}{F_0^*(t) - F_0^*(s)} \leq \frac{p_0(z)}{F_0(t) - F_0(s)}.
\end{equation}
Assume that at least one of $s,t$ is finite, since otherwise~\eqref{eq:p0-star-s-t} holds trivially. Let $t_{\min} := \inf\mathcal{T}$ and $t_{\max} := \sup\mathcal{T}$. If $-\infty < s \vee t_{\min} \leq t \wedge t_{\max} < \infty$, then by Lemma~\ref{lem:pk-limit},
\begin{equation}
\label{eq:p0star-int-ineq}
p_0(z) \int_s^t p_0^* = \frac{p_0^*(s_0)}{p_0(s_0)} \cdot p_0(z) \lim_{k \to \infty} \int_s^t p_k \geq \lim_{k \to \infty} \frac{p_0^*(s_0)}{p_0(s_0)} \cdot p_k(z) \int_s^t p_0 = p_0^*(z) \int_s^t p_0
\end{equation}
for $z \in \{s,t\}$. Otherwise, it follows from~\eqref{eq:p0star-int-ineq} that
\begin{align*}
\displaystyle p_0(s) \int_s^\infty p_0^* = p_0(s) \cdot \sup_{t' \in \mathcal{T}} \int_s^{t'} p_0^* \geq p_0^*(s) \cdot \sup_{t' \in \mathcal{T}} \int_s^{t'} p_0 = p_0^*(s) \int_s^\infty p_0 \quad&\text{if }-\infty < s < t_{\max} = t = \infty \\
\displaystyle p_0(t) \int_{-\infty}^t p_0^* = p_0(t) \cdot \sup_{s' \in \mathcal{T}} \int_{s'}^t p_0^* \geq p_0^*(t) \cdot \sup_{s' \in \mathcal{T}} \int_{s'}^t p_0 = p_0^*(t) \int_{-\infty}^t p_0 \quad&\text{if }-\infty = s = t_{\min} < t < \infty,
\end{align*}
which completes the proof of~\eqref{eq:p0-star-s-t}. In particular, for $z \in \mathcal{T}$, taking $(s,t) = (-\infty,z)$ and $(s,t) = (z,\infty)$ in~\eqref{eq:p0-star-s-t}, we obtain
\[
p_0(z) \int_z^\infty p_0^* \geq p_0^*(z) \int_z^\infty p_0 \qquad\text{and}\qquad p_0(z) \int_{-\infty}^z p_0^* \geq p_0^*(z) \int_{-\infty}^z p_0,
\]
which proves~\eqref{eq:p0-star-hazard}. Summing these inequalities yields $p_0(z) \geq p_0^*(z)$, as required.

\medskip
\noindent
\textit{(b)} Since $s_0 \in \mathcal{T}$ and $p_0^*(z) = p_0^*(s_0)\exp\bigl(\int_{s_0}^z \psi_0^*\bigr) \leq p_0^*(s_0)$ for $z \in \R$, we deduce that $\norm{p_0^*}_\infty = p_0^*(s_0) \leq p_0(s_0) \leq \norm{p_0}_\infty$.
% Actually $p_0(s_0) = \norm{p_0}_\infty$
% Alternative 1: For $z \in \R$, we deduce from Lemmas~\ref{lem:pk-induction}\textit{(a)} and~\ref{lem:pk-limit} that
% \[
% p_0^*(z) \leq \frac{p_0(s_0)}{p_0^*(s_0)} \cdot p_0^*(z) = \lim_{k \to \infty} p_k(z) \leq \norm{p_0}_\infty.
% \]
% Therefore, $\norm{p_0^*}_\infty \leq \norm{p_0}_\infty$.
% Alternative 2: $p_0^* \leq p_0$ on $\mathcal{T}$, and for $z \in \mathcal{T}^c$, there exists a unique $\ell \in [K]$ such that $z \in (s_\ell,t_\ell)$ and hence $p_0^*(z) \leq p_0^*(s_\ell) \vee p_0^*(t_\ell) \leq \norm{p_0}_\infty$, as required.
Finally, by Theorem~\ref{thm:antitonic-score-proj}\textit{(c)}, $i^*(p_0) = \bigl(\int_\R (\psi_0^*)^2\,p_0\big/V_{p_0}(\psi_0^*)\bigr)^{1/2} = -\int_\R p_0\,d\psi_0^*$; see the equality case of~\eqref{eq:KKT}. Similarly, $i(p_0^*) = -\int_\R p_0^*\,d\psi_0^*$, so because $p_0^* \leq p_0$ on the support $\mathcal{T}$ of the Lebesgue--Stieltjes measure induced by $\psi_0^*$, we have
\[
i^*(p_0) = -\int_\R p_0\,d\psi_0^* \geq -\int_\R p_0^*\,d\psi_0^* = i(p_0^*),
\]
as required.
\end{proof}

\subsection{Additional examples for Section~\ref{sec:antitonic-proj}}
\label{subsec:appendix-examples}

\begin{example}
\label{rem:ARE-infinity}
For densities $p_0$ that are uniformly continuous and locally absolutely continuous on $\R$, it is possible to have $i^*(p_0) < \infty = i(p_0)$. This is unfavourable from the perspective of statistical efficiency because $\mathrm{ARE}^*(p_0) = 0$ in this case. To construct an example of such a density $p_0$, let
\[
J_0(u) := 
\begin{cases}
u\biggl(2 + \sin\Bigl(\dfrac{\pi}{4u}\Bigr)\biggr)\;&\text{for } u \in (0,1/2] \\
0 &\text{for }u = 0
\end{cases}
\qquad\text{and}\qquad J_0(u) := J_0(1 - u)\;\;\text{for } u \in (1/2,1],
\]
so that $J_0$ is continuous on $[0,1]$ and infinitely differentiable on $(0,1)$, with $u \leq J_0(u) \leq 3u$ for all $u \in [0,1/2]$. The least concave majorant $\hat{J}_0$ of $J_0$ is given by $\hat{J}_0(u) = 3\min(u, 1 - u)$ for $u \in [0,1]$. In addition, by the Cauchy--Schwarz inequality,
\[
\biggl(\int_0^1(J_0')^2\biggr)^{1/2} \geq \int_0^1 |J_0'| = 2\int_0^{1/2}\,\Bigl|2 + \sin\Bigl(\frac{u}{4\pi}\Bigr) - \frac{\pi}{4u}\cos\Bigl(\frac{u}{4\pi}\Bigr)\Bigr| = \infty,
\]
so $J_0$ is not of bounded variation on $(0,1)$. Next, defining $Q_0 \colon (0,1) \to \R$ by $Q_0(u) := \int_{1/2}^u 1/J_0$, we have $\lim_{u \nearrow 1}Q_0(u) = -\lim_{u \searrow 0}Q_0(u) = \int_0^{1/2} 1/J_0 \geq \int_0^{1/2} 1/(3u)\,du = \infty$. Then by Lemma~\ref{lem:density-quantile-reverse}, $Q_0$ is a strictly increasing, continuously differentiable bijection from $(0,1)$ to $\R$, and $p_0 := (Q_0^{-1})'$ is a strictly positive, uniformly continuous density on $\R$ with density quantile function $J_0$. Moreover, by Lemma~\ref{lem:psi0-star}, $p_0' = (J_0' \circ F_0)\,p_0$ is continuous, so $p_0$ is locally absolutely continuous on $\R$. By Remark~\ref{rem:fisher-J}, 
\[
i^*(p_0) = \int_0^1 \bigl(\hat{J}_0^{(\mathrm{R})}\bigr)^2 = \int_0^1 3^2 = 9 < \infty = \int_0^1 (J_0')^2 = i(p_0),
\]
so indeed $\mathrm{ARE}^*(p_0) = 0$. By modifying this construction slightly, one can also exhibit a unimodal density~$p_0$ with $\mathrm{ARE}^*(p_0) = 0$.
\end{example}

\begin{example}
\label{ex:t2} 
For $z \in \R$, let
\[
p_0(z) := \frac{1}{2}(1 + z^2)^{-3/2},
\]
so that $p_0$ is the density of $W/\sqrt{2}$ when $W \sim t_2$. Then by~\citet[Example~2.9]{dumbgen2011approximation}, the log-concave maximum likelihood projection $p_0^{\mathrm{ML}}$ is a standard Laplace density given by
\[
p_0^{\mathrm{ML}}(z) = \frac{1}{2}e^{-|z|}
\]
for $z \in \R$, with score function $\psi_0^{\mathrm{ML}}(\cdot) = -\sgn(\cdot)$. Therefore, the corresponding regression $M$-estimator $\hat{\beta}_{\psi_0^{\mathrm{ML}}} \in \argmax_{\beta \in \R^d} \sum_{i=1}^n \log p_0^{\mathrm{ML}}(Y_i - X_i^\top\beta) = \argmin_{\beta \in \R^d} \sum_{i=1}^n |Y_i - X_i^\top\beta|$ is the LAD estimator. On the other hand, by straightforward computations,
% \begin{align*}
% F_0(z) &= \frac{1}{2}\Bigl(1 + \frac{z}{\sqrt{1 + z^2}}\Bigr) \;\;\text{for }z \leq 0, \qquad F_0^{-1}(u) = -\frac{\sqrt{1 - 2u}}{2\sqrt{u(1 - u)}} \;\;\text{for }u \in (0,1/2), \\
% J_0(u) &= 4u^{3/2}(1 - u)^{3/2} \;\;\text{for }u \in [0,1],
% \end{align*}
% so that $\hat{J}_0$ is linear on $[0,1/4]$ and on $[3/4,1]$, with $\hat{J}_0(u) = J_0(u)$ for $u \in [1/4,3/4] \cup \{0,1\}$. Since $F_0^{-1}(3/4) = 3^{-1/2} = -F_0^{-1}(1/4)$, it follows that
\begin{align*}
\psi_0(z) &= (\log p_0)'(z) = -\frac{3z}{1 + z^2}, \\
\psi_0^*(z) &= (\hat{J}_0^{(\mathrm{R})} \circ F_0)(z) = \psi_0\bigl((z \wedge 3^{-1/2}) \vee (-3^{-1/2})\bigr) = \Bigl(-\frac{3\sqrt{3}}{4}\Bigr) \vee \psi_0(z) \wedge \frac{3\sqrt{3}}{4} 
\end{align*}
for $z \in \R$, so similarly to Example~\ref{ex:cauchy}, the optimal convex loss is a Huber-like function given by
\[
\ell_0^*(z) = -\log p_0^*(z) = -\int_0^z \psi_0^* + \log 2 = 
\begin{cases}
\,\dfrac{3\log\bigl((1 + z^2)/2\bigr)}{2} &\text{for }z \in [-3^{-1/2},3^{-1/2}] \\[6pt]
\,\dfrac{3\bigl(\sqrt{3}|z| - 1 + \log(4/9)\bigr)}{4} &\text{for }z \in \R \setminus [-3^{-1/2},3^{-1/2}].
\end{cases}
\]
This means that $p_0^* \neq p_0^{\mathrm{ML}}$, and $\hat{\beta}_{\psi_0^*} \in \argmin_{\beta \in \R^d} \sum_{i=1}^n \ell_0^*(Y_i - X_i^\top\beta)$ is usually different from $\hat{\beta}_{\psi_0^{\mathrm{ML}}}$, with the ratio of their asymptotic covariances being equal to
\[
\frac{V_{p_0}(\psi_0^*)}{V_{p_0}(\psi_0^{\mathrm{ML}})} = \frac{1/i^*(p_0)}{1/\bigl(4p_0(0)^2\bigr)} = \frac{1}{i^*(p_0)} = -\biggl(\int_{-1/\sqrt{3}}^{1/\sqrt{3}}\;p_0\,d\psi_0\biggr)^{-1} = \frac{80}{93} \approx 0.860.
\]
% Again, $i^*(p_0) = i(p_0^*)$ because $p_0 = p_0^*$ on $[-3^{-1/2},3^{-1/2}]$.
\end{example}

\begin{example}
\label{ex:pareto}
For $\alpha,\sigma > 0$, consider the symmetrised Pareto density $p_0 \colon \R \to \R$ given by
\[
p_0(z) := \frac{\alpha\sigma^\alpha}{2(|z| + \sigma)^{\alpha + 1}}.
\]
The corresponding distribution has a finite mean if and only if $\alpha > 1$, in which case its log-concave maximum likelihood projection~\citep{dumbgen2011approximation} is a Laplace density given by
\begin{equation}
\label{eq:pareto-sym}
p_0^{\mathrm{ML}}(z) := \frac{\alpha - 1}{2\sigma}\exp\biggl(-\frac{(\alpha - 1)
|z|}{\sigma}\biggr)
\end{equation}
for $z \in \R$; see~\citet[p.~1382]{chen2013smoothed} or~\citet[Exercise~9.14]{samworth24modern}. On the other hand, for general $\alpha,\sigma > 0$,
routine calculations yield
\[
% F_0(z) = \frac{1}{2} + \frac{\sgn(z)}{2}\biggl(1 - \Bigl(\frac{\sigma}{|z| + \sigma}\Bigr)^{\alpha}\biggr)
% F_0^{-1}(u) = -\bigl((2u)^{-1/\alpha} - 1\bigr)\sigma for u \leq 1/2
% F_0^{-1}(u) = \bigl((2 - 2u)^{-1/\alpha} - 1\bigr)\sigma for u \geq 1/2
J_0(u) = \frac{\alpha \bigl(2\min(u, 1 - u)\bigr)^{1 + 1/\alpha}}{2\sigma} \quad\text{and}\quad \hat{J}_0(u) = \frac{\alpha\min(u, 1 - u)}{\sigma}
\]
for $u \in [0,1]$. Since $p_0$ is symmetric about 0, it follows that $\psi_0^*(z) = (1 - 2\Ind_{\{z \geq 0\}})\alpha/\sigma$ and hence
\[
\ell_0^*(z) = -\log p_0^*(z) = \log\Bigl(\frac{\alpha}{2\sigma}\Bigr) - \frac{\alpha|z|}{\sigma}, \qquad p_0^*(z) = \frac{\alpha}{2\sigma}\exp\Bigl(-\frac{\alpha
|z|}{\sigma}\Bigr)
\]
for $z \in \R$. Therefore, $p_0^*$ is also a Laplace density, but while $p_0^{\mathrm{ML}} \neq p_0^*$ when $\alpha > 1$, both corresponding regression $M$-estimators $\hat{\beta}_{\psi_0^{\mathrm{ML}}}$ and $\hat{\beta}_{\psi_0^*}$ in~\eqref{eq:betahat-ML-fisher} minimise $\beta \mapsto \sum_{i=1}^n |Y_i - X_i^\top\beta|$ over $\R^d$, and hence coincide when there is a unique minimiser (i.e.~least absolute deviation estimator).
\end{example}

% NB: $p_0$ can be unimodal but not log-concave (when $|\mu_1 - \mu_2| > 2\sigma$ and $q$ is very close to 0 or 1, e.g.~$q = 10^{-3}$, $mu_1 = 0$ and $\mu_2 = 2.1$)
% When the variances of the mixture components are unequal, there can be two disjoint intervals on which $p_0$ is log-convex, e.g.~$0.5 N(0,1) + 0.5 N(2, 0.1^2)$

\begin{example}
\label{ex:laplace-mixture}
For $\rho \in (0,1)$ and $\mu > 0$, let $p_0$ be the two-component Laplace mixture density given by
\[
p_0(z) = \frac{1 - \rho}{2}e^{-|z + \mu|} + \frac{\rho}{2}e^{-|z - \mu|}
\]
for $z \in \R$. The corresponding score function $\psi_0 = (\log p_0)' = p_0'/p_0$ satisfies
\[
\psi_0(z) =
\begin{cases}
1 \;\;&\text{for }z < -\mu \\
\dfrac{\rho e^z - (1 - \rho)e^{-z}}{\rho e^z + (1 - \rho)e^{-z}} \;\;&\text{for }z \in (-\mu,\mu) \\
-1 \;\;&\text{for }z > \mu, 
\end{cases}
\]
so $\psi_0$ is strictly increasing on $(-\mu,\mu)$ and constant on both $(-\infty,\mu)$ and $(\mu,\infty)$. Since $J_0' = \psi_0 \circ F_0^{-1}$ on $(0,1) \setminus \{F_0(-\mu),F_0(\mu)\}$, it follows that $J_0$ is convex on $[F_0(-\mu),F_0(\mu)]$ while being linear on both $[0,F_0(-\mu)]$ and $[F_0(\mu),1]$. Therefore, $\hat{J}_0 = J_0$ on $[0,F_0(-\mu)] \cup [F_0(\mu),1]$ and $\hat{J}_0$ is linear on $[F_0(-\mu),F_0(\mu)]$ with $J_0\bigl(F_0(\pm\mu)\bigr) = p_0(\pm\mu)$, so
\[
\psi_0^*(z) = \hat{J}_0^{(\mathrm{R})}\bigl(F_0(z)\bigr) =
\begin{cases}
1 \;\;&\text{for }z < -\mu \\
\dfrac{p_0(\mu) - p_0(-\mu)}{F_0(\mu) - F_0(-\mu)} = 2\rho - 1 \;\;&\text{for }z \in [-\mu,\mu) \\
-1 \;\;&\text{for }z \geq \mu,
\end{cases}
\]
and hence
\[
p_0^*(z) =
\begin{cases}
Ce^{z + \mu} \;\;&\text{for }z \leq -\mu \\
Ce^{(2\rho - 1)(z + \mu)} \;\;&\text{for }z \in [-\mu,\mu] \\
Ce^{(2\rho - 1)2\mu} \cdot e^{-(z - \mu)} \;\;&\text{for }z \geq \mu,
\end{cases}
\qquad \text{where }C := \frac{1}{1 + e^{(2\rho - 1)2\mu} + \int_0^{2\mu} e^{(2\rho - 1)z}\,dz}.
\]
By direct calculation, $p_0^*(-\mu) = C < p_0(-\mu)$ and $p_0^*(\mu) = Ce^{(2\rho - 1)2\mu} < p_0(\mu)$, so $\norm{p_0^*}_\infty = p_0^*(-\mu) \vee p_0^*(\mu) < p_0(-\mu) \vee p_0(\mu) = \norm{p_0}_\infty$ and
\[
i^*(p_0) =  -\int_\R p_0\,d\psi_0^* = 2(1 - \rho) \cdot p_0(-\mu) + 2(1 + \rho) \cdot p_0(\mu) > 2(1 - \rho) \cdot p_0^*(-\mu) + 2(1 + \rho) \cdot p_0^*(\mu) = i(p_0^*).
\]
\end{example}

\begin{proposition}
\label{prop:gaussian-mixture}
For $\rho \in (0,1)$, $\mu_1,\mu_2 \in \R$ and $\sigma > 0$, let $P_0 := (1 - \rho)N(\mu_1,\sigma^2) + \rho N(\mu_2,\sigma^2)$ and denote by $p_0^{\mathrm{ML}}$ and $p_0^*$ its log-concave maximum likelihood and Fisher divergence projections respectively.
\begin{enumerate}[label=(\alph*)]
\item If $|\mu_1 - \mu_2| \leq 2\sigma$, then $P_0$ has density $p_0^{\mathrm{ML}} = p_0^*$.
\item If $|\mu_1 - \mu_2| > 2\sigma$, then $p_0^{\mathrm{ML}} \neq p_0^*$. Moreover, $\psi_0^*$ is continuous on $\R$, and there exist $z_1 \in (-\infty,a]$ and $z_2 \in [b,\infty)$ such that $\psi_0^*$ is decreasing on $(-\infty,z_1] \cup [z_2,\infty)$ and constant on $[z_1,z_2]$.
\end{enumerate}
\end{proposition}

\begin{proof}
By~\citet{cule2010maximum} and~\citet[Exercise~9.11]{samworth24modern}, $p_0$ is log-concave if and only if $|\mu_1 - \mu_2| \leq 2\sigma$. We will slightly refine this result by first noting that
\begin{align*}
p_0(z) &= \frac{1 - \rho}{\sqrt{2\pi}\sigma}\exp\biggl(-\frac{(z - \mu_1)^2}{2\sigma^2}\biggr) + \frac{\rho}{\sqrt{2\pi}\sigma}\exp\biggl(-\frac{(z - \mu_2)^2}{2\sigma^2}\biggr) \\
\psi_0(z) &= (\log p_0)'(z) = \frac{1}{p_0(z)}\biggl\{\frac{(1 - \rho)(\mu_1 - z)}{\sqrt{2\pi}\sigma^3}\exp\biggl(-\frac{(z - \mu_1)^2}{2\sigma^2}\biggr) + \frac{\rho(\mu_2 - z)}{\sqrt{2\pi}\sigma^3}\exp\biggl(-\frac{(z - \mu_2)^2}{2\sigma^2}\biggr)\biggr\} \\
\psi_0'(z) &= (\log p_0)''(z) = -\frac{g(z)}{2\pi\sigma^4 p_0(z)^2}\exp\biggl(-\frac{(z -\mu_1)^2 + (z - \mu_2)^2}{2\sigma^2}\biggr) 
\end{align*}
for $z \in \R$, where
\begin{align*}
g(z) := (1 - \rho)^2 \exp\Bigl(-\frac{\mu_2 -\mu_1}{2\sigma^2}(2z - \mu_1 - \mu_2)\Bigr) &+ \rho^2 \exp\Bigl(\frac{\mu_2 -\mu_1}{2\sigma^2}(2z - \mu_1 - \mu_2)\Bigr) \\
&+ \rho(1 - \rho)\Bigl\{2 - \Bigl(\frac{\mu_1 - \mu_2}{\sigma}\Bigr)^2\Bigr\}
% \\ &= \biggl\{(1 - q)\exp\biggl(-\frac{(\mu_2 -\mu_1)(2z - \mu_1 - \mu_2)}{4\sigma^2}\biggr) - q\exp\biggl(\frac{(\mu_2 -\mu_1)(2z - \mu_1 - \mu_2)}{4\sigma^2}\biggr)\biggr\}^2 + q(1 - q)\Bigl\{4 - \Bigl(\frac{\mu_1 - \mu_2}{\sigma}\Bigr)^2\Bigr\}
\end{align*}
is a convex function of $z$ with
\[
\min_{z \in \R}g(z) = \rho(1 - \rho)\Bigl\{4 - \Bigl(\frac{\mu_1 - \mu_2}{\sigma}\Bigr)^2\Bigr\}.
\]
\textit{(a)} If $|\mu_1 - \mu_2| \leq 2\sigma$, then $g \geq 0$ on $\R$ and hence $P_0$ has a log-concave density $p_0 = p_0^{\mathrm{ML}} = p_0^*$.

\medskip
\noindent
\textit{(b)} Otherwise, if $|\mu_1 - \mu_2| > 2\sigma$, then there exist $a < b$ such that $g \geq 0$ and hence $\psi_0' \leq 0$ on $(-\infty,a] \cup [b,\infty)$, while $g \leq 0$ and $\psi_0' \geq 0$ on $[a,b]$. Therefore, by~\citet[Examples~2.11(ii) and~2.12]{dumbgen2011approximation}, there exist $a' \in (-\infty,a]$ and $b' \in [b,\infty)$ such that $\phi := \log p_0^{\mathrm{ML}}$ agrees with $\phi_0 := \log p_0$ on $(-\infty,a'] \cup [b',\infty)$, and $\phi$ is affine on $[a',b']$. 

Suppose for a contradiction that $p_0^{\mathrm{ML}}$ is differentiable on $\R$. Then $\psi_0(a') = \phi_0'(a') = \phi'(a') = \phi'(b') = \psi_0(b')$, so defining $\ell \colon \R \to \R$ by $\ell(z) := \phi_0(a') + \psi_0(a')(z - a')$, we have $\phi = \ell$ on $[a',b']$ and $\ell(z) = \phi_0(b') + \psi_0(b')(z - b')$ for $z \in \R$. Since $\phi_0$ is concave on both $(-\infty,a]$ and $[b,\infty)$ while being convex on $[a,b]$, it follows that $\phi_0 \leq \ell$ on $(-\infty,a] \cup [b,\infty)$ and
\[
\phi_0(z) \leq \frac{b - z}{b - a}\,\phi_0(a) + \frac{z - a}{b - a}\,\phi_0(b) \leq \frac{b - z}{b - a}\,\ell(a) + \frac{z - a}{b - a}\,\ell(b) = \ell(z) = \phi(z)
\]
for all $z \in [a,b]$. Therefore, $\log p_0 = \phi_0 \leq \phi = \log p_0^{\mathrm{ML}}$ on $\R$, but since $\int_\R p_0 = \int_\R p_0^{\mathrm{ML}} = 1$, we must have $p_0 = p_0^{\mathrm{ML}}$ because both functions are continuous. However, this contradicts the fact that $p_0$ is not log-concave, so $p_0^{\mathrm{ML}}$ is not differentiable (at either $a'$ or $b'$).

On the other hand, $p_0$ and its corresponding quantile function $F_0^{-1}$ are differentiable on $\R$, so $J_0 = p_0 \circ F_0^{-1}$ is differentiable on $(0,1)$. By Lemma~\ref{lem:lcm-deriv}, $\hat{J}_0$ is continuously differentiable on $(0,1)$, so $(\log p_0^*)^{(\mathrm{R})} = \psi_0^* = \hat{J}_0^{(\mathrm{R})} \circ F_0$ is continuous on $\R$. Therefore, $\log p_0^*$ is differentiable on $\R$ and hence $p_0^{\mathrm{ML}} \neq p_0^*$.

Moreover, by the first line of the proof of \textit{(b)}, $J_0' = \psi_0 \circ F_0^{-1}$ is increasing on $[F_0(a),F_0(b)]$, and decreasing on both $(0,F_0(a)]$ and $[F_0(b),1)$. We deduce from Lemma~\ref{lem:lcm-concave-cvx} that there exist $z_1 \in (-\infty,a]$ and $z_2 \in [b,\infty)$ such that $\psi_0^*$ is decreasing on $(-\infty,z_1] \cup [z_2,\infty)$ and constant on $[z_1,z_2]$.
\end{proof}

\subsection{Background and proofs for Section~\ref{sec:regression}}
\label{subsec:regression-proofs}

Section~\ref{subsec:joint-opt} uses population-level results about the antitonic score projection from Section~\ref{sec:antitonic-proj} to motivate the alternating algorithms in Section~\ref{sec:regression}; this material is intended as background and is not required elsewhere. In the subsequent three subsections, we prove the antitonic efficiency results in Sections~\ref{subsec:linreg-sym} and~\ref{subsec:linreg-intercept}. Section~\ref{subsubsec:score-estimation} establishes key consistency results for our kernel-based score estimators, which are then used to prove Theorems~\ref{thm:linreg-score-sym} and~\ref{thm:linreg-score-intercept} in Sections~\ref{subsec:linreg-sym-proofs} and~\ref{subsec:linreg-intercept-proofs} respectively. Some techniques from the classical $M$-estimation literature are adapted to provide a common template for both semiparametric models, whose efficient scores and information matrices are derived in Section~\ref{subsec:linreg-semiparametric}. Despite the similarities among Lemmas~\ref{lem:linreg-sym-equicontinuity} to~\ref{lem:linreg-zeta-consistency} and their proofs, there are nevertheless some subtle differences.

\subsubsection{Joint optimisation for the score and regression coefficients}
\label{subsec:joint-opt}

For $\beta \in \R^d$, define $q_\beta \colon \R \to \R$ by
\[
q_\beta(z) := \E p_0\bigl(z - X_1^\top(\beta_0 - \beta)\bigr),
\]
so that $q_\beta$ is the density of $Y_1 - X_1^\top\beta = \varepsilon_1 + X_1^\top(\beta_0 - \beta)$. Our approach to estimating $\beta_0$ with data-driven convex loss functions is motivated by the following population-level joint optimisation problem. We seek to minimise the augmented score matching objective
\begin{align}
\label{eq:Q-beta-psi}
Q(\beta,\psi) := D_{q_\beta}(\psi) = \E\biggl(\int_\R \psi^2\bigl(z + X_1^\top(\beta_0 - \beta)\bigr)\,p_0(z)\,dz + 2\int_\R p_0\bigl(z - X_1^\top(\beta_0 - \beta)\bigr)\,d\psi(z)\biggr)
\end{align}
jointly over $\beta \in \R^d$ and $\psi \in \Psi_\downarrow(q_\beta)$ satisfying 
\begin{equation}
\label{eq:constraint-Z-est}
\E\bigl(X_1 \psi(Y_1 - X_1^\top \beta)\bigr) = 0.
\end{equation}
Recall from~\eqref{eq:Dp0-deriv} that if $\psi \in \Psi_\downarrow(q_\beta)$ is locally absolutely continuous on $\R$, then 
\[
Q(\beta,\psi) = \E\psi^2(Y_1 - X_1^\top\beta) + 2\E\psi'(Y_1 - X_1^\top\beta).
\]
The constraint~\eqref{eq:constraint-Z-est} is the population analogue of the estimating equations~\eqref{eq:linreg-Z-est} based on $\psi$. It forces $\beta$ to be a minimiser of the convex population risk function $b \mapsto \E\ell(Y_1 - X_1^\top b) =: L_\psi(b)$ over $\mathcal{D}_\psi$,\footnote{\label{footnote:population-risk}Indeed, $\ell$ is convex with $\psi = -\ell^{(\mathrm{R})}$, so for $\beta,\beta' \in \R^d$, we have
\[
\ell(Y_1 - X_1^\top\beta') \geq \ell(Y_1 - X_1^\top\beta) + (\beta - \beta')^\top X_1\psi(Y_1 - X_1^\top\beta).
\]
If $\beta$ satisfies~\eqref{eq:constraint-Z-est}, then taking expectations in the display above shows that $L_\psi(\beta') \geq L_\psi(\beta)$ for all $\beta' \in \mathcal{D}_\psi$.} where $\ell$ is any negative antiderivative of~$\psi$, and where $\mathcal{D}_\psi$ is the convex set of $b \in \R^d$ such that $L_\psi(b)$ is well-defined in $[-\infty,\infty]$.

The following proposition characterises the global joint minimiser of our constrained score matching optimisation problem. For $j \in [d]$, write $\mathsf{e}_j$ for the $j$th standard basis vector in $\R^d$.
\begin{proposition}
\label{prop:joint-min}
For an absolutely continuous density $p_0$ with $\mathcal{S}(p_0) = \R$, let $\psi^*_0$ be the projected score function defined in~\eqref{eq:psi0-star}.  Assume that $i^*(p_0) = \int_\R (\psi_0^*)^2\,p_0 < \infty$ and that $\E(X_1 X_1^\top) \in \R^{d \times d}$ is positive definite. Let
\[
\Gamma := \bigl\{(\beta,\psi) : \beta \in \R^d,\,\psi \in \Psi_\downarrow(q_\beta),\,\E\bigl(X_1\psi(Y_1 - X_1^\top \beta)\bigr) = 0\bigr\}.
\]
\begin{enumerate}[label=(\alph*)]
\item Suppose that $p_0$ is symmetric. Denote by $\Psi_\downarrow^{\mathrm{anti}}(q_\beta)$ the set of all $\psi \in \Psi_\downarrow(q_\beta)$ such that $\psi(-z) = -\lim_{z' \nearrow z}\psi(z')$, and let $\Gamma^{\mathrm{anti}} := \{(\beta,\psi) \in \Gamma : \psi \in \Psi_\downarrow^{\mathrm{anti}}(q_\beta)\}$. Then
\begin{align*}
(\beta_0,\psi_0^*) = \argmin_{(\beta,\psi) \in \Gamma^{\mathrm{anti}}} Q(\beta,\psi).
\end{align*}
\item If instead $X_{1d} = 1$ almost surely, then 
\begin{align*}
\bigl(\beta_0 + c \mathsf{e}_d, \psi_0^*(\cdot + c)\bigr) \in \argmin_{(\beta,\psi) \in \Gamma}\,Q(\beta,\psi)
\end{align*}
for all $c \in \R$. Moreover, all minimisers are of this form.
\end{enumerate}
\end{proposition}
Thus, when we restrict attention to right-continuous versions of antisymmetric functions $\psi$ in the setting of Proposition~\ref{prop:joint-min}\emph{(a)}, the unique minimiser of our population-level optimisation problem is given by the pair consisting of the vector $\beta_0$ of true regression coefficients, and the antitonic score projection~$\psi_0^*$ of the error density $p_0$.  This motivates our statistical methodology below, where we seek to minimise sample versions of this population-level objective. The proof below shows that for any $(\beta,\psi) \in \Gamma^{\mathrm{anti}}$, we also have $(\beta_0,\psi) \in \Gamma^{\mathrm{anti}}$; moreover, $Q(\beta_0,\psi_0^*) = \inf_{\psi \in \Psi_{\downarrow}^{\mathrm{anti}}(q_{\beta_0})} Q(\beta_0, \psi) = -1/i^*(p_0) < 0$.  On the other hand, if $\beta \neq \beta_0$, and if $\psi \in \Psi_{\downarrow}^{\mathrm{anti}}(q_\beta)$ is such that $(\beta,\psi) \in \Gamma^{\mathrm{anti}}$, then $Q(\beta,\psi) \geq 0$.  In other words, feasible pairs $(\beta,\psi)$ with $\beta \neq \beta_0$ are well-separated from the optimal solution in terms of their objective function values.

In Proposition~\ref{prop:joint-min}\emph{(b)}, where the errors need not have mean zero and where we no longer restrict attention to antisymmetric~$\psi$, the pair $(\beta_0,\psi_0^*)$ still minimises our objective up to appropriate translations to account for the lack of identifiability of the intercept term. Indeed, the joint distribution of $(X_1,Y_1)$ does not change if we replace $\beta_0$ and $\varepsilon_1$ with $\beta_0 + c \mathsf{e}_d$ and $\varepsilon_1 - c$ respectively, for any $c \in \R$. This explains why the minimiser in \emph{(b)} is not unique; observe that $\psi_0^*(\cdot + c)$ is the projected score corresponding to the density $p_0(\cdot + c)$ of $\varepsilon_1 - c$. Nevertheless, \emph{(b)} indicates that translations of the intercept term are the \textit{only} source of non-uniqueness, so that the first $d-1$ components of $\beta_0$ can still be recovered by solving the constrained optimisation problem above on the population level. This provides some insight into the reason that~\eqref{eq:thetahat} in Section~\ref{subsec:linreg-intercept} holds without any centring assumption on the errors. The constraint $\E\zeta(\varepsilon_1) = 0$ restores the identifiability of the intercept $\mu_0$.

\begin{proof}
In both \textit{(a)} and \textit{(b)}, $X_1$ and $\varepsilon_1$ are independent, so by the final assertion of Lemma~\ref{lem:psi0-star}, we have $\E\bigl\{X_1\psi_0^*(Y_1 - X_1^\top\beta_0)\bigr\} = \E(X_1)\,\E\psi_0^*(\varepsilon_1) = 0$. Therefore, $(\beta_0,\psi_0^*)$ satisfies the constraint~\eqref{eq:constraint-Z-est}. Moreover, $Y_1 - X_1^\top\beta_0 = \varepsilon_1$ has density $q_{\beta_0} = p_0$, so by Theorem~\ref{thm:antitonic-score-proj}(\emph{b},\,\emph{c}),
\[
Q(\beta_0,\psi_0^*) = D_{p_0}(\psi_0^*) = -i^*(p_0) = \inf_{\psi \in \Psi_\downarrow(p_0)} D_{p_0}(\psi) = \inf_{\psi \in \Psi_\downarrow(p_0)} Q(\beta_0,\psi).
\]
\noindent
\emph{(a)} Since $p_0$ is symmetric and has a continuous distribution function $F_0$, we have $F_0(z) = 1 - F_0(-z)$ for every $z \in \R$, so $J_0(u) := (p_0 \circ F_0^{-1})(u) = p_0\bigl(-F_0^{-1}(1 - u)\bigr) = J_0(1 - u)$ and hence $\hat{J}_0(u) = \hat{J}_0(1 - u)$ for all $u \in [0,1]$. For every $z \in \R$, we have $F_0(z) \in (0,1)$ because $\mathcal{S}(p_0) = \R$, so
\begin{equation}
\label{eq:psi0star-antisym}
\psi_0^*(-z) = \hat{J}_0^{(\mathrm{R})}\bigl(F_0(-z)\bigr) = \hat{J}_0^{(\mathrm{R})}\bigl(1 - F_0(z)\bigr) = -\hat{J}_0^{(\mathrm{L})}\bigl(F_0(z)\bigr) = -\lim_{u \nearrow F_0(z)}\hat{J}_0^{(\mathrm{R})}(u) = -\lim_{z' \nearrow z}\psi_0^*(z'),
\end{equation}
where the penultimate equality follows from~\citet[Theorem~24.1]{rockafellar97convex}. Thus, $\psi_0^* \in \Psi_\downarrow^{\mathrm{anti}}(p_0) = \Psi_\downarrow^{\mathrm{anti}}(q_{\beta_0})$.

Next, we claim that if $(\beta,\psi) \in \Gamma^{\mathrm{anti}}$ is such that $\beta \neq \beta_0$, then necessarily $Q(\beta,\psi) \geq 0 > Q(\beta_0,\psi_0^*)$, which proves \emph{(a)}. Indeed, $\E(X_1 X_1^\top)$ is positive definite by assumption, so $T := X_1^\top (\beta - \beta_0)$ satisfies $\E(T^2) = (\beta - \beta_0)^\top \E(X_1 X_1^\top) (\beta - \beta_0) > 0$ and hence $\Pr(T \neq 0) > 0$. By~\eqref{eq:constraint-Z-est}, $\E\bigl(T\psi(\varepsilon_1 - T)\bigr) = 0$, so because $T$ and $\varepsilon_1$ are independent, there exists $t \neq 0$ such that $\E|\psi(\varepsilon_1 - t)| < \infty$. Since $\varepsilon_1 \eqd -\varepsilon_1$ and $\psi$ agrees Lebesgue almost everywhere on $\R$ with an antisymmetric function, we have $\E\psi(\varepsilon_1 + t) = \E\psi(-\varepsilon_1 + t) = -\E\psi(\varepsilon_1 - t) \in \R$. Finally, since $|\psi(\varepsilon_1)| \leq |\psi(\varepsilon_1 - t)| \vee |\psi(\varepsilon_1 + t)|$, this shows that $\E|\psi(\varepsilon_1)| < \infty$. Thus, $\E\psi(\varepsilon_1) = \E\psi(-\varepsilon_1) = -\E\psi(\varepsilon_1)$, so again because $T$ and $\varepsilon_1$ are independent,
\begin{equation}
\label{eq:non-identifiable}
\E\bigl(T\psi(\varepsilon_1)\bigr) = \E(T)\,\E\psi(\varepsilon_1) = 0 = \E\bigl(T\psi(\varepsilon_1 - T)\bigr).
\end{equation}
For $z \in \R$, we have
\[
T^2\int_0^1 p_0(z + sT)\,ds = T\int_z^{z + T} p_0(w)\,dw = \int_\R Tp_0(w)\bigl(\Ind_{\{w - T < z \leq w\}} - \Ind_{\{w < z \leq w - T\}}\bigr)\,dw.
\]
Hence, by Fubini's theorem and the independence of $T$ and $\varepsilon_1$, we have
\begin{align*}
\E\int_0^1 \int_\R T^2 p_0(z + sT)\,d\psi(z)\,ds &= \E\int_\R p_0(w)\int_\R T\bigl(\Ind_{\{w - T < z \leq w\}} - \Ind_{\{w < z \leq w - T\}}\bigr)\,d\psi(z)\,dw \\
&= \E \int_\R T\bigl(\psi(w) - \psi(w - T)\bigr)\,p_0(w)\,dw \\
&= \E\bigl(T\psi(\varepsilon_1)\bigr) - \E\bigl(T\psi(\varepsilon_1 - T)\bigr) = 0.
\end{align*}
Defining $g \colon [0,1] \to \R$ by $g(s) := -\E \int_\R T^2 p_0(z + sT)\,d\psi(z) \geq 0$, we therefore have $\int_0^1 g = 0$, so $g = 0$ Lebesgue almost everywhere. Now $p_0$ is continuous on $\R$, so by Fatou's lemma,
\begin{align}
\label{eq:joint-min-deriv-1}
0 &\leq -\E \int_\R T^2 p_0(z + T)\,d\psi(z) = g(1) \leq \liminf_{s \nearrow 1} g(s) = 0, \\
\label{eq:joint-min-deriv-2}
0 &\leq -\E(T^2) \int_\R p_0(z)\,d\psi(z) = -\E \int_\R T^2 p_0(z)\,d\psi(z) = g(0) \leq \liminf_{s \searrow 0} g(s) = 0.
\end{align}
By~\eqref{eq:joint-min-deriv-1}, $\int_\R p_0(z + T)\Ind_{\{T \neq 0\}}\,d\psi(z) = 0$ almost surely. Moreover, since $\E(T^2) > 0$,~\eqref{eq:joint-min-deriv-2} implies that $\int_\R p_0(z + T)\Ind_{\{T = 0\}}\,d\psi(z) = \Ind_{\{T = 0\}}\int_\R p_0(z)\,d\psi(z) = 0$ almost surely, so
\begin{equation}
\label{eq:joint-min-deriv}
\E\int_\R p_0\bigl(z + X_1^\top(\beta - \beta_0)\bigr)\,d\psi(z) = \E\int_\R p_0(z + T)\,d\psi(z) = 0.
\end{equation}
Hence, by the definition of $Q(\beta,\psi)$ in~\eqref{eq:Q-beta-psi},
\[
Q(\beta,\psi) = \E\int_\R \psi^2\bigl(z - X_1^\top(\beta - \beta_0)\bigr)\,p_0(z)\,dz \geq 0 > Q(\beta_0,\psi_0^*),
\]
as claimed.

\medskip
\noindent
\emph{(b)} For $c \in \R$, $\beta \in \R^d$, $\psi \in \Psi_\downarrow(q_\beta)$ and $x = (\tilde{x},1) \in \R^{d-1} \times \R$, we have $x^\top (\beta_0 - \beta) = x^\top (\beta_0 - \beta - c\mathsf{e}_d) + c$, so $\psi_c(\cdot) := \psi(\cdot + c)$ satisfies
\begin{align*}
\int_\R \psi_c^2\bigl(z + x^\top(\beta_0 - \beta - c\mathsf{e}_d)\bigr)\,p_0(z)\,dz &= \int_\R \psi^2\bigl(z + x^\top(\beta_0 - \beta)\bigr)\,p_0(z)\,dz, \\
\text{and}\quad \int_\R p_0\bigl(z - x^\top(\beta_0 - \beta - c\mathsf{e}_d)\bigr)\,d\psi_c(z) &= \int_\R p_0\bigl(z - x^\top(\beta_0 - \beta - c\mathsf{e}_d) - c\bigr)\,d\psi(z) \\
&= \int_\R p_0\bigl(z - x^\top(\beta_0 - \beta)\bigr)\,d\psi(z).
\end{align*}
It follows that if $X_{1d} = 1$ almost surely, then $Q(\beta + c\mathsf{e}_d,\psi_c) = Q(\beta,\psi)$. 

For $\beta_c := \beta_0 + c\mathsf{e}_d$, we have $q_{\beta_c}(\cdot) = p_0(\cdot + c)$, and $\psi_c \in \Psi_\downarrow(q_{\beta_c})$ if and only if $\psi \in \Psi_\downarrow(q_{\beta_0}) = \Psi_\downarrow(p_0)$.  Thus,
\begin{equation}
\label{eq:beta-c}
0 > -i^*(p_0) = Q(\beta_0,\psi_0^*) = \inf_{\psi \in \Psi_\downarrow(p_0)} Q(\beta_0,\psi) = \inf_{\psi_c \in \Psi_\downarrow(q_{\beta_c})} Q(\beta_c,\psi_c)
\end{equation}
for all $c \in \R$.

Next, for any $(\beta,\psi) \in \Gamma$, we claim that there exists $c \in \R$ such that $\E\psi(\varepsilon_1 - c) = 0$. Indeed, $X_{1d} = 1$ almost surely, so $T = X_1^\top (\beta - \beta_0)$ satisfies $\E\psi(\varepsilon_1 - T) = \E\psi(Y_1 - X_1^\top\beta) = 0$. By the independence of $T$ and $\varepsilon$, it follows that $\varPsi(t) := \int_\R \psi(z - t)\,p_0(z)\,dz = \E\psi(\varepsilon_1 - t)$ is well-defined and finite for $P_T$-almost every $t \in \R$, where $P_T$ denotes the distribution of $T$. Since $\psi$ is decreasing, $\varPsi$ is increasing, so there exist $t_1 \leq t_2$ in the support of $P_T$ such that $-\infty < \varPsi(t_1) \leq 0 \leq \varPsi(t_2) < \infty$, and $|\psi(\varepsilon_1 - t)| \leq |\psi(\varepsilon_1 - t_1)| \vee |\psi(\varepsilon_1 - t_2)|$ for all $t \in [t_1,t_2]$. It follows from the dominated convergence theorem that $\varPsi(t)$ is continuous in $t \in [t_1,t_2]$, so by the intermediate value theorem, there exists $c \in [t_1,t_2]$ such that $\E\psi(\varepsilon_1 - c) = \varPsi(c) = 0$, as claimed.

By~\eqref{eq:constraint-Z-est} together with the independence of $\tilde{T} := T - c = X_1^\top (\beta - \beta_c)$ and $\varepsilon_1$, we deduce that  
\[
\E\bigl(\tilde{T}\psi(\varepsilon_1 - c)\bigr) = 0 = \E\bigl(\tilde{T}\psi(\varepsilon_1 - c - \tilde{T})\bigr).
\]
This is similar to~\eqref{eq:non-identifiable} in the proof of \emph{(a)}, but with $T$ and $\varepsilon_1$ replaced by $\tilde{T}$ and $\varepsilon_1 - c$ respectively. 

Suppose now that $\beta \neq \beta_{c'}$ for any $c' \in \R$. Since $\E(X_1 X_1^\top)$ is positive definite by assumption, $\E(\tilde{T}^2) = (\beta - \beta_c)^\top \E(X_1 X_1^\top) (\beta - \beta_c) > 0$. By arguing similarly to \emph{(a)} and noting that $p_c(\cdot) := p_0(\cdot + c)$ is the density of $\varepsilon_1 - c$, we conclude as in~\eqref{eq:joint-min-deriv} that
\[
\E\int_\R p_0\bigl(z + X_1^\top(\beta - \beta_0)\bigr)\,d\psi(z) = \E\int_\R p_0(z + T)\,d\psi(z) = \E\int_\R p_c(z + \tilde{T})\,d\psi(z) = 0.
\]
Thus, if $(\beta,\psi) \in \Gamma$ is such that $\beta \neq \beta_{c'}$ for any $c' \in \R$, then
\[
Q(\beta,\psi) = \E\int_\R \psi^2\bigl(z - X_1^\top(\beta - \beta_0)\bigr)\,p_0(z)\,dz \geq 0 > Q(\beta_0,\psi_0^*).
\]
Together with~\eqref{eq:beta-c}, this completes the proof.
\end{proof}

\subsubsection{Consistency of kernel-based score estimators}
\label{subsubsec:score-estimation}

Lemma~\ref{lem:kernel-score-consistency} below shows that the initial score estimates $\tilde{\psi}_{n,j}$ are consistent in $L^2(P_0)$. It extends the proof of~\citet[Lemma~25.64]{vdV1998asymptotic} to kernel density estimators $\tilde{p}_{n,j}$ based on out-of-sample residuals $\{\hat{\varepsilon}_i : i \in I_{j+1}\}$ instead of the unobserved regression errors~$\varepsilon_i$. The three folds of the data will be denoted by $\mathcal{D}_j \equiv \mathcal{D}_{n,j} := \{(X_i, Y_i) : i \in I_j\}$ for $j \in \{1,2,3\}$. 

\begin{lemma}
\label{lem:kernel-score-consistency}
Suppose that for $j \in \{1,2,3\}$, we have $\hat{\varepsilon}_i = \varepsilon_i + \mu + X_i^\top b_n$ for all $i \in I_{j+1}$, where $\mu \in \R$ is fixed and $b_n = O_p(n^{-1/2})$ as $n \to \infty$. If in addition~\emph{\ref{ass:fisher-finite}},~\emph{\ref{ass:kernel}} and~\emph{\ref{ass:alpha-gamma-h}} hold, then for each $j$, we have
\[
\int_\R \bigl(\tilde{\psi}_{n,j}(z) - \psi_0(z - \mu)\bigr)^2\,p_0(z - \mu)\,dz \cvp 0
\]
as $n \to \infty$.
\end{lemma}

\begin{proof}
It suffices to prove the result for $\tilde{\psi}_n \equiv \tilde{\psi}_{n,1}$ and $\mu = 0$, since $\tilde{\psi}_{n,2},\tilde{\psi}_{n,3}$ are obtained by permuting the three folds of the data (which have roughly the same size), and the density and score function of $\varepsilon_i' := \varepsilon_i + \mu$ are $p_0(\cdot - \mu)$ and $\psi_0(\cdot - \mu)$ respectively. By~\ref{ass:fisher-finite} and the Cauchy--Schwarz inequality, $\norm{p_0}_\infty \leq \int_\R |p_0'| \leq \bigl(\int_{\{p_0 > 0\}} (p_0')^2/p_0\bigr)^{1/2} = i(p_0)^{1/2} < \infty$. Since $K$ is twice continuously differentiable and supported on $[-1,1]$ by~\ref{ass:kernel}, we have $\int_\R K^2 \vee \int_\R (K')^2 < \infty$. Denote by $P_X$ the distribution of $X_1$ on $\R^d$. For $y \in \R$, $n \in \N$ and $h \equiv h_n$, let $\tilde{p}_n(y) \equiv \tilde{p}_{n,1}(y) = |I_2|^{-1}\sum_{i \in I_2}K_h(y - \hat{\varepsilon}_i)$ and
\begin{align*}
p_n(y) := \E\bigl(\tilde{p}_n(y) \!\bigm|\! \mathcal{D}_1\bigr) &= \int_{\R^d} \int_\R K_h(y - z - x^\top b_n)p_0(z)\,dz\,dP_X(x) \notag \\
&= \int_{\R^d} \int_\R p_0(y - uh - x^\top b_n)K(u)\,du\,dP_X(x).
\end{align*}
We have $nh_n^2 = (nh_n^3\gamma_n^2)(h_n\gamma_n^2)^{-1} \to \infty$ by~\ref{ass:alpha-gamma-h}, so $\norm{b_n} = o_p(h)$ as $n \to \infty$. Moreover, since $p_0$ is uniformly continuous on $\R$, it follows from the bounded convergence theorem that $p_n(y) \cvp p_0(y)$ for every $y \in \R$. By Fubini's theorem,
\begin{align}
\int_\R |p_n - p_0|
&= \int_\R\,\biggl|\int_{\R^d} \int_\R \bigl\{p_0(y - uh - x^\top b_n) - p_0(y)\bigr\}K(u)\,du\,dP_X(x)\biggr|\,dy \notag \\
&\leq \int_\R \int_{\R^d} \int_\R \int_\R |p_0'(z)|(\Ind_{\{y < z \leq y - uh - x^\top b_n\}} + \Ind_{\{y - uh - x^\top b_n < z \leq y\}})\,dz\,K(u)\,du\,dP_X(x)\,dy \notag \\
&= \biggl(\int_\R |p_0'(z)|\,dz\biggr) \int_{\R^d} \int_\R |uh + x^\top b_n|K(u)\,du\,dP_X(x) \notag \\
\label{eq:kernel-L1-bias}
&\leq \biggl(\int_\R |p_0'|\biggr)\biggl(h\int_\R |u|K(u)\,du + \norm{b_n} \cdot \E\norm{X_1}\biggr) = O_p(h).
\end{align}
By~\ref{ass:alpha-gamma-h} and the dominated convergence theorem, we may differentiate under the integral sign to obtain
\begin{align*}
p_n'(y) = \E\bigl(\tilde{p}_n'(y) \!\bigm|\! \mathcal{D}_1\bigr) &= \int_{\R^d} \int_\R K_h'(y - z - x^\top b_n)p_0(z)\,dz\,dP_X(x).
\end{align*}
For $y \in \R$, we have
\begin{align}
\label{eq:var-pnh}
\Var\bigl(\tilde{p}_n(y) \!\bigm|\! \mathcal{D}_1\bigr) &\leq \frac{1}{|I_2|}\int_{\R^d} \int_\R K_h(y - z - x^\top b_n)^2\,p_0(z)\,dz\,dP_X(x) \leq \frac{\norm{p_0}_\infty \int_\R K^2}{|I_2| h},\\
\label{eq:var-pnh'}
\Var\bigl(\tilde{p}_n'(y) \!\bigm|\! \mathcal{D}_1\bigr) &\leq \frac{1}{|I_2|}\int_{\R^d} \int_\R K_h'(y - z - x^\top b_n)^2\,p_0(z)\,dz\,dP_X(x) \leq \frac{\norm{p_0}_\infty \int_\R (K')^2}{|I_2| h^3}.
\end{align}
Since $|I_2| h_n^3 \to \infty$ by~\ref{ass:alpha-gamma-h}, it follows from Lemma~\ref{lem:cond-cvg} that $\tilde{p}_n(y) - p_n(y) \cvp 0$ and $\tilde{p}_n'(y) - p_n'(y) \cvp 0$. Defining $\tau(z) := \int_\R |p_0'(y - z) - p_0'(y)|\,dy \leq 2\int_\R |p_0'|$ for $z \in \R$, we have $\lim_{z \to 0}\tau(z) = 0$ by continuity of translation~\citep[Proposition~8.5]{folland1999real}. Thus, since $b_n \cvp 0$, it follows from the bounded convergence theorem that
\begin{align*}
\int_\R |p_n' - p_0'| &= \int_\R\,\biggl|\int_{\R^d} \int_\R \bigl\{p_0'(y - uh - x^\top b_n) - p_0'(y)\bigr\}K(u)\,du\,dP_X(x)\biggr|\,dy \\
&\leq \int_{\R^d}\int_\R \tau(uh + x^\top b_n)K(u)\,du\,dP_X(x) \cvp 0.
\end{align*}
This means that every subsequence of $(p_n)$ has a further subsequence $(p_{n_k})$ such that with probability~1, we have $p_{n_k}' \to p_0'$ and hence $(p_{n_k}')^2/p_{n_k} \to (p_0')^2/p_0$ Lebesgue almost everywhere. By the Cauchy--Schwarz inequality,
\begin{align*}
\frac{p_n'(y)^2}{p_n(y)} &= \frac{\bigl(\int_{\R^d} \int_\R p_0'(y - uh - x^\top b_n)K(u)\,du\,dP_X(x)\bigr)^2}{\int_{\R^d} \int_\R p_0(y - uh - x^\top b_n)K(u)\,du\,dP_X(x)} \\
&\leq \int_{\R^d} \int_\R \frac{(p_0')^2}{p_0}(y - uh - x^\top b_n)K(u)\,du\,dP_X(x)
\end{align*}
for $y \in \R$ and $n \in \N$, so by Fubini's theorem,
\begin{equation}
\int_\R \frac{(p_n')^2}{p_n} \leq \int_\R \frac{(p_0')^2}{p_0}
\end{equation}
for all $n$. Thus, by applying a slight generalisation of Scheff\'e's lemma~\citep[Lemma~2.29]{vdV1998asymptotic} to subsequences of $(p_n)$, we obtain
\begin{align}
\label{eq:pnh-inequality}
\int_\R \Bigl(\frac{p_n'}{\sqrt{p_n}} - \frac{p_0'}{\sqrt{p_0}}\Bigr)^2 \cvp 0
\end{align}
as $n \to \infty$. Since $\int_\R \psi_0^2\,p_0 = i(p_0) < \infty$ and $\alpha_n \to \infty$ while $\gamma_n \to 0$, we can argue similarly to the proof of~\citet[Lemma~25.64]{vdV1998asymptotic} and deduce by the dominated convergence theorem that
\[
\E\biggl(\int_{\tilde{S}_{n,1}^c} (\tilde{\psi}_n - \psi_0)^2\,p_0 \biggm| \mathcal{D}_1\biggr) \leq \int_\R \psi_0(y)^2\,p_0(y) \cdot \bigl\{\Pr\bigl(|\tilde{p}_n'(y)| > \alpha_n \!\bigm|\! \mathcal{D}_1\bigr) + \Pr\bigl(\tilde{p}_n(y) < \gamma_n \!\bigm|\! \mathcal{D}_1\bigr)\bigr\}\,dy \cvp 0.
\]
Moreover, on $\tilde{S}_{n,1}$, we have
\begin{align*}
|\tilde{\psi}_n - \psi_0|\sqrt{p_0} &\leq \frac{|\tilde{p}_n'|}{\tilde{p}_n}|\sqrt{p_0} - \sqrt{p_n}| + \frac{\sqrt{p_n}}{\tilde{p}_n}|\tilde{p}_n' - p_n'| + \frac{|p_n'|}{\sqrt{p_n}}\frac{|p_n - \tilde{p}_n|}{|\tilde{p}_n|} + \Bigl|\frac{p_n'}{\sqrt{p_n}} - \frac{p_0'}{\sqrt{p_0}}\Bigr| \\
&\leq \frac{\alpha_n}{\gamma_n}|p_0 - p_n|^{1/2} + \frac{\sqrt{p_n}}{\gamma_n}|\tilde{p}_n' - p_n'| + \frac{|p_n'|}{\sqrt{p_n}}\frac{|p_n - \tilde{p}_n|}{\gamma_n} + \Bigl|\frac{p_n'}{\sqrt{p_n}} - \frac{p_0'}{\sqrt{p_0}}\Bigr|,
\end{align*}
so by~\eqref{eq:var-pnh}--\eqref{eq:pnh-inequality} and~\ref{ass:alpha-gamma-h},
\begin{align*}
&\frac{1}{4}\,\E\biggl(\int_{\tilde{S}_{n,1}} (\tilde{\psi}_n - \psi_0)^2\,p_0 \biggm| \mathcal{D}_1\biggr) \\
&\quad\leq \frac{\alpha_n^2}{\gamma_n^2}\int_\R |p_n - p_0| + \int_\R \frac{p_n(y)}{\gamma_n^2}\Var\bigl(\tilde{p}_n'(y)\,|\,\mathcal{D}_1\bigr)\,dy + \int_\R \frac{(p_n')^2}{\gamma_n^2\,p_n}(y)\Var\bigl(\tilde{p}_n(y)\,|\,\mathcal{D}_1\bigr)\,dy \\
&\hspace{11cm}+ \int_\R \Bigl(\frac{p_n'}{\sqrt{p_n}} - \frac{p_0'}{\sqrt{p_0}}\Bigr)^2\\
&\quad\leq \frac{\alpha_n^2}{\gamma_n^2}\int_\R |p_n - p_0| + \frac{\norm{p_0}_\infty \int_\R (K')^2}{|I_2| h_n^3\gamma_n^2} + \frac{\norm{p_0}_\infty \int_\R K^2}{|I_2| h_n\gamma_n^2}\int_\R \frac{(p_0')^2}{p_0} + \int_\R \Bigl(\frac{p_n'}{\sqrt{p_n}} - \frac{p_0'}{\sqrt{p_0}}\Bigr)^2 \\
&\quad= O_p\Bigl(\frac{\alpha_n^2}{\gamma_n^2}h_n\Bigr) + O_p\Bigl(\frac{1}{|I_2| h_n^3\gamma_n^2}\Bigr) + o_p(1) = o_p(1).
\end{align*}
Therefore, by Lemma~\ref{lem:cond-cvg}, $\int_\R (\tilde{\psi}_n - \psi_0)^2\,p_0 \cvp 0$.
\end{proof}

Next, we use Lemma~\ref{lem:kernel-score-consistency} and properties of antitonic projections (Section~\ref{sec:isoproj}) to show that the projected score estimates $\hat{\psi}_{n,j}$ are consistent in $L^2(P_0)$ for the population-level antitonic projection $\psi_0^*$ of $\psi_0$.

\begin{lemma}
\label{lem:proj-score-consistency}
In the setting of Lemma~\ref{lem:kernel-score-consistency}, suppose that~\emph{\ref{ass:fisher-finite}--\ref{ass:covariates}} hold. Then for $j \in \{1,2,3\}$, we have
\begin{equation}
\label{eq:proj-score-consistency}
\int_\R \bigl(\hat{\psi}_{n,j}(z) - \psi_0^*(z - \mu)\bigr)^2\,p_0(z - \mu)\,dz \cvp 0.
\end{equation}
Moreover, $\hat{\Psi}_{n,j}(t) := \int_\R \hat{\psi}_{n,j}(z - t)\,p_0(z - \mu)\,dz \in \R$ for all $t \in \R$, and for any sequence $(v_n)$ satisfying $v_n(\alpha_n/\gamma_n) \to 0$, we have
\begin{align}
\label{eq:proj-score-consistency-1}
\sup_{s,t \in [-v_n,v_n]} \int_\R \bigl(\hat{\psi}_{n,j}(z - t) - \hat{\psi}_{n,j}(z - s)\bigr)^2\,p_0(z - \mu)\,dz &\cvp 0, \\ 
\label{eq:proj-score-consistency-2}
\Delta_n(v_n) := \sup_{\substack{s,t \in [-v_n,v_n] : \\ s \neq t}}\, \biggl|\frac{\hat{\Psi}_{n,j}(t) - \hat{\Psi}_{n,j}(s)}{t - s} - i^*(p_0)\biggr| &\cvp 0
\end{align}
as $n \to \infty$.
\end{lemma}

\begin{proof}
Similarly to the previous lemma, it suffices to prove this result when $\mu = 0$ and $j = 1$, so we will drop the index~$j$ in all of our notation below. In the terminology of Section~\ref{sec:isoproj}, $\hat{\psi}_n \equiv \hat{\psi}_{n,1} \in \Pi_\downarrow(\tilde{\psi}_n, \tilde{P}_n)$ is defined as an $L^2(\tilde{P}_n)$-antitonic projection of the initial score estimator $\tilde{\psi}_n \equiv \tilde{\psi}_{n,1}$, where $\tilde{P}_n \equiv \tilde{P}_{n,1}$ denotes the distribution with density $\tilde{p}_n \equiv \tilde{p}_{n.1}$. For each $n$, let 
\[
\tilde{\psi}_n^* := \widehat{\mathcal{M}}_\mathrm{R}(\tilde{\psi}_n \circ F_0^{-1}) \circ F_0 \in \Pi_\downarrow(\tilde{\psi}_n, P_0)
\]
denote a corresponding $L^2(P_0)$-antitonic projection of $\tilde{\psi}_n$. Then by~\eqref{eq:Linfty-proj-contraction}
% need the $L^\infty$ rather than the $L^\infty(P)$ version of this
and the definition~\eqref{eq:psi-kernel} of $\tilde{\psi}_n$, we have $\norm{\hat{\psi}_n}_\infty \vee \norm{\tilde{\psi}_n^*}_\infty \leq \norm{\tilde{\psi}_n}_\infty \leq \alpha_n/\gamma_n < \infty$. We have $\tilde{\psi}_n \in L^2(P_0) \cap L^2(\tilde{P}_n)$ for all $n$, so by~\eqref{eq:L2-proj-ineq} in Lemma~\ref{lem:L2-proj-ineq} together with the triangle inequality,
\begin{align}
\norm{\hat{\psi}_n - \tilde{\psi}_n^*}_{L^2(P_0)}^2 &\leq \int_\R (\tilde{\psi}_n^* - \hat{\psi}_n)(\hat{\psi}_n - \tilde{\psi}_n)\,d(\tilde{P}_n - P_0) \notag \\
\label{eq:iso-kernel-crude-bd}
&\leq \norm{(\tilde{\psi}_n^* - \hat{\psi}_n)(\hat{\psi}_n - \tilde{\psi}_n)}_\infty \int_\R |\tilde{p}_n - p_0| \leq 4\Bigl(\frac{\alpha_n}{\gamma_n}\Bigr)^2 \int_\R (|\tilde{p}_n - p_n| + |p_n - p_0|),
\end{align}
where $p_n(y) = \E\bigl(\tilde{p}_n(y)\!\bigm|\!\mathcal{D}_1\bigr)$ for $y \in \R$. By~\ref{ass:fisher-finite}, $\E(|\varepsilon_1|^\delta) = \int_\R |z|^\delta\,p_0(z)\,dz < \infty$ for some $\delta \in (0,1]$. Recall from~\ref{ass:alpha-gamma-h} that $\rho \in \bigl(0,\delta/(\delta + 1)\bigr)$. By Lemma~\ref{lem:L1-kde},~\ref{ass:covariates},~\ref{ass:kernel} and the fact that $b_n \cvp 0$, we have
\begin{align*}
&\E\biggl(\int_\R |\tilde{p}_n - p_n| \Bigm| \mathcal{D}_1\biggr) \\
&\quad\leq \frac{2^{1 - 2\rho} \norm{K}_\infty C_{\delta,\rho}^\rho}{(|I_2|h)^\rho}(1 + h)^{\delta(1 - \rho)}\Bigl(1 + \int_\R |z|^\delta \int_{\R^d} p_0(z - x^\top b_n)\,dP_X(x)\,dz\Bigr) \\
&\quad= \frac{2^{1 - 2\rho} \norm{K}_\infty C_{\delta,\rho}^\rho}{(|I_2|h)^\rho}(1 + h)^{\delta(1 - \rho)}\bigl(1 + \E\bigl\{|\varepsilon_1 + X_1^\top b_n|^\delta\!\bigm|\!\mathcal{D}_1\bigr\}\bigr) \\
&\quad\leq \frac{2^{1 - 2\rho} \norm{K}_\infty C_{\delta,\rho}^\rho}{(|I_2|h)^\rho}(1 + h)^{\delta(1 - \rho)}\bigl\{1 + \E(|\varepsilon_1|^\delta) + \E(\norm{X_1}^\delta) \cdot \norm{b_n}^\delta\bigr\} = O_p\bigl((nh)^{-\rho}\bigr)
\end{align*}
as $n \to \infty$ and $h \to 0$ with $nh \to \infty$, where $C_{\delta,\rho} = \int_\R (1 + |z|)^{-\delta(1 - \rho)/\rho}\,dz < \infty$. Thus, by Lemma~\ref{lem:cond-cvg} and~\ref{ass:alpha-gamma-h}, $\int_\R |\tilde{p}_n - p_n| = O_p\bigl((nh)^{-\rho}\bigr) = O_p(n^{-2\rho/3})$. Combining this with~\eqref{eq:kernel-L1-bias},~\eqref{eq:iso-kernel-crude-bd} and~\ref{ass:alpha-gamma-h} yields $\norm{\hat{\psi}_n - \tilde{\psi}_n^*}_{L^2(P_0)}^2 = O_p\bigl((h_n \vee n^{-2\rho/3})(\alpha_n/\gamma_n)^2\bigr) = o_p(1)$ as $n \to \infty$. 
% See also Goldenshluger and Lepski (2014, Theorem~4), where we take $p = 1$ and e.g.\ $\theta < 1/2$ in the tail dominance condition
Hence, by applying the triangle inequality for $\norm{{\cdot}}_{L^2(P_0)}$ together with~\eqref{eq:L2-proj-contraction} and Lemma~\ref{lem:kernel-score-consistency}, we obtain
\begin{align*}
\biggl(\int_\R (\hat{\psi}_n - \psi_0^*)^2\,p_0\biggr)^{1/2} = \norm{\hat{\psi}_n - \psi_0^*}_{L^2(P_0)} &\leq \norm{\hat{\psi}_n - \tilde{\psi}_n^*}_{L^2(P_0)} + \norm{\tilde{\psi}_n^* - \psi_0^*}_{L^2(P_0)}\\
&\leq \norm{\hat{\psi}_n - \tilde{\psi}_n^*}_{L^2(P_0)} + \norm{\tilde{\psi}_n - \psi_0}_{L^2(P_0)} \cvp 0,
\end{align*}
which proves~\eqref{eq:proj-score-consistency}.

For $t \in \R$, let $p_t(\cdot) := p_0(\cdot + t)$. We have $\norm{\hat{\psi}_n}_\infty \leq \norm{\tilde{\psi}_n}_\infty \leq \alpha_n/\gamma_n$ for each $n$, so $\hat{\Psi}_n(t) \equiv\hat{\Psi}_{n,1}(t) = \int_\R \hat{\psi}_n\,p_t$ is finite for every $t$. Moreover,
\begin{align}
\int_\R (\hat{\psi}_n - \psi_0^*)^2\,p_t
&\leq 3\biggl(\int_\R \hat{\psi}_n^2\bigl(\sqrt{p_t} - \sqrt{p_0}\bigr)^2 + \int_\R \bigl(\hat{\psi}_n\sqrt{p_0} - \psi_0^*\sqrt{p_0}\bigr)^2 + \int_\R (\psi_0^*)^2\bigl(\sqrt{p_0} - \sqrt{p_t}\bigr)^2\biggr) \notag \\
\label{eq:psi-kernel-consistency-1}
&\leq 3\biggl(\Bigl(\frac{\alpha_n}{\gamma_n}\Bigr)^2 \int_\R \bigl(\sqrt{p_t} - \sqrt{p_0}\bigr)^2 + \int_\R (\hat{\psi}_n - \psi_0^*)^2\,p_0 + \int_\R (\psi_0^*)^2\bigl(\sqrt{p_0} - \sqrt{p_t}\bigr)^2\biggr)
\end{align}
and
\begin{align}
&\int_\R \bigl(\hat{\psi}_n(z - t) - \hat{\psi}_n(z)\bigr)^2\,p_0(z)\,dz \notag \\
&\qquad\leq 3\int_\R \bigl\{\bigl(\hat{\psi}_n(z - t) - \psi_0^*(z - t)\bigr)^2 + \bigl(\psi_0^*(z - t) - \psi_0^*(z)\bigr)^2 + \bigl(\psi_0^*(z) - \hat{\psi}_n(z)\bigr)^2\bigr\}\,p_0(z)\,dz \notag \\
\label{eq:psi-kernel-consistency-2}
&\qquad= 3\biggl(\int_\R (\hat{\psi}_n - \psi_0^*)^2\,p_t + \int_\R \bigl(\psi_0^*(z - t) - \psi_0^*(z)\bigr)^2\,p_0(z)\,dz + \int_\R (\hat{\psi}_n - \psi_0^*)^2\,p_0\biggr).
\end{align}
Since $i(p_0) < \infty$ by~\ref{ass:fisher-finite}, Lemma~\ref{lem:location-DQM} implies that 
\begin{equation}
\label{eq:location-DQM-hellinger}
\int_\R \,\bigl(\sqrt{p_t} - \sqrt{p_0}\bigr)^2 = O(t^2)
\end{equation}
as $t \to 0$. Moreover, by~\ref{ass:psi0-star}, $\psi_0^*$ satisfies the hypotheses of Lemma~\ref{lem:varPsi-deriv}, so
\begin{equation}
\label{eq:psi0-shift}
\int_\R (\psi_0^*)^2\bigl(\sqrt{p_0} - \sqrt{p_t}\bigr)^2 \to 0 \quad\text{and}\quad \int_\R \bigl(\psi_0^*(z - t) - \psi_0^*(z)\bigr)^2 p_0(z)\,dz \to 0
\end{equation}
as $t \to 0$. Therefore, letting $(v_n)$ be such that $v_n(\alpha_n/\gamma_n) \to 0$, we deduce from~\eqref{eq:proj-score-consistency},~\eqref{eq:psi-kernel-consistency-1} and~\eqref{eq:psi0-shift} that
\begin{equation}
\label{eq:psi-kernel-consistency-3}
\sup_{t \in [-v_n,v_n]}\int_\R (\hat{\psi}_n - \psi_0^*)^2\,p_t \cvp 0.
\end{equation}
Hence by~\eqref{eq:psi-kernel-consistency-2} and~\eqref{eq:psi0-shift},
\[
\sup_{s,t \in [-v_n,v_n]}\int_\R \bigl(\hat{\psi}_n(z - t) - \hat{\psi}_n(z - s)\bigr)^2\,p_0(z)\,dz \leq 4\sup_{t \in [-v_n,v_n]} \int_\R \bigl(\hat{\psi}_n(z - t) - \hat{\psi}_n(z)\bigr)^2\,p_0(z)\,dz \cvp 0.
\]
This yields~\eqref{eq:proj-score-consistency-1}. Next, by Lemma~\ref{lem:varPsi-deriv}, $\Psi_0^*(t) := \int_\R \psi_0^*(z - t)\,p_0(z)\,dz = \int_\R \psi_0^*\,p_t \in \R$ for all $t \in [-t_0,t_0]$, and by the Cauchy--Schwarz inequality,
\begin{align*}
\bigl|\hat{\Psi}_n(t) - \hat{\Psi}_n(s) - \bigl(\Psi_0^*(t) - \Psi_0^*(s)\bigr)\bigr|^2 &=
\biggl|\int_\R (\hat{\psi}_n - \psi_0^*)(p_t - p_s)\biggr|^2 \\
&\leq \biggl\{\int_\R (\hat{\psi}_n - \psi_0^*)^2 \bigl(\sqrt{p_t} + \sqrt{p_s}\bigr)^2\biggr\}\biggl\{\int_\R \bigl(\sqrt{p_t} - \sqrt{p_s}\bigr)^2\biggr\} \\
&\leq \biggl\{2\int_\R (\hat{\psi}_n - \psi_0^*)^2 (p_t + p_s)\biggr\}\biggl\{\int_\R \bigl(\sqrt{p_t} - \sqrt{p_s}\bigr)^2\biggr\}
\end{align*}
for $s,t \in [-t_0,t_0]$. By~\eqref{eq:location-DQM-hellinger},
% $\int_\R (\hat{\psi}_n - \psi_0^*)^2 (p_t + p_s) \to 0$ by~\eqref{eq:psi-kernel-consistency-1} and 
$\int_\R \bigl(\sqrt{p_t} - \sqrt{p_s}\bigr)^2 = \int_\R \bigl(\sqrt{p_{t - s}} - \sqrt{p_0}\bigr)^2 = O\bigl((t - s)^2\bigr)$ as $s,t \to 0$, so it follows from~\eqref{eq:psi-kernel-consistency-3} that
\begin{equation}
\label{eq:Psi_n-hat-1}
\sup_{\substack{s,t \in [-v_n,v_n] : \\ s \neq t}}\, \biggl|\frac{\hat{\Psi}_n(t) - \hat{\Psi}_n(s) - \bigl(\Psi_0^*(t) - \Psi_0^*(s)\bigr)}{t - s}\biggr| = o_p(1)
\end{equation}
as $n \to \infty$. By Lemma~\ref{lem:varPsi-deriv} and Theorem~\ref{thm:antitonic-score-proj}\emph{(c)}, $(\Psi_0^*)'(0) = -\int_\R p_0\,d\psi_0^* = \bigl(\int_\R (\psi_0^*)^2\,p_0\big/V_{p_0}(\psi_0^*)\bigr)^{1/2} = i^*(p_0) \in (0,\infty)$, so
\[
\sup_{\substack{s,t \in [-v_n,v_n] : \\ s \neq t}}\, \biggl|\frac{\Psi_0^*(t) - \Psi_0^*(s)}{t - s} - i^*(p_0)\biggr| \to 0
\]
as $n \to \infty$. Combining this with~\eqref{eq:Psi_n-hat-1} yields~\eqref{eq:proj-score-consistency-2}.
\end{proof}

\subsubsection{Proof of Theorem~\ref{thm:linreg-score-sym}}
\label{subsec:linreg-sym-proofs}

\begin{corollary}
\label{cor:sym-score-consistency}
For $n \in \N$ and $j \in \{1,2,3\}$, define $\hat{\Psi}_{n,j}^{\mathrm{anti}} \colon \R \to \R$ by $\hat{\Psi}_{n,j}^{\mathrm{anti}}(t) := \int_\R \hat{\psi}_{n,j}^{\mathrm{anti}}(z - t)\,p_0(z)\,dz$. Then under the hypotheses of Theorem~\ref{thm:linreg-score-sym}, the conclusions of Lemma~\ref{lem:proj-score-consistency} hold with $\mu = 0$ for $(\hat{\psi}_{n,j}^{\mathrm{anti}})$ and $(\hat{\Psi}_{n,j}^{\mathrm{anti}})$ in place of $(\hat{\psi}_{n,j})$ and $(\hat{\Psi}_{n,j})$ respectively.
\end{corollary}

\begin{proof}
It suffices to consider $j = 1$, so we drop the index $j$ in all of our notation below. The pilot residuals in Step~3 can be written as $\hat{\varepsilon}_i = Y_i - X_i^\top\bar{\beta}_n^{(1)} = \varepsilon_i + X_i^\top b_n$ for $i \in I_2$, where $b_n := \beta_0 - \bar{\beta}_n^{(1)} = O_p(n^{-1/2})$ by assumption. Therefore, the hypotheses of Lemmas~\ref{lem:kernel-score-consistency} and~\ref{lem:proj-score-consistency} are satisfied with $\mu = 0$. Since $p_0$ is symmetric, $\psi_0(z) = p_0'(z)/p_0(z) = -\psi_0(-z)$ for all $z \in \R$. By~\eqref{eq:psi0star-antisym}, $\psi_0^*(z) = -\psi_0^*(-z)$ for Lebesgue almost every $z \in \R$, so
\begin{align*}
\int_\R (\hat{\psi}_n^{\mathrm{anti}} - \psi_0^*)^2\,p_0 &= \int_\R \biggl(\frac{\hat{\psi}_n(z) - \hat{\psi}_n(-z)}{2} - \frac{\psi_0^*(z) - \psi_0^*(-z)}{2}\biggr)^2\,p_0(z)\,dz \\
&\leq \int_\R \frac{\bigl(\hat{\psi}_n(z) - \psi_0^*(z)\bigr)^2 + \bigl(\hat{\psi}_n(-z) - \psi_0^*(-z)\bigr)^2}{2}\,p_0(z)\,dz = \int_\R (\hat{\psi}_n - \psi_0^*)^2\,p_0
\end{align*}
for each $n$. Similarly, for $s,t \in \R$, we have
\begin{align*}
\int_\R \bigl(\hat{\psi}_n^{\mathrm{anti}}(z - t) &- \hat{\psi}_n^{\mathrm{anti}}(z - s)\bigr)^2\,p_0(z)\,dz \\
&\leq \int_\R \frac{\bigl(\hat{\psi}_n(z - t) - \hat{\psi}_n(z - s)\bigr)^2 + \bigl(\hat{\psi}_n(z + t) - \hat{\psi}_n(z + s)\bigr)^2}{2}\,p_0(z)\,dz
\end{align*}
and
\begin{align*}
\hat{\Psi}_n^{\mathrm{anti}}(t) = \int_\R \frac{\hat{\psi}_n(z - t) - \hat{\psi}_n(t - z)}{2}\,p_0(z)\,dz &= \int_\R \frac{\hat{\psi}_n(z - t) - \hat{\psi}_n(t + z)}{2}\,p_0(z)\,dz \\
&= \frac{\hat{\Psi}_n(t) - \hat{\Psi}_n(-t)}{2}.
\end{align*}
The desired conclusions for $(\hat{\psi}_{n,j}^{\mathrm{anti}})$ and $(\hat{\Psi}_{n,j}^{\mathrm{anti}})$ therefore follow from Lemma~\ref{lem:proj-score-consistency}.
\end{proof}

The proof of Theorem~\ref{thm:linreg-score-sym} relies on the following `asymptotic equicontinuity' result,
% NB: Modulus of continuity of an empirical process (local increments); see van de Geer (2000, Lemma~5.13)
which is similar to~\citet[Lemma~4.1]{bickel1975one}.

\begin{lemma}
\label{lem:linreg-sym-equicontinuity}
For $b \in \R^d$, $n \in \N$ and $j \in \{1,2,3\}$, define
\[
R_{n,j}(b) := \frac{1}{\sqrt{n}}\sum_{i \in I_{j+2}} X_i\Bigl\{\hat{\psi}_{n,j}^{\mathrm{anti}}\Bigl(\varepsilon_i - \frac{X_i^\top b}{\sqrt{n}}\Bigr) - \hat{\Psi}_{n,j}^{\mathrm{anti}}\Bigl(\frac{X_i^\top b}{\sqrt{n}}\Bigr)\Bigr\}.
\]
Fix $M > 0$ and let $B_M := \{b \in \R^d : \norm{b} \leq M\}$. Under the hypotheses of Theorem~\ref{thm:linreg-score-sym}, it holds for every $j \in \{1,2,3\}$ that
\[
\sup_{b \in B_M} \norm{R_{n,j}(b) - R_{n,j}(0)} = o_p(1)
\] 
as $n \to \infty$.
% \green{Using chaining and empirical process techniques to refine the crude proof below, we should strengthen this to a maximal inequality in expectation.}
\end{lemma}

\begin{proof}
It suffices to consider $j = 1$, so we drop the $j$ subscript from $\hat{\psi}_{n,j}^{\mathrm{anti}}$, $\hat{\Psi}_{n,j}^{\mathrm{anti}}$ and $R_{n,j}$; we also write $\mathcal{D}' \equiv \mathcal{D}_n' := \mathcal{D}_1 \cup \mathcal{D}_2 \cup \{X_i : i \in I_3\}$.  We initially fix $b \in B_M$. For $i \in I_3$, we have $\E\bigl\{\hat{\psi}_n^{\mathrm{anti}}(\varepsilon_i - X_i^\top bn^{-1/2})\,|\,\mathcal{D}'\bigr\} = \hat{\Psi}_n^{\mathrm{anti}}(X_i^\top bn^{-1/2})$, so $\E\bigl(R_n(b)\!\bigm|\!\mathcal{D}'\bigr) = 0$. By~\ref{ass:covariates} and Cauchy--Schwarz, $\max_{i \in [n]}|X_i^\top b|\,\alpha_n/\gamma_n = o_p(n^{1/2})$, so it follows from the weak law of large numbers and Corollary~\ref{cor:sym-score-consistency} that
\begin{align}
\E\bigl(\norm{R_n(b) - R_n(0)}^2 \!\bigm|\! \mathcal{D}'\bigr) &= \tr\Cov\bigl(R_n(b) - R_n(0) \bigm| \mathcal{D}'\bigr) \notag \\
&= \frac{1}{n}\sum_{i \in I_3} \tr\Cov\bigl(X_i\{\hat{\psi}_n^{\mathrm{anti}}(\varepsilon_i - X_i^\top bn^{-1/2}) - \hat{\psi}_n^{\mathrm{anti}}(\varepsilon_i)\} \bigm| \mathcal{D}'\bigr) \notag \\
&= \frac{1}{n}\sum_{i \in I_3} \norm{X_i}^2 \Var\bigl(\hat{\psi}_n^{\mathrm{anti}}(\varepsilon_i - X_i^\top bn^{-1/2}) - \hat{\psi}_n^{\mathrm{anti}}(\varepsilon_i) \bigm| \mathcal{D}'\bigr) \notag  \\
&\leq \Bigl(\frac{1}{n}\sum_{i \in I_3} \norm{X_i}^2\Bigr) \max_{i \in I_3}\int_\R \bigl(\hat{\psi}_n^{\mathrm{anti}}(z - X_i^\top bn^{-1/2}) - \hat{\psi}_n^{\mathrm{anti}}(z)\bigr)^2\,p_0(z)\,dz \notag \\
\label{eq:cov-Rn-beta}
&= O_p(1)\,o_p(1) = o_p(1)
\end{align}
as $n \to \infty$, where the second equality holds because the summands $X_i \{\hat{\psi}_n^{\mathrm{anti}}(\varepsilon_i - X_i^\top bn^{-1/2}) - \hat{\psi}_n^{\mathrm{anti}}(\varepsilon_i)\}$ are conditionally independent given $\mathcal{D}'$. Thus, by Lemma~\ref{lem:cond-cvg}, $\norm{R_n(b) - R_n(0)} \cvp 0$ for every $b \in B_M$.

Fix $\epsilon \in (0,1)$. There exists a $\epsilon M$-Euclidean covering set $B_{M,\epsilon} \subseteq B_M$ such that $|B_{M,\epsilon}| \leq 3\epsilon^{-d}$~\citep[e.g.][Example~5.8]{wainwright2019high}, so since $d$ is fixed as $n \to \infty$, it follows by a union bound that
\[
\max_{b \in B_{M,\epsilon}}\norm{R_n(b) - R_n(0)} \cvp 0.
\]
For every $b \in B_M$, there exists $\pi_b \in B_{M,\epsilon}$ such that $\norm{b - \pi_b} \leq \epsilon M$. Since $\hat{\psi}_n^{\mathrm{anti}}$ is decreasing and $\hat{\Psi}_n^{\mathrm{anti}}$ is increasing for each $n$, we have 
\begin{align*}
\hat{\psi}_n^{\mathrm{anti}}\biggl(\varepsilon_i - \frac{X_i^\top\pi_b - \epsilon M\norm{X_i}}{\sqrt{n}}\biggr) &\leq \hat{\psi}_n^{\mathrm{anti}}\biggl(\varepsilon_i - \frac{X_i^\top b'}{\sqrt{n}}\biggr) \leq \hat{\psi}_n^{\mathrm{anti}}\biggl(\varepsilon_i - \frac{X_i^\top\pi_b + \epsilon M\norm{X_i}}{\sqrt{n}}\biggr), \\[3pt]
\hat{\Psi}_n^{\mathrm{anti}}\biggl(\frac{X_i^\top\pi_b - \epsilon M\norm{X_i}}{\sqrt{n}}\biggr) &\leq \hat{\Psi}_n^{\mathrm{anti}}\Bigl(\frac{X_i^\top b'}{\sqrt{n}}\Bigr) \leq \hat{\Psi}_n^{\mathrm{anti}}\biggl(\frac{X_i^\top\pi_b + \epsilon M\norm{X_i}}{\sqrt{n}}\biggr)
\end{align*}
for $b' \in \{b,\pi_b\}$. Now for $b' \in B_{M,\epsilon}$, let
\[
r_n(b') := \sum_{i \in I_3} \frac{\norm{X_i}}{\sqrt{n}} \biggl\{\hat{\psi}_n^{\mathrm{anti}}\biggl(\varepsilon_i - \frac{X_i^\top b' + \epsilon M\norm{X_i}}{\sqrt{n}}\biggr) - \hat{\psi}_n^{\mathrm{anti}}\biggl(\varepsilon_i - \frac{X_i^\top b' - \epsilon M\norm{X_i}}{\sqrt{n}}\biggr)\biggr\}.
\]
Then
\begin{align*}
&\sup_{b \in B_M}\|R_n(b) - R_n(\pi_b)\| \leq \max_{b' \in B_{M,\epsilon}} \bigl\{r_n(b') + \E\bigl(r_n(b') \bigm| \mathcal{D}'\bigr)\bigr\}.
% \\
% &\hspace{1cm}\leq \max_{b' \in B_{M,\epsilon}}\sum_{i \in I_3} \frac{\norm{X_i}}{\sqrt{n}} \biggl\{\hat{\psi}_n^{\mathrm{anti}}\biggl(\varepsilon_i - X_i^\top b' - \frac{\epsilon M\norm{X_i}}{\sqrt{n}}\biggr) - \hat{\psi}_n^{\mathrm{anti}}\biggl(\varepsilon_i - X_i^\top b' + \frac{\epsilon M\norm{X_i}}{\sqrt{n}}\biggr) \\
% &\hspace{5cm}+ \hat{\Psi}_n^{\mathrm{anti}}\biggl(X_i^\top b' + \frac{\epsilon M\norm{X_i}}{\sqrt{n}}\biggr) - \hat{\Psi}_n^{\mathrm{anti}}\biggl(X_i^\top b' - \frac{\epsilon M\norm{X_i}}{\sqrt{n}}\biggr)\biggr\}.
\end{align*}
By~\ref{ass:covariates}, $V_n := \max_{i \in I_3} M(1 + \epsilon)n^{-1/2}\norm{X_i} = o_p(\gamma_n/\alpha_n)$, so we deduce from Corollary~\ref{cor:sym-score-consistency}, Lemma~\ref{lem:cvp-plugin} and the independence of $\mathcal{D}_1,\mathcal{D}_2,\mathcal{D}_3$ that
\begin{align*}
&\Var\bigl(r_n(b') \bigm| \mathcal{D}'\bigr) \\
&\quad\leq \sum_{k \in I_3} \frac{\norm{X_k}^2}{n} \max_{i \in I_3} \int_\R\biggl\{\hat{\psi}_n^{\mathrm{anti}}\biggl(z - \frac{X_i^\top b' + \epsilon M\norm{X_i}}{\sqrt{n}}\biggr) - \hat{\psi}_n^{\mathrm{anti}}\biggl(z - \frac{X_i^\top b' - \epsilon M\norm{X_i}}{\sqrt{n}}\biggr)\biggr\}^2 p_0(z)\,dz \cvp 0.
\end{align*}
Together with a union bound over $b' \in B_{M,\epsilon}$, this shows that $\max_{b' \in B_{M,\epsilon}} \bigl\{r_n(b') - \E\bigl(r_n(b') \bigm| \mathcal{D}'\bigr)\bigr\} = o_p(1)$. Moreover, recalling the definition of $\Delta_n$ in~\eqref{eq:proj-score-consistency-2}, we have $\Delta_n(V_n) = o_p(1)$ by~\ref{ass:covariates} and Corollary~\ref{cor:sym-score-consistency}, so
\begin{align*}
\max_{b' \in B_{M,\epsilon}}\E\bigl(r_n(b') \bigm| \mathcal{D}'\bigr) &= \max_{b' \in B_{M,\epsilon}}\sum_{i \in I_3} \frac{\norm{X_i}}{\sqrt{n}} \biggl\{\hat{\Psi}_n^{\mathrm{anti}}\biggl(\frac{X_i^\top b' + \epsilon M\norm{X_i}}{\sqrt{n}}\biggr) - \hat{\Psi}_n^{\mathrm{anti}}\biggl(\frac{X_i^\top b' - \epsilon M\norm{X_i}}{\sqrt{n}}\biggr)\biggr\} \\
&\leq \sum_{i \in I_3} \frac{2\epsilon M\norm{X_i}^2}{n}\bigl(i^*(p_0) + \Delta_n(V_n)\bigr) = \frac{2\epsilon Mi^*(p_0)}{3}\,\E(\norm{X_1}^2) + o_p(1).
\end{align*}
Thus,
\begin{align*}
&\sup_{b \in B_M}\norm{R_n(b) - R_n(0)} \leq \sup_{b \in B_M}\|R_n(b) - R_n(\pi_b)\| + \max_{b' \in B_{M,\epsilon}}\norm{R_n(b') - R_n(0)} \\
&\quad\leq \max_{b' \in B_{M,\epsilon}} \bigl\{r_n(b') - \E\bigl(r_n(b') \bigm| \mathcal{D}'\bigr)\bigr\} + 2\max_{b' \in B_{M,\epsilon}}\E\bigl(r_n(b') \bigm| \mathcal{D}'\bigr) + \max_{b' \in B_{M,\epsilon}}\norm{R_n(b') - R_n(0)} \\
&\quad\leq \frac{4\epsilon Mi^*(p_0)}{3}\,\E(\norm{X_1}^2) + o_p(1)
\end{align*}
as $n \to \infty$. Since this holds for all $\epsilon \in (0,1)$, the desired conclusion follows.
\end{proof}

As a first step towards proving the asymptotic normality of $\hat{\beta}_n^\dagger$ and $\hat{\beta}_n^\ddagger$ in Theorem~\ref{thm:linreg-score-sym}, we show in Lemma~\ref{lem:linreg-sym-consistency} below that they are $\sqrt{n}$-consistent estimators of $\beta_0$. To this end, we exploit the convexity of the induced loss functions $z \mapsto \hat{\ell}_{n,j}^{\mathrm{sym}}(z) = -\int_0^z \hat{\psi}_{n,j}^{\mathrm{anti}}$, similarly to~\citet[Theorem~2.1]{he2000parameters}.  We denote the Euclidean unit sphere in $\mathbb{R}^d$ by $S^{d-1} := \{u \in \R^d : \|u\| = 1\}$.
\begin{lemma}
\label{lem:linreg-sym-consistency}
For $n \in \N$, $j \in \{1,2,3\}$ and $t > 0$, we have
\begin{align}
\label{eq:linreg-sym-sphere-j}
\bigl\{\norm{\hat{\beta}_n^{(j)} - \beta_0} > t\bigr\} &\subseteq \biggl\{\inf_{u \in S^{d-1}}u^\top \sum_{i \in I_{j+2}} X_i\hat{\psi}_{n,j}^{\mathrm{anti}}(\varepsilon_i - tX_i^\top u) \leq 0\biggr\}, \\
\label{eq:linreg-sym-sphere}
\bigl\{\norm{\hat{\beta}_n^\ddagger - \beta_0} > t\bigr\} &\subseteq \biggl\{\inf_{u \in S^{d-1}}u^\top \sum_{j=1}^3 \sum_{i \in I_{j+2}} X_i\hat{\psi}_{n,j}^{\mathrm{anti}}(\varepsilon_i - tX_i^\top u) \leq 0\biggr\}.
\end{align}
Consequently, under the hypotheses of Theorem~\ref{thm:linreg-score-sym}, the following statements hold as $n \to \infty$:
\begin{enumerate}[label=(\alph*)]
\item $\sqrt{n}(\hat{\beta}_n^{(j)} - \beta_0) = O_p(1)$ and $\sqrt{n}(\hat{\beta}_n^\ddagger - \beta_0) = O_p(1)$;
\item $\displaystyle\frac{1}{\sqrt{n}}\sum_{i \in I_{j+2}} X_i\hat{\psi}_{n,j}^{\mathrm{anti}}\bigl(Y_i - X_i^\top\hat{\beta}_n^{(j)}\bigr) = o_p(1)$ \; and \; $\displaystyle\frac{1}{\sqrt{n}}\sum_{j=1}^3 \sum_{i \in I_{j+2}} X_i\hat{\psi}_{n,j}^{\mathrm{anti}}\bigl(Y_i - X_i^\top\hat{\beta}_n^\ddagger\bigr) = o_p(1)$.
\end{enumerate}
\end{lemma}

\begin{proof}
It suffices to consider $j = 1$. We write $\hat{\psi}_n^{\mathrm{anti}} \equiv \hat{\psi}_{n,1}^{\mathrm{anti}}$ and $\hat{\ell}_n^{\mathrm{sym}} \equiv \hat{\ell}_{n,1}^{\mathrm{sym}}$. The function $\beta \mapsto \sum_{i \in I_3} \hat{\ell}_n^{\mathrm{sym}}(Y_i - X_i^\top\beta) =: \hat{L}_n^{\mathrm{sym}}(\beta)$ is convex on $\R^d$, with $\hat{\beta}_n \equiv \hat{\beta}_n^{(1)} \in \argmin_{\beta \in \R^d}\hat{L}_n^{\mathrm{sym}}(\beta)$, so $h \mapsto \hat{L}_n^{\mathrm{sym}}\bigl((1 - h)\hat{\beta}_n + h\beta_0\bigr) =: G_n(h)$ is convex and increasing on $[0,\infty)$. Since
\[
g_n(h) := -(\hat{\beta}_n - \beta_0)^\top\sum_{i \in I_3} X_i\hat{\psi}_n^{\mathrm{anti}}\bigl(\varepsilon_i - (1-h)X_i^\top(\hat{\beta}_n - \beta_0)\bigr)
\]
is a subgradient of $G_n$ at $h$, it follows that $g_n(h) \geq 0$ for every $h > 0$. Hence, if $\norm{\hat{\beta}_n - \beta_0} > t > 0$, then 
\begin{align}
0 \leq g\biggl(1 - \frac{t}{\norm{\hat{\beta}_n - \beta_0}}\biggr) &= -(\hat{\beta}_n - \beta_0)^\top\sum_{i \in I_3} X_i\hat{\psi}_n^{\mathrm{anti}}\biggl(\varepsilon_i - tX_i^\top\frac{\hat{\beta}_n - \beta_0}{\norm{\hat{\beta}_n - \beta_0}}\biggr) \notag \\
\label{eq:linreg-sym-sphere-j-1}
&\leq -\norm{\hat{\beta}_n - \beta_0}\inf_{u \in S^{d-1}} u^\top\sum_{i \in I_3}X_i\hat{\psi}_n^{\mathrm{anti}}(\varepsilon_i - tX_i^\top u),
\end{align}
so~\eqref{eq:linreg-sym-sphere-j} holds. By analogous reasoning based on $\beta \mapsto \hat{L}_n^\ddagger(\beta) := \sum_{j=1}^3 \sum_{i \in I_{j+2}} \hat{\ell}_{n,j}^{\mathrm{sym}}(Y_i - X_i^\top\beta)$, we obtain~\eqref{eq:linreg-sym-sphere}.

\medskip
\noindent \textit{(a)} Fix $M > 0$ and write $\hat{\Psi}_n^{\mathrm{anti}} \equiv \hat{\Psi}_{n,1}^{\mathrm{anti}}$. By~\ref{ass:covariates}, $V_n := \max_{i \in I_3} Mn^{-1/2}\norm{X_i} = o_p(\gamma_n/\alpha_n)$. We have $\hat{\Psi}_n^{\mathrm{anti}}(0) = 0$ and $|I_3|/n \to 1/3$, so it follows from Corollary~\ref{cor:sym-score-consistency},~\eqref{eq:proj-score-consistency-2} and Lemma~\ref{lem:cvp-plugin} that
\begin{align}
&\sup_{u \in S^{d-1}}\Biggl\|\,\sum_{i \in I_3} \frac{X_i}{\sqrt{n}}\hat{\Psi}_n^{\mathrm{anti}}\Bigl(\frac{M X_i^\top u}{\sqrt{n}}\Bigr) - \frac{Mi^*(p_0)}{3}\,\E(X_1 X_1^\top) u\,\Biggr\| \notag \\
&\qquad\leq \sup_{u \in S^{d-1}}\sum_{i \in I_3} \frac{\norm{X_i}}{\sqrt{n}}\biggl|\hat{\Psi}_n^{\mathrm{anti}}\Bigl(\frac{MX_i^\top u}{\sqrt{n}}\Bigr) - \frac{Mi^*(p_0)X_i^\top u}{\sqrt{n}}\biggr| + Mi^*(p_0)\,\biggl\|\,\sum_{i \in I_3}\frac{X_i X_i^\top}{n} - \frac{\E(X_1 X_1^\top)}{3}\biggr\|_{\mathrm{op}}  \notag \\
\label{eq:linreg-sym-Psi}
&\qquad\leq \frac{M\sum_{i \in I_3}\norm{X_i}^2}{n} \cdot \Delta_n(V_n) + Mi^*(p_0)\biggl\|\,\sum_{i \in I_3}\frac{X_i X_i^\top}{n} - \frac{\E(X_1 X_1^\top)}{3}\biggr\|_{\mathrm{op}} = o_p(1).
\end{align}
Since $\E(X_1 X_1^\top)$ is positive definite by~\ref{ass:covariates}, its minimum eigenvalue $\lambda_{\min}$ is strictly positive. Then by~\eqref{eq:linreg-sym-Psi} and Lemma~\ref{lem:linreg-sym-equicontinuity},
\begin{align}
\inf_{u \in S^{d-1}}u^\top \sum_{i \in I_3} \frac{X_i}{\sqrt{n}}\hat{\psi}_n^{\mathrm{anti}}\Bigl(\varepsilon_i - \frac{M X_i^\top u}{\sqrt{n}}\Bigr) &= \inf_{u \in S^{d-1}}u^\top\biggl\{R_n(Mu) + \sum_{i \in I_3} \frac{X_i}{\sqrt{n}}\hat{\Psi}_n^{\mathrm{anti}}\Bigl(\frac{M X_i^\top u}{\sqrt{n}}\Bigr)\biggr\} \notag \\
&= \inf_{u \in S^{d-1}} u^\top\Bigl(R_n(0) + \frac{M i^*(p_0)}{3}\,\E(X_1 X_1^\top)u\Bigr) + o_p(1) \notag \\
\label{eq:linreg-sym-consistency}
&\geq -\norm{R_n(0)} + \frac{M\lambda_{\min}i^*(p_0)}{3} + o_p(1).
\end{align}
Writing $\mathcal{D}' := \mathcal{D}_1 \cup \mathcal{D}_2 \cup \{X_i : i \in I_3\}$, we have 
\[
\E\biggl(\frac{1}{\sqrt{n}}\sum_{i \in I_3} X_i\bigl(\hat{\psi}_n^{\mathrm{anti}}(\varepsilon_i) - \psi_0^*(\varepsilon_i)\bigr) \biggm| \mathcal{D}'\biggr) = \frac{1}{\sqrt{n}}\sum_{i \in I_3} X_i \int_\R (\hat{\psi}_n^{\mathrm{anti}} - \psi_0^*)\,p_0 = 0
\]
because $p_0$ is symmetric and $\hat{\psi}_n^{\mathrm{anti}},\psi_0^*$ are antisymmetric. Together with Corollary~\ref{cor:sym-score-consistency},
% and~\eqref{eq:proj-score-consistency-1}, 
this implies that
\begin{align*}
\E\biggl(\biggl\|\frac{1}{\sqrt{n}}\sum_{i \in I_3} X_i\bigl(\hat{\psi}_n^{\mathrm{anti}}(\varepsilon_i) - \psi_0^*(\varepsilon_i)\bigr)\biggr\|^2 \biggm| \mathcal{D}'\biggr) &\leq \frac{1}{n}\sum_{i \in I_3}\E\bigl\{\norm{X_i}^2 (\hat{\psi}_n^{\mathrm{anti}} - \psi_0^*)^2(\varepsilon_i)\!\bigm|\! \mathcal{D}'\bigr\} \\
&= \frac{1}{n}\sum_{i \in I_3}\norm{X_i}^2 \int_\R (\hat{\psi}_n^{\mathrm{anti}} - \psi_0^*)^2\,p_0 \cvp 0.
\end{align*}
We have $\E\bigl(X_i\psi_0^*(\varepsilon_i)\bigr) = \E(X_i)\int_\R \psi_0^*\,p_0 = 0$ for each $i$, so by the central limit theorem,
\begin{equation}
\label{eq:linreg-sym-clt}
R_n(0) = \frac{1}{\sqrt{n}}\sum_{i \in I_3} X_i\hat{\psi}_n^{\mathrm{anti}}(\varepsilon_i) = \frac{1}{\sqrt{n}}\sum_{i \in I_3} X_i\psi_0^*(\varepsilon_i) + o_p(1) = O_p(1).
\end{equation}
Since $M\lambda_{\min}i^*(p_0) > 0$, we deduce from~\eqref{eq:linreg-sym-sphere-j} and~\eqref{eq:linreg-sym-consistency} that $\limsup_{n \to \infty}\Pr(\norm{\hat{\beta}_n - \beta_0} \geq M/\sqrt{n}) \to 0$ as $M \to \infty$, so $\sqrt{n}(\hat{\beta}_n - \beta_0) = O_p(1)$ as $n \to \infty$. The $\sqrt{n}$-consistency of $\hat{\beta}_n^\ddagger$ follows similarly from~\eqref{eq:linreg-sym-sphere}.

\medskip
\noindent
\textit{(b)} Since the errors $\varepsilon_i$ have an absolutely continuous density $p_0$, the conditional distribution of $(Y_i : i \in I_3)$ given $\mathcal{D}' = \mathcal{D}_1 \cup \mathcal{D}_2 \cup \{X_i : i \in I_3\}$ is absolutely continuous with respect to Lebesgue measure on $\R^{I_3}$. We used $\mathcal{D}_1 \cup \mathcal{D}_2$ to obtain the convex function $\hat{\ell}_n^{\mathrm{sym}}$, whose subdifferential at $z \in \R$ is
\[
\partial\hat{\ell}_n^{\mathrm{sym}}(z) = \bigl[\hat{\psi}_n^{\mathrm{anti}}(z),\hat{\psi}_n^{\mathrm{anti}}(z-)\bigr],
\]
where $\hat{\psi}_n^{\mathrm{anti}}(z-) := \lim_{z' \nearrow z}\hat{\psi}_n^{\mathrm{anti}}(z')$. Since $\hat{\psi}_n^{\mathrm{anti}}$ is decreasing, $A := \bigl\{z \in \R : \hat{\psi}_n^{\mathrm{anti}}(z) < \hat{\psi}_n^{\mathrm{anti}}(z-)\bigr\}$ is countable. Thus, applying Lemma~\ref{lem:subspace-countable} to the linear subspace $W := \bigl\{(X_i^\top\beta)_{i \in I_3} : \beta \in \R^d\bigr\}$ of dimension at most $d$, we have
\begin{equation}
\label{eq:convex-loss-kinks}
\Pr\biggl(\max_{\beta \in \R^d}\,\sum_{i \in I_3} \Ind_{\{Y_i - X_i^\top\beta \in A\}} \leq d \biggm| \mathcal{D}'\biggr) = 1.
\end{equation}
Moreover, $\hat{\beta}_n \in \argmin_{\beta \in \R^d}\hat{L}_n^{\mathrm{sym}}(\beta)$, so writing $\partial\hat{L}_n^{\mathrm{sym}}(\beta)$ for the subdifferential of the convex function $\hat{L}_n^{\mathrm{sym}}$ at $\beta \in \R^d$~\citep[p.~315]{rockafellar97convex}, we have
\[
0 \in \partial\hat{L}_n^{\mathrm{sym}}(\hat{\beta}_n) = \sum_{i \in I_3} X_i\partial\hat{\ell}_n^{\mathrm{sym}}(Y_i - X_i^\top\hat{\beta}_n),
\]
so there exist $a_i \in \partial\hat{L}_n^{\mathrm{sym}}(Y_i - X_i^\top\hat{\beta}_n)$ for $i \in I_3$ such that $\sum_{i \in I_3} X_i a_i = 0$. For each $i \in I_3$, we have $\bigl|a_i - \hat{\psi}_n^{\mathrm{anti}}(Y_i - X_i^\top \hat{\beta}_n)\bigr| \leq 2\norm{\hat{\psi}_n^{\mathrm{anti}}}_\infty \Ind_{\{Y_i - X_i^\top\hat{\beta}_n \in A\}} \leq 2(\alpha_n/\gamma_n)\,\Ind_{\{Y_i - X_i^\top\hat{\beta}_n \in A\}}$, so by~\eqref{eq:convex-loss-kinks},
\[
\Biggl|\,\sum_{i \in I_3}X_i\hat{\psi}_n^{\mathrm{anti}}(Y_i - X_i^\top\hat{\beta}_n)\Biggr| \leq \sum_{i \in I_3}\norm{X_i} \cdot \bigl|a_i - \hat{\psi}_n^{\mathrm{anti}}(Y_i - X_i^\top\hat{\beta}_n)\bigr| \leq 2d\max_{i \in I_3}\norm{X_i}\,\frac{\alpha_n}{\gamma_n}
\]
almost surely.  Assumption~\ref{ass:covariates} ensures that the right-hand side is $o_p(n^{1/2})$ as $n \to \infty$, so \textit{(b)} holds for $\hat{\beta}_n = \hat{\beta}_n^{(1)}$, and the result for $\hat{\beta}_n^\ddagger$ follows similarly by considering $\hat{L}_n^\ddagger$. 
\end{proof}

\begin{proof}[Proof of Theorem~\ref{thm:linreg-score-sym}]
For $j \in \{1,2,3\}$, Lemma~\ref{lem:linreg-sym-consistency}\emph{(b)} implies that
\begin{align}
&\frac{1}{\sqrt{n}}\sum_{i \in I_{j+2}} X_i\hat{\Psi}_{n,j}^{\mathrm{anti}}\bigl(X_i^\top(\hat{\beta}_n^{(j)} - \beta_0)\bigr) + R_{n,j}\bigl(\sqrt{n}(\hat{\beta}_n^{(j)} - \beta_0)\bigr) \notag \\
\label{eq:linreg-sym-estimating}
&\quad= \frac{1}{\sqrt{n}}\sum_{i \in I_{j+2}} X_i\hat{\psi}_{n,j}^{\mathrm{anti}}\bigl(\varepsilon_i - X_i^\top(\hat{\beta}_n^{(j)} - \beta_0)\bigr) = \frac{1}{\sqrt{n}}\sum_{i \in I_{j+2}} X_i\hat{\psi}_{n,j}^{\mathrm{anti}}(Y_i - X_i^\top\hat{\beta}_n^{(j)}) = o_p(1).
\end{align}
and Lemma~\ref{lem:linreg-sym-consistency}\emph{(a)} yields
\[
\biggl\|\frac{1}{\sqrt{n}}\sum_{i \in I_{j+2}} X_i X_i^\top(\hat{\beta}_n^{(j)} - \beta_0)\biggr\|
% \leq \biggl\|\frac{1}{n_1}\sum_{i \in I_{j+2}} X_i X_i^\top\biggr\|_{\mathrm{op}} \, \norm{\sqrt{n}(\hat{\beta}_n^{(j)} - \beta_0)}
\leq \frac{1}{n}\sum_{i \in I_{j+2}} \norm{X_i}^2 \norm{\sqrt{n}(\hat{\beta}_n^{(j)} - \beta_0)} = O_p(1).
\]
Moreover, by~\ref{ass:covariates} and Lemma~\ref{lem:linreg-sym-consistency}\textit{(a)}, $\max_{i \in I_{j+2}}|X_i^\top (\hat{\beta}_n^{(j)} - \beta_0)|\,\alpha_n/\gamma_n = o_p(1)$ as $n \to \infty$, so arguing similarly to~\eqref{eq:linreg-sym-Psi}, we have
\begin{align}
\label{eq:PsiBound}
\frac{1}{\sqrt{n}}\sum_{i \in I_{j+2}} X_i\hat{\Psi}_{n,j}^{\mathrm{anti}}\bigl(X_i^\top(\hat{\beta}_n^{(j)} - \beta_0)\bigr)
= \frac{i^*(p_0)}{3}\,\E(X_1 X_1^\top) \sqrt{n}(\hat{\beta}_n^{(j)} - \beta_0) + o_p(1).
\end{align}
From~\eqref{eq:PsiBound} and~\eqref{eq:linreg-sym-estimating}, followed by Lemma~\ref{lem:linreg-sym-equicontinuity} and Lemma~\ref{lem:linreg-sym-consistency}\textit{(a)}, and then~\eqref{eq:linreg-sym-clt}, we deduce that 
\begin{align*}
\frac{i^*(p_0)}{3}\E(X_1 X_1^\top) \sqrt{n}(\hat{\beta}_n^{(j)} - \beta_0) &= - R_{n,j}\bigl(\sqrt{n}(\hat{\beta}_n^{(j)} - \beta_0)\bigr) + o_p(1) \\
&= -R_{n,j}(0) + o_p(1) = -\frac{1}{\sqrt{n}}\sum_{i \in I_{j+2}} X_i\psi_0^*(\varepsilon_i) + o_p(1)
\end{align*}
for each $j \in \{1,2,3\}$.  Hence, by the central limit theorem,
\begin{align*}
\sqrt{n}(\hat{\beta}_n^\dagger - \beta_0) = \sum_{j=1}^3 \frac{\sqrt{n}}{3}(\hat{\beta}_n^{(j)} - \beta_0) &= -\frac{\{\E(X_1 X_1^\top)\}^{-1}}{i^*(p_0)\sqrt{n}}\sum_{j=1}^3 \sum_{i \in I_{j+2}} X_i\psi_0^*(\varepsilon_i) + o_p(1) \\
&= -\frac{\{\E(X_1 X_1^\top)\}^{-1}}{i^*(p_0)\sqrt{n}}\sum_{i=1}^n X_i\psi_0^*(\varepsilon_i) + o_p(1) \cvd N_d\biggl(0, \frac{\{\E(X_1 X_1^\top)\}^{-1}}{i^*(p_0)}\biggr)
\end{align*}
as $n \to \infty$. By analogous reasoning based on Lemmas~\ref{lem:linreg-sym-equicontinuity} and~\ref{lem:linreg-sym-consistency}, $\sqrt{n}(\hat{\beta}_n^\ddagger - \beta_0)$ has the same limiting distribution; note in particular that $R_n^\ddagger := \sum_{j=1}^3 R_{n,j}^\ddagger$ satisfies $\sup_{\norm{\beta} \in B_M} \norm{R_n^\ddagger(b) - R_n^\ddagger(0)} = o_p(1)$.
\end{proof}

\subsubsection{Proof of Theorem~\ref{thm:linreg-score-intercept}}
\label{subsec:linreg-intercept-proofs}

For $n \in \N$ and $j \in \{1,2,3\}$, let $\check{\psi}_{n,j}(\cdot) := \hat{\psi}_{n,j}(\cdot + \mu_0)$ and $\bar{\delta}_{n,j} := \bar{X}_{n,j}^\top(\bar{\theta}_n^{(j)} - \theta_0)$, and define $\tilde{W}_i := \tilde{X}_i - \bar{X}_{n,j}$ for $i \in I_{j+2}$. The pilot residuals in Step~$3'$ can be written as $\hat{\varepsilon}_i = \varepsilon_i + \mu_0 + X_i^\top b_n$ for $i \in I_{j+1}$, where $b_n := (\theta_0 - \bar{\theta}_n^{(j)}, 0) = O_p(n^{-1/2})$ under the hypotheses of Theorem~\ref{thm:linreg-score-intercept}, so the conclusions of Lemma~\ref{lem:proj-score-consistency} hold with $\mu = \mu_0$ and $\hat{\Psi}_{n,j}(t) = \int_\R \hat{\psi}_{n,j}(z - t)\,p_0(z - \mu_0)\,dz = \int_\R \check{\psi}_{n,j}(z - t)\,p_0(z)\,dz$ for $t \in \R$. In particular, $\int_\R (\check{\psi}_{n,j} - \psi_0^*)^2\,p_0 \cvp 0$ by~\eqref{eq:proj-score-consistency}.

The first two lemmas below are similar to those in the previous subsection, and enable us to prove the limiting distribution~\eqref{eq:thetahat} of our estimate of $\theta_0$ in Step~$4'$, even in the absence of condition~\ref{ass:zeta}. By replacing $X_i$ with $\tilde{W}_i$ and $\hat{\psi}_{n,j}^{\mathrm{anti}}$ with $\check{\psi}_{n,j}$ in the proof of Lemma~\ref{lem:linreg-sym-equicontinuity}, we obtain the following result.

\begin{lemma}
\label{lem:linreg-intercept-equicontinuity}
For $b \equiv (s,a) \in \R^{d-1} \times \R$, $n \in \N$ and $j \in \{1,2,3\}$, define
\[
\tilde{R}_{n,j}(b) \equiv \tilde{R}_{n,j}(s,a) := \frac{1}{\sqrt{n}}\sum_{i \in I_{j+2}} \tilde{W}_i\Bigl\{\check{\psi}_{n,j}\Bigl(\varepsilon_i - \frac{a + \tilde{W}_i^\top s}{\sqrt{n}}\Bigr) - \hat{\Psi}_{n,j}\Bigl(\frac{a + \tilde{W}_i^\top s}{\sqrt{n}}\Bigr)\Bigr\}.
\] 
If the hypotheses of Theorem~\ref{thm:linreg-score-intercept} are satisfied, then for every $j$ and $M > 0$, we have
\[
\sup_{b \in B_M}\norm{\tilde{R}_{n,j}(b) - \tilde{R}_{n,j}(0)} = o_p(1)
\]
as $n \to \infty$, where $B_M = \{b \in \R^d : \norm{b} \leq M\}$.
\end{lemma}

\begin{lemma}
\label{lem:linreg-intercept-consistency}
For $n \in \N$, $j \in \{1,2,3\}$ and $t > 0$, we have
\begin{align}
\label{eq:linreg-intercept-sphere-j}
\bigl\{\norm{\hat{\theta}_n^{(j)} - \theta_0} > t\bigr\} &\subseteq \biggl\{\inf_{u \in S^{d-2}}u^\top \sum_{i \in I_{j+2}} \tilde{W}_i\check{\psi}_{n,j}\bigl(\varepsilon_i - \bar{\delta}_{n,j} - t\tilde{W}_i^\top u\bigr) \leq 0\biggr\}, \\
\label{eq:linreg-intercept-sphere}
\bigl\{\norm{\hat{\theta}_n^\ddagger - \theta_0} > t\bigr\} &\subseteq \biggl\{\inf_{u \in S^{d-2}}u^\top \sum_{j=1}^3 \sum_{i \in I_{j+2}} \tilde{W}_i\check{\psi}_{n,j}\bigl(\varepsilon_i - \bar{\delta}_{n,j} - t\tilde{W}_i^\top u\bigr) \leq 0\biggr\}.
\end{align}
Consequently, under the hypotheses of Lemma~\ref{lem:linreg-intercept-equicontinuity}, we have
\begin{enumerate}[label=(\alph*)]
\item $\sqrt{n}(\hat{\theta}_n^{(j)} - \theta_0) = O_p(1)$ and $\sqrt{n}(\hat{\theta}_n^\ddagger - \theta_0) = O_p(1)$;
\item $\displaystyle\frac{1}{\sqrt{n}}\sum_{i \in I_{j+2}}\!\!\tilde{W}_i\check{\psi}_{n,j}\bigl(\varepsilon_i - \bar{\delta}_{n,j} - \tilde{W}_i^\top(\hat{\theta}_n^{(j)} - \theta_0)\bigr) = o_p(1)$ and \\
$\displaystyle\frac{1}{\sqrt{n}}\sum_{j=1}^3 \sum_{i \in I_{j+2}}\!\!\tilde{W}_i\check{\psi}_{n,j}\bigl(\varepsilon_i - \bar{\delta}_{n,j} - \tilde{W}_i^\top(\hat{\theta}_n^\ddagger - \theta_0)\bigr) = o_p(1)$.
\end{enumerate}
\end{lemma}

\begin{proof}
Similarly to the proof of Lemma~\ref{lem:linreg-sym-consistency}, it suffices to consider $j = 1$, and we write $\check{\psi}_n \equiv \check{\psi}_{n,1}$ and $\hat{\Psi}_n \equiv \hat{\Psi}_{n,1}$. By definition, $\hat{\theta}_n \equiv \hat{\theta}_n^{(1)}$ minimises the convex function 
\[
\theta \mapsto \sum_{i \in I_3} \hat{\ell}_{n,1}\bigl(Y_i - \bar{X}_{n,1}^\top\bar{\theta}_n^{(1)} - \tilde{W}_i^\top\theta\bigr) = \sum_{i \in I_3} \check{\ell}_n\bigl(\varepsilon_i - \bar{\delta}_n - \tilde{W}_i^\top(\theta - \theta_0)\bigr) =: \check{L}_n(\theta) 
\]
over $\R^{d-1}$, where $\check{\ell}_n \equiv \hat{\ell}_{n,1}(\cdot + \mu_0)$ and $\bar{\delta}_n \equiv \bar{\delta}_{n,1}$. Thus, $h \mapsto \check{L}_n\bigl((1 - h)\hat{\theta}_n + h\theta_0\bigr)$ is convex and increasing on $[0,\infty)$, so 
\[
(\hat{\theta}_n - \theta_0)^\top\sum_{i \in I_3} \tilde{W}_i\check{\psi}_n\bigl(\varepsilon_i - \bar{\delta}_n - (1-h)\tilde{W}_i^\top(\hat{\theta}_n - \theta_0)\bigr) \leq 0
\]
for $h \geq 0$. This means that if $\norm{\hat{\theta}_n - \theta_0} > t > 0$, then we can take $h = 1 - t/\norm{\hat{\theta}_n - \theta_0} \in [0,1]$ and argue as in~\eqref{eq:linreg-sym-sphere-j-1} to obtain~\eqref{eq:linreg-intercept-sphere-j}. By similar reasoning based on $\theta \mapsto \sum_{j=1}^3 \sum_{i \in I_{j+2}} \check{\ell}_n\bigl(\varepsilon_i - \bar{\delta}_{n,j} - \tilde{W}_i^\top(\theta - \theta_0)\bigr)$, we obtain~\eqref{eq:linreg-intercept-sphere}.

\medskip
\noindent
\textit{(a)} By~\ref{ass:covariates}, $\Cov(\tilde{X}_1)$ is positive definite, so its minimum eigenvalue $\lambda_{\min}$ is strictly positive. Since $\bar{\theta}_n^{(1)} - \theta_0 = O_p(n^{-1/2})$, it follows from~\ref{ass:covariates} that $\bar{\delta}_n = O_p(n^{-1/2}) = o_p(\gamma_n/\alpha_n)$ and hence $V_n := |\bar{\delta}_n| + Mn^{-1/2}\max_{i \in I_3}\norm{\tilde{W}_i} = o_p(\gamma_n/\alpha_n)$ as $n \to \infty$. Moreover, $\sum_{i \in I_3} \tilde{W}_i = 0$ and $n^{-1}\sum_{i \in I_3}\tilde{W}_i \tilde{W}_i^\top \cvp \Cov(\tilde{X}_1)/3$. Arguing similarly to~\eqref{eq:linreg-sym-Psi}, we deduce from~\eqref{eq:proj-score-consistency-2} and Lemma~\ref{lem:cvp-plugin} that
\begin{align}
\sup_{u \in S^{d-2}}&\biggl\|\,\sum_{i \in I_3} \frac{\tilde{W}_i}{\sqrt{n}}\hat{\Psi}_n\Bigl(\bar{\delta}_n + \frac{M \tilde{W}_i^\top u}{\sqrt{n}}\Bigr) - \frac{Mi^*(p_0)}{3}\Cov(\tilde{X}_1)u\,\biggr\| \notag \\
&\leq \sup_{u \in S^{d-2}}\biggl\|\,\sum_{i \in I_3} \frac{\tilde{W}_i}{\sqrt{n}}\biggl\{\hat{\Psi}_n\Bigl(\bar{\delta}_n + \frac{M \tilde{W}_i^\top u}{\sqrt{n}}\Bigr) - \hat{\Psi}_n(\bar{\delta}_n) - \frac{Mi^*(p_0)\tilde{W}_i^\top u}{\sqrt{n}}\biggr\}\,\biggr\| \notag\\
&\hspace{6.5cm}+ Mi^*(p_0)\,\biggl\|\sum_{i \in I_3}\frac{\tilde{W}_i \tilde{W}_i^\top}{n} - \frac{\Cov(\tilde{X}_1)}{3}\,\biggr\|_{\mathrm{op}} \notag \\
\label{eq:linreg-intercept-Psi}
&\leq \frac{M\sum_{i \in I_3}\norm{\tilde{W}_i}^2}{n} \cdot \Delta_n(V_n) + Mi^*(p_0)\,\biggl\|\sum_{i \in I_3}\frac{\tilde{W}_i \tilde{W}_i^\top}{n} - \frac{\Cov(\tilde{X}_1)}{3}\,\biggr\|_{\mathrm{op}} = o_p(1).
\end{align}
Therefore, by Lemma~\ref{lem:linreg-intercept-equicontinuity},
\begin{align}
\inf_{u \in S^{d-2}}u^\top \sum_{i \in I_3} \frac{\tilde{W}_i}{\sqrt{n}}\check{\psi}_n\Bigl(\varepsilon_i - \bar{\delta}_n - \frac{M \tilde{W}_i^\top u}{\sqrt{n}}\Bigr) &= \inf_{u \in S^{d-2}} u^\top\biggl\{\tilde{R}_{n,1}(Mu,\sqrt{n}\bar{\delta}_n) + \sum_{i \in I_3}\!\frac{\tilde{W}_i}{\sqrt{n}}\hat{\Psi}_n\Bigl(\bar{\delta}_n + \frac{M \tilde{W}_i^\top u}{\sqrt{n}}\Bigr)\biggr\} \notag \\
&= \inf_{u \in S^{d-2}} u^\top\Bigl(\tilde{R}_{n,1}(0,0) + \frac{Mi^*(p_0)}{3}\Cov(\tilde{X}_1)u\Bigr) + o_p(1) \notag \\
\label{eq:linreg-intercept-consistency}
&\geq -\norm{\tilde{R}_{n,1}(0,0)} + \frac{M\lambda_{\min}i^*(p_0)}{3} + o_p(1).
\end{align}
Writing $\mathcal{D}' = \mathcal{D}_1 \cup \mathcal{D}_2 \cup \{X_i : i \in I_3\}$, we have
\[
\E\Bigl(\frac{1}{\sqrt{n}}\sum_{i \in I_3} \tilde{W}_i\bigl(\check{\psi}_n(\varepsilon_i) - \psi_0^*(\varepsilon_i)\bigr) \Bigm| \mathcal{D}'\Bigr) = \frac{1}{\sqrt{n}}\sum_{i \in I_3} \tilde{W}_i \int_\R (\check{\psi}_n - \psi_0^*)\,p_0 = 0,
\]
which together with Lemma~\ref{lem:proj-score-consistency} implies that
\begin{align*}
\E\Bigl(\Bigl\|\frac{1}{\sqrt{n}}\sum_{i \in I_3} \tilde{W}_i\bigl(\check{\psi}_n(\varepsilon_i) - \psi_0^*(\varepsilon_i)\bigr)\Bigr\|^2 \Bigm| \mathcal{D}'\Bigr) &\leq \frac{1}{n}\sum_{i \in I_3}\norm{\tilde{W}_i}^2\,\E\bigl\{(\check{\psi}_n - \psi_0^*)^2(\varepsilon_i) \!\bigm|\! \mathcal{D}'\bigr\} \\
&= \frac{1}{n}\sum_{i \in I_3}\norm{\tilde{W}_i}^2 \int_\R (\check{\psi}_n - \psi_0^*)^2\,p_0 \cvp 0.
\end{align*}
Moreover, by Lemma~\ref{lem:psi0-star}, $n^{-1}\sum_{i \in I_3}\psi_0^*(\varepsilon_i) \cvp \int_\R \psi_0^*\,p_0 = 0$, so writing $\tilde{m} = \E(\tilde{X}_1)$, we deduce by the central limit theorem that
\begin{align}
\tilde{R}_{n,1}(0,0) = \frac{1}{\sqrt{n}}\sum_{i \in I_3} \tilde{W}_i\check{\psi}_n(\varepsilon_i) 
% &= \frac{1}{\sqrt{n}}\sum_{i \in I_3} \tilde{W}_i\psi_0^*(\varepsilon_i) + o_p(1) \notag \\
&= \frac{1}{\sqrt{n}}\sum_{i \in I_3} (\tilde{X}_i - \tilde{m})\psi_0^*(\varepsilon_i) - \sqrt{n}(\bar{X}_{n,1} - \tilde{m}) \sum_{i \in I_3} \frac{\psi_0^*(\varepsilon_i)}{n} + o_p(1) \notag \\
\label{eq:linreg-intercept-clt}
&= \frac{1}{\sqrt{n}}\sum_{i \in I_3} (\tilde{X}_i - \tilde{m})\psi_0^*(\varepsilon_i) + o_p(1) = O_p(1)
\end{align}
as $n \to \infty$. We have $M\lambda_{\min}i^*(p_0) > 0$, so by~\eqref{eq:linreg-intercept-sphere-j} and~\eqref{eq:linreg-intercept-consistency}, $\limsup_{n \to \infty}\Pr(\norm{\hat{\theta}_n - \theta_0} \geq M/\sqrt{n}) \to 0$ as $M \to \infty$, so $\sqrt{n}(\hat{\theta}_n - \theta_0) = O_p(1)$ as $n \to \infty$. The $\sqrt{n}$-consistency of $\hat{\theta}_n^\ddagger$ follows similarly from~\eqref{eq:linreg-intercept-sphere}.

\medskip
\noindent
\textit{(b)} Arguing similarly to the proof of Lemma~\ref{lem:linreg-sym-consistency}, we deduce from Lemma~\ref{lem:subspace-countable} that
\[
\Pr\biggl(\max_{\theta \in \R^{d-1}}\,\sum_{i \in I_3} \Ind_{\{\varepsilon_i - \bar{\delta}_n - \tilde{W}_i^\top(\theta - \theta_0) \in A\}} \leq d - 1 \biggm| \mathcal{D}'\biggr) = 1,
\]
where $A$ denotes the countable set of discontinuities of the decreasing function $\check{\psi}_n$. Thus, by~\ref{ass:covariates} and the fact that $\hat{\theta}_n \in \argmin_{\theta \in \R^{d-1}} \check{L}_n(\theta)$, we have
\[
\Biggl|\,\sum_{i \in I_3}\tilde{W}_i\check{\psi}_n\bigl(\varepsilon_i - \bar{\delta}_n - \tilde{W}_i^\top(\hat{\theta}_n - \theta_0)\bigr)\Biggr| \leq (d - 1)\max_{i \in I_3}\norm{\tilde{W}_i}\,\frac{\alpha_n}{\gamma_n} = o_p(n^{1/2})
\]
as $n \to \infty$. The proof of the analogous conclusion for $\hat{\theta}_n^\ddagger$ is similar.
\end{proof}

To derive the limiting behaviour of the intercept estimate~\eqref{eq:intercept-est}, we require two additional lemmas whose proofs are slightly more straightforward because unlike the antitonic score estimates $\hat{\psi}_{n,j}$ above, the function $\zeta$ is not random. This means that sample splitting is not required here. For $i \in [n]$, let $\breve{W}_i := \tilde{X}_i - \tilde{m}$ and $W_i := (\breve{W}_i,1)$ and note that since $\E(\norm{X_1}^2) < \infty$, we also have $\E(\norm{W_1}^2) < \infty$. By a union bound, $n^{-1}\,\Pr(\max_{i \in [n]} \norm{W_i}^2 > t) \leq \Pr(\norm{W_1}^2 > t)$ for all $t > 0$. It follows from the dominated convergence theorem that
\begin{equation}
\label{eq:expected-max}
\frac{1}{n}\E\Bigl(\max_{i \in [n]}\norm{W_i}^2\Bigr) = \int_0^\infty \frac{1}{n}\Pr\Bigl(\max_{i \in [n]}\norm{W_i}^2 > t\Bigr)\,dt \to 0
\end{equation}
as $n \to \infty$; see also~\citet[Exercise~10.6.4\emph{(a)}]{samworth24modern}. Thus, $\max_{i \in [n]}\norm{W_i} = o_p(n^{1/2})$.

\begin{lemma}
\label{lem:linreg-zeta-equicontinuity}
Suppose that $\zeta \colon \R \to \R$ satisfies~\emph{\ref{ass:zeta}} for some $t_0 > 0$, and that $\E\zeta(\varepsilon_1) = 0$. Define $\Lambda \colon \R \to \R$ by
\begin{equation}
\label{eq:zeta-Lambda}
\Lambda(t) :=
\begin{cases}
\int_\R \zeta(z - t)\,p_0(z)\,dz \quad&\text{if }z \in [-t_0,t_0] \\
0 \quad&\text{otherwise}.
\end{cases}
\end{equation}
For $n \in \N$ and $b \in \R^d$, let
\[
R_n^\zeta(b) \equiv R_n^\zeta(s,a) := \frac{1}{\sqrt{n}}\sum_{i=1}^n \Bigl\{\zeta\Bigl(\varepsilon_i - \frac{a + \breve{W}_i^\top s}{\sqrt{n}}\Bigr) - \Lambda\Bigl(\frac{a + \breve{W}_i^\top s}{\sqrt{n}}\Bigr)\Bigr\}.
\]
Then for each $M > 0$, we have
\[
\sup_{b \in B_M}\norm{R_n^\zeta(b) - R_n^\zeta(0)} = o_p(1).
\]
\end{lemma}
\begin{proof}
This requires only a few minor modifications to the proof of Lemma~\ref{lem:linreg-intercept-equicontinuity}. Let $\mathcal{W} := (W_1,\dotsc,W_n)$, which is independent of $\varepsilon_1,\dotsc,\varepsilon_n$. Define the event $\Omega_n := \{\max_{i \in [n]}\norm{W_i} \leq t_0 n^{1/2}/M\}$, which as noted below~\eqref{eq:expected-max} has probability tending to 1 as $n \to \infty$. By~\ref{ass:zeta}, $\zeta$ satisfies the hypothesis of Lemma~\ref{lem:varPsi-deriv}, so $\Lambda$ is well-defined. Moreover, for each fixed $b \in B_M$, we have $\E\bigl\{\zeta(\varepsilon_i - W_i^\top bn^{-1/2})\!\bigm|\!\mathcal{W}\bigr\} = \Lambda(W_i^\top bn^{-1/2})$ for all $i \in [n]$ on the event $\Omega_n$, so by~\eqref{eq:psi-shift-limit} in Lemma~\ref{lem:varPsi-deriv},
\begin{align}
\E\bigl(\{R_n^\zeta(b) - R_n^\zeta(0)\}^2\!\bigm|\!\mathcal{W}\bigr) \cdot \Ind_{\Omega_n} &= \frac{1}{n}\sum_{i=1}^n \Var\bigl(\zeta(\varepsilon_i - W_i^\top bn^{-1/2}) - \zeta(\varepsilon_i)\!\bigm|\!\mathcal{W}\bigr) \cdot \Ind_{\Omega_n} \notag \\
\label{eq:zeta-shift}
&\leq \max_{i \in [n]}\int_\R\,\bigl(\zeta(z - W_i^\top bn^{-1/2}) - \zeta(z)\bigr)^2\,p_0(z)\,dz \cdot \Ind_{\Omega_n} \cvp 0
\end{align}
as $n \to \infty$. Thus, $\E\bigl(\{R_n^\zeta(b) - R_n^\zeta(0)\}^2\!\bigm|\!\mathcal{W}\bigr) = o_p(1)$ for each $b \in B_M$. Fix $\epsilon \in (0,1)$. As in the proof of Lemma~\ref{lem:linreg-sym-equicontinuity}, we can find a $\epsilon M$-cover $B_{M,\epsilon}$ of $B_M$ with $|B_{M,\epsilon}| \leq 3\epsilon^{-d}$, and deduce that $\max_{b \in B_{M,\epsilon}} \norm{R_n^\zeta(b) - R_n^\zeta(0)} \cvp 0$. For $b' \in B_{M,\epsilon}$, Define
\[
r_n^\zeta(b') := \frac{1}{\sqrt{n}} \sum_{i=1}^n \biggl\{\zeta\biggl(\varepsilon_i - \frac{W_i^\top b' + \epsilon M\norm{W_i}}{\sqrt{n}}\biggr) - \zeta\biggl(\varepsilon_i - \frac{W_i^\top b' - \epsilon M\norm{W_i}}{\sqrt{n}}\biggr)\biggr\}
\]
Similarly to~\eqref{eq:zeta-shift}, it follows from~\eqref{eq:psi-shift-limit} that
\begin{align*}
\Var\bigl(r_n^\zeta(b')\!\bigm|\!\mathcal{W}\bigr) \cdot \Ind_{\Omega_n} \leq \max_{i \in [n]} \int_\R\biggl\{\zeta\biggl(z - \frac{W_i^\top b' + \epsilon M\norm{W_i}}{\sqrt{n}}\biggr) - \zeta\biggl(z - \frac{W_i^\top b' - \epsilon M\norm{W_i}}{\sqrt{n}}\biggr)\biggr\}^2 p_0(z)\,dz \cvp 0.
\end{align*}
Therefore, $\max_{b' \in B_{M,\epsilon}} \bigl\{r_n^\zeta(b') - \E\bigl(r_n^\zeta(b')\!\bigm|\!\mathcal{W}\bigr)\bigr\} = o_p(1)$. Furthermore, by~\ref{ass:zeta} and~\eqref{eq:varPsi-deriv} in Lemma~\ref{lem:varPsi-deriv}, $\Lambda$ is increasing on $[-t_0,t_0]$ with $\Lambda'(0) = -\int_\R p_0\,d\zeta \in (0,\infty)$ and
\begin{equation}
\label{eq:zeta-Lambda-deriv}
\Delta_n^\zeta(v) := \sup_{\substack{s,t \in [-v,v] : \\ s \neq t}}\,\biggl|\frac{\Lambda(t) - \Lambda(s)}{t - s} + \int_\R p_0\,d\zeta\biggr| \to 0
\end{equation}
as $v \to 0$. Consequently, on $\Omega_n$, we have
\begin{align*}
\max_{b' \in B_{M,\epsilon}}\E\bigl(r_n^\zeta(b') \bigm| \mathcal{W}\bigr) &= \max_{b' \in B_{M,\epsilon}}\,\frac{1}{\sqrt{n}}\sum_{i=1}^n \biggl\{\Lambda\biggl(\frac{W_i^\top b' + \epsilon M\norm{W_i}}{\sqrt{n}}\biggr) - \Lambda\biggl(\frac{W_i^\top b' - \epsilon M\norm{W_i}}{\sqrt{n}}\biggr)\biggr\} \\
&\leq \sum_{i=1}^n\frac{2\epsilon M\norm{W_i}}{n}\Bigl\{\Lambda'(0) + \Delta_n^\zeta\Bigl(\frac{2M\max_{i \in [n]}\norm{W_i}}{\sqrt{n}}\Bigr)\Bigr\} \\
&= 2\epsilon M\Lambda'(0)\,\E(\norm{W_1}) + o_p(1),
\end{align*}
and hence $\max_{b' \in B_{M,\epsilon}}\E\bigl(r_n^\zeta(b')\!\bigm|\!\mathcal{W}\bigr) = O_p(1)$. On $\Omega_n$, we conclude as in the proof of Lemma~\ref{lem:linreg-sym-equicontinuity} that
\begin{align*}
&\sup_{b \in B_M}\norm{R_n(b) - R_n(0)} \\
&\quad \leq \max_{b' \in B_{M,\epsilon}} \bigl\{r_n^\zeta(b') - \E\bigl(r_n^\zeta(b') \bigm| \mathcal{D}'\bigr)\bigr\} + 2\max_{b' \in B_{M,\epsilon}}\E\bigl(r_n^\zeta(b') \bigm| \mathcal{D}'\bigr) + \max_{b' \in B_{M,\epsilon}}\norm{R_n^\zeta(b') - R_n^\zeta(0)} \\
&\quad = 2\epsilon M\Lambda'(0)\,\E(\norm{W_1}) + o_p(1).
\end{align*}
Since $\epsilon \in (0,1)$ was arbitrary and $\Pr(\Omega_n) \to 1$, the result follows.
\end{proof}
\begin{lemma}
\label{lem:linreg-zeta-consistency}
Suppose that $\zeta \colon \R \to \R$ satisfies~\emph{\ref{ass:zeta}}. For $n \in \N$, let $\hat{\mu}_n^\zeta$ be the intercept estimate in~\eqref{eq:intercept-est} based on $\hat{\theta}_n \in \{\hat{\theta}_n^\dagger,\hat{\theta}_n^\ddagger\}$ and define $\hat{\delta}_n := \hat{\mu}_n^\zeta - \mu_0 + \tilde{m}^\top(\hat{\theta}_n - \theta_0)$. Then for $t > 0$, we have
\begin{align}
\label{eq:linreg-zeta-event}
\bigl\{|\hat{\delta}_n| > t\bigr\} &\subseteq \biggl\{\min_{u \in \{-1,1\}} u\sum_{i=1}^n \zeta\bigl(\varepsilon_i - \breve{W}_i^\top(\hat{\theta}_n - \theta_0) - tu\bigr) \leq 0\biggr\}.
\end{align}
Moreover, if $\E\zeta(\varepsilon_1) = 0$, then $\sqrt{n}\hat{\delta}_n  = O_p(1)$ and $n^{-1/2}\sum_{i=1}^n \zeta\bigl(\varepsilon_i - \breve{W}_i^\top(\hat{\theta}_n - \theta_0) - \hat{\delta}_n\bigr) = o_p(1)$.
\end{lemma}

\begin{proof}
By~\eqref{eq:intercept-est}, $\hat{\mu}_n^\zeta = \hat{\delta}_n + \mu_0 - \tilde{m}^\top(\hat{\theta}_n - \theta_0)$ minimises
\[
\sum_{i=1}^n L_\zeta(Y_i - \tilde{X}_i^\top\hat{\theta}_n - \mu) = \sum_{i=1}^n L_\zeta\bigl(\varepsilon_i - (\breve{W}_i + \tilde{m})^\top(\hat{\theta}_n - \theta_0) - (\mu - \mu_0)\bigr)
\]
over $\mu \in \R$, so $\hat{\delta}_n$ minimises the convex function 
\begin{equation}
\label{eq:Gn-alpha}
\alpha \mapsto \sum_{i=1}^n L_\zeta\bigl(\varepsilon_i - \breve{W}_i^\top(\hat{\theta}_n - \theta_0) - \alpha\bigr) =: G_n(\alpha)
\end{equation}
over $\R$. For $t > 0$, if $\hat{\delta}_n > t$, then $0 \geq G_n^{(\mathrm{R})}(t) = \sum_{i=1}^n \zeta\bigl(\varepsilon_i - \breve{W}_i^\top(\hat{\theta}_n - \theta_0) - t\bigr)$, whereas if $\hat{\delta}_n < -t$, then $0 \leq G_n^{(\mathrm{R})}(-t) = \sum_{i=1}^n \zeta\bigl(\varepsilon_i - \breve{W}_i^\top(\hat{\theta}_n - \theta_0) + t\bigr)$. This yields~\eqref{eq:linreg-zeta-event}.

Next, fix $M > 0$. By Lemma~\ref{lem:linreg-intercept-consistency}\textit{(a)}, $\sqrt{n}(\hat{\theta}_n - \theta_0) = O_p(1)$, which together with~\eqref{eq:expected-max} implies that $V_n := Mn^{-1/2} + \max_{i \in [n]}\norm{W_i}\,|\hat{\theta}_n - \theta_0| = o_p(1)$. Since $\Lambda(0) = \E\zeta(\varepsilon_1) = 0$ and $n^{-1/2}\sum_{i=1}^n \breve{W}_i = n^{-1/2}\sum_{i=1}^n (\tilde{X}_i - \tilde{m}) = O_p(1)$, we deduce from~\eqref{eq:zeta-Lambda-deriv} and Lemma~\ref{lem:cvp-plugin} that
\begin{align}
% NB: Neyman orthogonality is crucial here!
\max_{u \in \{-1,1\}}\,&\biggl|\sum_{i=1}^n \frac{1}{\sqrt{n}}\Lambda\Bigl(\breve{W}_i^\top(\hat{\theta}_n - \theta_0) + \frac{Mu}{\sqrt{n}}\Bigr) + Mu\int_\R p_0\,d\zeta\biggr| \notag \\
&\leq \max_{u \in \{-1,1\}}\,\sum_{i=1}^n \frac{1}{\sqrt{n}} \,\biggl|\Lambda\Bigl(\breve{W}_i^\top(\hat{\theta}_n - \theta_0) + \frac{Mu}{\sqrt{n}}\Bigr) - \Lambda'(0)\Bigl(\breve{W}_i^\top(\hat{\theta}_n - \theta_0) + \frac{Mu}{\sqrt{n}}\Bigr)\biggr| + O_p(n^{-1/2}) \notag \\
\label{eq:Lambda-deriv-1}
&\leq \sum_{i=1}^n \frac{\sqrt{n}\norm{W_i}\,|\hat{\theta}_n - \theta_0| + M}{n} \cdot \Delta_n^\zeta(V_n) + O_p(n^{-1/2}) = o_p(1).
\end{align}
By Lemma~\ref{lem:linreg-zeta-equicontinuity} and~\eqref{eq:Lambda-deriv-1},
\begin{align*}
\min_{u \in \{-1,1\}} &u\sum_{i=1}^n \zeta\Bigl(\varepsilon_i - \breve{W}_i^\top(\hat{\theta}_n - \theta_0) - \frac{Mu}{\sqrt{n}}\Bigr) \\
&= \min_{u \in \{-1,1\}} u\,\biggl\{R_n^\zeta\bigl(\sqrt{n}(\hat{\theta}_n - \theta_0),Mu\bigr) + \frac{1}{\sqrt{n}}\sum_{i=1}^n \Lambda\Bigl(\breve{W}_i^\top(\hat{\theta}_n - \theta_0) + \frac{Mu}{\sqrt{n}}\Bigr)\biggr\} \\
&= \min_{u \in \{-1,1\}} u\Bigl(R_n^\zeta(0,0) - Mu\int_\R p_0\,d\zeta\Bigr) + o_p(1) = -M\int_\R p_0\,d\zeta - |R_n^\zeta(0,0)| + o_p(1).
\end{align*}
Now $R_n^\zeta(0,0) = n^{-1/2}\sum_{i=1}^n \zeta(\varepsilon_i) = O_p(1)$ and $-\int_\R p_0\,d\zeta \in (0,\infty)$, so by~\eqref{eq:linreg-zeta-event}, $\limsup_{n \to \infty}\Pr(\hat{\delta}_n > M/\sqrt{n}) \to 0$ as $M \to \infty$, i.e.~$\sqrt{n}\hat{\delta}_n = O_p(1)$. 

For the final assertion, we have $\max_{i \in [n]} |\breve{W}_i^\top(\hat{\theta}_n - \theta_0) + \hat{\delta}_n| = o_p(1)$, so letting $\xi(z) := \lim_{z' \nearrow z}\zeta(z') - \zeta(z) = \max\{|a - b| : a,b \in \partial L_\zeta(z)\}$ for $z \in \R$, we deduce that
\[
\max_{i \in [n]}\,\xi\bigl(\varepsilon_i - \breve{W}_i^\top(\hat{\theta}_n - \theta_0) - \hat{\delta}_n\bigr) \leq \max_{i \in [n]}\,\bigl(|\zeta(\varepsilon_i - t_0)| + |\zeta(\varepsilon_i + t_0)|\bigr)
\]
with probability tending to 1 as $n \to \infty$. By~\ref{ass:zeta} and~\eqref{eq:expected-max}, it follows that
\begin{equation}
\label{eq:d-zeta-max}
\max_{i \in [n]}\,\xi\bigl(\varepsilon_i - \breve{W}_i^\top(\hat{\theta}_n - \theta_0) - \hat{\delta}_n\bigr) = o_p(n^{1/2}).
\end{equation}
Moreover, the set $A := \{z \in \R : \xi(z) > 0\}$ of discontinuities of $\zeta$ is countable, and $(\varepsilon_1,\dotsc,\varepsilon_n)$ has a density with respect to Lebesgue measure on $\R^n$, so as in~\eqref{eq:convex-loss-kinks}, we deduce from Lemma~\ref{lem:subspace-countable} that
\[
\Pr\biggl(\max_{\mu \in \R,\,\theta \in \R^{d-1}}\,\sum_{i=1}^n \Ind_{\{\varepsilon_i - \mu - \breve{W}_i^\top\theta \in A\}} \leq d \biggm| \mathcal{W}\biggr) = 1
\]
almost surely. Since $\hat{\delta}_n$ minimises the convex function $\alpha \mapsto G_n(\alpha)$ in~\eqref{eq:Gn-alpha} over $\R$, we must have $0 \in \partial G_n(\hat{\delta}_n) = -\sum_{i=1}^n \partial L_\zeta\bigl(\varepsilon_i - \breve{W}_i^\top(\hat{\theta}_n - \theta_0) - \hat{\delta}_n\bigr)$. Therefore,
\[
\Biggl|\,\sum_{i=1}^n \zeta\bigl(\varepsilon_i - \breve{W}_i^\top(\hat{\theta}_n - \theta_0) - \hat{\delta}_n\bigr)\Biggr| \leq \sum_{i=1}^n \xi\bigl(\varepsilon_i - \breve{W}_i^\top(\hat{\theta}_n - \theta_0) - \hat{\delta}_n\bigr) \leq d\max_{i \in [n]}\,\xi\bigl(\varepsilon_i - \breve{W}_i^\top(\hat{\theta}_n - \theta_0) - \hat{\delta}_n\bigr)
\]
almost surely, so the conclusion follows by~\eqref{eq:d-zeta-max}.
\end{proof}
\begin{proof}[Proof of Theorem~\ref{thm:linreg-score-intercept}]
By Lemma~\ref{lem:p0-star}, $\lim_{z \to -\infty}\psi_0^*(z) > 0 > \lim_{z \to \infty}\psi_0^*(z)$, and for $j \in \{1,2,3\}$, we have $\int_\R (\check{\psi}_{n,j} - \psi_0^*)^2\,p_0 \cvp 0$ by~\eqref{eq:proj-score-consistency}. Thus, with probability tending to 1 as $n \to \infty$, we have $\lim_{z \to -\infty}\check{\psi}_{n,j}(z) > 0 > \lim_{z \to \infty}\check{\psi}_{n,j}(z)$. In this case, there exists $\hat{\theta}_n^{(j)}$ satisfying~\eqref{eq:betahat-intercept} for each $j$, and by Lemma~\ref{lem:linreg-intercept-consistency},
\begin{align}
\frac{1}{\sqrt{n}}\sum_{i \in I_{j+2}} \tilde{W}_i\hat{\Psi}_{n,j}\bigl(\bar{\delta}_{n,j} &+ \tilde{W}_i^\top(\hat{\theta}_n^{(j)} - \theta_0)\bigr) + \tilde{R}_{n,j}\bigl(\sqrt{n}(\hat{\theta}_n^{(j)} - \theta_0), \sqrt{n}\bar{\delta}_{n,j}\bigr) \notag \\
\label{eq:linreg-intercept-efficient}
&\quad= \frac{1}{\sqrt{n}}\sum_{i \in I_{j+2}} \tilde{W}_i\check{\psi}_{n,j}\bigl(\varepsilon_i - \bar{\delta}_{n,j} - \tilde{W}_i^\top(\hat{\theta}_n^{(j)} - \theta_0)\bigr) = o_p(1)
\end{align}
and
\[
\biggl\|\frac{1}{\sqrt{n}}\sum_{i \in I_{j+2}} \tilde{W}_i \tilde{W}_i^\top(\hat{\theta}_n^{(j)} - \theta_0)\biggr\|
\leq \frac{1}{n}\sum_{i \in I_{j+2}} \norm{\tilde{W}_i}^2 \norm{\sqrt{n}(\hat{\theta}_n^{(j)} - \theta_0)} = O_p(1).
\]
Moreover, by~\ref{ass:covariates} and Lemma~\ref{lem:linreg-intercept-consistency}\textit{(a)}, $\bigl(|\bar{\delta}_{n,j}| + \max_{i \in I_{j+2}}|\tilde{W}_i^\top (\hat{\theta}_n^{(j)} - \theta_0)|\bigr)\,\alpha_n/\gamma_n = o_p(1)$ as $n \to \infty$. Hence, using the fact that $\sum_{i \in I_{j+2}}\tilde{W}_i = 0$ and arguing similarly to~\eqref{eq:linreg-intercept-Psi}, we have
\begin{align}
\label{eq:linreg-intercept-psi}
\frac{1}{\sqrt{n}}\sum_{i \in I_{j+2}} \tilde{W}_i\hat{\Psi}_{n,j}\bigl(\bar{\delta}_{n,j} + \tilde{W}_i^\top(\hat{\theta}_n^{(j)} - \theta_0)\bigr)
% &= \frac{i^*(p_0)}{n}\sum_{i \in I_{j+2}} \tilde{W}_i \tilde{W}_i^\top \sqrt{n}(\hat{\theta}_n^{(j)} - \theta_0) + o_p(1) \notag \\
&= \frac{i^*(p_0)}{3}\Sigma\sqrt{n}(\hat{\theta}_n^{(j)} - \theta_0) + o_p(1),
\end{align}
where $\Sigma = \Cov(\tilde{X}_1)$. Therefore, by~\eqref{eq:linreg-intercept-psi} and~\eqref{eq:linreg-intercept-efficient} followed by Lemmas~\ref{lem:linreg-intercept-equicontinuity} and~\ref{lem:linreg-intercept-consistency}\textit{(a)}, and then~\eqref{eq:linreg-intercept-clt},
\begin{align*}
\frac{i^*(p_0)}{3}\Sigma\sqrt{n}(\hat{\theta}_n^{(j)} - \theta_0) &= -\tilde{R}_{n,j}\bigl(\sqrt{n}(\hat{\theta}_n^{(j)} - \theta_0), \sqrt{n}\bar{\delta}_{n,j}\bigr) + o_p(1) \\
&= -\tilde{R}_{n,j}(0,0) + o_p(1) = -\frac{1}{\sqrt{n}}\sum_{i \in I_{j+2}} (\tilde{X}_i - \tilde{m})\psi_0^*(\varepsilon_i) + o_p(1)
\end{align*}
for each $j \in \{1,2,3\}$, where $\tilde{m} = \E(\tilde{X}_1)$. By the central limit theorem, we conclude that
\begin{align}
\sqrt{n}(\hat{\theta}_n^\dagger - \theta_0) = \sum_{j=1}^3 \frac{\sqrt{n}}{3}(\hat{\theta}_n^{(j)} - \theta_0) &= -\frac{\Sigma^{-1}}{i^*(p_0)\sqrt{n}}\sum_{j=1}^3 \sum_{i \in I_{j+2}} (\tilde{X}_i - \tilde{m})\psi_0^*(\varepsilon_i) + o_p(1) \notag \\
\label{eq:theta-asymp-linear}
&= -\frac{\Sigma^{-1}}{i^*(p_0)\sqrt{n}}\sum_{i=1}^n (\tilde{X}_i - \tilde{m})\psi_0^*(\varepsilon_i) + o_p(1) \cvd N_{d-1}\biggl(0,\,\frac{\Sigma^{-1}}{i^*(p_0)}\biggr)
\end{align}
as $n \to \infty$. By similar reasoning based on Lemmas~\ref{lem:linreg-intercept-equicontinuity} and~\ref{lem:linreg-intercept-consistency}, $\sqrt{n}(\hat{\theta}_n^\ddagger - \theta_0)$ has the same limiting distribution.

Next, taking $\hat{\theta}_n \in \{\hat{\theta}_n^\dagger,\hat{\theta}_n^\ddagger\}$, we deduce from Lemma~\ref{lem:linreg-zeta-consistency} that
\begin{align*}
\frac{1}{\sqrt{n}}\sum_{i=1}^n \Lambda\bigl(\breve{W}_i^\top(\hat{\theta}_n - \theta_0) + \hat{\delta}_n\bigr) + R_n^\zeta\bigl(\sqrt{n}(\hat{\theta}_n - \theta_0),\sqrt{n}\hat{\delta}_n\bigr) = \frac{1}{\sqrt{n}}\sum_{i=1}^n \zeta\bigl(\varepsilon_i - \breve{W}_i^\top(\hat{\theta}_n - \theta_0) - \hat{\delta}_n\bigr) = o_p(1)
\end{align*}
and moreover that $|\hat{\delta}_n| + \max_{i \in [n]}|\breve{W}_i^\top(\hat{\theta}_n - \theta_0)| = o_p(1)$ as $n \to \infty$. Therefore, arguing similarly to~\eqref{eq:Lambda-deriv-1}, we have
\begin{align*}
% NB: Neyman orthogonality is crucial here!
\frac{1}{\sqrt{n}}\sum_{i=1}^n \Lambda\bigl(\breve{W}_i^\top(\hat{\theta}_n - \theta_0) + \hat{\delta}_n\bigr) 
% &= \frac{\Lambda'(0)}{n}\sum_{i=1}^n \bigl(\breve{W}_i^\top\sqrt{n}(\hat{\theta}_n - \theta_0) + \sqrt{n}\hat{\delta}_n\bigr) + o_p(1) \notag \\
= -\sqrt{n}\hat{\delta}_n\int_\R p_0\,d\zeta + o_p(1).
\end{align*}
By applying Lemmas~\ref{lem:linreg-zeta-equicontinuity} and~\ref{lem:linreg-zeta-consistency}, we deduce that
\begin{align}
\label{eq:alpha-asymp-linear}
\sqrt{n}\hat{\delta}_n\int_\R p_0\,d\zeta &= R_n^\zeta\bigl(\sqrt{n}(\hat{\theta}_n - \theta_0), \sqrt{n}\hat{\delta}_n\bigr) + o_p(1) = R_n^\zeta(0,0) + o_p(1) = \frac{1}{\sqrt{n}}\sum_{i=1}^n \zeta(\varepsilon_i) + o_p(1).
\end{align}
Since $\tilde{X_1}$ has mean $\tilde{m}$ and is independent of $\varepsilon_1$, we have 
\[
\Cov\bigl(\zeta(\varepsilon_1),(\tilde{X}_1 - \tilde{m})\psi_0^*(\varepsilon_1)\bigr) = \E(\tilde{X}_1 - \tilde{m})\,\E\bigl(\zeta(\varepsilon_1)\psi_0^*(\varepsilon_1)\bigr) = 0.
\]
Thus, since $\upsilon_{p_0} = (\int_\R \zeta^2\,p_0)/(\int_\R p_0\,d\zeta)^2 \in (0,\infty)$ by~\ref{ass:zeta}, we conclude  from~\eqref{eq:theta-asymp-linear} and~\eqref{eq:alpha-asymp-linear} that
\begin{align*}
\sqrt{n}(\hat{\beta}_n - \beta_0) = 
\sqrt{n}\begin{pmatrix}
\hat{\theta}_n - \theta_0 \\
\hat{\delta}_n - \tilde{m}^\top(\hat{\theta}_n - \theta_0)
\end{pmatrix}
&= \begin{pmatrix}
-\displaystyle\frac{1}{\sqrt{n}}\sum_{i=1}^n \frac{\Sigma^{-1}(\tilde{X}_i - \tilde{m})}{i^*(p_0)}\psi_0^*(\varepsilon_i) \\[8pt]
\displaystyle\frac{1}{\sqrt{n}}\sum_{i=1}^n \biggl(\frac{\zeta(\varepsilon_i)}{\int_\R p_0\,d\zeta} + \frac{\tilde{m}^\top\Sigma^{-1}(\tilde{X}_i - \tilde{m})}{i^*(p_0)}\psi_0^*(\varepsilon_i)\biggr)
\end{pmatrix} + o_p(1) \\[3pt]
&\cvd N_d\bigl(0, (\tilde{I}_{\beta_0}^*)^{-1}\bigr),
\end{align*}
where $(\tilde{I}_{\beta_0}^*)^{-1}$ is the limiting covariance matrix  in~\eqref{eq:antitonic-efficient-information-beta}.
\end{proof}

\subsubsection{Semiparametric calculations for Section~\ref{subsec:linreg-intercept}}
\label{subsec:linreg-semiparametric}

The purpose of this subsection is to formalise the semiparametric model underlying Theorem~\ref{thm:linreg-score-intercept}, and derive the efficient score functions and information matrices that underpin our theory and methodology. Suppose that $X_1 = (\tilde{X}_1,1)$ has a density $q_X$ with respect to some $\sigma$-finite measure $\nu$ on $\R^d$. Then for $\beta_0 = (\theta_0,\mu_0) \in \R^{d-1} \times \R$, the independence of $X_1$ and $\varepsilon_1$ means that the joint density of $(X_1,Y_1)$ in the linear model~\eqref{eq:linear-model-intercept} with respect to $\nu \otimes \mathrm{Leb}$ is
\[
(x,y) \mapsto q_{\beta_0,p_0}(x,y) = q_X(x)\,p_0(y - x^\top\beta_0).
\]
Denote by $\mathcal{Q}$ the set of all such joint densities $q_{\beta_0,p_0}$ under which $\E(X_1 X_1^\top)$ is invertible and $i(p_0) < \infty$. For a fixed decreasing, right-continuous function $\zeta \colon \R \to \R$ such that $\inf_{z \in \R}\zeta(z) < 0 < \sup_{z \in \R}\zeta(z)$, define $\mathcal{Q}_\zeta$ to be the subclass of $\mathcal{Q}$ for which the density $p_0$ satisfies condition~\ref{ass:zeta} with $\varphi = \zeta$; this is the model in Section~\ref{subsec:linreg-intercept}.
% i.e.~$\E\zeta(\varepsilon_1) = 0$ and $\upsilon_{p_0} = (\int_\R \zeta^2\,p_0)/(\int_\R p_0\,d\zeta)^2 \in (0,\infty)$
Both $\mathcal{Q}$ and $\mathcal{Q}_\zeta$ are \textit{separated semiparametric models}~\citep[Section~25.4]{vdV1998asymptotic} in which $p_0$ is regarded as an unknown nuisance parameter. In the former, the intercept~$\mu_0$ is unidentifiable and hence also treated as a nuisance parameter.

\begin{proposition}
\label{prop:linreg-semiparametric}
Let $m := \E(\tilde{X}_1)$ and $\Sigma := \Cov(\tilde{X}_1)$.
\begin{enumerate}[label=(\alph*)]
\item In the model $\mathcal{Q}$, the efficient score function for $\theta_0$ is
\begin{equation}
\label{eq:efficient-score-theta}
(x,y) \mapsto \tilde{\ell}_{\theta_0}(x,y) := -(\tilde{x} - \tilde{m})\,\psi_0(y - x^\top\beta_0)
\end{equation}
and the corresponding efficient information matrix is $\tilde{I}_{\theta_0} = i(p_0)\Sigma$.
\item In the model $\mathcal{Q}_\zeta$, the efficient score function for $\beta_0$ is
\begin{equation}
\label{eq:efficient-score-beta}
\tilde{\ell}_{\beta_0}(x,y) = -\bigl(x - \E(X_1)\bigr)\psi_0(y - x^\top\beta_0) + \frac{\E(X_1)\int_\R p_0\,d\zeta}{\int_\R \zeta^2 p_0}\,\zeta(y - x^\top\beta_0),
\end{equation}
and the corresponding efficient information matrix and its inverse are
\begin{equation}
\label{eq:efficient-information-beta}
\tilde{I}_{\beta_0} =
\begin{pmatrix}
i(p_0)\Sigma + \tilde{m}\tilde{m}^\top/\upsilon_{p_0} & \tilde{m}/\upsilon_{p_0} \\[3pt]
\tilde{m}^\top/\upsilon_{p_0} & 1/\upsilon_{p_0}
\end{pmatrix},
\qquad\quad
\tilde{I}_{\beta_0}^{-1} =
\begin{pmatrix}
\dfrac{\Sigma^{-1}}{i(p_0)} & -\dfrac{\Sigma^{-1}\tilde{m}}{i(p_0)} \\[10pt]
-\dfrac{\tilde{m}^\top\Sigma^{-1}}{i(p_0)} & \upsilon_{p_0} + \dfrac{\tilde{m}^\top\Sigma^{-1}\tilde{m}}{i(p_0)}
\end{pmatrix}.
\end{equation}
\end{enumerate}
\end{proposition}

For completeness, we will prove this result in full. The model~$\mathcal{Q}$ is the focus of \citet[Example~3]{bickel1982adaptive}, while~\citet[Section~4.1]{kosorok2008introduction} considers two special cases of $\mathcal{Q}_\zeta$ where the errors either have mean 0 or median 0, corresponding to the functions~$\zeta \colon \R \to \R$ given by $\zeta(z) = -z$ and $\zeta(z) = -\sgn(z)$ respectively. The calculations for general $\zeta$ are similar.

\begin{proof}
Fix $q_{\beta_0,p_0} \in \mathcal{Q}$ for some $\beta_0,p_0$ as above, and write $Q_0$ for the corresponding distribution on $\R^d \times \R$. To define a path in $\mathcal{Q}$, take $p_1$ to be any other Lebesgue density on $\R$ such that $i(p_1) < \infty$, and let $\gamma := p_1 - p_0$. For $b = (s,a) \in \R^{d-1} \times \R$ and $t \in [0,1]$, the function $p_t := p_0 + t\gamma = (1 - t)p_0 + tp$ is a density, so
\[
(x,y) \mapsto q_{\beta_0 + tb, p_t}(x,y) = q_X(x)\,p_t\bigl(y - x^\top(\beta_0 + tb)\bigr) 
\]
is a joint density that belongs to $\mathcal{Q}$. Then
for $(x,y) \in \R^d \times \R$, we have
\begin{align}
\frac{\partial}{\partial t}\log q_{\beta_0 + tb, p_t}(x,y)\Big|_{t=0} &= -b^\top x\psi_0(y - x^\top\beta_0) + \frac{\gamma}{p_0}(y - x^\top\beta_0) \notag \\
\label{eq:score-path}
&= -s^\top\tilde{x} \psi_0(y - x^\top\beta_0) - a\psi_0(y - x^\top\beta_0) + \frac{\gamma}{p_0}(y - x^\top\beta_0),
\end{align}
where here and below, $\tilde{x} \in \R^{d-1}$ denotes the subvector comprising the first $d - 1$ components of $x$. The first two terms on the right-hand side of~\eqref{eq:score-path} indicate that in a parametric submodel where $p_0$ is known, the score function for $\beta_0$ is
\[
(x,y) \mapsto -x\psi_0(y - x^\top\beta_0) =: \dot{\ell}_{\beta_0}(x,y) =
\begin{pmatrix}
\dot{\ell}_{\theta_0}(x,y) \\
\dot{\ell}_{\mu_0}(x,y)
\end{pmatrix},
\]
which decomposes into the parametric scores for $\theta_0$ and $\mu_0$ given by
\[
\dot{\ell}_{\theta_0}(x,y) := -\tilde{x}\psi_0(y - x^\top\beta_0) \quad\text{and}\quad \dot{\ell}_{\mu_0}(x,y) := -\psi_0(y - x^\top\beta_0)
\]
respectively. If $(X_1,Y_1) \sim Q_0$, then $\E\bigl(\norm{\dot{\ell}_{\beta_0}(X_1,Y_1)}^2\bigr) = \E(\norm{X_1}^2)\,i(p_0) < \infty$, so $\dot{\ell}_{\beta_0} \in L^2(Q_0)^d$. The third term in~\eqref{eq:score-path} yields an element $(x,y) \mapsto \Gamma(y - x^\top\beta_0)$ of the \textit{nuisance tangent set} $\mathcal{N}$ for $p_0$~\citep[Section~25.4]{vdV1998asymptotic},
% pages 362 and 369
where $\Gamma := \gamma/p_0$. 

\medskip
\noindent \textit{(a)} In the model $\mathcal{Q}$, the parameter of interest is $\theta_0$ while both $\mu_0$ and $p_0$ are nuisance parameters. We have $\int_\R \Gamma(z)\,p_0(z)\,dz = \int_\R \gamma = 0$ above, so considering the second and third terms in~\eqref{eq:score-path} and arguing as in~\citet[Examples~25.16,~25.24 and~25.28]{vdV1998asymptotic}, we deduce that the full nuisance tangent set is
\[
\mathcal{N} = \bigl\{(x,y) \mapsto a\psi_0(y - x^\top\beta_0) + \Gamma(y - x^\top\beta_0) : a \in \R,\,\Gamma \in L^2(P_0)\bigr\}.
\]
By definition, the efficient score function for $\theta_0$ is the $L^2(Q_0)$ orthogonal projection of $\dot{\ell}_{\theta_0}$ onto the orthogonal complement of $\mathcal{N}$,
% which is again a closed subspace of $L^2(Q_0)$ because the linear span of $(x,y) \mapsto \dot{\ell}_{\mu_0}(x,y) = \psi_0(y - x^\top\beta_0)$ is finite-dimensional
which comprises functions of $y - x^\top\beta_0$ only. By the independence of $X_1$ and $\varepsilon_1$, we have
\begin{align}
\label{eq:score-indep-orthogonal}
\E\bigl\{(\tilde{X}_1 - \tilde{m})\psi_0(\varepsilon_1) \cdot g(\varepsilon_1)\bigr\} = 0
\end{align}
for all $g \in L^2(P_0)$. Moreover, every component of $(x,y) \mapsto \tilde{m}\psi_0(y - x^\top\beta_0) = -\tilde{m}\dot{\ell}_{\mu_0}(x,y)$ belongs to $\mathcal{N}$, so
\[
(x,y) \mapsto \tilde{\ell}_{\theta_0}(x,y) := -(\tilde{x} - \tilde{m})\psi_0(y - x^\top\beta_0)
\]
is indeed the efficient score function for $\theta_0$ in the model $\mathcal{Q}$. Consequently, the efficient information matrix for $\theta_0$ is
\[
\tilde{I}_{\theta_0} := \E\bigl(\tilde{\ell}_{\theta_0}(X_1,Y_1)\,\tilde{\ell}_{\theta_0}(X_1,Y_1)^\top\bigr) = \E\bigl(\psi_0(\varepsilon_1)^2\bigr)\Cov(\tilde{X}_1) = i(p_0)\Sigma.
\]
\textit{(b)} In the model $\mathcal{Q}_\zeta$, the parameter of interest is $\beta_0$ and the nuisance parameter is $p_0$. Returning to~\eqref{eq:score-path}, we begin with $q_{\beta_0,p_0} \in \mathcal{Q}_\zeta$ and must ensure that $q_{\beta_0 + tb,p_t} \in \mathcal{Q}_\zeta$ for all $t \in [0,1]$. This restricts us to consider Lebesgue densities $p_1$ with $i(p_1) < \infty$ that also satisfy $\int_\R \zeta(z)\,p_1(z)\,dz = 0$ and $\upsilon_{p_1} \in (0,\infty)$. In this case, $\gamma = p_1 - p_0$ satisfies $\int_\R \gamma(z)\,dz = \int_\R \zeta(z)\gamma(z)\,dz = 0$, and arguments similar to those for \textit{(a)} show that the full nuisance tangent set is 
\[
\mathcal{N}_\zeta = \bigl\{(x,y) \mapsto \Gamma(y - x^\top\beta_0) : \Gamma \in L^2(P_0),\,\E\bigl(\Gamma(\varepsilon_1)\bigr) = \E\bigl(\zeta(\varepsilon_1)\Gamma(\varepsilon_1)\bigr) = 0\bigr\},
\]
which is a closed subspace of $L^2(Q_0)$. Indeed, $\mathcal{U} := \{c + d\zeta : c,d \in \R\}$ is a finite-dimensional (and hence closed) subspace of $L^2(P_0)$ whose orthogonal complement $\mathcal{U}^\perp$ is also closed, and we can write $\mathcal{N}_\zeta = \bigl\{(x,y) \mapsto \Gamma(y - x^\top\beta_0) : \Gamma \in \mathcal{U}^\perp\}$.

Consider a square-integrable random variable $\phi(X_1,\varepsilon_1)$ such that $\E\bigl(\phi(X_1,\varepsilon_1)\,\Gamma(\varepsilon_1)\bigr) = 0$ for all $\Gamma \in L^2(P_0)$ satisfying $\E\bigl(\Gamma(\varepsilon_1)\bigr) = \E\bigl(\zeta(\varepsilon_1)\Gamma(\varepsilon_1)\bigr) = 0$. Then by the tower property of expectation, $\E\bigl(\E\{\phi(X_1,\varepsilon_1)\,|\,\varepsilon_1\} \cdot \Gamma(\varepsilon_1)\bigr) = 0$ for all $\Gamma \in \mathcal{U}^\perp$, so because $(\mathcal{U}^\perp)^\perp = \mathcal{U}$, there exist $c,d \in \R$ such that
\[
\E\{\phi(X_1,\varepsilon_1)\,|\,\varepsilon_1\} = c + d\zeta(\varepsilon_1).
\]
Moreover,
\[
\E\bigl(\psi_0(\varepsilon_1)\bigr) = \int_\R p_0'(z)\,dz = 0 \quad\text{and}\quad \E\bigl(\zeta(\varepsilon_1)\psi_0(\varepsilon_1)\bigr) = \int_\R \zeta(z)p_0'(z)\,dz = -\int_\R p_0\,d\zeta,
\]
where the final equality follows from Fubini's theorem, similarly to~\eqref{eq:fubini-parts}. Write $(x,y) \mapsto \Gamma^*(y - x^\top\beta_0)$ for the componentwise $L^2(Q_0)$ orthogonal projection of $\dot{\ell}_{\beta_0}$ onto $\mathcal{N}_\zeta$. Then each of the $d$ components of $\dot{\ell}_{\beta_0}(X_1,Y_1) - \Gamma^*(\varepsilon_1)$ is a function of the form $\phi(X_1,\varepsilon_1)$ with the above properties, so
\[
\Gamma^*(\varepsilon_1) = \E\bigl(\dot{\ell}_{\beta_0}(X_1,Y_1)\,|\,\varepsilon_1\bigr)- c^* - d^*\zeta(\varepsilon_1) = -\E(X_1)\psi_0(\varepsilon_1) - c^* - d^*\zeta(\varepsilon_1)
\]
for $c^*,d^* \in \R^d$ satisfying
\begin{align*}
0 = \E\bigl(\Gamma^*(\varepsilon_1)\bigr) = -c^*, \qquad 
0 = \E\bigl(\zeta(\varepsilon_1)\Gamma^*(\varepsilon_1)\bigr) = \E(X_1)\int_\R p_0\,d\zeta - d^*\int_\R \zeta^2 p_0.
\end{align*}
By definition, the efficient score function for $\beta_0$ in $\mathcal{Q}_\zeta$ is the componentwise $L^2(Q_0)$ orthogonal projection $\tilde{\ell}_{\beta_0}$ of $\dot{\ell}_{\beta_0}$ onto the orthogonal complement of $\mathcal{N}_\zeta$, so $\tilde{\ell}_{\beta_0}(x,y) = \dot{\ell}_{\beta_0}(x,y) - \Gamma^*(y - x^\top\beta_0)$ is indeed given by~\eqref{eq:efficient-score-beta}. The efficient information matrix for $\beta_0$ is
\begin{align*}
\tilde{I}_{\beta_0} := \E\bigl(\tilde{\ell}_{\beta_0}(X_1,Y_1)\,\tilde{\ell}_{\beta_0}(X_1,Y_1)^\top\bigr) &= \E\bigl\{\bigl(X_1 - \E(X_1)\bigr)\bigl(X_1 - \E(X_1)\bigr)^\top\bigr\}\,\E\bigl(\psi_0(\varepsilon_1)^2\bigr) + \frac{\E(X_1)\E(X_1)^\top}{\upsilon_{p_0}},
\end{align*}
% actually everything except the final expression holds for a general linear model (not necessarily with an intercept)
which simplifies to the expression in~\eqref{eq:efficient-information-beta}. Therefore, the block matrix inversion formula~\citep[e.g.][Proposition~10.10.2]{samworth24modern} yields the expression for $\tilde{I}_{\beta_0}^{-1}$.
\end{proof}

By Proposition~\ref{prop:linreg-semiparametric}\textit{(b)}, the efficient score equations for estimating $\beta_0$ in the model $\mathcal{Q}_\zeta$ based on $(X_1,Y_1),\dotsc,(X_n,Y_n)$ and a plug-in estimate $\tilde{\psi}_n$ of $\psi_0$ are given by
\begin{align*}
\frac{1}{n}\sum_{i=1}^n (X_i - \bar{X}_n)\,\tilde{\psi}_n(Y_i - X_i^\top\beta) - \frac{\bar{X}_n\int_\R p_0\,d\zeta}{\int_\R \zeta^2 p_0}\sum_{i=1}^n \zeta(Y_i - X_i^\top\beta) = 0
\end{align*}
for $\beta \in \R^d$, where $\bar{X}_n := n^{-1}\sum_{i=1}^n X_i$. Since $X_{id} = 1$ for all $i \in [n]$, this system is equivalent to 
\begin{equation}
\label{eq:efficient-score-eqn-beta}
\frac{1}{n}\sum_{i=1}^n (\tilde{X}_i - \bar{\tilde{X}}_n)\,\tilde{\psi}_n(Y_i - \mu - \tilde{X}_i^\top\theta) = 0 \quad\text{and}\quad \frac{1}{n}\sum_{i=1}^n \zeta(Y_i - \mu - \tilde{X}_i^\top\theta) = 0,
\end{equation}
where $\bar{\tilde{X}}_n := n^{-1}\sum_{i=1}^n \tilde{X}_i$. Our estimator $\hat{\beta}_n = (\hat{\theta}_n,\hat{\mu}_n^\zeta)$ of $\beta = (\theta_0,\mu_0)$ in Theorem~\ref{thm:linreg-score-intercept} solves a variant of~\eqref{eq:efficient-score-eqn-beta} where $\tilde{\psi}_n$ is a kernel-based estimate of the antitonic projected score $\psi_0^*$ instead of $\psi_0$.

\subsubsection{Comparison with composite quantile regression}
\label{subsec:CQR}

\citet{zou2008composite} introduced a robust alternative to least squares called the \textit{composite quantile regression estimator}, which remains $\sqrt{n}$-consistent when the error variance is infinite, but also has asymptotic efficiency at least 70\% that of OLS under suitable conditions on the error density $p_0$. This is achieved by borrowing strength across several quantiles of the conditional distribution of $Y_1$ given $X_1$, rather than targeting just a single quantile (e.g.~the conditional median) using the \textit{quantile loss} $\ell_\tau \colon \R \to \R$ given by $\ell_\tau(z) := (\tau - \Ind_{\{z < 0\}})z$ for $\tau \in (0,1)$. More precisely, for $K \in (0,\infty)$, let $\tau_k := k/(K + 1)$ for $k \in [K]$ and define
\begin{equation}
\label{eq:CQR}
(\hat{\theta}_n^{\mathrm{CQ},K},\hat{\mu}_{n,1},\dotsc,\hat{\mu}_{n,K}) \in \argmin_{(\theta,\mu_1,\dotsc,\mu_K)}\,\sum_{k=1}^K \sum_{i=1}^n \ell_{\tau_k}(Y_i - \mu_k - \tilde{X}_i^\top\theta),
\end{equation}
where the $\argmin$ is taken over all $\theta \in \R^{d-1}$ and $\mu_1,\dotsc,\mu_K \in \R$. Under assumption~(2) of~\citet{zou2008composite} in our random design setting, it follows from their Theorems~2.1 and~3.1 that
\[
\sqrt{n}(\hat{\theta}_n^{\mathrm{CQ},K} - \theta_0) \cvd N_{d-1}\bigl(0,\,V_{p_0,\mathrm{CQ},K} \cdot \Cov(\tilde{X}_1)^{-1}\bigr),
\]
where writing $J_0 = p_0 \circ F_0^{-1}$ for the density quantile function of the errors, we have
\[
V_{p_0,\mathrm{CQ},K} := \frac{\sum_{k,k'=1}^K \tau_{k \wedge k'}(1 - \tau_{k \vee k'})}{\bigl(\sum_{k=1}^K J_0(\tau_k)\bigr)^2} \to \frac{1}{12\,\bigl(\int_0^1 J_0\bigr)^2} =: V_{p_0,\mathrm{CQ}}
\]
as $K \to \infty$. For every $p_0$ satisfying their assumptions,~\citet[Theorem~3.1]{zou2008composite} established that the CQR estimator (in the notional limit $K \to \infty$) has asymptotic relative efficiency
\[
\frac{V_{p_0}(\psi_{\mathrm{OLS}})}{V_{p_0,\mathrm{CQ}}} = 12\biggl(\int_0^1 J_0\biggr)^2 \int_\R z^2\,p_0(z)\,dz > \frac{6}{e\pi} \approx 0.703
\]
relative to OLS, where $\psi_{\mathrm{OLS}}(z) := -z$ for $z \in \R$. Nevertheless, Theorem~\ref{thm:linreg-score-intercept} and the following result imply that the asymptotic covariance of the `limiting' CQR estimator is always at least that of the semiparametric convex $M$-estimator $\hat{\theta}_n$ in our framework. Moreover, the former estimator can have arbitrarily low efficiency relative to the latter, even when $p_0$ is log-concave.

\begin{lemma}
\label{lem:CQR-suboptimal}
For every uniformly continuous density $p_0$, we have $V_{p_0,\mathrm{CQ}} \geq 1/i^*(p_0)$, with equality if and only if either $i^*(p_0) = \infty$ or $p_0$ is a logistic density of the form
\[
p_0(z) = \frac{\lambda e^{-\lambda(z - \mu)}}{(1 + e^{-\lambda(z - \mu)})^2}
\]
for $z \in \R$, where $\mu \in \R$ and $\lambda > 0$. 
% See~\citet[Example~3]{zou2008composite}
Moreover,
\[
\inf_{p_0 \in \mathcal{P}_{\mathrm{LC}} : i(p_0) < \infty}\,\frac{1}{i^*(p_0) \, V_{p_0,\mathrm{CQ}}} = 0.
\]
\end{lemma}

\begin{proof}
By~\citet[Corollary~24.2.1]{rockafellar97convex} and~\eqref{eq:J01}, we have $\hat{J}_0(v) = \int_0^v \hat{J}_0^{(\mathrm{R})}$ for all $v \in [0,1]$, so applying Fubini's theorem and the Cauchy--Schwarz inequality yields
\begin{align*}
\biggl(\int_0^1 J_0\biggr)^2 \leq \biggl(\int_0^1 \hat{J}_0\biggr)^2 &= \biggl(\int_0^1 \int_0^1 \hat{J}_0^{(\mathrm{R})}(u)\,\Ind_{\{u \leq v\}}\,du\,dv\biggr)^2 = \biggl(\int_0^1 (1 - u)\,\hat{J}_0^{(\mathrm{R})}(u)\,du\biggr)^2 \\
&= \biggl(\int_0^1 \Bigl(\frac{1}{2} - u\Bigr)\,\hat{J}_0^{(\mathrm{R})}(u)\,du\biggr)^2 \\
&\leq \biggl(\int_0^1 \Bigl(\frac{1}{2} - u\Bigr)^2\,du\biggr)\int_0^1 \bigl(\hat{J}_0^{(\mathrm{R})}\bigr)^2 = \frac{1}{12}\int_0^1 \bigl(\hat{J}_0^{(\mathrm{R})}\bigr)^2,
\end{align*}
where the equality in the second line holds because $\int_0^1 \hat{J}_0^{(\mathrm{R})} = \hat{J}_0(1) = 0$. Thus, by Remark~\ref{rem:fisher-J},
\[
i^*(p_0) = \int_0^1 \bigl(\hat{J}_0^{(\mathrm{R})}\bigr)^2 \geq 12\,\biggl(\int_0^1 J_0\biggr)^2 = \frac{1}{V_{p_0,\mathrm{CQ}}},
\]
and equality holds if and only if $J_0 = \hat{J}_0$ and there exists $\lambda > 0$ such that $\hat{J}_0^{(\mathrm{R})}(u) = \lambda(1 - 2u)$ for $u \in [0,1]$, i.e.~$
J_0(u) = \lambda u(1 - u)$ for all such $u$. Now by direct calculation, the logistic density $q_\lambda \colon \R \to \R$ given by
\[
q_\lambda(z) := \frac{\lambda e^{-\lambda z}}{(1 + e^{-\lambda z})^2}
% \quad\text{and}\quad F_\lambda(z) := \frac{1}{1 + e^{-\lambda z}}
\]
has corresponding quantile function $J_0$. Since $J_0 > 0$ on $(0,1)$, it follows from the last assertion of Lemma~\ref{lem:density-quantile-reverse} that $p_0(\cdot) = q_\lambda(\cdot - \mu)$ for some $\mu \in \R$, as claimed.

Finally, given $\epsilon \in (0,1/2]$, define the log-concave density $p_0 \colon \R \to \R$ by
\[
p_0(z) := \exp\biggl(\frac{(1 - 2\epsilon - 2|z|) \wedge 0}{2\epsilon}\biggr).
\]
Then the corresponding density quantile function $J_0 = p_0 \circ F_0^{-1}$ satisfies
\[
J_0(u) =
\begin{cases}
\min(u/\epsilon, 1) \;&\text{for }u \in [0,1/2] \\
J_0(1 - u) \;&\text{for }u \in [1/2,1],
\end{cases}
\]
so $i(p_0) = i^*(p_0) = \int_0^1 \bigl(\hat{J}_0^{(\mathrm{R})}\bigr)^2 = 2/\epsilon < \infty$ and hence
\[
\frac{1}{i^*(p_0)\,V_{p_0,\mathrm{CQ}}} = \frac{12\,\bigl(\int_0^1 J_0\bigr)^2}{\int_0^1 \bigl(\hat{J}_0^{(\mathrm{R})}\bigr)^2} = 6\epsilon(1 - \epsilon)^2.
\]
By taking $\epsilon$ to be arbitrarily small, we obtain the final assertion of the lemma.
\end{proof}

\subsubsection{Proofs for Section~\ref{subsec:linreg-inference}}

\begin{proof}[Proof of Lemma~\ref{lem:observed-information}]

First consider the setting of Theorem~\ref{thm:linreg-score-sym} and let $j = 1$. Defining the residuals $\check{\varepsilon}_i := \varepsilon_i + X_i^\top b_n$ for $i \in I_3$, we have that $(\check{\varepsilon}_i)_{i \in I_3}$ are conditionally independent given $\mathcal{D}' \equiv \mathcal{D}_n' = \mathcal{D}_1 \cup \mathcal{D}_2 \cup (X_i)_{i \in I_3}$, where $b_n := \beta_0 - \bar{\beta}_n^{(1)} = O_p(n^{-1/2})$ by assumption. Thus, by~\ref{ass:covariates} and the Cauchy--Schwarz inequality, $\max_{i \in I_3}|X_i^\top b_n|\,\alpha_n/\gamma_n = o_p(1)$. Let $\hat{\psi}_n \in \{\hat{\psi}_{n,1},\hat{\psi}_{n,1}^{\mathrm{anti}}\}$. Then 
\begin{align}
\biggl|\biggl(\frac{1}{n}\sum_{i \in I_3} \hat{\psi}_n(\check{\varepsilon}_i)^2\biggr)^{1/2} &- \biggl(\frac{1}{n}\sum_{i \in I_3} \psi_0^*(\varepsilon_i)^2\biggr)^{1/2}\biggr| \leq \biggl(\frac{1}{n}\sum_{i \in I_3} \bigl(\hat{\psi}_n(\check{\varepsilon}_i) - \psi_0^*(\varepsilon_i)\bigr)^2\biggr)^{1/2} \notag \\
\label{eq:observed-information-0}
&\leq \biggl(\frac{1}{n}\sum_{i \in I_3} \bigl(\hat{\psi}_n(\check{\varepsilon}_i) - \hat{\psi}_n(\varepsilon_i)\bigr)^2\biggr)^{1/2} + \biggl(\frac{1}{n}\sum_{i \in I_3} \bigl(\hat{\psi}_n(\varepsilon_i) - \psi_0^*(\varepsilon_i)\bigr)^2\biggr)^{1/2}.
\end{align}
Similarly to~\eqref{eq:cov-Rn-beta}, it follows from~\eqref{eq:proj-score-consistency-1} in Lemma~\ref{lem:proj-score-consistency} together with Corollary~\ref{cor:sym-score-consistency} that
\begin{align}
\label{eq:observed-information-1}
\E\Bigl(\frac{1}{n}\sum_{i \in I_3} \bigl(\hat{\psi}_n(\check{\varepsilon}_i) - \hat{\psi}_n(\varepsilon_i)\bigr)^2 \Bigm| \mathcal{D}'\Bigr) \leq \max_{i \in I_3}\int_\R \bigl(\hat{\psi}_n(z + X_i^\top b_n) - \hat{\psi}_n(z)\bigr)^2\,p_0(z)\,dz \cvp 0
\end{align}
as $n \to \infty$. Moreover, by~\eqref{eq:proj-score-consistency},
\begin{align}
\label{eq:observed-information-2}
\E\Bigl(\frac{1}{n}\sum_{i \in I_3} \bigl(\hat{\psi}_n(\varepsilon_i) - \psi_0^*(\varepsilon_i)\bigr)^2 \Bigm| \mathcal{D}'\Bigr) \leq \int_\R (\hat{\psi}_n - \psi_0^*)^2\,p_0 \cvp 0,
\end{align}
so together with Lemma~\ref{lem:cond-cvg} and the weak law of large numbers, it follows from~\eqref{eq:observed-information-0}--\eqref{eq:observed-information-2} that
\begin{equation}
\label{eq:observed-information-3}
\frac{1}{n}\sum_{i \in I_3} \hat{\psi}_n(\check{\varepsilon}_i)^2 = \frac{1}{n}\sum_{i \in I_3} \psi_0^*(\varepsilon_i)^2 + o_p(1) = \frac{i^*(p_0)}{3} + o_p(1).
\end{equation}
Arguing similarly for $j \in \{2,3\}$, we conclude that $\hat{\imath}_n \cvp i^*(p_0)$ under the hypotheses of Theorem~\ref{thm:linreg-score-sym}.

The proof is similar in the setting of Theorem~\ref{thm:linreg-score-intercept}, except that when $j = 1$ we instead take $\hat{\psi}_n(\cdot) = \hat{\psi}_{n,1}(\cdot + \mu_0)$ and use residuals $\check{\varepsilon}_i := Y_i - \tilde{X}_i^\top\bar{\theta}_n^{(1)} = \varepsilon_i + \mu_0 + X_i^\top b_n$ that are conditionally independent given $\mathcal{D}'$ for $i \in [n]$, where $b_n := (\theta_0 - \bar{\theta}_n^{(1)}, 0) = O_p(n^{-1/2})$ by assumption. As explained at the start of Section~\ref{subsec:linreg-intercept-proofs}, it follows from Lemma~\ref{lem:proj-score-consistency} that~\eqref{eq:observed-information-1},~\eqref{eq:observed-information-2} and hence~\eqref{eq:observed-information-3} still hold for $n^{-1}\sum_{i \in I_3} \hat{\psi}_n(\check{\varepsilon}_i)^2 = n^{-1}\sum_{i \in I_3} \hat{\psi}_{n,1}(Y_i - \tilde{X}_i^\top\bar{\theta}_n^{(1)})$. By analogous reasoning for $j \in \{2,3\}$, we obtain
\[
\hat{\imath}_n = \frac{1}{n}\sum_{j=1}^3 \sum_{i \in I_{j+2}} \hat{\psi}_{n,j}(Y_i - \tilde{X}_i^\top\bar{\theta}_n^{(j)}) = \frac{1}{n}\sum_{i=1}^n \psi_0^*(\varepsilon_i)^2 + o_p(1) \cvp i^*(p_0). \qedhere
\]
\end{proof}

\begin{proof}[Proof of Lemma~\ref{lem:upsilon-estimate}]
First consider $j = 1$. As in the proof of Lemma~\ref{lem:observed-information}, the residuals $\check{\varepsilon}_i = Y_i - X_i^\top\bar{\beta}_n^{(1)} = \varepsilon_i + X_i^\top b_n$ are conditionally independent given $\mathcal{D}'$, where $b_n := \beta_0 - \bar{\beta}_n^{(1)} = O_p(n^{-1/2})$ by assumption and hence $\max_{i \in [n]} |X_i^\top b_n| = o_p(1)$ by~\eqref{eq:expected-max}. We have
\[
\biggl|\biggl(\frac{1}{n}\sum_{i \in I_3} \zeta(\check{\varepsilon}_i)^2\biggr)^{1/2} - \biggl(\frac{1}{n}\sum_{i \in I_3} \zeta(\varepsilon_i)^2\biggr)^{1/2}\biggr| \leq \biggl(\frac{1}{n}\sum_{i \in I_3} \bigl(\zeta(\check{\varepsilon}_i) - \zeta(\varepsilon_i)\bigr)^2\biggr)^{1/2},
\]
and $\zeta$ satisfies~\ref{ass:zeta} under the hypotheses of Theorem~\ref{thm:linreg-score-intercept}, so similarly to~\eqref{eq:observed-information-1}, it follows from~\eqref{eq:psi-shift-limit} that
\begin{align*}
\E\Bigl(\frac{1}{n}\sum_{i \in I_3} \bigl(\zeta(\check{\varepsilon}_i) - \zeta(\varepsilon_i)\bigr)^2 \Bigm| \mathcal{D}'\Bigr) &\leq \max_{i \in I_3} \int_\R\bigl(\zeta(z + X_i^\top b_n) - \zeta(z)\bigr)^2\,p_0(z)\,dz \cvp 0.
\end{align*}
Together with Lemma~\ref{lem:cond-cvg} and the weak law of large numbers, this implies that $n^{-1}\sum_{i \in I_3} \zeta(\check{\varepsilon}_i)^2 = n^{-1}\sum_{i \in I_3} \zeta(\varepsilon_i)^2 + o_p(1) \cvp \E\bigl(\zeta(\varepsilon_1)^2\bigr)/3$, so arguing similarly for $j \in \{2,3\}$, we obtain
\begin{equation}
\label{eq:zeta-squared-avg}
\frac{1}{n}\sum_{i=1}^n \zeta(\check{\varepsilon}_i)^2 \cvp \E\bigl(\zeta(\varepsilon_1)^2\bigr)
\end{equation}
as $n \to \infty$. Next, since $\zeta$ satisfies $\int_\R \zeta(z + t)^2\,p_0(z)\,dz \in \R$ for $t \in \{-t_0,t_0\}$, the same is true of $\zeta_{\mathrm{ac}}$, so by Lemma~\ref{lem:varPsi-deriv}, $\int_\R \zeta_{\mathrm{ac}}'(z + t)\,p_0(z)\,dz \in \R$ for all $t \in [-t_0,t_0]$ and
\[
\int_\R \zeta_{\mathrm{ac}}'(z + t)\,p_0(z)\,dz 
% = \int_\R p_0(z - t)\,d\zeta_{\mathrm{ac}}(z) \to \int_\R p_0\,d\zeta_{\mathrm{ac}} = 
\to \int_\R \zeta_{\mathrm{ac}}'(z)\,p_0(z)\,dz
\]
as $t \to 0$. Moreover, since $-\zeta_{\mathrm{ac}}'$ is non-negative and continuous Lebesgue almost everywhere, it follows from Scheff\'e's lemma~\citep[e.g.][Lemma~2.29]{vdV1998asymptotic} that $\int_\R |\zeta_{\mathrm{ac}}'(z + t) - \zeta_{\mathrm{ac}}'(z)|\,p_0(z)\,dz \to 0$ as $t \to 0$, so
\[
\E\Bigl(\frac{1}{n}\sum_{i \in I_3} |\zeta_{\mathrm{ac}}'(\check{\varepsilon}_i) - \zeta_{\mathrm{ac}}'(\varepsilon_i)| \Bigm| \mathcal{D}'\Bigr) \leq \max_{i \in I_3} \int_\R |\zeta_{\mathrm{ac}}'(z + X_i^\top b_n) - \zeta_{\mathrm{ac}}'(z)|\,p_0(z)\,dz \cvp 0.
\]
By analogous reasoning for $j \in \{2,3\}$, we deduce that
\begin{equation}
\label{eq:zeta-deriv-avg}
\frac{1}{n}\sum_{i=1}^n \zeta_{\mathrm{ac}}'(\check{\varepsilon}_i) = \frac{1}{n}\sum_{i=1}^n \zeta_{\mathrm{ac}}'(\varepsilon_i) + o_p(1) \cvp \E\zeta_{\mathrm{ac}}'(\varepsilon_1).
\end{equation}
Finally, for $m \in [M]$ and $j \in \{1,2,3\}$, define $\tilde{p}_{n,j}(z_m) := |I_{j+2}|^{-1}\sum_{i \in I_{j+2}} K_h(z_m - \check{\varepsilon}_i)$ and first consider $j = 1$. Letting $p_{n,1}(z_m) := \E\bigl(\tilde{p}_{n,1}(z_m) \!\bigm|\! \mathcal{D}_1\bigr)$, recall from~\eqref{eq:kernel-L1-bias} and~\eqref{eq:var-pnh} that by the bounded convergence theorem,
% \norm{p_0}_\infty < \infty
\begin{align*}
|p_{n,1}(z_m) - p_0(z_m)| &\leq \int_{\R^d} \int_\R |p_0(z_m - uh - x^\top b_n) - p_0(z_m)|\,K(u)\,du\,dP_X(x) \cvp 0 \\
\Var\bigl(\tilde{p}_{n,1}(z_m) \!\bigm|\! \mathcal{D}_1\bigr) &\leq \frac{\norm{p_0}_\infty \int_\R K^2}{|I_3| h} \to 0.
\end{align*}
Thus, $\tilde{p}_{n,1}(z_m) \cvp p_0(z_m)$ by Lemma~\ref{lem:cond-cvg}, and similarly $\tilde{p}_{n,j}(z_m) \cvp p_0(z_m)$ for $j \in \{2,3\}$, so $\tilde{p}_n(z_m) = n^{-1}\sum_{j=1}^3 |I_{j+2}|\,\tilde{p}_{n,j}(z_m) \cvp p_0(z_m)$ for all $m \in [M]$. Combining this with~\eqref{eq:zeta-deriv-avg} shows that
\[
\frac{1}{n}\sum_{i=1}^n \zeta_{\mathrm{ac}}'(\check{\varepsilon}_i) - \sum_{m=1}^M \zeta_m\tilde{p}_n(z_m) \cvp \int_\R \zeta_{\mathrm{ac}}'\,p_0 - \sum_{m=1}^M \zeta_m p_0(z_m) = \int_\R p_0\,d\zeta,
\]
which together with~\eqref{eq:zeta-squared-avg} yields
\[
\hat{\upsilon}_n := \frac{n^{-1}\sum_{i=1}^n \zeta(\check{\varepsilon}_i)^2}{\bigl\{n^{-1}\sum_{i=1}^n \zeta_{\mathrm{ac}}'(\check{\varepsilon}_i) - \sum_{m=1}^M \zeta_m\tilde{p}_n(z_m)\bigr\}^2} \cvp \frac{\int_\R \zeta^2\,p_0}{(\int_\R p_0\,d\zeta)^2} = \upsilon_{p_0}. \qedhere
\]
\end{proof}

\subsection{Auxiliary results and proofs}
\label{sec:auxiliary}

For reference, we provide a formal statement of the classical asymptotic result in~\eqref{eq:M-estimator-asymp} for generic convex $M$-estimators. Since the error distributions in this paper have uniformly continuous densities $p_0$, we only require mild assumptions on the convex loss function $L$. 
% \citet[Corollary~3.5]{huber1964robust}
Our condition~\ref{ass:zeta} is simple and neither forces~$L$ to be differentiable everywhere, nor places any explicit restrictions on the growth of $L(z)$ as $|z| \to \infty$.

\begin{proposition}
\label{prop:cvx-M-est-asymp}
Consider a random design linear model~\eqref{eq:linear-model} where $\E(X_1 X_1^\top)$ is invertible and the error density $p_0$ is absolutely continuous with $i(p_0) < \infty$. If $L \colon \R \to \R$ is a convex loss function for which $\varphi = -L^{(\mathrm{R})}$ satisfies~\emph{\ref{ass:zeta}} and $\E\varphi(\varepsilon_1) = 0$, then the following statements hold.
\begin{enumerate}[label=(\alph*)]
\item For each $n$, there always exists $\hat{\beta}_n \in \argmin_{\beta \in \R^d} \sum_{i=1}^n L(Y_i - X_i^\top\beta)$, and this is unique if the design matrix $X = (X_1 \; \cdots \; X_n)^\top$ has full column rank and $\varphi$ is strictly decreasing.
\item $\sqrt{n}(\hat{\beta}_n - \beta_0) \cvd N_d\bigl(0, V_{p_0}(\varphi) \cdot \{\E(X_1 X_1^\top)\}^{-1}\bigr)$, where $V_{p_0}(\varphi) = (\int_\R \varphi^2\,p_0)/(\int_\R p_0\,d\varphi)^2$ is as in~\eqref{eq:Vp0}.
\end{enumerate}
\end{proposition}

A necessary condition for uniqueness in \textit{(a)} is that $X$ has full column rank, which happens with probability tending to 1 as $n \to \infty$ since $\E(X_1 X_1^\top)$ is invertible. The hypothesis that $p_0$ is absolutely continuous with $i(p_0) < \infty$ is only used to prove \textit{(b)}, and can be relaxed provided that the function $\Lambda \colon \R \to \R$ given by~\eqref{eq:zeta-Lambda} satisfies~\eqref{eq:zeta-Lambda-deriv}. In particular, this ensures that $\Lambda'(0) = -\int_\R p_0\,d\zeta \in (0,\infty)$; see Lemmas~\ref{lem:Psi-deriv} and~\ref{lem:varPsi-deriv} below.

\begin{proof}
\textit{(a)} By assumption, $\inf_{z \in \R}\varphi(z) < 0 < \sup_{z \in \R}\varphi(z)$, so $L$ is coercive in the sense that $L(z) \to \infty$ as $|z| \to \infty$. Thus, $\theta \equiv (\theta_1,\dotsc,\theta_n) \mapsto \sum_{i=1}^n L(Y_i - \theta_i) =: \mathcal{L}_n(\theta)$ is convex and coercive on $\mathrm{Im}(X) := \{X\beta : \beta \in \R^d\}$, so $\beta \mapsto \mathcal{L}_n(X\beta)$ attains its minimum on $\R^d$. If $\zeta$ is strictly decreasing, then $\mathcal{L}_n$ is strictly convex on $\mathrm{Im}(X)$, so has a unique minimiser $m \in \mathrm{Im}(X)$. If in addition $X$ has full column rank, then there exists a unique $\hat{\beta}_n$ such that $m = X\hat{\beta}_n$, i.e.~$\hat{\beta}_n = \argmin_{\beta \in \R^d} \sum_{i=1}^n L(Y_i - X_i^\top\beta)$.

\medskip
\noindent
\textit{(b)} This follows from simpler versions of Lemmas~\ref{lem:linreg-sym-equicontinuity} and~\ref{lem:linreg-sym-consistency}, whose key steps we now outline; see also~\citet[Lemma~4.1]{bickel1975one} and~\citet[Theorem~2.1]{he2000parameters}. Define
\[
R_n(b) := \frac{1}{\sqrt{n}}\sum_{i=1}^n X_i\Bigl\{\zeta\Bigl(\varepsilon_i - \frac{X_i^\top b}{\sqrt{n}}\Bigr) - \Lambda\Bigl(\frac{X_i^\top b}{\sqrt{n}}\Bigr)\Bigr\}
\]
for $b \in \R^d$. Then for every $M > 0$, we deduce by almost identical arguments to those in the proof of Lemma~\ref{lem:linreg-zeta-equicontinuity} that $\sup_{b \in B_M} \norm{R_n(b) - R_n(0)} = o_p(1)$ as $n \to \infty$, where $B_M = \{b \in \R^d : \norm{b} \leq M\}$. Consequently, imitating the proofs of Lemmas~\ref{lem:linreg-sym-consistency} and~\ref{lem:linreg-zeta-consistency} yields
\[
\bigl\{\norm{\hat{\beta}_n - \beta_0} > t\bigr\} \subseteq \biggl\{\inf_{u \in S^{d-1}}u^\top \sum_{i=1}^n X_i\zeta(\varepsilon_i - tX_i^\top u) \leq 0\biggr\},
\]
which together with~\eqref{eq:zeta-Lambda-deriv} allows us to conclude that $\sqrt{n}(\hat{\beta}_n - \beta_0) = O_p(1)$ and $n^{-1/2}\sum_{i=1}^n L(Y_i - X_i^\top\hat{\beta}_n) = o_p(1)$. With these ingredients in place, the rest of the proof proceeds very similarly to that of Theorem~\ref{thm:linreg-score-sym}, except with $\zeta$ in place of both $\hat{\psi}_{n,j}^{\mathrm{anti}}$ and $\psi_0^*$, and $\Lambda$ instead of $\hat{\Psi}_{n,j}^{\mathrm{anti}}$.
\end{proof}

\citet[Theorem~3]{huber1964robust} established an equivalent variational characterisation of the Fisher information for information; see~\citet[Theorem~4.2]{huber2009robust} for an alternative proof of the following fact based on Hilbert space theory.

\begin{proposition}
\label{prop:fisher-inf-variational}
For a distribution $P_0$ on $\R$, the following are equivalent:
\begin{enumerate}[label=(\roman*)]
\item $P_0$ has an absolutely continuous density $p_0$ on $\R$ with respect to Lebesgue measure, with $i(p_0) = \int_{\{p_0 > 0\}}(p_0')^2/p_0 < \infty$.
\item $I(P_0) := \displaystyle\sup_\psi \frac{\bigl(\int_\R\psi'\,dP_0\bigr)^2}{\int_\R\psi^2\,dP_0} < \infty$,
where the supremum is taken over all compactly supported, continuously differentiable $\psi \colon \R \to \R$ such that $\int_\R\psi^2\,dP_0 > 0$.
\end{enumerate}
Furthermore, if either (i) or (ii) holds, then $I(P_0) = i(p_0)$.
\end{proposition}

The proofs of Lemma~\ref{lem:proj-score-consistency} and the results in Section~\ref{subsec:linreg-intercept-proofs} require the next four lemmas, the first of which is adapted from~\citet[Lemma~7.6 and Example~7.8]{vdV1998asymptotic}.

\begin{lemma}
\label{lem:location-DQM}
Suppose that $p_0$ is an absolutely continuous density on $\R$ with $i(p_0) < \infty$. For $t \in \R$, let $p_t(\cdot) := p_0(\cdot + t)$, so that $p_t'(\cdot) = p_0'(\cdot + t)$. Then the location model $\{p_t : t \in \R\}$ is differentiable in quadratic mean at every $t \in \R$, i.e.
\[
\lim_{h \to 0} \int_\R \,\biggl(\frac{\sqrt{p_{t + h}} - \sqrt{p_t}}{h} + \frac{p_t'}{2\sqrt{p_t}}\biggr)^2 = 0
\]
for every $t \in \R$. In particular, $\int_\R (\sqrt{p_{t + h}} - \sqrt{p_t})^2 = O(h^2)$ as $h \to 0$.
\end{lemma}

\begin{proof}
Under our assumptions on $p_0$, we will show first that $\sqrt{p_0}$ is locally absolutely continuous on $\R$ with derivative $p_0'/(2\sqrt{p_0})$ Lebesgue almost everywhere. This essentially follows from~\citet[Theorem~3 and Corollary~8]{serrin1969general}, but we give a direct argument here for completeness. For $n \in \N$, define $\varphi_n, \Phi_n \colon \R \to \R$ by $\varphi_n(z) := 1/(2\sqrt{z})$ for $z > 1/n$ and $\varphi_n(z) := 0$ otherwise, and $\Phi_n(z) := \int_0^z \varphi_n$. Then for every $z \in \R$, we have $\Phi_n(z) \nearrow z_+^{1/2}$ and $\varphi_n(z) \nearrow 1/(2\sqrt{z})\Ind_{\{z > 0\}} =: \varphi(x)$ as $n \to \infty$. Now $p_0$ is absolutely continuous on $\R$ and each $\Phi_n$ is Lipschitz on $\R$ with weak derivative $\varphi_n$, so $\Phi_n \circ p_0$ is absolutely continuous on $\R$ for each $n$, with $(\Phi_n \circ p_0)'(z) = (\varphi_n \circ p_0)(z) \cdot p_0'(z)$ for Lebesgue almost every $z \in \R$. Thus, for all $x,y \in \R$, we have 
\[
(\Phi_n \circ p_0)(y) - (\Phi_n \circ p_0)(x) = \int_x^y (\varphi_n \circ p_0)\,p_0', 
\]
and by the Cauchy--Schwarz inequality, $\bigl(\int_x^y |(\varphi \circ p_0)\,p_0'|\bigr)^2 \leq |y - x|\int_\R \bigl((\varphi \circ p_0)\,p_0'\bigr)^2 = |y - x|\,i(p_0)/4 < \infty$. It follows by the dominated convergence theorem that
\begin{align}
\sqrt{p_0(y)} - \sqrt{p_0(x)} &= \lim_{n \to \infty} \bigl((\Phi_n \circ p_0)(y) - (\Phi_n \circ p_0)(x)\bigr) \notag \\
\label{eq:sqrt-p0-ac}
&= \lim_{n \to \infty} \int_x^y (\varphi_n \circ p_0)\,p_0' = \int_x^y (\varphi \circ p_0)\,p_0' = \int_x^y \frac{p_0'}{2\sqrt{p_0}}
\end{align}
for all $x,y \in \R$, so $\sqrt{p_0}$ is indeed locally absolutely continuous on $\R$. 

Thus, for $t \in \R$ and Lebesgue almost every $z \in \R$, we have
\begin{align}
\label{eq:location-DQM-1}
\frac{\sqrt{p_{t + h}(z)} - \sqrt{p_t(z)}}{h} &= \frac{\sqrt{p_0(z + t + h)} - \sqrt{p_0(z + t)}}{h} \notag \\
&\to \frac{p_0'(z + t)}{2\sqrt{p_0(z + t)}} = \frac{p_t'(z)}{2\sqrt{p_t(z)}} =: q_t(z)
\end{align}
as $h \to 0$. Moreover, for $t \in \R$ and $h \neq 0$, we deduce from~\eqref{eq:sqrt-p0-ac}, the Cauchy--Schwarz inequality and Fubini's theorem that
\begin{align}
\int_\R \,\biggl(\frac{\sqrt{p_{t + h}(z)} - \sqrt{p_t(z)}}{h}\biggr)^2\,dz &= \int_\R \,\biggl(\int_0^1 q_{t + sh}(z)\,ds\biggr)^2 dz \notag \\
&\leq \int_\R \int_0^1 q_{t + sh}(z)^2\,ds\,dz = \int_0^1 \int_\R \frac{p_{t + sh}'(z)^2}{4p_{t + sh}(z)}\,dz\,ds \notag \\
\label{eq:location-DQM-2}
&= \int_0^1 \int_\R \frac{p_t'(z)^2}{4p_t(z)}\,dz\,ds = \int_\R q_t^2(z)\,dz = \frac{i(p_0)}{4} < \infty.
\end{align}
It follows from~\eqref{eq:location-DQM-1},~\eqref{eq:location-DQM-2} and~\citet[Proposition~2.29]{vdV1998asymptotic} that 
\[
\int_\R \,\biggl(\frac{\sqrt{p_{t + h}} - \sqrt{p_t}}{h} - q_t\biggr)^2 \to 0
\]
as $h \to 0$. Therefore,
\[
\int_\R \,\biggl(\frac{\sqrt{p_{t + h}} - \sqrt{p_t}}{h}\biggr)^2 \leq 2\int_\R \,\biggl(\frac{\sqrt{p_{t + h}} - \sqrt{p_t}}{h} - q_t\biggr)^2 + 2\int_\R q_t^2 = O(1)
\]
as $h \to 0$, as required.
\end{proof}

\begin{lemma}
\label{lem:Psi-deriv}
Let $\psi \colon \R \to \R$ be decreasing and right-continuous, and let $g \colon \R \to \R$ be a continuous function of bounded variation such that $\lim_{|z| \to \infty} g(z) = 0$. Suppose that there exists $t_0 > 0$ such that 
\[
\varPsi(t) := \int_\R \psi(z - t)\,g(z)\,dz \quad\text{and}\quad \int_\R g(z - t)\,d\psi(z)
\]
are finite for $t \in \{-t_0, t_0\}$. Then 
% $\varPsi(t) \in \R$ and $\int_\R g(z - t)\,d\psi(z) = -\int_\R \psi(z + t)\,dg(z) \in \R$
the same is true of all $t \in [-t_0,t_0]$, and $t \mapsto \varPsi(t)$ is differentiable at 0 with
\[
\varPsi'(0) = -\int_\R g\,d\psi = -\lim_{t \to 0} \int_\R g(z - t)\,d\psi(z) \in \mathbb{R}.
\]
\end{lemma}

\begin{proof}
Since $\psi$ is decreasing, $|\psi(z - t)| \leq |\psi(z - t_0)| + |\psi(z + t_0)| =: \psi_*(z)$ for all $z \in \R$ and $t \in [-t_0,t_0]$, so because $\int_\R \psi_*(z)\,g(z)\,dz < \infty$, it follows that $\varPsi(t) \in \R$ for all $t \in [-t_0,t_0]$. Moreover, since $\lim_{|z| \to \infty} g(z) = 0$, we have $\int_{(-\infty,t)}dg = -\int_{[t,\infty)}dg = g(t)$ for all $t \in \R$. 

Define $h(z,t) := \Ind_{\{z < t \leq 0\}} - \Ind_{\{0 < t \leq z\}}$ for $z,t \in \R$. Then $\int_\R h(z, t - t_0)\,d\psi(t) = \psi(t_0) - \psi(z + t_0)$ for all $z \in \R$, and $\int_\R h(z, t - t_0)\,dg(z) = g(t - t_0)$ for all $t \in \R$. Since $\int_\R g(t - t_0)\,d\psi(t) \in \R$, it follows from Fubini's theorem that
\begin{align}
-\int_\R g(t - t_0)\,d\psi(t) = -\int_\R \int_\R h(z, t - t_0)\,dg(z)\,d\psi(t) &= \int_\R \bigl(\psi(z + t_0) - \psi(t_0)\bigr)\,dg(z) \notag \\
\label{eq:fubini-parts}
&= \int_\R \psi(z + t_0)\,dg(z).
\end{align}
Similarly, $\int_\R\psi(z - t_0)\,dg(z) = -\int_\R g(t + t_0)\,d\psi(t) \in \R$, so $\int_\R \psi_*\,dg \in \R$. We can find increasing, continuous functions $g_1,g_2$ such that $g = g_1 - g_2$ and $\int_\R \psi_*\,d(g_1 + g_2) < \infty$~\citep[e.g.][p.~94 and 103--104]{folland1999real}. 
% Hahn--Jordan decomposition of the atomless measure $\nu_0 = \nu_0^+ - \nu_0^-$
Therefore, for all $t \in [-t_0,t_0]$, we have
\[
\biggl|\int_\R \psi(z + t)\,dg(z)\biggr| \leq \int_\R |\psi(z + t)|\,d(g_1 + g_2)(z) \leq \int_\R \psi_*\,d(g_1 + g_2) < \infty,
\]
so it follows similarly to~\eqref{eq:fubini-parts} that $\int_\R \psi(z + t)\,dg(z) = -\int_\R g(z - t)\,d\psi(z) \in  \R$. For all $z \in \R$ and $t \in [-t_0,t_0] \setminus \{0\}$, we have $|\psi(z + t)| \vee |t^{-1}\int_z^{z + t} \psi| \leq \psi_*(z)$. Moreover, denoting by $A$ the set of continuity points of $\psi$, we have $\lim_{t \to 0} \psi(z + t) = \psi(z) = \lim_{t \to 0} t^{-1}\int_z^{z + t} \psi$ for all $z \in A$. Since $\R \setminus A$ is countable and $g$ is continuous, it follows from Fubini's theorem and the dominated convergence theorem that for $t \in [-t_0,t_0]$,
\[
-\int_\R g(z - t)\,d\psi(z) = \int_A \psi(z + t)\,dg(z) \to \int_A \psi(z)\,dg(z) = -\int_\R g\,d\psi
\]
as $t \to 0$. For $t \in (0, t_0]$, a further application of Fubini's theorem yields
\begin{align*}
\frac{\varPsi(t) - \varPsi(0)}{t} = \int_\R\frac{g(z + t) - g(z)}{t}\,\psi(z)\,dz &= \int_\R \frac{1}{t}\int_A \Ind_{\{z < s \leq z + t\}}\,dg(s)\,\psi(z)\,dz \\
&= \int_A \frac{1}{t}\int_{s - t}^s \psi(z)\,dz\,dg(s),
\end{align*}
which converges to $\int_\R \psi\,dg = -\int_\R g\,d\psi$ as $t \searrow 0$ by the dominated convergence theorem. Similarly, again for $t \in (0,t_0]$, we have
\[
\frac{\varPsi(0) - \varPsi(-t)}{t} = \int_\R \frac{1}{t}\int_\R \Ind_{\{z - t < s \leq z\}}\,dg(s)\,\psi(z)\,dz = \int_\R\frac{1}{t}\int_s^{s + t}\psi(z)\,dz\,dg(s) \searrow -\int_\R g\,d\psi,
\]
as required.
\end{proof}

\begin{lemma}
\label{lem:varPsi-deriv}
Let $p_0$ be an absolutely continuous density with $i(p_0) < \infty$, and define $p_t(\cdot) := p_0(\cdot + t)$ for $t \in \R$. Suppose that $\psi \colon \R \to \R$ is a decreasing, right-continuous function for which there exists $t_0 > 0$ such that $\int_\R \psi(z - t)^2\,p_0(z)\,dz < \infty$ for $t \in \{-t_0,t_0\}$. Then
\begin{equation}
\label{eq:psi-shift-limit}
\int_\R \psi(z)^2\bigl(\sqrt{p_t(z)} - \sqrt{p_0(z)}\bigr)^2\,dz \to 0 \quad\text{and}\quad \int_\R \bigl(\psi(z - t) - \psi(z)\bigr)^2\,p_0(z)\,dz \to 0
\end{equation}
as $t \to 0$. Moreover, $\varPsi(t) := \int_\R \psi(z - t)\,p_0(z)\,dz$ and $\int_\R p_0(z - t)\,d\psi(z)$ 
% = -\int_\R \psi(z + t)\,p_0'(z)\,dz
are finite for all $t \in [-t_0,t_0]$, and
\begin{equation}
\label{eq:varPsi-deriv}
\lim_{v \to 0}\,\sup_{\substack{s,t \in [-v,v] : \\ s \neq t}}\,\biggl|\frac{\varPsi(t) - \varPsi(s)}{t - s} + \int_\R p_0\,d\psi\biggr| = 0,
\end{equation}
where $-\int_\R p_0\,d\psi = -\lim_{t \to 0} \int_\R p_0(z - t)\,d\psi(z) = \varPsi'(0) \in [0,\infty)$.
\end{lemma}
\begin{proof}
Since $\psi$ is decreasing, it is continuous Lebesgue almost everywhere and $\psi(z - t)^2 \leq \psi(z - t_0)^2 \vee \psi(z + t_0)^2$ for all $z \in \R$ and $t \in [-t_0, t_0]$. Thus, by the dominated convergence theorem, $\int_\R \psi^2\,p_t =  \int_\R \psi(z - t)^2\,p_0(z)\,dz \to \int_\R \psi^2\,p_0$ as $t \to 0$. Hence by the continuity of $p_0$ and a slight generalisation of Scheff\'e's lemma~\citep[Lemma~2.29]{vdV1998asymptotic}, we obtain~\eqref{eq:psi-shift-limit}.

By the Cauchy--Schwarz inequality, $|\varPsi(t)| \leq (\int_\R \psi(z - t)^2\,p_0(z)\,dz)^{1/2} < \infty$ for all $t \in [-t_0,t_0]$, and
\begin{align*}
\bigl\{\varPsi(t) - \varPsi(s) - \bigl(\varPsi(t - s) - \varPsi(0)\bigr)\bigr\}^2 &= \biggl\{\int_\R \bigl(\psi(z - s) - \psi(z)\bigr)\bigl(p_0(z + t - s) - p_0(z)\bigr)\,dz\biggr\}^2 \\
&\leq 2\!\int_\R \bigl(\psi(z - s) - \psi(z)\bigr)^2(p_{t - s} + p_0)(z)\,dz \!\int_\R \bigl(\sqrt{p_{t - s}} - \sqrt{p_0}\bigr)^2
\end{align*}
for $s,t \in [-t_0,t_0]$. 
Moreover, $\int_\R \bigl(\psi(z - s) - \psi(z)\bigr)^2 p_{t - s}(z)\,dz = \int_\R \bigl(\psi(z - t) - \psi(z + s - t)\bigr)^2 p_0(z)\,dz$, so by the second limit in~\eqref{eq:psi-shift-limit} and the final assertion of Lemma~\ref{lem:location-DQM},
\begin{equation}
\label{eq:varPsi-shift}
\lim_{v \to 0}\,\sup_{\substack{s,t \in [-v,v] : \\ s \neq t}}\,\biggl|\frac{\varPsi(t) - \varPsi(s) - \bigl(\varPsi(t - s) - \varPsi(0)\bigr)}{t - s}\biggr| = 0.
\end{equation}
We have $\int_\R |p_0'| \leq i(p_0)^{1/2} < \infty$, so $p_0$ is of bounded variation on $\R$. Moreover, since $p_0$ is a density, we have $\lim_{|z| \to \infty} p_0(z) = 0$, so writing $\psi_0 = p_0'/p_0$, it follows from~\eqref{eq:fubini-parts} and Cauchy--Schwarz that
\begin{align*}
0 \leq -\int_\R p_0(z + t)\,d\psi(z) &= \int_\R \psi(z - t) p_0'(z)\,dz = \int_\R \psi(z - t)\,\psi_0(z)\,p_0(z)\,dz \\
&\leq \biggl(\int_\R \psi(z - t)^2\,p_0(z)\,dz\biggr)^{1/2} i(p_0)^{1/2} < \infty
\end{align*}
for $t \in \{-t_0,t_0\}$. Therefore, by Lemma~\ref{lem:Psi-deriv}, $\int_\R p_0(z + t)\,d\psi(z)$ is finite for all $t \in [-t_0,t_0]$ and $\varPsi'(0) = -\int_\R p_0\,d\psi = -\lim_{t \to 0} \int_\R p_0(z - t)\,d\psi(z) \in [0,\infty)$, so
\[
\lim_{v \to 0}\,\sup_{\substack{s,t \in [-v,v] : \\ s \neq t}}\,\biggl|\frac{\varPsi(t - s) - \varPsi(0)}{t - s} + \int_\R p_0\,d\psi\biggr| = 0,
\]
which together with~\eqref{eq:varPsi-shift} yields~\eqref{eq:varPsi-deriv}.
\end{proof}

\begin{lemma}
\label{lem:L1-kde}
Let $p_0 \colon \R \to [0,\infty)$ be a Lebesgue density. For $h > 0$, define $p_{0,h} \colon \R \to [0,\infty]$ by $p_{0,h}(z) := (2h)^{-1}\int_{z - h}^{z + h} p_0$. Given $\varepsilon_1,\dotsc,\varepsilon_n \iid p_0$ and a kernel $K \colon \R \to \R$ that is supported on $[-1,1]$, define $\hat{p}_{n,h} \colon \R \to [0,\infty)$ by $\hat{p}_{n,h}(z) := n^{-1}\sum_{i=1}^n K_h(z - \varepsilon_i)$, where $K_h(\cdot) = h^{-1}K(\cdot/h)$. Then for $\rho \in [0,1/2]$ and $\delta > \rho/(1 - \rho)$, we have
\begin{align*}
\E\int_\R |\hat{p}_{n,h}(z) - \E\hat{p}_{n,h}(z)|\,dz &\leq \norm{K}_\infty \int_\R\,\min\biggl\{\frac{p_{0,h}^{1/2}}{(nh)^{1/2}}, 2p_{0,h}\biggr\} \leq \frac{2^{1 - 2\rho} \norm{K}_\infty \int_\R p_{0,h}^{1 - \rho}}{(nh)^\rho} \\ 
&\leq \frac{2^{1 - 2\rho} \norm{K}_\infty C_{\delta,\rho}^\rho(1 + h)^{\delta(1 - \rho)}\bigl(1 + \int_\R |z|^\delta\,p_0(z)\,dz\bigr)^{1 - \rho}}{(nh)^\rho},
\end{align*}
where $C_{\rho,\delta} := \int_\R (1 + |z|)^{-\delta(1 - \rho)/\rho}\,dz = 2\rho/\bigl((1 - \rho)\delta - \rho\bigr) \in (0,\infty)$.
\end{lemma}
% By applying a `local maximal operator' to $p_0$, we obtain the function $p_0^*$, which is guaranteed to be measurable and finite almost everywhere~\citep[Section~4]{goldenshluger2014adaptive}.
% Mention symmetric decreasing rearrangements?

\begin{proof}
For $z \in \R$ and $h > 0$, we have
\[
\E|\hat{p}_{n,h}(z)| = \E |K_h(z - \varepsilon_1)| = \int_\R |K(u)|\,p_0(z - uh)\,du \leq \norm{K}_\infty\,p_{0,h}(z)
\]
and similarly
\[
\Var\hat{p}_{n,h}(z) = \frac{\Var K_h(z - \varepsilon_1)}{n} \leq \frac{\E\{K_h(z - \varepsilon_1)^2\}}{n} = \int_\R \frac{K(u)^2}{nh}\,p_0(z - uh)\,du \leq \frac{\norm{K}_\infty^2\,p_{0,h}(z)}{nh},
\]
so
\[
\E|\hat{p}_{n,h}(z) - \E\hat{p}_{n,h}(z)| \leq \min\bigl\{\Var^{1/2}\hat{p}_{n,h}(z), 2\E|\hat{p}_{n,h}(z)|\bigr\} \leq \norm{K}_\infty \min\biggl\{\frac{p_{0,h}(z)^{1/2}}{(nh)^{1/2}}, 2p_{0,h}(z)\biggr\}.
\]
By Fubini's theorem and the fact that $\min(a,b) \leq a^{2\rho}b^{1 - 2\rho}$ for $a,b \geq 0$ and $\rho \in [0,1/2]$, we have
\begin{align*}
\E\int_\R |\hat{p}_{n,h}(z) - \E\hat{p}_{n,h}(z)|\,dz &\leq \norm{K}_\infty \int_\R\,\min\biggl\{\frac{p_{0,h}(z)^{1/2}}{(nh)^{1/2}}, 2p_{0,h}(z)\biggr\}\,dz \notag \\
&\leq \norm{K}_\infty \int_\R\,\biggl(\frac{p_{0,h}(z)}{nh}\biggr)^\rho \bigl(2p_{0,h}(z)\bigr)^{1 - 2\rho}\,dz
= \frac{2^{1 - 2\rho} \norm{K}_\infty \int_\R p_{0,h}^{1 - \rho}}{(nh)^\rho}.
\end{align*}
Finally, if $U \sim U[-1,1]$ is independent of $\varepsilon_1$, then $hU + \varepsilon_1$ has density $p_{0,h}$, so by H\"older's inequality,
\begin{align*}
\int_\R p_{0,h}^{1 - \rho} &\leq \biggl(\int_\R (1 + |z|)^{-\delta(1 - \rho)/\rho}\,dz\biggr)^\rho \biggl(\int_\R (1 + |z|)^\delta\,p_{0,h}(z)\,dz\biggr)^{1 - \rho} \\
&= C_{p,\delta}^\rho\,\bigl[\E\bigl\{(1 + |hU + \varepsilon_1|)^\delta\bigr\}\bigr]^{1 - \rho} \\
&\leq C_{p,\delta}^\rho\,\bigl[\E\bigl\{(1 + h|U|)^\delta \cdot (1 + |\varepsilon_1|)^\delta\bigr\}\bigr]^{1 - \rho} \leq C_{p,\delta}^\rho(1 + h)^{\delta(1 - \rho)} \biggl(1 + \int_\R |z|^\delta\,p_0(z)\,dz\biggr)^{1 - \rho}.
\end{align*}
This completes the proof.
\end{proof}

As an application of  Lemma~\ref{lem:L1-kde}, when $p_0$ is a Cauchy density and $h > 8/n$, taking $\rho = 1/2 - 1/\log(nh)$ and $\delta = 2\rho = 1 - 2/\log(nh)$ in Lemma~\ref{lem:L1-kde} yields a bound on $\E\int_\R |\hat{p}_{n,h}(z) - \E\hat{p}_{n,h}(z)|\,dz$ of order $\log(nh)/\sqrt{nh}$, which is tight up to a universal constant.
% \[
% \int_\R p_{0,h}^\rho \leq \int_\R p_0\bigl(\max(|z| - h),0\bigr) \, dz = \frac{2h}{\pi} + \frac{2}{\pi^\rho}\int_h^\infty \frac{1}{\{1+(z-h)^2\}^\rho} \, dz = \frac{2h}{\pi} + \frac{\Gamma(\rho-1/2)}{\pi^{\rho-1/2}\Gamma(\rho)}.
% \]

\begin{lemma}
\label{lem:subspace-countable}
Let $W \subseteq \R^n$ be a linear subspace of dimension $d < n$. If $A \subseteq \R$ is countable, then
\[
\mathcal{A} := \Bigl\{(y_1,\dotsc,y_n) \in \R^n : \sum_{i=1}^n \Ind_{\{y_i - w_i \in A\}} \geq d + 1 \text{ for some } (w_1,\dotsc,w_n) \in W\Bigr\}
\]
has Lebesgue measure 0.
% This result is not true for general null sets $A$ (e.g.~the Cantor set).
\end{lemma}

\begin{proof}
For $I \subseteq [n]$ with $|I| = d + 1$, denote by $E_I$ the linear span of $(\mathsf{e}_i : i \in I^c)$, which has dimension $n - d - 1$. Then $E_I + W = \{z + w : z \in E_I,\,w \in W\} \subseteq \R^n$ is a linear subspace of dimension at most $(n - d - 1) + d = n - 1$, so it is a null set (i.e.~has Lebesgue measure 0). Moreover, $A_I := \{\sum_{i \in I}a_i\mathsf{e}_i : a_i \in A \text{ for all }i \in I\}$ is countable. Therefore,
\[
\mathcal{A} = \bigcup_{\substack{I \subseteq [n] \\ |I| = d+1}}\bigcup_{a \in A_I} (a + E_I + W)
\]
is a countable union of null sets, and hence is also a null set.
\end{proof}

\begin{lemma}[\citealp{chernozhukov2018double},~Lemma~6.1]
\label{lem:cond-cvg}
Suppose that $(X_n)$ is a sequence of random vectors and $(\mathcal{G}_n)$ is a sequence of $\sigma$-algebras. If $\E(\norm{X_n}\,|\,\mathcal{G}_n) = o_p(1)$, then $\norm{X_n} = o_p(1)$ as $n \to \infty$. Similarly, if $\E(\norm{X_n}\,|\,\mathcal{G}_n) = O_p(1)$, then $\norm{X_n} = O_p(1)$.
\end{lemma}

\begin{proof}
Fix $\varepsilon > 0$. If $\E(\norm{X_n}\,|\,\mathcal{G}_n) = o_p(1)$, then $\Pr(\norm{X_n} > \varepsilon\,|\,\mathcal{G}_n) \leq \varepsilon^{-1}\E(\norm{X_n}\,|\,\mathcal{G}_n) \cvp 0$ by Markov's inequality, so the bounded convergence theorem implies that $\Pr(\norm{X_n} > \varepsilon) = \E\,\Pr(\norm{X_n} > \varepsilon\,|\,\mathcal{G}_n) \to 0$ as $n \to \infty$. Since $\varepsilon > 0$ was arbitrary, we conclude that $\norm{X_n} = o_p(1)$. 

On the other hand, suppose that $\E(\norm{X_n}\,|\,\mathcal{G}_n) = O_p(1)$. Then for any sequence $M_n \to \infty$, we have $\Pr(\norm{X_n} > M_n\,|\,\mathcal{G}_n) \leq M_n^{-1}\E(\norm{X_n}\,|\,\mathcal{G}_n) = o_p(1)$ by Markov's inequality, so similarly $\Pr(\norm{X_n} > M_n) = \E\,\Pr(\norm{X_n} > M_n\,|\,\mathcal{G}_n) \to 0$ as $n \to \infty$. Thus, $\norm{X_n} = O_p(1)$. 
\end{proof}

\begin{lemma}
\label{lem:cvp-plugin}
Let $(\Delta_n)$ be a sequence of random measurable functions\footnote{More precisely, writing $\Omega$ for the underlying probability space, suppose that there exist jointly measurable functions $\tilde{\Delta}_n \colon \R \times \Omega \to \R$ such that $\Delta_n(v)(\cdot) = \tilde{\Delta}_n(v,\cdot) \colon \Omega \to \R$ for every $n \in \N$ and $v \in \R$.} from $\R$ to $\R$. Assume that for some deterministic sequence $(w_n)$ with $w_n \to 0$, we have $\Delta_n(v_n) \cvp 0$ as $n \to \infty$ whenever $(v_n)$ is a deterministic sequence such that $v_n = o(w_n)$. Then for any sequence of random variables $(V_n)$ independent of~$(\Delta_n)$, we have $\Delta_n(V_n) \cvp 0$ whenever $V_n = o_p(w_n)$.
\end{lemma}

\begin{proof}
Fix $\epsilon > 0$ and let $g_n(v) := \Pr(|\Delta_n(v)| > \epsilon)$ for $v \in \R$ and $n \in \N$. By assumption, $g_n(v_n) \to 0$ whenever $v_n = o(w_n)$. Since $(V_n)$ and $(\Delta_n)$ are independent, we have $\Pr\bigl(|\Delta_n(V_n)| > \epsilon\,|\,V_n\bigr) = g_n(V_n) \cvp 0$ whenever $V_n = o_p(w_n)$. 
% where $\Delta_n(V_n) \colon \Omega \to \R$ is indeed a measurable function (i.e.~random variable) because $V_n$ and $\tilde{\Delta}_n$ are measurable (with the latter being jointly measurable)
Therefore, by the bounded convergence theorem, $\Pr(|\Delta_n(V_n)| > \epsilon) = \E\,\Pr\bigl(|\Delta_n(V_n)| > \epsilon\,|\,V_n\bigr) \to 0$, as required.
\end{proof}

\subsection{Antitonic projections and least concave majorants}
\label{sec:isoproj}

Denote by $\Psi_\downarrow$ the set of all decreasing functions on $\R$. For an integrable function $f \colon (0,1) \to \R$, recall from Section~\ref{sec:notation} the definitions of $\widehat{\mathcal{M}}_\mathrm{L}f$ and $\widehat{\mathcal{M}}_\mathrm{R}f$ on $[0,1]$.

\begin{proposition}[\protect{\citealp[Exercise~9.24]{samworth24modern}}]
\label{prop:isoproj}
If $\psi \in L^2(P)$ for some Borel probability measure $P$ on $\R$ with distribution function $F$, then
\begin{equation}
\label{eq:isoproj}
\psi_P^* := \widehat{\mathcal{M}}_\mathrm{L}(\psi \circ F^{-1}) \circ F \in \argmin_{g \in \Psi_\downarrow}\,\int_I (g - \psi)^2\,dP =: \Pi_\downarrow(\psi, P).
\end{equation}
We have $\psi_P^* \in L^2(P)$, and moreover $\psi^\star \in \Pi_\downarrow(\psi, P)$ if and only if $\psi^\star = \psi_P^*$ $P$-almost everywhere. Furthermore, if $F$ is continuous, then also $\psi_P^{**} := \widehat{\mathcal{M}}_\mathrm{R}(\psi \circ F^{-1}) \circ F \in \Pi_\downarrow(\psi, P)$.
\end{proposition}

We refer to $\Pi_\downarrow(\psi, P)$ as the \emph{$L^2(P)$ antitonic \emph{(decreasing isotonic)} projection} of $\psi$. Since $\Psi_\downarrow$ is a convex class of functions,
\begin{equation}
\label{eq:iso-ipr}
\int_I (g - \psi_P^*)(\psi - \psi_P^*)\,dP \leq 0 \;\;\text{for all }g \in \Psi_\downarrow \text{ and }\psi_P^* \in \Pi_\downarrow(\psi, P);
\end{equation}
see e.g.~\citet[Theorem~9.36]{samworth24modern}. This can be used to derive the basic inequality~\eqref{eq:L2-proj-ineq} below. For a measurable function $g$ on $\R$, let $\norm{g}_{L^\infty(P)} := \sup\bigl\{\lambda \geq 0 : P(\{z \in \R : |g(z)| \geq \lambda\}) > 0\bigr\}$.

\begin{lemma}
\label{lem:L2-proj-ineq}
Given Borel probability measures $P,Q$ on $\R$ and $\psi \in L^2(P) \cap L^2(Q)$, let $\psi_P^* \in \Pi_\downarrow(\psi, P)$ and $\psi_Q^* \in \Pi_\downarrow(\psi, Q)$. Then
\begin{equation}
\label{eq:L2-proj-ineq}
\norm{\psi_P^* - \psi_Q^*}_{L^2(P)}^2 \leq \int_\R (\psi - \psi_Q^*)(\psi_Q^* - \psi_P^*)\,d(Q - P).
\end{equation}
Moreover, if $\psi_1,\psi_2 \in L^2(P)$ and $\psi_\ell^\star \in \Pi_\downarrow(\psi_\ell,P)$ for $\ell \in \{1,2\}$, then
\begin{align}
\label{eq:L2-proj-contraction}
\norm{\psi_1^\star - \psi_2^\star}_{L^2(P)} &\leq \norm{\psi_1 - \psi_2}_{L^2(P)},
\end{align}
% Easier version: $\norm{\psi^*}_{L^\infty(P)} \leq \norm{\psi}_{L^\infty(P)}$ for all measurable $\psi \colon \R \to \R$
% Indeed, if $M := \norm{\psi}_{L^\infty(P)} < \infty$, then defining $\eta_M(x) := (x \wedge M) \vee (-M)$ for $x \in \R$, we have $\eta_M \circ \psi^* \in \Psi_\downarrow$ and $\norm{\eta_M \circ \psi^* - \psi}_{L^2(P)} = \norm{\eta_M \circ \psi^* - \eta_M \circ \psi}_{L^2(P)} \leq \norm{\psi^* - \psi}_{L^2(P)}$. Thus, by the definition of $\psi^*$, we have $\psi^* = \eta_M \circ \psi$ $P$-almost everywhere, so $\norm{\psi}_{L^\infty(P)} \leq M = \norm{\psi}_{L^\infty(P)}$.
and writing $\psi_\ell^* \equiv (\psi_\ell)_P^*$ and $\psi_\ell^{**} \equiv (\psi_\ell)_P^{**}$ for $\ell \in \{1,2\}$, we have
\begin{align}
\label{eq:Linfty-proj-contraction}
\norm{\psi_1^* - \psi_2^*}_\infty = \norm{\psi_1^{**} - \psi_2^{**}}_\infty \leq \norm{\psi_1 - \psi_2}_{L^\infty(P)}.
\end{align}
\end{lemma}

\begin{proof}
By~\eqref{eq:iso-ipr},
\[
\int_\R (\psi_Q^* - \psi_P^*)(\psi - \psi_P^*)\,dP \leq 0 \leq \int_\R (\psi_P^* - \psi_Q^*)(\psi_Q^* - \psi)\,dQ,
\]
and adding $\int_\R \,(\psi_Q^* - \psi_P^*)(\psi_Q^* - \psi)\,dP$ to both sides yields~\eqref{eq:L2-proj-ineq}. 

For the second assertion, define
\begin{align*}
D(t) &:= \bigl\|(1 - t)\psi_1^\star + t\psi_1 - \{(1 - t)\psi_2^\star + t\psi_2\}\bigr\|_{L^2(P)}^2 = \bigl\|(\psi_1^\star - \psi_2^\star) + t(\psi_1 - \psi_1^\star + \psi_2^\star - \psi_2)\bigr\|_{L^2(P)}^2 \\
&\phantom{:}= \norm{\psi_1^\star - \psi_2^\star}_{L^2(P)}^2 + 2t\int_\R (\psi_1^\star - \psi_2^\star)(\psi_1 - \psi_1^\star + \psi_2^\star - \psi_2)\,dP + t^2\norm{\psi_1 - \psi_1^\star + \psi_2^\star - \psi_2}_{L^2(P)}^2,
\end{align*}
which is a quadratic function of $t \in \R$. By~\eqref{eq:iso-ipr}, $\int_\R (\psi_1^\star - \psi_2^\star)(\psi_1 - \psi_1^\star + \psi_2^\star - \psi_2)\,dP \geq 0$, so $D$ is non-decreasing on $[0,\infty)$. Thus, $\norm{\psi_1^\star - \psi_2^\star}_{L^2(P)}^2 = D(0) \leq D(1) = \norm{\psi_1 - \psi_2}_{L^2(P)}^2$, which proves~\eqref{eq:L2-proj-contraction}. 

For $\varphi_1,\varphi_2 \in L^2(P)$, it follows from the min-max formulae for the isotonic projection~\citep[Exercise~9.25\emph{(b)} and Theorem~9.37]{samworth24modern} that,
\begin{equation}
\label{eq:isoproj-ptwise}
\varphi_1 \leq \varphi_2 \;\;P\text{-almost everywhere} \quad \Rightarrow \quad \varphi_1^* \leq \varphi_2^* \;\;P\text{-almost everywhere}.
\end{equation}
Now defining $d := \norm{\psi_1 - \psi_2}_{L^\infty(P)}$, we have $\psi_1 - d \leq \psi_2 \leq \psi_1 + d$ $P$-almost everywhere on $\R$, so~\eqref{eq:isoproj-ptwise} implies that $\psi_1^* - d = (\psi_1 - d)^* \leq \psi_2^* \leq (\psi_1 + d)^* = \psi_1^* + d$ $P$-almost everywhere. In other words, $|\psi_1^* - \psi_2^*| \leq d$ on a set of the form $F^{-1}(E) := \{F^{-1}(u) : u \in E\}$ for some $E \subseteq (0,1)$ with Lebesgue measure 1.
% In other words, $\norm{\psi_1^* - \psi_2^*}_\infty \leq d$, so~\eqref{eq:Linfty-proj-contraction} holds.

It remains to show that $|\psi_1^* - \psi_2^*| \vee |\psi_1^{**} - \psi_2^{**}| \leq d$ on $\R$. For $\ell \in \{1,2\}$, define $\Psi_\ell \colon [0,1] \to \R$ by $\Psi_\ell(u) := \int_0^u \psi_\ell \circ F^{-1}$. Fix $z \in \R$ and suppose first that $u := F(z) \in (0,1]$. Since $E$ is dense in $[0,1]$, we can find a sequence $(u_n)$ in $E$ such that $u_n \nearrow u$ and hence $u_n \leq (F \circ F^{-1})(u_n) \nearrow (F \circ F^{-1})(u) = u$. Therefore, since $\hat{\Psi}_1^{(\mathrm{L})} - \hat{\Psi}_2^{(\mathrm{L})}$ is left-continuous on $(0,1]$ by~\citet[Theorem~24.1]{rockafellar97convex}, we have
\[
|\psi_1^*(z) - \psi_2^*(z)| = \bigl|\hat{\Psi}_1^{(\mathrm{L})}(u) - \hat{\Psi}_2^{(\mathrm{L})}(u)\bigr| = \lim_{n \to \infty} \bigl|\psi_1^*\bigl(F^{-1}(u_n)\bigr) - \psi_2^*\bigl(F^{-1}(u_n)\bigr)\bigr| \leq d.
\]
Next, when $u \in [0,1)$, we can find a sequence $(u_n)$ in $E$ such that $u_n \searrow u$. Let $u' := \lim_{n \to \infty}F\bigl(F^{-1}(u_n)\bigr)$. 
% Since $F$ is right-continuous, $u' \in \Im F$, so $(F \circ F_{-1})(u') = u'$.
If $u' = u$, then by~\citet[Theorem~24.1]{rockafellar97convex},
\begin{align}
|\psi_1^{**}(z) - \psi_2^{**}(z)| = \bigl|\hat{\Psi}_1^{(\mathrm{R})}(u) - \hat{\Psi}_2^{(\mathrm{R})}(u)\bigr| &= \lim_{n \to \infty}\bigl|\hat{\Psi}_1^{(\mathrm{L})}\bigl((F \circ F^{-1})(u_n)\bigr) - \hat{\Psi}_2^{(\mathrm{L})}\bigl((F \circ F^{-1})(u_n)\bigr)\bigr| \notag \\
\label{eq:F-quantile}
&= \lim_{n \to \infty} \bigl|\psi_1^*\bigl(F^{-1}(u_n)\bigr) - \psi_2^*\bigl(F^{-1}(u_n)\bigr)\bigr| \leq d.
\end{align}
Otherwise, if $u' > u$, then $(F \circ F^{-1})(u') = u'$ and $F^{-1}(v) = F^{-1}(u')$ for all $v \in (u,u']$. Thus, for $\ell \in \{1,2\}$, it follows that $\psi_\ell \circ F^{-1}$ is constant and hence $\Psi_\ell,\hat{\Psi}_\ell$ are both affine on $[u,u']$ with $\hat{\Psi}_\ell^{(\mathrm{R})}(u) = \hat{\Psi}_\ell^{(\mathrm{L})}(u')$, so~\eqref{eq:F-quantile} remains valid. Since $|\psi_1^*(z) - \psi_2^*(z)| = |\psi_1^{**}(z) - \psi_2^{**}(z)|$ when $u = F(z) \in \{0,1\}$, the proof is complete.
\end{proof}

% Common fact underpinning the previous paragraph as well as Lemmas~\ref{lem:lcm-affine} and~\ref{lem:lcm-concave-cvx}: if $0 \leq u < v \leq 1$ and there is an affine function $L$ such that $F \leq L \leq \hat{F}$ on $[u,v]$, then $\hat{F}$ is affine on $[u,v]$.
\begin{lemma}
\label{lem:lcm-affine}
For a function $F \colon [0,1] \to \R$, suppose that its least concave majorant $\hat{F}$ on $[0,1]$ satisfies $F(v) < \hat{F}(v) < \infty$ for some $v \in (0,1)$ at which $F$ is continuous. Then there exists $\delta \in \bigl(0,\min(v, 1 - v)\bigr)$ such that $\hat{F}$ is affine on $[v - \delta, v + \delta]$.
\end{lemma}

\begin{proof}
Since $\hat{F}$ is concave on $[0,1]$ with $\hat{F} \geq F > -\infty$ and $\hat{F}(v) \in \R$, \citet[Theorem~10.1]{rockafellar97convex} ensures that $\hat{F}$ is continuous at $v$. Therefore, because $F$ is also continuous at $v$, there exists $\delta \in \bigl(0,\min(v, 1 - v)\bigr)$ such that $\inf_{u \in [v - \delta, v + \delta]} \hat{F}(u) > \sup_{u \in [v - \delta, v + \delta]} F(u)$. Define $\ell \colon [0,1] \rightarrow \R$ to be the affine function that agrees with $\hat{F}$ at $v - \delta$ and $v + \delta$.  Then $\ell > F$ on $[v - \delta, v + \delta]$ and moreover $\ell \geq \hat{F} \geq F$ on $[0,1] \setminus [v - \delta, v + \delta]$, so $\ell \geq F$ on $[0,1]$. It follows from the definition of the least concave majorant that $\hat{F} \wedge \ell = \hat{F}$ and hence that $\hat{F} = \ell$ on $[v - \delta, v + \delta]$, as required.
\end{proof}

\begin{lemma}
\label{lem:lcm-deriv-0}
For a continuous function $F \colon [0,1] \to \R$, we have $\hat{F}(u) = F(u)$ for $u \in \{0,1\}$ and
\[
\sup_{u \in (0,1)}\frac{F(u) - F(0)}{u} = \hat{F}^{(\mathrm{R})}(0), \qquad \inf_{u \in (0,1)}\frac{F(1) - F(u)}{1 - u} = \hat{F}^{(\mathrm{L})}(1).
\]
\end{lemma}

\begin{proof}
Let $C := \sup_{u \in (0,1)}\bigl(F(u) - F(0)\bigr)/u$, so that $F(u) \leq F(0) + Cu$ for all $u \in [0,1]$, where we adopt the convention $0 \times \infty = 0$. Since $\hat{F}$ is the least concave majorant of $F$, we deduce that $\hat{F}(u) \leq F(0) + Cu$ for all $u \in (0,1]$ and $\hat{F}(0) = F(0)$. Thus, $\sup_{u \in (0,1)}\bigl(\hat{F}(u) - \hat{F}(0)\bigr)/u = C$. Moreover, by the concavity of~$\hat{F}$ and~\citet[Corollary~24.2.1]{rockafellar97convex},
\[
\hat{F}^{(\mathrm{R})}(0) = \sup_{u \in (0,1)} \hat{F}^{(\mathrm{R})}(u) \leq \sup_{u \in (0,1)}\frac{\int_0^u \hat{F}^{(\mathrm{R})}}{u} = \sup_{u \in (0,1)}\frac{\hat{F}(u) - \hat{F}(0)}{u} \leq \hat{F}^{(\mathrm{R})}(0),
\]
so
\[
\sup_{u \in (0,1)} \frac{F(u) - F(0)}{u} = \sup_{u \in (0,1)} \frac{\hat{F}(u) - \hat{F}(0)}{u} = \hat{F}^{(\mathrm{R})}(0).
\]
The remaining assertions of the lemma follow similarly by considering $u \mapsto F(1 - u)$ instead.
\end{proof}

\begin{lemma}
\label{lem:lcm-deriv}
If $F \colon [0,1] \to \R$ is differentiable at $v \in (0,1)$ and its least concave majorant $\hat{F}$ is finite at~$v$, then $\hat{F}$ is differentiable at $v$. Moreover, $\hat{F}^{(\mathrm{L})}$ and $\hat{F}^{(\mathrm{R})}$ are both continuous (and coincide) at $v$.
\end{lemma}

\begin{proof}
If $F(v) < \hat{F}(v)$, then by Lemma~\ref{lem:lcm-affine}, $\hat{F}$ is affine on some open interval around $v$, so $\hat{F}$ is differentiable at $v$ in this case. On the other hand, if $F(v) = \hat{F}(v)$, then
\begin{align*}
\liminf_{u \nearrow v}\frac{F(v) - F(u)}{v - u} \geq \lim_{u \nearrow v}\frac{\hat{F}(v) - \hat{F}(u)}{v - u} &= \hat{F}^{(\mathrm{L})}(v) \\
&\geq \hat{F}^{(\mathrm{R})}(v) = \lim_{u \searrow v}\frac{\hat{F}(u) - \hat{F}(v)}{u - v} \geq \limsup_{u \searrow v}\frac{F(u) - F(v)}{u - v}.
\end{align*}
Thus, if $F$ is also differentiable at $v$, then $\hat{F}^{(\mathrm{L})}(v) = \hat{F}^{(\mathrm{R})}(v) = F'(v)$, and~\citet[Theorem~24.1]{rockafellar97convex} ensures that $\hat{F}^{(\mathrm{L})}$ and $\hat{F}^{(\mathrm{R})}$ are both continuous at $v$.
\end{proof}

\begin{lemma}
\label{lem:lcm-concave-cvx}
Given $0 \leq u \leq v \leq 1$, suppose that $F \colon [0,1] \to \R$ is convex on $[u,v]$, and concave on both $[0,u]$ and $[v,1]$. Then there exist $u' \in [0,u]$ and $v' \in [v,1]$ such that the least concave majorant $\hat{F}$ is affine on $[u',v']$ and coincides with $F$ on $[0,1] \setminus [u',v']$.
\end{lemma}

\begin{proof}
If $u = v$, then the conclusion holds with $u' = u = v = v'$, so suppose now that $u < v$ and define $\ell \colon [0,1] \to \R$ to be the affine function that agrees with $\hat{F}$ at $u$ and $v$. Then 
% $\hat{F}(w) = \ell(w)$ for $w \in \{u,v\}$, so 
$\ell \geq \hat{F} \geq F$ on $[0,1] \setminus [u,v]$ by the concavity of $\hat{F}$, and moreover since $F$ is convex on $[u,v]$, we have
\[
F(w) \leq \frac{v - w}{v - u}\,F(u) + \frac{w - u}{v - u}\,F(v) \leq \frac{v - w}{v - u}\,\hat{F}(u) + \frac{w - u}{v - u}\,\hat{F}(v) = \ell(w) \leq \hat{F}(w)
\]
for all $w \in [u,v]$. Therefore, $\ell \geq F$ on $[0,1]$, so by the definition of the least concave majorant, $\hat{F} \wedge \ell = \hat{F}$. Let $u' := \inf\{w \in [0,u] : \hat{F}(w) = \ell(w)\}$ and $v' := \sup\{w \in [v,1] : \hat{F}(w) = \ell(w)\}$. Then $\hat{F} = \ell$ on $[u',v']$ but $\hat{F}$ is not locally affine at either $u'$ or $v'$, so by Lemma~\ref{lem:lcm-affine}, $\hat{F}(w) = F(w)$ for $w \in \{u',v'\}$. Since $F \leq \ell$ is concave on both $[0,u']$ and $[v',1]$, we conclude that $\hat{F} = F$ on $[0,1] \setminus [u',v']$, as claimed.
\end{proof}

\subsection{Additional simulation results}

In Section~\ref{sec:experiments}, we analysed the empirical performance of the \textbf{ASM} estimator without constraining the convex loss function to be symmetric. Here, we present simulation studies of \textbf{ASM} where we assume the noise distribution to be symmetric and hence enforce symmetry on the estimated loss function; see Section~\ref{subsec:linreg-sym}. We generate data from the same linear model
\[
Y_i = \mu_0 + \tilde{X}_i^\top \theta_0 + \varepsilon_i \quad \text{for $i \in [n]$}
\]
as in Section~\ref{sec:experiments} and take the noise distribution to be either (i) standard Gaussian, (ii) standard Cauchy, (iii) Gaussian scale mixture, (iv) Gaussian location mixture, or (v) smoothed uniform. We omit the smoothed exponential distribution since it is not symmetric. As the intercept $\mu_0$ is now identifiable, we consider the estimation error of not just the coefficient $\theta_0$ but also the intercept $\mu_0$, that is, we measure $\|\theta_0 - \hat{\theta}\|_2^2 + (\mu_0 - \hat{\mu})^2$. 

We compare the \textbf{ASM} estimator to the same group of competitors that we use in Section~\ref{sec:experiments}. To make the comparison fair, we modify the one-step estimator (1S) and the log-concave MLE (LCMLE) to also incorporate the symmetry assumption on the noise distribution.

\begin{table}[ht]
\centering
\begin{tabular}{|l|r||r|r|r|r|r|r|}
\hline
& Oracle & ASM & Alt & LCMLE & 1S & LAD & OLS \\ 
\hline
Standard Gaussian & 9.84 & 10.15 & 10.04 & 11.00 & 10.67 & 16.04 & $\bf{9.84}$ \\ 
Standard Cauchy & 24.90 & $\bf{25.36}$ & 25.51 & 26.99 & 27.40 & 26.95 & 5471151.67 \\ 
Gaussian scale mixture & 35.32 & $\bf{35.80}$ & 35.93 & 39.75 & 42.51 & 40.49 & 82.59 \\ 
Gaussian location mixture & 0.20 & 0.25 & $\bf{0.20}$ & 1.60 & 19.53 & 886.44 & 21.44 \\ 
Smoothed uniform & 1.16 & 1.49 & $\bf{1.33}$ & 1.50 & 2.40 & 9.70 & 3.42 \\ 
\hline
\end{tabular}
\caption{Squared estimation error ($\times 10^3$) for different estimators, with $n = 600$ and $d=6$.}
\label{tab:MSE-comparison-symm}
\end{table}

\begin{figure}[ht]
\centering
\includegraphics[width=0.47\textwidth]{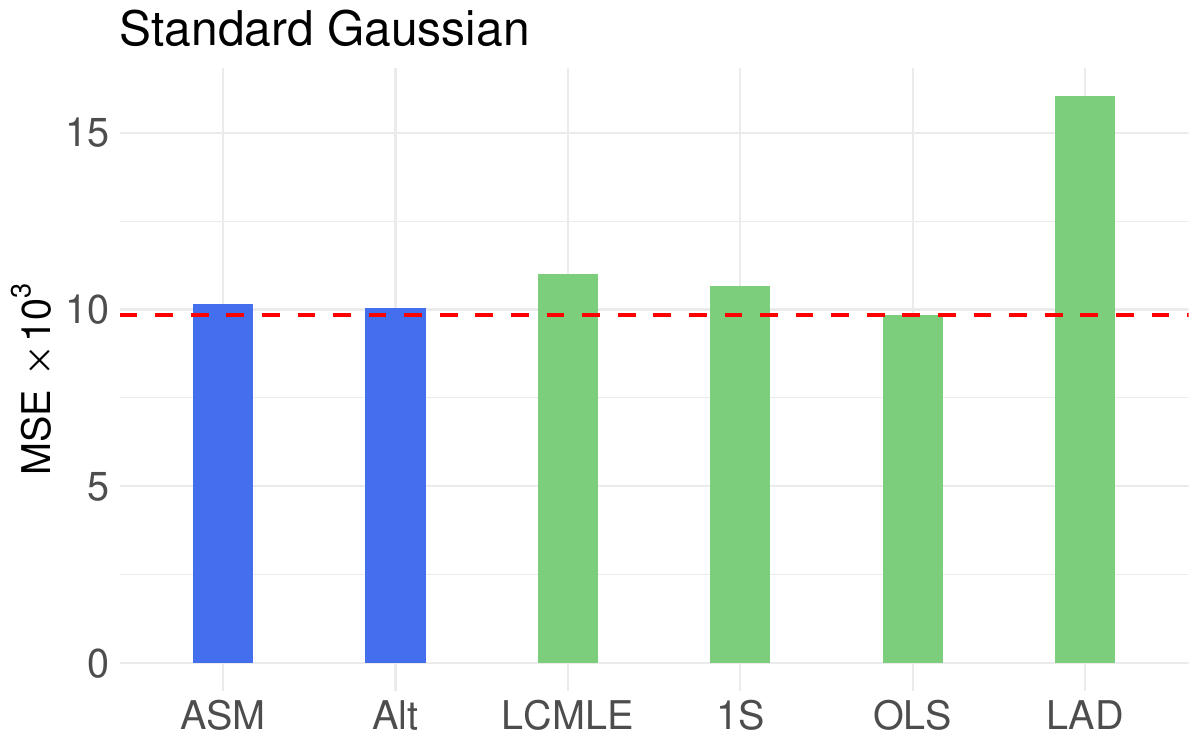}
\hfill
\includegraphics[width=0.47\textwidth]{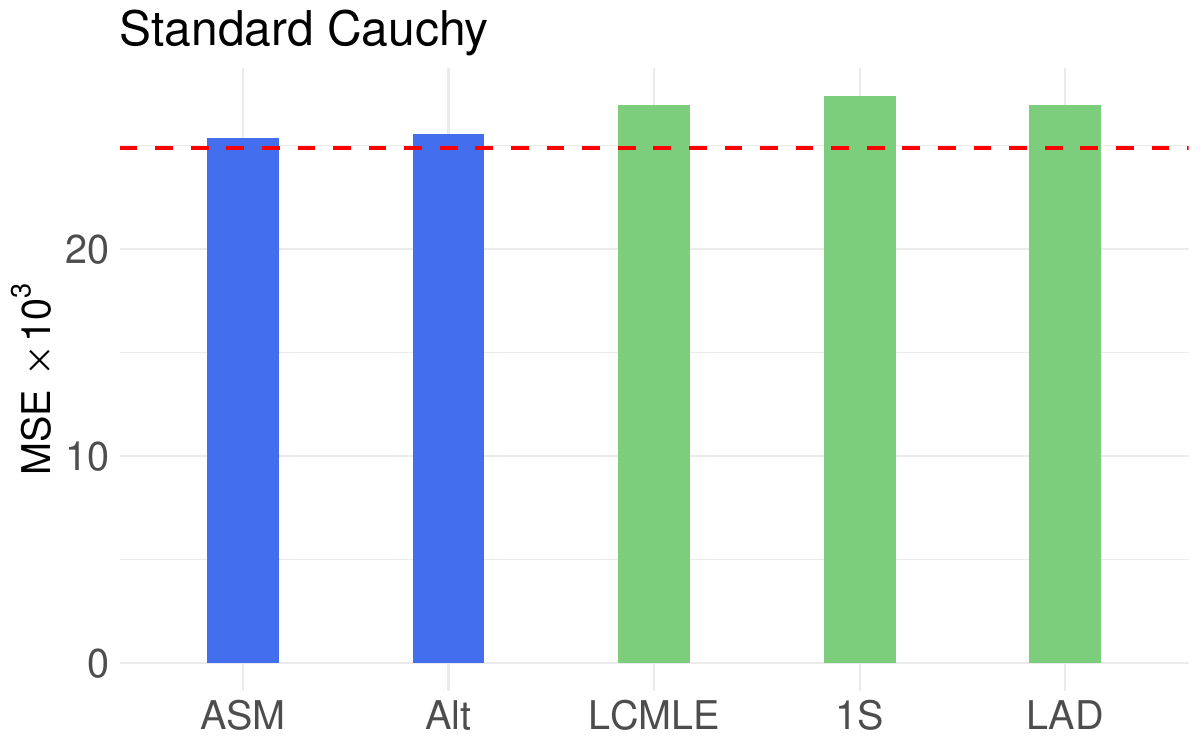}
\includegraphics[width=0.47\textwidth]{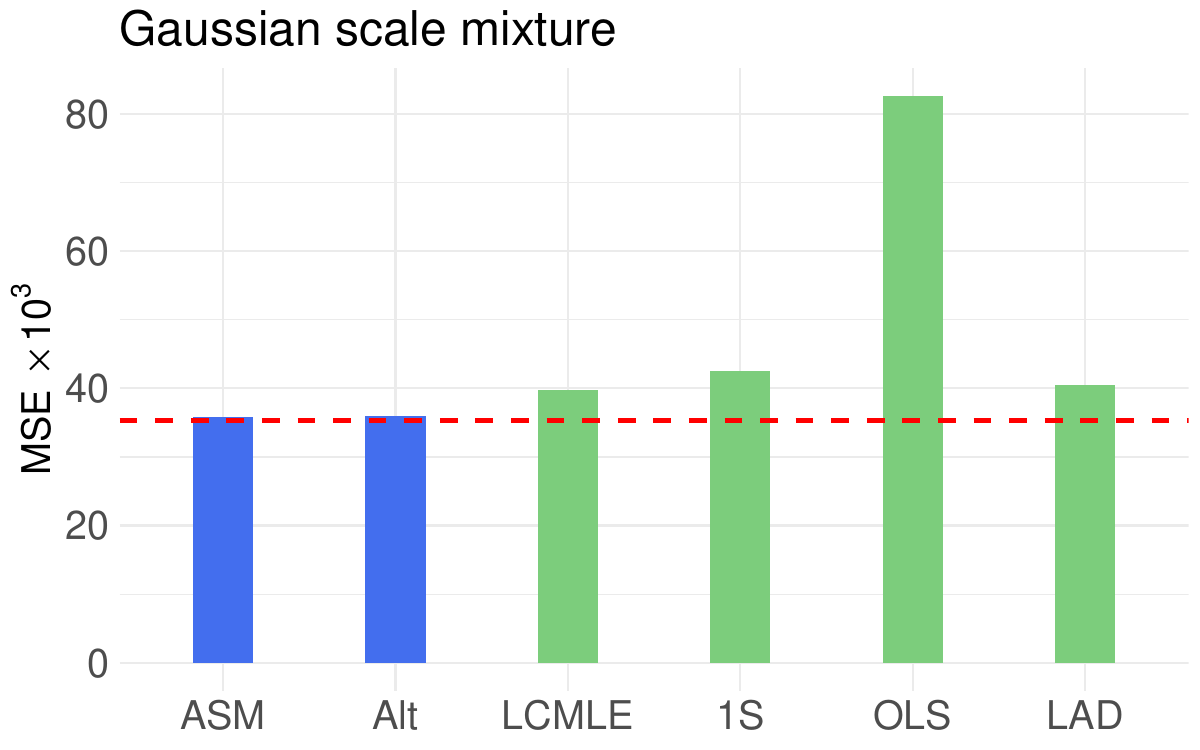}
\hfill
\includegraphics[width=0.47\textwidth]{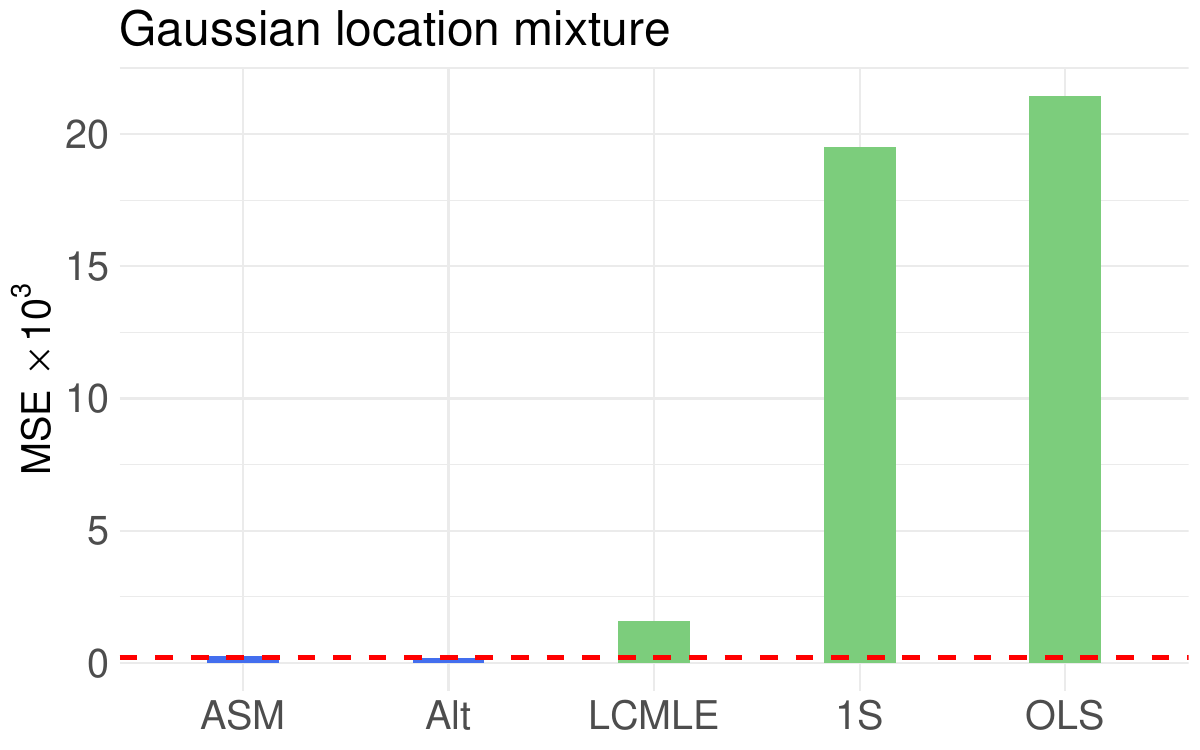}
\includegraphics[width=0.47\textwidth]{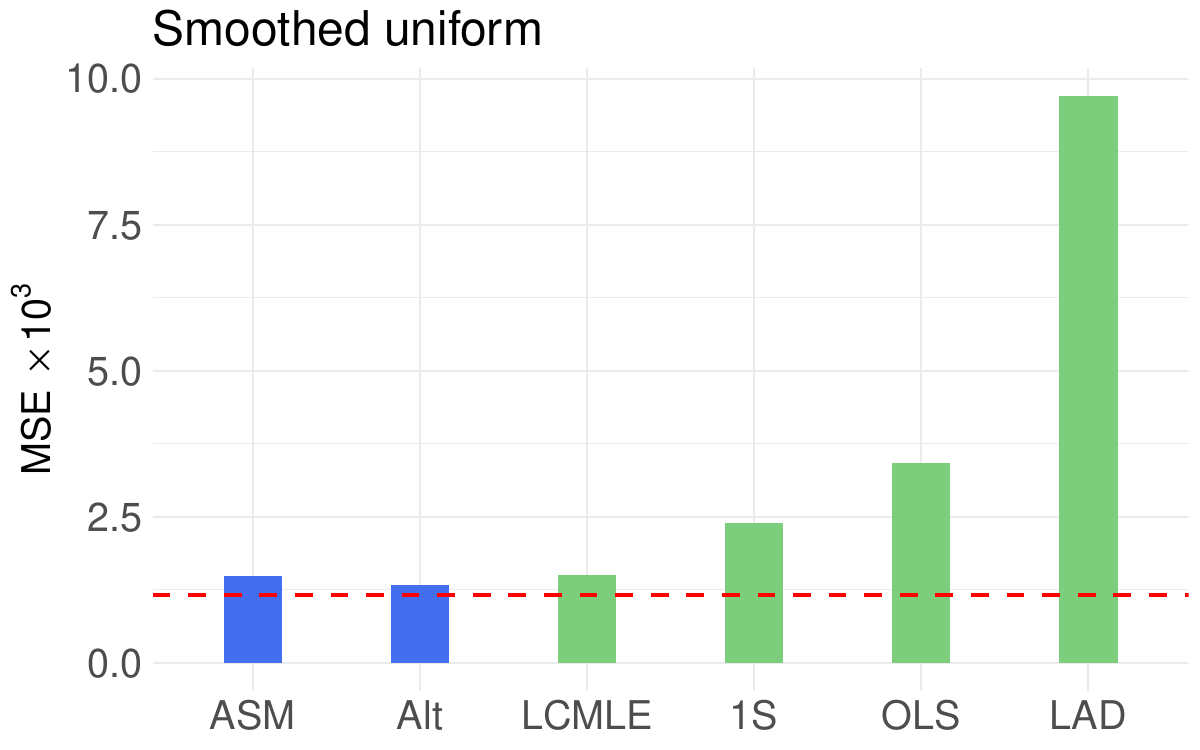}
\hfill 
\caption{Plots of the average squared error loss ($\times 10^3$) of different estimators for noise distributions (i)--(v), with $n = 600$ and $d = 6$. In each plot, the red dashed line indicates the corresponding value for the oracle convex $M$-estimator, and we omit the estimators that have very large estimation error (see Table~\ref{tab:MSE-comparison-symm} for full details).}
\label{fig:MSE-comparison-symm}
\end{figure}
\end{document}